\documentclass[10pt,a4paper]{article}
\usepackage[latin1]{inputenc}

\usepackage[protrusion=false]{microtype} 

\usepackage{titletoc}
\newcommand\DoToC{%
  \startcontents
  \printcontents{}{1}{\textbf{Contents}\vskip3pt\hrule\vskip5pt}
  \vskip3pt\hrule\vskip5pt
}

\usepackage{setspace}
\usepackage{amsmath}
\usepackage{amsfonts}
\usepackage{amssymb}
\usepackage{amsthm}
\makeatletter
\toks@\expandafter{\@endtheorem\@endpetrue}
\edef\@endtheorem{\the\toks@}
\makeatother
\usepackage{bm}
\usepackage{graphicx}
\usepackage{enumitem}
\usepackage{accents} 
\graphicspath{ {./figures/} }

\usepackage[left=1in, right=1in, top=1in, bottom=1in]{geometry}
\usepackage[title]{appendix}
\usepackage{authblk}
\usepackage{multirow}
\usepackage[dvipsnames]{xcolor}
\definecolor{Bleu}{RGB}{0,0,204}
\usepackage[hypertexnames=false]{hyperref}
\hypersetup{
colorlinks,
  citecolor=Bleu,
  linkcolor=Bleu,
  urlcolor=Bleu}

\usepackage{algorithm}
\usepackage{algorithmic}

\usepackage{multirow}

\usepackage{caption}
\usepackage{subcaption}

\newcommand{\iidsim}{\raisebox{-2pt}{$\,\overset{\textnormal{\scriptsize iid}}{\sim}\,$}}

\usepackage[mathscr]{euscript}
\usepackage{bbm}

\usepackage{thmtools} 
\renewcommand\thmcontinues[1]{continued}

\newtheorem{theorem}{Theorem}
\newtheorem{lemma}{Lemma}

\newtheorem{corollary}{Corollary}

\theoremstyle{definition}

\newtheorem{conditionC}{Condition}

\newtheorem{example}{Example}
\usepackage{chngcntr}
\newcounter{parentnumber}

\makeatletter 
\newenvironment{subexample}[1]{%
  \counterwithin*{example}{parentnumber}
  \def\subexamplecounter{#1}%
  \refstepcounter{#1}%
  \protected@edef\theparentnumber{\csname the#1\endcsname}%
  \setcounter{parentnumber}{\value{#1}}%
  \setcounter{#1}{0}%
  \expandafter\def\csname the#1\endcsname{\theparentnumber\alph{#1}}%
  \ignorespaces
}{%
  \setcounter{\subexamplecounter}{\value{parentnumber}}%
  \counterwithout*{example}{parentnumber} 
  \ignorespacesafterend
}
\makeatother
\usepackage{cleveref}
\crefname{example}{Example}{Examples} 

\DeclareMathOperator*{\argmin}{argmin}

\DeclareMathOperator*{\esssup}{ess\,sup}
\DeclareMathOperator*{\essinf}{ess\,inf}

\newcommand{\indep}{\perp \!\!\! \perp}

\usepackage[width=\textwidth]{caption}

\usepackage{natbib}
\title{One-Step Estimation of Differentiable Hilbert-Valued Parameters}
\author{Alex Luedtke}
\author{Incheoul Chung}
\affil{Department of Statistics, University of Washington}

\date{\today}

\begin{document}
\allowdisplaybreaks
\maketitle

\begin{abstract}
We present estimators for smooth Hilbert-valued parameters, where smoothness is characterized by a pathwise differentiability condition. When the parameter space is a reproducing kernel Hilbert space, we provide a means to obtain efficient, root-$n$ rate estimators and corresponding confidence sets. These estimators correspond to generalizations of cross-fitted one-step estimators based on Hilbert-valued efficient influence functions. We give theoretical guarantees even when arbitrary estimators of nuisance functions are used, including those based on machine learning techniques. We show that these results naturally extend to Hilbert spaces that lack a reproducing kernel, as long as the parameter has an efficient influence function. However, we also uncover the unfortunate fact that, when there is no reproducing kernel, many interesting parameters fail to have an efficient influence function, even though they are pathwise differentiable. To handle these cases, we propose a regularized one-step estimator and associated confidence sets. We also show that pathwise differentiability, which is a central requirement of our approach, holds in many cases. Specifically, we provide multiple examples of pathwise differentiable parameters and develop corresponding estimators and confidence sets. Among these examples, four are particularly relevant to ongoing research by the causal inference community: the counterfactual density function, dose-response function, conditional average treatment effect function, and counterfactual kernel mean embedding.
\end{abstract}

\section{Introduction}

There has been much recent work on combining tools from semiparametric efficiency and machine learning to estimate finite-dimensional parameters \citep{baiardi2021value,kennedy2022semiparametric,hines2022demystifying}. These works often focus on pathwise differentiable parameters, which are characterized by their smoothness along regular univariate submodels of the statistical model \citep{pfanzagl1990estimation,van1991differentiable,bickel1993efficient}. When a finite-dimensional parameter is pathwise differentiable, it also has an efficient influence function (EIF), which corresponds to the Riesz representation of its pathwise derivative. Efficient influence functions are the critical ingredient used to define various estimation strategies, such as those based on one-step estimation \citep{pfanzagl1982lecture,newey1994large}, estimating equations \citep{van2003unified,tsiatis2006semiparametric}, targeted learning \citep{van2006targeted,van2011targeted}, and double machine learning \citep{chernozhukov2017double,chernozhukov2018double}. When paired with cross-fitting \citep{schick1986asymptotically,klaassen1987consistent}, these frameworks yield asymptotically efficient estimators provided the nuisance functions are estimated well enough as the sample size $n$ grows to make a certain remainder term negligible. Often, this amounts to requiring an $n^{-1/4}$-rate condition, which will most plausibly hold if the nuisance functions are estimated flexibly.

Another line of research has focused on leveraging machine learning tools to estimate function-valued parameters, such as the causal dose-response function \citep{diaz2013targeted}, counterfactual density function \citep{kennedy2021semiparametric}, and conditional average treatment effect function \citep{nie2021quasi}. Possibly owing to the wealth of available methods for estimating real-valued functionals, many of these works have focused on the evaluation of these functions at a point. As has been noted in \cite{van2018cv} and \cite{chernozhukov2018biased}, the resulting point evaluations tend not to be pathwise differentiable except in trivial cases (e.g., when the data are discrete). To overcome this challenge, kernel-smoothed approximations of the function evaluation parameter have been considered in those two works and others \citep{colangelo2020double,luedtke2020efficient,chernozhukov2021simple,jung2021double}, and local polynomial approximations have also been introduced for several parameters \citep{kennedy2017non,takatsu2022debiased,kennedy2022minimax}. These smoothed approximations tend to yield pathwise differentiable parameters, which enables the use of one-step estimators. Slower-than-$n^{1/2}$ convergence rates are typically attained because the fineness of the approximation must improve with sample size. The guarantees provided for these estimators tend to be pointwise in nature. Given that pointwise convergence does not generally imply norm convergence or uniform convergence without additional regularity conditions, these pointwise-based estimators usually only facilitate inference for the evaluation of the unknown function at one or finitely many points, rather than for the entire function.

Some works have focused on estimating unknown function-valued parameters in a norm sense. Many of these works incorporate objects from semiparametric efficiency theory. For example, in the context of conditional average treatment effect estimation, risk functions have been developed \citep{van2006statistical,luedtke2016super,nie2021quasi}, and \cite{kennedy2020optimal} develops rate-of-convergence guarantees for the corresponding empirical risk minimizers. These estimators incorporate (weighted) variants of the EIF of the marginal average treatment effect in their construction. As further examples, EIFs have been used to construct norm-convergent estimators of the counterfactual density function \citep{kennedy2021semiparametric} and dose-response function \citep{takatsu2022debiased}. A drawback to these approaches to estimating function-valued parameters is that, to date, it has seemed that a new estimator must be derived and new regularity conditions established for each new parameter considered. Others have presented general approaches to learning unknown functions based on empirical risk minimization, where the population risk depends on unknown nuisance functions that can be orthogonalized by conducting statistical learning using an efficient estimator of the risk function as an objective function \citep{van2003unifiedCV,foster2019orthogonal}. When the population regret takes the form of a squared norm, these approaches provide a means to derive estimators with norm-convergence guarantees. However, unlike standard approaches such as one-step-estimation that are used for estimating finite-dimensional quantities, these methods do not appear to easily lend themselves to the construction of confidence sets for the unknown functions. 
Instead, the available approaches to construct confidence sets rely on approaches that are generally distinct from those used for estimating the function, such as building them using higher-order influence functions \citep{robins2008higher}, a restricted score test \citep{hudson2021inference}, or a maximum mean discrepancy (MMD) criterion \citep{luedtke2019omnibus}.

In this work, we establish that the one-step estimation methodology can be extended to estimate and make inference about pathwise differentiable parameters that take values in a Hilbert space. This pathwise differentiability condition turns out to be quite reasonable for many function-valued parameters of current interest. Indeed, we show that all of the parameters mentioned earlier in this Introduction satisfy it under regularity conditions. This is true in spite of the fact that these Hilbert-valued parameters are not pathwise differentiable when composed with an evaluation map.

The notion of pathwise differentiability that we focus on in this work is that studied in some early literature on semiparametric efficiency theory, which defined pathwise differentiability and EIFs not just for finite-dimensional parameters, but for general Banach-valued parameters (\citealp{van1989prohorov}; \citealp{van1991differentiable}; page 179 of \citealp{bickel1993efficient}). 
Since all Hilbert spaces are Banach spaces, their definitions apply in our case, as do some useful results that they present, such as a convolution theorem. Nevertheless, existing works did not provide any examples of how to evaluate the pathwise differentiability of infinite-dimensional Hilbert-valued parameters --- for example, see \cite{van1991differentiable} and Chapter 5.3 of \cite{bickel1993efficient}, whose infinite-dimensional examples all pertain to parameters taking values in a Banach space equipped with the uniform norm. Since they are not even pathwise differentiable at a point, none of the aforementioned function-valued parameters are pathwise differentiable in such a Banach space. 
Previous works also do not indicate whether or how the pathwise differentiability of an infinite-dimensional Hilbert-valued parameter can be used to facilitate estimation or inference, whether via the one-step estimation methodology or otherwise. While a brief, one-paragraph sketch was given on page 405 of \cite{bickel1993efficient} suggesting that providing a general, efficient $n^{-1/2}$-rate estimation framework may be difficult for infinite-dimensional spaces, this sketch only discusses a single example where $n^{-1/2}$-rate estimation may not even be possible. Moreover, neither that work, nor any subsequent ones, appear to evaluate whether leveraging the pathwise differentiability of a Hilbert-valued parameter would be useful for constructing a performant, but slower than $n^{-1/2}$-rate, estimator, or for constructing a confidence set.

The main contributions of this work are as follows:
\begin{enumerate}
    \item We characterize the EIF of a pathwise differentiable Hilbert-valued parameter, when it exists, and provide a means to obtain a regularized version thereof, when it does not.
    \item We construct one-step estimators using (possibly regularized) EIFs. Any method can be used to estimate the needed nuisance functions provided it converges at a suitable rate.
    \item We provide root-$n$-rate weak convergence and efficiency guarantees for these estimators, when an EIF exists, and slower rate-of-convergence guarantees, when one does not.
    \item We show how to construct asymptotically-valid confidence sets for the Hilbert-valued estimands. These confidence sets take different forms depending on whether an EIF exists.
    \item We study our framework in examples of current interest to the causal inference community and establish the pathwise differentiability of several more traditional parameters.
\end{enumerate}
When the estimand is a function, our confidence sets will contain it with a specified probability. Thus, if the aim is to infer about the whole function, our methodology is likely preferable to pointwise approaches. To accomplish the last point above, we derive a general lemma that facilitates the evaluation of the pathwise differentiability of Hilbert-valued parameters. 
Finally, we have conducted a simulation study to evaluate the proposed approach. 
All proofs can be found in the appendix.

\section{Pathwise differentiability in Hilbert spaces and constructing estimators}\label{sec:pd}

\subsection{Notation}
We work on a Polish space $(\mathcal{Z},\mathbf{B})$ with a collection of distributions $\mathcal{P}$, which we refer to as the model. 
Let $Z_1, Z_2, \cdots, Z_n \sim P_0$ be an independent and identically distributed (iid) sample from a distribution $P_0\in\mathcal{P}$, and let $P_n$ denote the corresponding empirical distribution. Let $\widehat{P}_n \in \mathcal{P}$ be an estimate of $P_0$. 
To ease notation, for now we consider a sample splitting approach wherein $\widehat{P}_n$ is fitted using an iid sample that is independent of $Z_1, Z_2, \cdots, Z_n$; in Section~\ref{sec:cf}, we describe the case where cross-fitting is used  \citep{schick1986asymptotically,klaassen1987consistent}, which is our preferred approach.  
For $Q$ a signed measure on $(\mathcal{Z},\mathbf{B})$ and a measurable function $f: \mathcal{Z} \to \mathbb{R}$, we use the shorthand $Qf := \int f dQ$. 
For any object indexed by $P_0$, we will abbreviate the notation by replacing `$P_0$' by `$0$'; for example, we will write $f_0$ rather than $f_{P_0}$. Similarly, we will replace `$\widehat{P}_n$' by `$n$' and write $f_n$ rather than $f_{\widehat{P}_n}$.

All Hilbert spaces mentioned in this paper are real Hilbert spaces. For a measure $\mu$ on a measurable space $(\mathcal{X},\Sigma)$, we write $L^2(\mu)$ to denote the Hilbert space of $\mu$-a.s. equivalence classes of $\mathcal{X}\rightarrow \mathbb{R}$ functions equipped with inner product $\langle f,g \rangle_{L^2(\mu)}:=\int fg\,d\mu$. If $\mathcal{X}\subset \mathbb{R}^d$, $\mu$ is the Lebesgue measure, and $\Sigma$ is the Borel $\sigma$-algebra on $\mathbb{R}^d$, we will sometimes write $L^2(\mathcal{X})$ instead of $L^2(\mu)$. 
In what follows $\mathcal{V}$ denotes a generic Hilbert space. 
We let $\| \cdot \|_{\mathcal{V}}$ and $\langle \cdot, \cdot \rangle_{\mathcal{V}}$ denote the norm and inner product associated with $\mathcal{V}$.
The space $L^2(P;\mathcal{V})$ is the Hilbert space containing all Bochner measurable functions $f: \mathcal{Z} \to \mathcal{V}$ such that
$$\|f\|_{L^2(P;\mathcal{V})} := \left( \int \| f(z) \|_{\mathcal{V}}^2\, P(dz) \right)^{1/2}<\infty.$$
The operator norm of a linear functional $f : \mathcal{V}\rightarrow\mathbb{R}$ is defined as $\|f\|_{\textnormal{op}}:=\inf\{c\ge 0 : |f(v)|\le c\|v\|_{\mathcal{V}}\textnormal{ for all }v\in\mathcal{V}\}$. 
If $\mathcal{W}$ is a closed subspace of $\mathcal{V}$, then let $\Pi_{\mathcal{V}}[h \mid \mathcal{W}]$ denote the orthogonal projection of $h$ to $\mathcal{W}$. We let $\ell^2$ denote the space of all square-summable sequences and $\|b\|_{\ell^2}:=[\sum_{k=1}^\infty b_k^2]^{1/2}$. We also let $[0,1]^{\mathbb{N}}$ denote the space of all $[0,1]$-valued sequences. To avoid having to use different notation to treat finite- and infinite-dimensional Hilbert spaces, throughout we use the convention that, if $\mathcal{V}$ is finite-dimensional, then we call $(v_k)_{k=1}^{\infty}$ an orthonormal basis of $\mathcal{V}$ if $(v_k)_{k=1}^{\mathrm{dim}(\mathcal{V})}$ is an orthonormal system that spans $\mathcal{V}$ and $v_k=0$ for all $k>\mathrm{dim}(\mathcal{V})$.

\subsection{Pathwise differentiability in Hilbert spaces}
We start with a brief review of important definitions that can be used to characterize the smoothness of a Hilbert-valued parameter. These definitions are adapted from those given in \citep{bickel1993efficient} for more general Banach-valued parameter settings. The subsequent parts of this section will involve developing estimators for our more specialized, but understudied, setting, where we heavily leverage the availability of an inner product in our Hilbert parameter space.

Let $\mathcal{P}$ be a collection of distributions defined on a common Polish space $(\mathcal{Z},\mathbf{B})$, which we refer to as the model. Suppose that the model is dominated by a $\sigma$-finite measure $\lambda$. 
A submodel $\{P_\epsilon : \epsilon \in [0,\delta)\} \subset \mathcal{P}$ is said to be quadratic mean differentiable at $P$ if and only if there exists a score function $s\in L^2(P)$ such that
\begin{align}
\left\|p_\epsilon^{1/2} - p^{1/2} - \epsilon s p^{1/2}/2 \right\|_{L^2(\lambda)} = o(\epsilon),\label{eq:qmd}
\end{align}
where, for $\epsilon\ge 0$, $p_\epsilon^{1/2} = \sqrt{\frac{dP_\epsilon}{d\lambda}}$ and $p^{1/2}=\sqrt{\frac{dP}{d\lambda}}$. 
Let $\mathscr{P}(P,\mathcal{P},s)$ refer to the set of quadratic mean differentiable submodels at $P$ with score function $s$. The set $\{s \in L^2(P) : \mathscr{P}(P,\mathcal{P},s) \neq \emptyset\}$ is called the tangent set, and its closed linear span is called the tangent space of $\mathcal{P}$ at $P$, denoted by $\dot{\mathcal{P}}_P$. For all $s \in \dot{\mathcal{P}}_P$, $P s = \int s dP = 0$. We let $L^2_0(P) := \{h \in L^2(P) : Ph = 0\}$, which is the largest possible tangent space at $P$. Any model with this tangent space at all distributions $P$ it contains is referred to as locally nonparametric.

Let $\mathcal{H}$ be a set known as the action space and $\nu: \mathcal{P} \to \mathcal{H}$ a parameter whose value is to be estimated. Throughout we assume that $\mathcal{H}$ is a real separable Hilbert space. 
The parameter $\nu$ is said to be pathwise differentiable at $P$ if and only if there exists a continuous linear operator $\dot{\nu}_P: \dot{\mathcal{P}}_P \to \mathcal{H}$ such that, for all $\{P_\epsilon : \epsilon \in [0,\delta)\} \in \mathscr{P}(P,\mathcal{P},s)$,
\begin{align}
    \| \nu(P_\epsilon) - \nu(P) - \epsilon \dot{\nu}_P(s) \|_{\mathcal{H}} = o(\epsilon). \label{eq:pdA}
\end{align}
The operator $\dot{\nu}_P$ is called the local parameter of $\nu$ at $P$ and its Hermitian adjoint, denoted by $\dot{\nu}_P^\ast: \mathcal{H} \to \dot{\mathcal{P}}_P$, is referred to as the efficient influence operator. 
The image of the local parameter $\dot{\nu}_P$, denoted by $\dot{\mathcal{H}}_P$, is a closed subspace of $\mathcal{H}$ that is referred to as the local parameter space. Throughout we equip $\dot{\mathcal{H}}_P$ with the inner product $\langle\cdot,\cdot\rangle_{\mathcal{H}}$, so that $\dot{\mathcal{H}}_P$ is itself a Hilbert space. 
The efficient influence operator can be shown to only depend on its argument through its projection onto the local parameter space, in the sense that $\dot{\nu}_P^*(h)=\dot{\nu}_P^*(\Pi_{\mathcal{H}}[h\mid \dot{\mathcal{H}}_P])$ for all $h\in\mathcal{H}$. At times in this work, we will consider pointwise evaluations of the efficient influence operator of the form $\dot{\nu}_P^*(h)(z)$. When doing so, we always assume that suitably `nice' elements of the $P$-a.s. equivalence classes defined by the elements $\dot{\nu}_P^*(h)$ of $L^2(P)$ are used to define these evaluations. In particular, we select these elements so that the efficient influence process, which we define as $\{\dot{\nu}_P^*(h) : h\in\mathcal{H}\}$, is a separable stochastic process, in the sense that there exists a countable dense subset $\mathcal{H}'$ of $\mathcal{H}$ and a $P$-probability one subset $\mathcal{Z}'$ of $\mathcal{Z}$ such that, for all $h\in\mathcal{H}$ and $z\in\mathcal{Z}'$, there exists an $\mathcal{H}'$-valued sequence $(h_j)_{j=1}^\infty$ that converges to $h$ and satisfies $\dot{\nu}_P^*(h_j)(z)\rightarrow \dot{\nu}_P^*(h)(z)$ as $j\rightarrow\infty$.

Analogous to the case for Euclidean parameters, in some semiparametric models it may be natural to describe a Hilbert-valued parameter as the restriction of a parameter defined on a larger, possibly nonparametric, model. If the true parameter lies in a model $\mathcal{P}' \subset \mathcal{P}$ with tangent space $\dot{\mathcal{P}}_P'$ and $\nu: \mathcal{P} \to \mathcal{H}$ is a parameter defined on $\mathcal{P}$, then the restriction $\nu \vert_{\mathcal{P}'}$ has local parameter $\dot{\nu} \vert_{\dot{\mathcal{P}}'_P}$ and efficient influence operator $h\mapsto \Pi_{L^2(P)}[\dot{\nu}_{P}^\ast(h) | \dot{\mathcal{P}}'_P]$. 
Armed with this fact, results can easily be transferred from a larger nonparametric model to a semiparametric model provided the form of the projection operator $\Pi_{L^2(P)}[\;\cdot\mid \dot{\mathcal{P}}'_P]$ is known. As a simple example, we may have that $\nu(P)(\cdot)=E_P[Y\mid X=\cdot\,]$ for each $P$ in a nonparametric model $\mathcal{P}$, and the model $\mathcal{P}'=\{P\in\mathcal{P} : \mathrm{var}_{P}(Y)=1\}$ may reflect knowledge that the variance of an outcome $Y$ is $1$. The form of the local parameter and efficient influence operator of $\nu$ relative to $\mathcal{P}$ are given in Example~\ref{ex:reg} in the appendix when $\mathcal{H}$ is an $L^2$ space, and the form of the projection onto $\dot{\mathcal{P}}_P'$ is given in Example~3.2.3 of \cite{bickel1993efficient}.

The parameter $\nu$ is said to have an EIF $\phi_P : \mathcal{Z}\rightarrow\mathcal{H}$ when there exists a $P$-probability-one set $\mathcal{Z}'$ such that
\begin{align}
\hspace{5em}\dot{\nu}_P^*(h)(z)=\langle h, \phi_P(z)\rangle_{\mathcal{H}}\ \ \textnormal{ for all $(h,z)\in\mathcal{H}\times \mathcal{Z}'$.} \label{eq:EIFRiesz}
\end{align}
By the Riesz representation theorem, $\nu$ has an EIF if and only if $\dot{\nu}_P^*(\cdot)(z) : \mathcal{H}\rightarrow\mathbb{R}$ is a bounded linear functional $P$-almost surely; in those cases, $\phi_P(z)$ is $P$-a.s. equal to the Riesz representation of $\dot{\nu}_P^*(\cdot)(z)$. The fact that $\dot{\nu}_P^*(h)=\dot{\nu}_P^*(\Pi_{\mathcal{H}}[h\mid \dot{\mathcal{H}}_P])$ implies that the image of $\phi_P$ is necessarily contained in $\dot{\mathcal{H}}_P$. 
When $\mathcal{H}=\mathbb{R}^d$, the EIF of $\nu$ at $P$ takes the form $\phi_P(z)=(\dot{\nu}_{P}^\ast(e_t)(z))_{t=1}^d$, where $\{e_t\}_{t=1}^d$ is the standard basis. To our knowledge, the existence and form of this object have not previously been studied in infinite-dimensional Hilbert spaces. 
Given that knowing the form of the EIF readily facilitates the construction of estimators in finite-dimensional settings, this appears to constitute an important gap in the literature. 
We therefore focus the remainder of this section on studying the existence of EIFs in Hilbert spaces and providing ways to construct estimators based on EIFs, when we can show they exist, or imitations thereof, when we cannot. 

The cases where EIFs do not exist become particularly salient in Section~\ref{sec:exPD}, where we demonstrate through examples that, for several interesting $L^2$-valued parameters, $\dot{\nu}_P^*(\cdot)(z) : \mathcal{H}\rightarrow\mathbb{R}$ depends on a point evaluation functional, and so is not bounded $P$-almost surely. Nevertheless, even when $\nu$ does not have an EIF, we will show in Section~\ref{sec:confSetsRegularized} that it is always possible to define an injective transformation of $\nu$ that has one. Consequently, the procedure we shall present to construct confidence sets for parameters with EIFs can also be used to construct them for those without one: first construct a confidence set for the transformation of $\nu(P_0)$, and then invert it to obtain one for $\nu(P_0)$.

Before proceeding, we note that inefficient influence operators can be defined when the model is semiparametric at $P$, in that $\dot{\mathcal{P}}_P$ is a strict subspace of $L_0^2(P)$. Under a condition akin to \eqref{eq:EIFRiesz}, inefficient influence functions can also be defined. To streamline presentation, we defer the presentation of these objects and their use for constructing estimators to Appendix~\ref{app:inefficient}. There, we also show that inefficient influence functions can only exist if an EIF exists.

\subsection{One-step estimation based on the efficient influence function} \label{sec:onestep}

On the one hand, if $\mathcal{H}$ is finite-dimensional, then the EIF can be used to construct what is known as a one-step estimator, which is known to be efficient under conditions \citep{pfanzagl1982lecture}. This estimator takes the form $\nu(\widehat{P}_n) + P_n \phi_n$, where $\widehat{P}_n\in\mathcal{P}$ is an initial estimate of the data-generating distribution $P_0$ and we recall the convention that $\phi_n:=\phi_{\widehat{P}_n}$. On the other hand, if $\mathcal{H}$ is infinite-dimensional, then previously studied one-step estimators cannot be applied. In this section, we provide a natural means to extend the one-step estimation framework to infinite-dimensional settings. 
Similarly to the one-step estimator in finite-dimensional settings, this one-step estimator takes the form $\widehat{\nu}_n:=\nu(\widehat{P}_n) + P_n \phi_n$ for an $\mathcal{H}$-valued EIF $\phi_n$. 
This estimator is applicable whenever $\nu$ has an EIF at $\widehat{P}_n$ with $P_0$-probability one. 
We will see that, under conditions that include that $\nu$ also has an EIF $\phi_0$ at $P_0$, $\widehat{\nu}_n$ is an asymptotically linear estimator of $\nu(P_0)$ with influence function $\phi_0$, in the sense that
\begin{align}
\widehat{\nu}_n - \nu(P_0)&= {\textstyle\frac{1}{n}\sum_{i=1}^n} \phi_0(Z_i) + o_p(n^{-1/2}), \label{eq:al}
\end{align}
where throughout we let Hilbert-valued quantities of the form $o_p(n^{-a})$ denote terms whose Hilbert norm goes to zero in probability even after being multiplied by $n^{a}$. 
We will be especially interested in cases where $n^{1/2}[\widehat{\nu}_n - \nu(P_0)]$ will converge weakly to a tight random element, since this can be used to facilitate the construction of confidence sets for $\nu(P_0)$ (see Section~\ref{sec:confSets}). To be able to apply a central limit theorem to establish the weak convergence of $n^{1/2}[\widehat{\nu}_n - \nu(P_0)]$, it suffices that $\phi_0$ is $P_0$-Bochner square integrable, in the sense that
$\|\phi_0\|_{L^2(P_0;\mathcal{H})}<\infty$ \citep[Example~1.8.5 of][]{van1996weak}. 
Therefore, we will focus on settings where $\phi_0\in L^2(P_0;\mathcal{H})$.

We begin by establishing the existence and form of the EIF at a generic $P\in\mathcal{P}$ in an interesting class of problems. In particular, we consider cases where $\mathcal{H}$ is an RKHS over a space $\mathcal{T}$ or, more generally, the local parameter space $\dot{\mathcal{H}}_P$ is an RKHS over $\mathcal{T}$. 
Denote the feature map of $\dot{\mathcal{H}}_P$ by $t\mapsto K_t$. For $P\in\mathcal{P}$, define $\tilde{\phi}_P : \mathcal{Z} \to \mathcal{H}$ as follows for each $t\in\mathcal{T}$:
\begin{align}
\tilde{\phi}_P(z)(t) = \dot{\nu}_P^*(K_t)(z) \quad \text{$P$-a.s. }z. \label{eq:eifDef}
\end{align}
The following result shows that $\tilde{\phi}_P$ both provides the form of the EIF of $\nu$, when it exists, and also a sufficient condition that can be used to verify this existence.
\begin{theorem}[Form of the efficient influence function in RKHS settings]\label{thm:eifProperties}
Suppose $\nu$ is pathwise differentiable at $P$ and $\dot{\mathcal{H}}_P$ is an RKHS. Both of the following implications hold:
\begin{enumerate}[label=(\roman*)]
    \item\label{it:eifForm} If $\nu$ has an EIF $\phi_P$ at $P$, then $\phi_P=\tilde{\phi}_P$ $P$-almost surely.
    \item\label{it:eifSuffCond} If $\|\tilde{\phi}_P\|_{L^2(P;\mathcal{H})}<\infty$, then $\nu$ has an EIF at $P$.
\end{enumerate}
\end{theorem}
\noindent The form of the EIF in \eqref{eq:eifDef} naturally generalizes its form in finite-dimensional spaces, where the feature $K_t$ replaces the $t$-th standard basis element $e_t$. The proof of \ref{it:eifForm} is a straightforward extension of results about the Riesz representation of a bounded linear functional to our setting, where $\dot{\nu}_P^*(\cdot)(z)$ is only known to be bounded and linear $P$-almost surely \citep[cf. Lemma~10 of][]{berlinet2011reproducing}. The proof of \ref{it:eifSuffCond} is more subtle, and involves showing that, when $\|\tilde{\phi}_P\|_{L^2(P;\mathcal{H})}<\infty$, \textit{any} separable version of the efficient influence process must $P$-a.s. have sample paths $\dot{\nu}_P^*(\cdot)(z)$ that are both bounded and linear. 
In the remainder of this subsection, we suppose that $\nu$ has an EIF $\phi_P$ at each $P\in\mathcal{P}$.

In Section~\ref{sec:oneStepGuarantee}, we establish that, under conditions, a cross-fitted variant of the one-step estimator $\widehat{\nu}_n:= \nu(\widehat{P}_n) + P_n \phi_n$ is efficient, in the sense that $\|\widehat{\nu}_n-\nu(P_0)\|_{\mathcal{H}}$ is as concentrated about zero as is possible for any estimator satisfying appropriate regularity conditions. Here, we provide two more heuristic arguments as to why the one-step correction should lead to improvements. The first, which applies specifically in cases where $\mathcal{H}$ is an RKHS, is based on the pointwise performance of the one-step estimator. In particular, the fact that norm convergence in an RKHS implies pointwise convergence can be used to show that the pathwise differentiability of $\nu : \mathcal{P}\rightarrow\mathcal{H}$ also implies the pathwise differentiability of $\nu(\cdot)(t) : \mathcal{P}\rightarrow\mathbb{R}$ for each $t\in\mathcal{T}$. Moreover, the EIF of $\nu(\cdot)(t)$ at $P\in\mathcal{P}$ is equal to $z\mapsto \dot{\nu}_P^*(K_t)(z)$, and so the one-step estimator for the real-valued parameter $\nu(P_0)(t)$ is equal to the evaluation of the $\mathcal{H}$-valued one-step estimator $\widehat{\nu}_n$ at the point $t$. While this pointwise justification of the one-step estimator $\widehat{\nu}_n$ is informative, it does not in itself explain why $\widehat{\nu}_n$ should be expected to perform well in a norm sense. Indeed, pointwise convergence in an RKHS does not necessarily imply norm convergence. For the same reason, pathwise differentiability of the real-valued parameters $\nu(\cdot)(t)$, $t\in\mathcal{T}$, does not generally imply pathwise differentiability of the RKHS-valued parameter $\nu$.

The second heuristic justification that we provide here provides initial insights into why the one-step estimator $\widehat{\nu}_n$ should perform well in a norm sense. This justification applies regardless of whether $\dot{\mathcal{H}}_P$ is an RKHS. Fix a submodel $\{P_\epsilon: \epsilon\in [0,\delta)\} \in \mathscr{P}(P,\mathcal{P},s)$. In the appendix, we establish that, under conditions on either the EIF (Lemma~\ref{lem:localParamToMean}) or the submodel (Lemma~\ref{lem:LinearQMDRemainder}), a first-order approximation to the local parameter $\dot{\nu}_P(s)$ is given by $\epsilon^{-1}(P_\epsilon - P)\phi_P$, in the sense that the difference between these quantities converges to zero as $\epsilon\rightarrow 0$. Combining this with \eqref{eq:pdA} and the fact that $P \phi_P = 0$, this yields the approximation $\nu(P_\epsilon) \approx \nu(P) + P_\epsilon \phi_P$, which is valid up to an additive $o(\epsilon)$ remainder term. Letting $P_0$ play the role of $P_\epsilon$ and $\widehat{P}_n$ play the role of $P$, this suggests that the von Mises approximation
\begin{align}
\nu(P_0) \approx \nu(\widehat{P}_n) + P_0 \phi_n \label{eq:firstOrderApproxRKHS}
\end{align}
may also be valid up to a term that goes to zero in probability at a reasonable rate. Admittedly, caution is needed when making the leap from the approximation along the fixed quadratic mean differentiable submodel $\{P_\epsilon: \epsilon\in [0,\delta)\}$ to an approximation that involves the random quantity $\widehat{P}_n$. For a given parameter $\nu$, the sense in which the above approximation holds can be made precise by directly studying the quantity $\nu(P_0)-\nu(\widehat{P}_n)-P_0\phi_n$. In any case, the approximation above is appealing in that it only relies on $P_0$ through an expectation, which can naturally be approximated by an expectation under the empirical distribution. This, therefore, suggests the one-step estimator $\widehat{\nu}_n:=\nu(\widehat{P}_n) + P_n \phi_n$. 

As will follow from the upcoming Lemma~\ref{lem:regParamEIF}, $\dot{\mathcal{H}}_P$ need not be an RKHS for an EIF to exist. 
Consequently, when one does, it is natural to wonder whether there is a general expression for its form. The Riesz representation theorem provides an affirmative answer to this question, showing that, when an EIF exists, it is $P$-a.s. equal to the following convergent sum:
\begin{align}
\phi_P(z)= \sum_{k=1}^\infty \dot{\nu}_P^*(h_k)(z) \,h_k, \label{eq:EIFinfiniteSum}
\end{align}
where here and throughout we let $(h_k)_{k=1}^\infty$ denote an orthonormal basis of $\mathcal{H}$. If $\dot{\mathcal{H}}_P$ is an RKHS, evaluating the expression for the EIF in \eqref{eq:eifDef} is typically easier than computing (or approximating) the infinite sum in \eqref{eq:EIFinfiniteSum}. However, in non-RKHS settings, \eqref{eq:EIFinfiniteSum} is useful both as an explicit expression for the EIF, if it exists, and as a basis for generalizing the one-step estimator to settings where it does not.

\subsection{Regularized one-step estimation when there is no efficient influence function} \label{sec:twostep}

We now introduce a generalization of the one-step estimator that can be employed regardless of whether an EIF exists. This estimator is a type of series estimator \citep{schwartz1967estimation,chen2007large} based on the Riesz representation of a regularized form of the efficient influence operator. This regularized form is motivated by the fact that, when $\nu$ has an EIF $\phi_P$, it is $P$-a.s. true that $\dot{\nu}_P^*(\cdot)(z) : h\mapsto \sum_{k=1}^\infty \langle h,h_k\rangle_{\mathcal{H}} \dot{\nu}_P^*(h_k)(z)$. The regularized form is designed to ensure that the terms in this sum must decay as $k$ grows sufficiently large. 
For a square summable $[0,1]$-valued sequence $\beta:=(\beta_k)_{k=1}^\infty$, the $\beta$-regularized efficient influence operator is given by $r_P^\beta(h)(z):= \sum_{k=1}^\infty \beta_k \langle h,h_k\rangle_{\mathcal{H}} \dot{\nu}_P^*(h_k)(z)$. 
We now show that $r_P^\beta(\cdot)(z)$ is always $P$-a.s. bounded and linear, and we also provide an explicit expression for its Riesz representation. In what follows we let $\ell_{*}^2:=\ell^2\cap[0,1]^{\mathbb{N}}$.
\begin{lemma}[$\beta$-regularized EIF based on $\beta$-regularized efficient influence operator]\label{lem:approximateEIF}
If $\nu$ is pathwise differentiable at $P$ and $\beta\in \ell_{*}^2$, then $r_P^\beta(\cdot)(z) : \mathcal{H}\rightarrow\mathbb{R}$ is a bounded linear functional on a $P$-probability one set $\mathcal{Z}^\beta$ with Riesz representation
\begin{align*}
\phi_P^\beta(z)&:= \sum_{k=1}^\infty \beta_k \dot{\nu}_P^*(h_k)(z) h_k.
\end{align*}
Moreover, $\sigma_P(\beta):=\|\phi_P^\beta\|_{L^2(P;\mathcal{H})}=[\sum_{k=1}^\infty \beta_k^2 P \dot{\nu}_P(h_k)^2]^{1/2}\le \|\dot{\nu}_P^*\|_{\mathrm{op}}\|\beta\|_{\ell^2}<\infty$.
\end{lemma}
When $\nu$ has an EIF , $\phi_P^\beta$ is similar to the expression for it given in \eqref{eq:EIFinfiniteSum}, but with the $k$-th term dampened by the multiplier $\beta_k\in [0,1]$. Given this similarity, we call $\phi_P^\beta$ the $\beta$-regularized EIF of $\nu$ at $P$. The corresponding $\beta$-regularized one-step estimator is $\widehat{\nu}_n^\beta:= \nu(\widehat{P}_n) + P_n \phi_n^\beta$.

We now provide a heuristic argument that is similar to one used in the previous subsection for justifying the (non-regularized) one-step estimator, but adapted to account for the regularization bias that arises from using a $\beta$-regularized EIF. Fix a submodel $\{P_\epsilon: \epsilon\} \in \mathscr{P}(P,\mathcal{P},s)$. In Lemma~\ref{lem:LinearQMDApproxRemainder} in the appendix, we show that, under a regularity condition on the submodel,
\begin{align}
&\left\|\nu(P)-\nu(P_\epsilon) + P_\epsilon \phi_P^\beta - {\textstyle\sum_{k=1}^\infty} (1-\beta_k) \langle\nu(P) - \nu(P_\epsilon),h_k\rangle_{\mathcal{H}} h_k\right\|_{\mathcal{H}} \label{eq:regularizedRemainder} \\
&\quad= \left(1 + \|\phi_P^\beta\|_{L^2(P;\mathcal{H})}\right)\cdot o(\epsilon), \nonumber
\end{align}
where the $o(\epsilon)$ term does not depend on the choice of $\beta$. Similarly to how we did when deriving \eqref{eq:firstOrderApproxRKHS}, we let $P_0$ play the role of $P_\epsilon$ and $\widehat{P}_n$ play the role of $P$. Recalling that $\widehat{\nu}_n^\beta:= \nu(\widehat{P}_n) + P_n \phi_n^\beta$ then yields that
\begin{align}
\widehat{\nu}_n^\beta-\nu(P_0)&\approx (P_n-P_0) \phi_n^\beta + \mathcal{B}_{\widehat{P}_n}^\beta, \label{eq:biasedAL}
\end{align}
where, for $P'\in\mathcal{P}$, we let $\mathcal{B}_{P'}^\beta:=\sum_{k=1}^\infty (1-\beta_k) \langle\nu(P') - \nu(P_0),h_k\rangle_{\mathcal{H}} h_k$. 
Our formal study of the regularized one-step estimator in Section~\ref{sec:regularizedOneStepGuarantee} builds on the above. Informally speaking, that study will show that the latter term above plays the role of a regularization bias term that decays as $\beta$ grows entrywise to $(1,1,1,\ldots)$ under conditions, and the leading term plays the role of a variance term whose magnitude typically grows with that of $\beta$. Hence, a bias-variance tradeoff must be considered when selecting a value for the tuning parameter $\beta$. In Section~\ref{sec:regularizedOneStepTuning}, we describe a cross-validation strategy for making this selection. There, we also discuss the selection of the basis $(h_k)_{k=1}^\infty$.

\subsection{Cross-fitted (regularized) one-step estimation}\label{sec:cf}

So far, the estimators we have defined have assumed the availability of an iid sample that is independent of $Z_1,Z_2,\ldots,Z_n$ that can be used to obtain the estimate $\widehat{P}_n$ of $P_0$. We now describe how cross-fitting \citep{schick1986asymptotically,klaassen1987consistent,zheng2011cross,chernozhukov2018double} can be used to avoid the need for this independent sample. For simplicity, we focus on the case of 2-fold cross-fitting and suppose that the sample size is an even number. The generalizations to $k$-fold cross-fitting ($k\ge 2$) and to the case where $n$ is not divisible by $k$ are straightforward and so are omitted. 
Let $\widehat{P}_n^1 \in \mathcal{P}$ denote an estimate of $P_0$ based on $\{Z_i\}_{i=1}^{n/2}$ and let $P_n^1$ denote the empirical distribution of the remainder of the sample $\{Z_i\}_{i=n/2+1}^{n}$. 
Define $\widehat{P}_n^2$ and $P_n^2$ similarly, but with the roles of the two subsamples reversed. 
We note that, in a slight abuse of notation, $P_n^j$ denotes an empirical distribution derived from $n/2$ observations rather than a $j$-fold product distribution derived from $j$ independent draws from the empirical distribution $P_n$ of the full sample $\{Z_i\}_{i=1}^n$. 
Cross-fitting enables the use of arbitrary estimation strategies when constructing $\widehat{P}_n^j$, $j\in\{1,2\}$, including those based on machine learning techniques.

We now present the form of our cross-fitted estimators. From a notational standpoint, these estimators will be denoted by replacing the hat accents used to denote the sample-splitting estimators in Sections~\ref{sec:onestep} and \ref{sec:twostep} by bar accents --- for example, the cross-fitted one-step estimator will be denoted by $\bar{\nu}_n$ rather than $\widehat{\nu}_n$. This cross-fitted one-step estimator takes the form $\bar{\nu}_n:= \frac{1}{2}\sum_{j=1}^2 [\nu(\widehat{P}_n^j) + P_n^j \phi_n^j]$, where $\phi_n^j:=\phi_{\widehat{P}_n^j}$. 
Let $\phi_n^{j,\beta}:=\phi_{\widehat{P}_n^j}^\beta$. The cross-fitted $\beta$-regularized one-step estimator takes the form $\bar{\nu}_n^\beta:= \frac{1}{2}\sum_{j=1}^2 [\nu(\widehat{P}_n^j) + P_n^j \phi_n^{j,\beta}]$.

\section{Examples of pathwise differentiable parameters}\label{sec:exPD}

In this section, we present examples of pathwise differentiable Hilbert-valued parameters and the forms of their efficient influence operators and, where applicable, EIFs. From these objects, cross-fitted (regularized) one-step estimators can be derived using the formulas at the end of the previous section. 
We study the performance of these estimators in Section~\ref{sec:exEst}. 

In the main text, we focus on parameters that have recently become objects of interest to the causal inference community. Two of these examples (Examples \ref{ex:cfdBandlimited} and \ref{ex:gkmed}) consider cases where the action space is an RKHS, and two (Examples \ref{ex:cfdNonparametric} and \ref{ex:conac}) consider cases where the action space is an $L^2$ space, and is therefore not an RKHS. 
In Appendix~\ref{app:exClassical}, we show that four more well-studied Hilbert-valued parameters are also pathwise differentiable. In particular, we show that regression functions, square-root density functions, and conditional average treatment effect functions are pathwise differentiable when viewed as elements of appropriate $L^2$ spaces, and we also show that a kernel mean embedding of a distribution \citep{gretton2012kernel} is pathwise differentiable when viewed as an element of an RKHS. 
We are not aware of any reference establishing the pathwise differentiability of any of the eight Hilbert-valued parameters that we consider in this work.

All derivations for our examples are deferred to Appendix~\ref{app:exDerivations}. 
For most of these examples, our derivations make use of the following lemma, which we prove in Appendix~\ref{app:lem}. In this lemma, $H(P,P'):=[\int (\sqrt{dP}-\sqrt{dP'})^2]^{1/2}$ denotes the Hellinger distance.
\begin{lemma}[Sufficient condition for pathwise differentiability] \label{lem:suffCondsPD}
$\nu : \mathcal{P}\rightarrow\mathcal{H}$ is pathwise differentiable at $P$ with local parameter $\dot{\nu}_P=\eta_P$ if both of the following hold:
\begin{enumerate}[label=(\roman*)]
    \item\label{it:closure} $\eta_P : \dot{\mathcal{P}}_P\rightarrow\mathcal{H}$ is bounded and linear and there exists a set of scores $\mathcal{S}(P)$ whose $L^2(P)$-closure is equal to $\dot{\mathcal{P}}_P$ such that, for all $s\in \mathcal{S}(P)$, there is at least one submodel $\{P_\epsilon : \epsilon\}\in \mathscr{P}(P,\mathcal{P},s)$ for which $\| \nu(P_\epsilon) - \nu(P) - \epsilon\, \eta_P(s) \|_{\mathcal{H}} = o(\epsilon)$; and
    \item\label{it:localLip} $\nu$ is locally Lipschitz at $P$ in the sense that there exist $(c,\delta)\in (0,\infty)^2$ such that 
\begin{align}
    \|\nu(P_1)-\nu(P_2)\|_{\mathcal{H}}&\le cH(P_1,P_2)\ \textnormal{ for all } P_1,P_2\in B_\delta(P), \label{eq:localLip}
\end{align}
where $B_\delta(P)$ consists of all $P'\in\mathcal{P}$ for which $H(P,P')\le \delta$.
\end{enumerate}
\end{lemma}
This lemma is most useful when the set $\mathcal{S}(P)$ and corresponding submodels in $\mathscr{P}(P,\mathcal{P},s)$ in \ref{it:closure} can be chosen to make establishing \eqref{eq:pdA} for those submodels simple. For example, in a locally nonparametric model, we may take $\mathcal{S}(P)$ to be the set of bounded, $P$-mean zero functions, and we may take the chosen submodel in $\mathscr{P}(P,\mathcal{P},s)$ to be such that, for all $\epsilon\in [0,1/\esssup_z|s(z)|)$, $\frac{dP_\epsilon}{dP}=1+\epsilon s$, where the essential supremum is under $P$.

All of the examples presented in the main text are motivated by questions arising in causal inference. The data structure is common across them, with $Z=(X,A,Y)\sim P$, where $X$ is a vector of covariates with support on $\mathcal{X}$, $A$ is a treatment with support on either $\{0,1\}$ or $\mathbb{R}$, and $Y$ is an outcome with support on $\mathcal{Y}$. In each example, we suppose that $\mathcal{P}$ is locally nonparametric. We further suppose, for simplicity, that all pairs of distributions in $\mathcal{P}$ are mutually absolutely continuous. For a given distribution $P$, we let $P_{Y\mid A,X}$ denote the conditional distribution of $Y$ given $(A,X)$ and $P_X$ denote the marginal distribution of $X$. We let $g_P(\cdot\mid x)$ denote the conditional probability mass function of $A$ given $X=x$ under $P$, when $A$ is binary, and the density of $A$ given $X=x$ under $P$, when $A$ is continuous. For $s\in \dot{\mathcal{P}}_P$, we let $s_X(x):=E_P[s(Z)\mid X=x]$ and $s_{Y\mid A,X}(y\mid a,x):=s(z)-E_P[s(Z)\mid A=a,X=x]$.

\begin{subexample}{example}\label{ex:cfd}
\begin{example}[Counterfactual density function] \label{ex:cfdNonparametric}

Suppose that the treatment $A$ is binary and the goal is to estimate the density function of the counterfactual outcome in a setting where everyone receives treatment $A=1$. 
This density can offer a more nuanced measure of causal effects than can a more commonly studied counterfactual mean outcome \citep{kennedy2021semiparametric}. 
Suppose that there is a $\sigma$-finite measure $\lambda_Y$ such that, for all $P \in \mathcal{P}$, there is a regular conditional probability $P_{Y \mid A,X}$ such that $P_{Y \mid A,X}(\cdot \mid a,x) \ll \lambda_Y$ for $P$-almost all $(a,x) \in \{0,1\} \times \mathcal{X}$. Define the propensity to receive treatment $a$ as $g_P(a\mid x) := P(A=a \mid X=x)$ and let $p_{Y\mid A,X}(\cdot\mid 1,x)$ denote the conditional density of $Y$ given $(A,X)=(1,x)$. 
The parameter of interest $\nu: \mathcal{P} \to L^2(\lambda_Y)$ takes the form
\begin{align}
\nu(P)(y) = \int p_{Y\mid A,X}(y\mid 1,x)\, P_X(dx). \label{eq:cdDef}
\end{align}
Under typical causal assumptions, $\nu$ corresponds to the density of the counterfactual outcome that would be seen under treatment $A=1$. We require that $\mathcal{P}$ satisfy the following conditions:
\begin{equation} \label{cond:cfDensAndstrPos}
\inf_{P\in\mathcal{P}}\essinf_{x \in \mathcal{X}} g_P(1\mid x) > 0\ \ \textnormal{ and } \  \sup_{P\in\mathcal{P}}\esssup_{(x,y)\in\mathcal{X}\times\mathcal{Y}} p_{Y\mid A,X}(y\mid 1,x) < \infty,
\end{equation}
where the essential infimum is under $P_X$ and the essential supremum is under $P_X\times \lambda_Y$. 
The first inequality, referred to as a strong positivity assumption, holds when the propensity to receive treatment 1 is not vanishingly small. The second holds when the conditional density of $Y$ given $(A,X)=(1,x)$ is uniformly bounded across all distributions in the model. While in principle the second condition could be weakened, this condition appears to be sufficiently general to capture many statistical models of interest.

The local parameter takes the form
\begin{align}
\dot{\nu}_P(s)(y) = \int \big\{&s_{Y\mid A,X}(y\mid 1,x) + s_X(x)\big\}p_{Y\mid A,X}(y\mid 1,x) P_X(dx), \label{eq:countDenLocalP}
\end{align}
and the efficient influence operator takes the form
\begin{align}
\dot{\nu}_P^\ast (h)(y,a,x) &= \frac{1\{a=1\}}{g_P(a\mid x)}\left\{h(y)-E_P\left[h(Y) \mid A=a, X=x \right]\right\} \nonumber \\
&\quad+ \Big(E_P\left[h(Y)\mid A=1, X=x\right] - \int E_P\left[h(Y) \mid A=1, X=x'\right]P_X(dx')\Big).
\label{eq:countDenAdjoint}
\end{align}
Unless $\lambda_Y$ is a discrete measure, $\dot{\nu}_P^*$ will not generally be a bounded operator. This can be shown to follow from the facts that point evaluation is not continuous in $L^2$ spaces \citep[page 8 of][]{berlinet2011reproducing} and $\dot{\nu}_P^\ast (h)(y,a,x)$ depends on the evaluation of $h$ at $y$.
\end{example}

\begin{example}[Bandlimited counterfactual density function]\label{ex:cfdBandlimited}

\begin{figure}[tb]
    \centering
    \includegraphics[width=\textwidth]{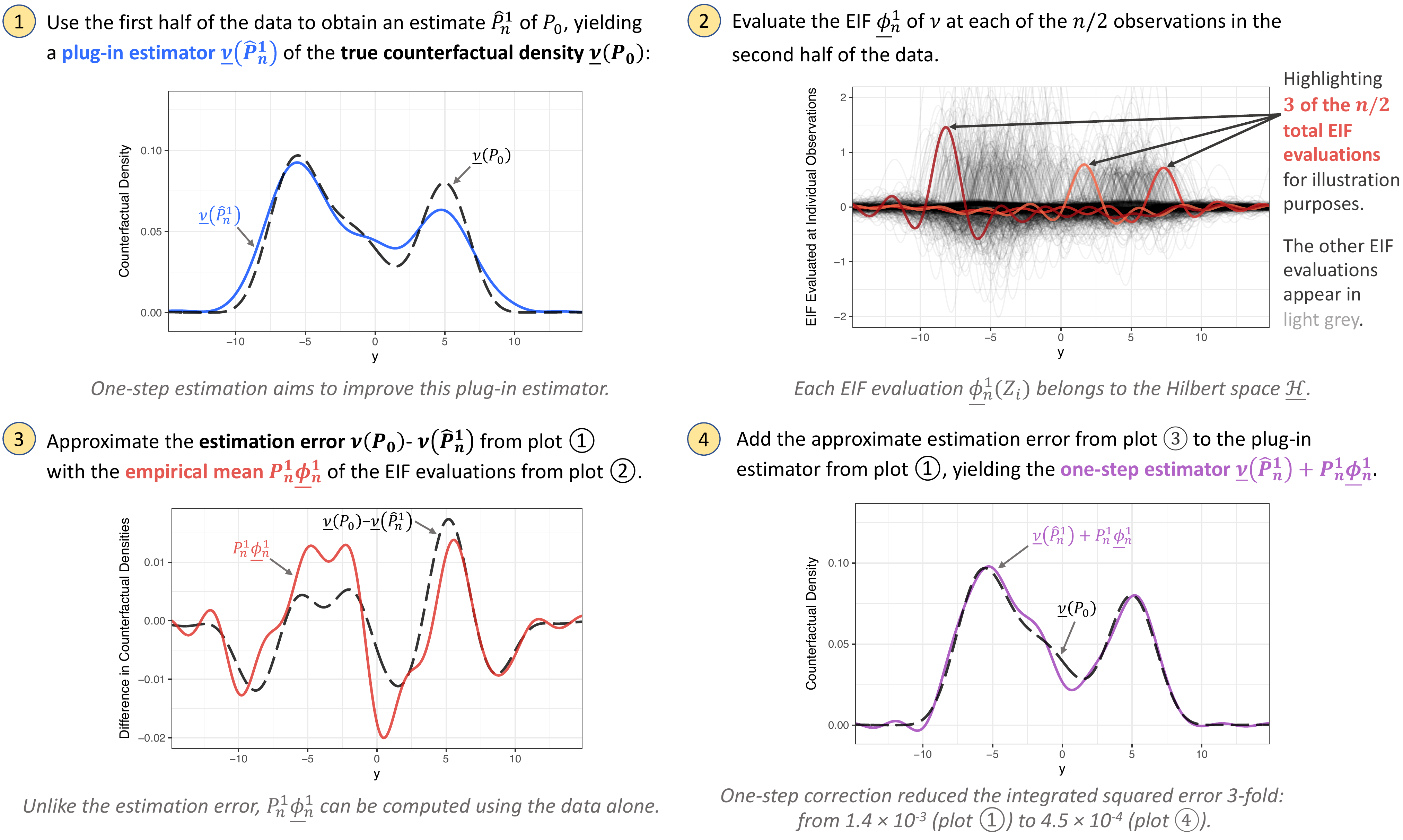}
    \caption{Illustration of how a sample-splitting one-step estimator is constructed in Example~\ref{ex:cfdBandlimited}. 
    The cross-fitted estimator averages two such estimators, one obtained as above and the other with the roles of the two halves of the data reversed.  
    Details of the data-generating process, along with a Monte Carlo assessment of the performance of a cross-fitted one-step estimator, are given in Appendix~\ref{app:simCfdBandlimited}.}
    \label{fig:onestepIllustration}
\end{figure}

The setting is the same as in Example~\ref{ex:cfdNonparametric}, except that $Y$ must be real-valued and continuous, $\lambda_Y$ must be the Lebesgue measure, and, for fixed $b>0$, the target of inference is the following transformation of the counterfactual density $\nu(P)$ that was defined in Example~\ref{ex:cfdNonparametric}:
\begin{align}
    \underline{\nu}(P)(y):= \int_{-\infty}^\infty \underline{K}_y(\tilde{y})\,\nu(P)(\tilde{y})\, \lambda_Y(d\tilde{y}), \label{eq:underlineNuDef}
\end{align}
where $\underline{K}_y(\tilde{y}):=\{\sin[b(\tilde{y}-y)]\}/[\pi (\tilde{y}-y)]$. This estimand corresponds to a bandlimiting of the counterfactual density function. In particular, letting $\mathcal{F}$ and $\mathcal{F}^{-1}$ denote the Fourier transform and inverse Fourier transform and fixing $b>0$, the $b$-bandlimiting of a Lebesgue square integrable function $f : \mathbb{R}\rightarrow\mathbb{R}$ is the function $\mathscr{B}(f) : \mathbb{R}\rightarrow\mathbb{R}$ given by
\begin{align}
\mathscr{B}(f) := \mathcal{F}^{-1}\big(1_{[-b,b]}\cdot \mathcal{F}(f)\big)(y) = \int_{-\infty}^\infty \underline{K}_y(\tilde{y})\,f(\tilde{y})\, \lambda_Y(d\tilde{y}), \label{eq:bandlimitingDef} 
\end{align}
where $1_{[-b,b]}\cdot \mathcal{F}(f)$ represents the function $\xi\mapsto 1_{[-b,b]}(\xi)\cdot \mathcal{F}(f)(\xi)$ and the latter equality holds by the convolution theorem. 
The estimand is equal to $\underline{\nu}(P):=\mathscr{B}(\nu(P))$. Lemma~\ref{lem:bandlimitedproj} in Appendix~\ref{lem:bandLimDen} shows that $\mathscr{B}(\nu(P))$ corresponds to an $L^2(\lambda_Y)$ projection of $\nu(P)$ onto
\begin{align*}
    \underline{\mathcal{H}}:= \left\{ h \in C(\mathbb{R}) \;\middle|\; \textnormal{support}[\mathcal{F}(h)]\subseteq [-b,b]\right\},
\end{align*}
where $C(\mathbb{R})$ denotes the set of continuous, Lebesgue square integrable functions. A little care is needed to make this result precise since $L^2(\lambda_Y)$ is a space of equivalence classes of functions whereas $\underline{\mathcal{H}}$ is a space of functions (see the statement of Lemma~\ref{lem:bandlimitedproj} for details). The space $\underline{\mathcal{H}}$ is an RKHS when equipped with the $L^2(\lambda_Y)$ inner product $\langle h,\tilde{h}\rangle_{\underline{\mathcal{H}}} = \int h(y)\tilde{h}(y)\lambda_Y(dy)$ \citep{yao1967applications}. The kernel function in this RKHS is given by 
$(y,\tilde{y})\mapsto \underline{K}_y(\tilde{y})$. This RKHS is a smoothness class consisting of all square integrable functions that have an analytic continuation to the complex plane that satisfies an exponential growth restriction \citep[Theorem~19.3 and page 372 of][]{rudin1987real}. 
When a (non-counterfactual) density function belongs to $\underline{\mathcal{H}}$, a particular kernel density estimator has been shown to attain mean integrated squared error (MISE) that decays at an $n^{-1}$ rate \citep{ibragimov1983estimation,agarwal2015nonparametric}. Our setting differs from earlier ones in that (i) we focus on a counterfactual density, and (ii) we define our estimand as a nonparametric projection onto the space of $b$-bandlimited functions, rather than requiring that our estimand belong to this class. 
Naturally, when the counterfactual density $\nu(P)$ already belongs to $\underline{\mathcal{H}}$, $\underline{\nu}(P)$ will equal this density and so our approach will yield estimators of it.

In the appendix, we rely heavily on the calculations performed in Example~\ref{ex:cfdNonparametric} when showing that $\underline{\nu} : \mathcal{P}\rightarrow\underline{\mathcal{H}}$ is pathwise differentiable. 
We show that the local parameter $\underline{\dot{\nu}}_P : \dot{\mathcal{P}}_P\rightarrow \underline{\mathcal{H}}$ of $\underline{\nu}$ at $P$ is closely related to that of $\nu$. In particular, $\underline{\dot{\nu}}_P(s)=\mathscr{B}(\dot{\nu}_P(s))$, where $\dot{\nu}_P$ is as defined in Example~\ref{ex:cfdNonparametric}. 
Letting $[h]$ denote the equivalence class of functions that are equal to some $h\in\underline{\mathcal{H}}$ Lebesgue-almost everywhere, the efficient influence operator $\underline{\dot{\nu}}_P^* : \underline{\mathcal{H}}\rightarrow\dot{\mathcal{P}}_P$ takes the form $\underline{\dot{\nu}}_P^*(h)=\dot{\nu}_P^*([h])$, where $\dot{\nu}_P^*$ is as defined in Example~\ref{ex:cfdNonparametric}. Since $\underline{\mathcal{H}}$ is an RKHS, we can also look to define the EIF $\underline{\phi}_P$ of $\underline{\nu}$ at $P$. In particular, \eqref{cond:cfDensAndstrPos} can be used to show that the function $\underline{\phi}_P(z) : y\mapsto \underline{\nu}_P^*(\underline{K}_y)(z)$ belongs to $L^2(P;\underline{\mathcal{H}})$, and so $\underline{\phi}_P$ is indeed the EIF of $\underline{\nu}$ at $P$. See \eqref{eq:cfDensEIF} in the appendix for a more explicit expression for the EIF $\underline{\phi}_P$.

Figure~\ref{fig:onestepIllustration} shows how a one-step estimator can improve an initial estimator in this example.
\end{example}
\end{subexample}

\begin{example}[Counterfactual mean outcome under a continuous treatment] \label{ex:conac}
In the previous example, the treatment was considered to be binary. In this example, we take $A$ to be a continuous treatment taking values in $\mathcal{A}=[0,1]$, such as a dosage, duration, or frequency of intervention. 
Denote the marginal distribution of $A$ by $P_A$ and suppose that the Lebesgue measure $\lambda_A$ dominates the conditional distribution $P_{A\mid X}(\cdot\mid x)$ of $A\mid X=x$ under $P$ for $P$-almost all $x$. Let $g_P(\,\cdot\mid x)$ denote the conditional density of $A$ given that $X=x$.  
The target of estimation is $\nu: \mathcal{P} \to L^2(\lambda_A)$, where $\nu(P)(a) = \int E_P[Y \mid A=a,X=x]P_X(dx)$. 
Under causal conditions, $\nu(P)(a)$ corresponds to the mean outcome under a continuous treatment \citep{diaz2013targeted}. We suppose the strong positivity assumption that $\inf_{P\in\mathcal{P}}\essinf_{(a,x)} g_P(a\mid x) > 0$, where the essential infimum is under $\lambda_A\times P_X$. We further suppose that $Y$ has a bounded conditional second moment, in the sense that $\sup_{P \in \mathcal{P}} \esssup_{(a, x)} E_P[Y^2 \mid A=a,X=x]<\infty$, where the essential supremum is under $\lambda_A\times P_X$.

The local parameter takes the form
\begin{align}
\dot{\nu}_P(s)(a) &= \iint \{y-\mu_P(a,x)\}s_{Y \mid A,X}(y\mid a,x)P_{Y \mid A,X}(dy \mid a,x) P_X(dx) \nonumber \\
&\quad+ \int [\mu_P(a,x)-\nu(P)(a)]s_X(x)P_X(dx), \label{eq:conacLocalP}
\end{align}
where $\mu_P(a,x) = E_P[Y \mid A=a,X=x]$. The efficient influence operator takes the form
\begin{align}
\dot{\nu}_P^\ast(h)(y,a,x) = \frac{y-\mu_P(a,x)}{g_P(a\mid x)}h(a) + \int [\mu_P(a',x)-\nu(P)(a')]h(a')\lambda_A(da'). \label{eq:conacEIO}
\end{align}
Similarly to Example~\ref{ex:cfdNonparametric}, $\dot{\nu}_P^*$ is not generally a bounded operator.
\end{example}

\begin{example}[Counterfactual kernel mean embedding] \label{ex:gkmed}
Let $A$ be a binary treatment and suppose the strong positivity assumption that $\inf_{P\in\mathcal{P}}\essinf_{x \in \mathcal{X}} g_P(1\mid x) > 0$, where the essential infimum is under $P_X$. Let $\kappa: \mathcal{Y} \times \mathcal{Y} \to \mathbb{R}$ be a bounded, symmetric, positive define function, $\mathcal{H}$ be the RKHS associated with the kernel $\kappa$, and $K_y:=\kappa(y,\cdot)$ be the associated feature map. The counterfactual kernel mean embedding \citep{muandet2021counterfactual,fawkes2022doubly} is a parameter $\nu: \mathcal{P} \to \mathcal{H}$ such that
$$\nu(P) = \int E_P[K_Y \mid A=1,X=x] P_X(dx).$$
Under standard causal conditions \citep{robins1986new}, $\nu(P)$ can be shown to be equal to the kernel mean embedding \citep{gretton2012kernel} of the distribution of a counterfactual outcome in a world where treatment $1$ was given to everyone. We suppose that the strong positivity assumption in \eqref{cond:cfDensAndstrPos} holds.

The local parameter takes the form
\begin{align}
\dot{\nu}_P(s) = \iint K_y\,[s_{Y \mid A,X}(y\mid 1,x) + s_X(x)]\,P_{Y \mid A,X}(dy \mid 1,x) P_X(dx). \label{eq:gkmedLocalP}
\end{align}
The efficient influence operator takes the form
\begin{align*}
\dot{\nu}_P^\ast (h)(y,a,x) &= \frac{a}{g_P(1\mid x)}\left\{h(y)-E_P[h(Y) \mid A=a,X=x]\right\} \\
&\quad+ E_P[h(Y) \mid A=1,X=x] - \int E_P[h(Y) \mid A=1,X=x'] P_X(dx'),
\end{align*}
and the EIF is $P$-Bochner square integrable and takes the form
\begin{align}
    \phi_P (y,a,x) = \frac{a}{g_P(1\mid x)}\left\{K_y - \mu_P^K(x)\right\} + \mu_P^K(x) - \nu(P), \label{eq:gkmedEIF}
\end{align}
where $\mu_P^K : \mathcal{X}\rightarrow\mathcal{H}$ is defined as $\mu_P^K(x):=E_P[K_Y \mid A=1,X=x]$.
\end{example}

\section{Performance guarantees and inference when there is an EIF} \label{sec:perfGuarantees}

\subsection{Performance guarantees for one-step estimators}\label{sec:oneStepGuarantee}

In this section, we provide conditions under which a cross-fitted one-step estimator $\bar{\nu}_n$ is both asymptotically linear and efficient. 
These conditions concern the terms arising in the following decomposition:
\begin{align}
\left\|\bar{\nu}_n - \nu(P_0) - P_n \phi_0\right\|_{\mathcal{H}} &= \frac{1}{2}\Bigg\|\sum_{j=1}^2\left[\nu(\widehat{P}_n^j) + P_0 \phi_n^j - \nu(P_0)\right] + \sum_{j=1}^2(P_n^j - P_0)(\phi_n^j - \phi_0)\Bigg\|_{\mathcal{H}} \nonumber \\
&\le \max_j \|\mathcal{R}_n^j\|_{\mathcal{H}} + \max_j\|\mathcal{D}_n^j\|_{\mathcal{H}}, \label{eq:cfExpansion}
\end{align}
where $\mathcal{R}_n^j := \nu(\widehat{P}_n^j)+P_0 \phi_n^j - \nu(P_0)$ and $\mathcal{D}_n^j := (P_n^j - P_0)(\phi_n^j - \phi_0)$. We call $\mathcal{R}_n^j$ the remainder terms and $\mathcal{D}_n^j$ the drift terms, $j\in \{1,2\}$. Asymptotic linearity, as defined in \eqref{eq:al}, holds whenever both of these quantities are negligible, in the sense that they are $o_p(n^{-1/2})$. The remainder terms $\mathcal{R}_n^j$, $j\in\{1,2\}$, quantify the error in the approximation \eqref{eq:firstOrderApproxRKHS} across the two splits of the sample. 
As described heuristically in the text surrounding \eqref{eq:firstOrderApproxRKHS}, it is reasonable to expect that this remainder term will be negligible under appropriate conditions. The following result provides a reasonable condition under which the drift terms will be negligible.
\begin{lemma}[Sufficient condition for negligible drift terms]\label{lem:driftTerm}
Suppose that $\nu$ is pathwise differentiable at $P_0$ with EIF $\phi_0$. For each $j\in\{1,2\}$, $\|\phi_n^j-\phi_0\|_{L^2(P_0;\mathcal{H})}=o_p(1)$ implies that $\|\mathcal{D}_n^j\|_{\mathcal{H}}=o_p(n^{-1/2})$.
\end{lemma}
In the appendix, this lemma is proved via a conditioning argument that makes use of Chebyshev's inequality for Hilbert-valued random variables \citep{grenander1963probabilities} and the dominated convergence theorem. 
In the next section, we provide lower-level conditions under which the remainder and drift terms will be small in the contexts of Examples~\ref{ex:cfdBandlimited} and \ref{ex:gkmed}. 

In the process of showing that $\bar{\nu}_n$ is asymptotically linear, we also show that it is regular. An estimator sequence $(\widetilde{\nu}_n)_{n=1}^\infty$ is said to be regular at $P_0 \in \mathcal{P}$ if and only if there is a tight $\mathcal{H}$-valued random variable $\widetilde{\mathbb{H}}$ such that, for every score in the tangent set, quadratic mean differentiable submodel $\{P_\epsilon: \epsilon\}\in \mathscr{P}(P_0,\mathcal{P},s)$, and every sequence $\epsilon_n = O(n^{-1/2})$, $\sqrt{n}[\widetilde{\nu}_n - \nu(P_{\epsilon_n})]$ converges weakly to $\widetilde{\mathbb{H}}$ under iid sampling of $n$ observations from $P_{\epsilon_n}$. We say that an estimator $\widetilde{\nu}_n$ is regular when the implied estimator sequence $(\widetilde{\nu}_n)_{n=1}^\infty$ is clear from context. In the upcoming theorem, we write `$\rightsquigarrow$' to denote weak convergence in $\mathcal{H}$.
\begin{theorem}[Asymptotic linearity and weak convergence of a one-step estimator]\label{thm:al}
Suppose that $\nu$ is pathwise differentiable at $P_0$ with EIF $\phi_0\in L^2(P_0;\mathcal{H})$ and, for $j\in\{1,2\}$,  $\mathcal{R}_n^j=o_p(n^{-1/2})$ and $\mathcal{D}_n^j = o_P(n^{-1/2})$. Under these conditions, \eqref{eq:al} holds, $\bar{\nu}_n$ is regular, and
\begin{align}
n^{1/2}\left[\bar{\nu}_n - \nu(P_0)\right]&\rightsquigarrow \mathbb{H}, \label{eq:oneStepWeakConv}
\end{align}
where $\mathbb{H}$ is a tight $\mathcal{H}$-valued Gaussian random variable that is such that, for each $h\in\mathcal{H}$, the marginal distribution $\langle\mathbb{H} ,h \rangle_{\mathcal{H}}$ is $N(0,E_0[\langle \phi_0(Z),h \rangle_{\mathcal{H}}^2])$.
\end{theorem}

The convolution theorem for Banach-valued estimators can be used to characterize one sense in which $\bar{\nu}_n$ is an efficient estimator of $\nu(P_0)$ \citep[see Theorem~3.11.2 and Lemma~3.11.4 of][for a convenient version]{van1996weak}. Rather than present the theorem in full generality, here we present a particularly interpretable consequence of it. Specifically, under the conditions of Theorem~\ref{thm:al}, that theorem shows that $\bar{\nu}_n$ is optimally concentrated about $\nu(P_0)$ in the sense that, for any regular estimator sequence $(\widetilde{\nu}_n)_{n=1}^\infty$ and $c\ge 0$,
\begin{align}
\lim_{n\rightarrow\infty}P_0^n\left\{n\|\bar{\nu}_n-\nu(P_0)\|_{\mathcal{H}}^2>c\right\}\le \lim_{n\rightarrow\infty}P_0^n\left\{n\|\widetilde{\nu}_n-\nu(P_0)\|_{\mathcal{H}}^2>c\right\}. \label{eq:convolution}
\end{align}
The above describes a sense in which $\bar{\nu}_n$ is optimal among all regular estimators of $\nu(P_0)$. Under the conditions of Theorem~\ref{thm:al}, $\bar{\nu}_n$ can also be shown to outperform all estimators, including non-regular ones, in a local asymptotic minimax sense --- see Theorem 3.11.5 in \cite{van1996weak} for details.

\subsection{Construction of confidence sets based on one-step estimators}\label{sec:confSets}

As we will now show, the weak convergence result in Theorem~\ref{thm:al} can be used to facilitate the construction of $(1-\alpha)$-level confidence sets for the Hilbert-valued parameter $\nu(P_0)$, where $\alpha\in(0,1)$ is some fixed constant. 
Our proposed confidence set is built based upon a quadratic form $w(\,\cdot\,;\Omega) : h\mapsto \langle \Omega(h),h\rangle_{\mathcal{H}}$ that is parameterized by a standardization operator $\Omega$ that belongs to the set $\mathcal{O}$ of continuous, self-adjoint, positive definite linear operators mapping from $\mathcal{H}$ to $\mathcal{H}$. 
In particular, letting $\zeta\ge 0$ be a specified threshold and $\Omega_n\in \mathcal{O}$ an estimator of a some possibly-$P_0$-dependent operator $\Omega_0\in\mathcal{O}$, our confidence set will take the form
\begin{align}
\mathcal{C}_n(\zeta) := \left\{h \in \mathcal{H} : w(\bar{\nu}_n - h;\Omega_n) \leq \zeta/n\right\}. \label{eq:confSet}
\end{align}
We will see that, for an appropriate choice of $\zeta$ and provided $\|\Omega_n-\Omega_0\|_{\mathrm{op}}=o_p(1)$, the continuous mapping theorem justifies the asymptotic validity of this confidence set. 

Before presenting that result, we discuss two interesting choices of $\Omega_0$. The first is the identity function. This choice yields a spherical confidence set that consists of all $h$ belonging to an $\mathcal{H}$-ball centered at $\bar{\nu}_n$ of radius $(\zeta/n)^{1/2}$. Since the form of $\Omega_0$ does not rely on $P_0$ in this case, $\Omega_n$ can be taken to be equal to $\Omega_0$. The second is proportional to the inverse of a regularized form of the covariance operator of $\mathbb{H}$, which will yield a Wald-type confidence set for $\nu(P_0)$ that has an elliptical shape. 
Compared to spherical confidence sets, Wald-type confidence sets have the benefit of being narrower in the direction of unit vectors $h$ where estimation is easier, in the sense that the variance of $\langle \mathbb{H},h \rangle_{\mathcal{H}}$ is smaller. 
Regularization is needed when defining a Wald-type confidence set because the covariance operator of $\mathbb{H}$ will not generally be invertible and, even if it is, this inverse will not be continuous when $\|\phi_0\|_{L^2(P_0;\mathcal{H})}<\infty$ unless $\mathcal{H}$ is finite-dimensional. A simple example of a regularized covariance operator $\Omega_0$, and an operator-norm consistent estimator thereof, is given in Appendix~\ref{app:regOmega0}. Many other types of regularization are also possible \citep{tikhonov1995numerical}.

The following result establishes the asymptotic validity of the confidence set in \eqref{eq:confSet} based on a threshold $\widehat{\zeta}_n$ that is measurable with respect to the $\sigma$-field generated by the iid sample $Z_1,\ldots,Z_n$ from $P_0$. This threshold is an estimate of the $(1-\alpha)$-quantile $\zeta_{1-\alpha}$ of $w(\mathbb{H};\Omega_0)$.
\begin{theorem}[Asymptotically valid confidence set] \label{thm:CIcoverage}
Suppose the conditions of Theorem~\ref{thm:al} hold. Further suppose that $\|\phi_0\|_{L^2(P_0;\mathcal{H})}>0$, $\Omega_n\in\mathcal{O}$, $\Omega_0\in\mathcal{O}$, and $\|\Omega_n - \Omega_0\|_{\mathrm{op}}=o_p(1)$.
\begin{enumerate}[label=(\roman*)]
    \item\label{it:exactCoverage} if $\widehat{\zeta}_n\rightarrow \zeta_{1-\alpha}$ in probability, then $\lim_{n\rightarrow\infty} P_0^n\{\nu(P_0)\in \mathcal{C}_n(\widehat{\zeta}_n)\}= 1-\alpha$.
    \item\label{it:consCoverage} if $\widehat{\zeta}_n$ is an asymptotically conservative estimator of $\zeta_{1-\alpha}$, in the sense that $P_0^n\{\widehat{\zeta}_n\ge \zeta_{1-\alpha}-\delta\}\overset{n\rightarrow\infty}{\longrightarrow} 1$ for all $\delta>0$, then $\liminf_{n\rightarrow\infty} P_0^n\{\nu(P_0)\in \mathcal{C}_n(\widehat{\zeta}_n)\}\ge 1-\alpha$.
\end{enumerate}
\end{theorem}
The proof of this result is similar in structure to those used to establish the validity of Wald-type confidence sets for finite-dimensional parameters. It consists in applying the continuous mapping theorem and Slutsky's lemma to show that $n\cdot w[\bar{\nu}_n-\nu(P_0);\Omega_n]\rightsquigarrow w(\mathbb{H};\Omega_0)$, showing that $w(\mathbb{H};\Omega_0)$ is a continuous random variable, and finally using that convergence in distribution implies convergence of cumulative distribution functions at continuity points.

A consistent estimator of $\zeta_{1-\alpha}$ can be defined using the bootstrap \citep{efron1979bootstrap}.   
To define this estimator, we let $Z_1^\sharp,\ldots,Z_{n/2}^\sharp\iidsim P_n^2$ be sampled independently of $Z_{n/2+1}^\sharp,\ldots,Z_n^\sharp\iidsim P_n^1$. We then let $P_n^{2,\sharp}$ be the empirical distribution of $Z_1^\sharp,\ldots,Z_{n/2}^\sharp$ and $P_n^{1,\sharp}$ be the empirical distribution of $Z_{n/2+1}^\sharp,\ldots,Z_n^\sharp$. 
We let $\mathbb{H}_n^\sharp:= n^{1/2}\sum_{j=1}^2 (P_n^{j,\sharp}-P_n^j) \phi_n^j/2$. 
The threshold $\widehat{\zeta}_n$ is taken to be equal to the $(1-\alpha)$-quantile of $w(\mathbb{H}_n^\sharp,\Omega_n)$, conditionally on the original sample $(Z_1,\ldots,Z_n)$ used to define $P_n^1$ and $P_n^2$. In practice this quantile can be well-approximated by selecting $m$ Monte Carlo draws of the sample $Z_1^\sharp,\ldots,Z_{n}^\sharp$ and then returning the empirical $(1-\alpha)$-quantile of $w(\mathbb{H}_n^\sharp,\Omega_n)$ over these draws. A computational benefit of this bootstrap procedure is that it does not require refitting the initial estimators $\widehat{P}_n^j$ of $P_0$; indeed, the same EIFs $\phi_n^j:=\phi_{\widehat{P}_n^j}$ are used for each bootstrap replication.
\begin{theorem}[Consistent estimation of $\zeta_{1-\alpha}$ via the bootstrap]\label{thm:threshEst}
Suppose the conditions of Theorem~\ref{thm:al} hold. Further suppose that  $\Omega_n\in\mathcal{O}$, $\Omega_0\in\mathcal{O}$, $\|\Omega_n - \Omega_0\|_{\mathrm{op}}=o_p(1)$, and $\max_{j\in\{1,2\}}\|\phi_n^j-\phi_0\|_{L^2(P_0;\mathcal{H})}=o_p(1)$. Under these conditions, $\widehat{\zeta}_n\rightarrow \zeta_{1-\alpha}$ in probability.
\end{theorem}
In brief, the proof of the above consists in showing that $\mathbb{H}_n^\sharp$ is asymptotically equivalent to $\mathbb{H}_{n,0}^\sharp:= n^{1/2}\frac{1}{2}\sum_{j=1}^2 (P_n^{j,\sharp}-P_n^j) \phi_0$ in probability, invoking a guarantee from \cite{gine1990bootstrapping} regarding the weak convergence of the bootstrap for Hilbert-valued sample means and the continuous mapping theorem to show that $w(\mathbb{H}_{n,0}^\sharp;\Omega_0)\rightsquigarrow w(\mathbb{H};\Omega_0)$ conditionally on $(Z_i)_{i=1}^\infty$ with probability one, and finally applying a Slutsky-type argument to replace the $P_0$-dependent quantities $\mathbb{H}_{n,0}^\sharp$ and $\Omega_0$ in $w(\mathbb{H}_{n,0}^\sharp;\Omega_0)$ by $\mathbb{H}_n^\sharp$ and $\Omega_n$, respectively.  

When $\Omega_0$ is the identity function, Theorem~1 in \cite{szekely2003extremal} provides a means to derive an alternative estimator of $\zeta_{1-\alpha}$. This estimator does not require the bootstrap, but is asymptotically conservative. See Appendix~\ref{app:szekely} for details.

In practice, it will typically be necessary to use numerical techniques to compute the quadratic form $w(\bar{\nu}_n-h;\Omega_n)$ that is used to define our confidence set. We discuss some such approaches in Appendix~\ref{app:practicalNonRegCS}.

\section{Performance guarantees and inference when there is no EIF}\label{sec:perfGuaranteesRegularized}

\subsection{Performance guarantees for regularized one-step estimation}\label{sec:regularizedOneStepGuarantee}

In this subsection, we provide performance guarantees for the cross-fitted $\beta_n$-regularized one-step estimator $\bar{\nu}_n^{\beta_n}$, where, for each $n$, $\beta_n$ is an $\ell_{*}^2$-valued regularization parameter. 
Before doing so, we acknowledge a minor abuse of notation. We will denote the $k^{\mathrm{th}}$ entry of a generic regularization parameter $\beta\in \ell_*^2$ by $\beta_k$, which should not be mistaken for the sample-size-$n$ dependent regularization parameter $\beta_n$, whose $k^{\mathrm{th}}$ entry we will denote by $\beta_{n,k}$. 
This should not cause confusion, as we always denote sample size by $n$ and a generic index of a vector in $\ell_*^2$ by $k$.

We will show that, under conditions, $\bar{\nu}_n^{\beta_n}$ satisfies the following biased and slower-than-$n^{-1/2}$-rate  asymptotically linear expansion, which formalizes the approximation in \eqref{eq:biasedAL} for a cross-fitted one-step estimator:
\begin{align}
\bar{\nu}_n^{\beta_n}-\nu(P_0)&= \frac{1}{2}\sum_{j=1}^2\mathcal{B}_n^{j,\beta_n} + P_n \phi_0^{\beta_n} +  O_p\left(\|\beta_n\|_{\ell^2}/n^{1/2}\right). \label{eq:regularizedMainDecomp}
\end{align}
Above $\phi_0^{\beta_n}$ is as defined in Lemma~\ref{lem:approximateEIF} and $\mathcal{B}_n^{j,\beta_n}:=\mathcal{B}_{\widehat{P}_n^j}^{\beta_n}$ denotes the bias term defined below \eqref{eq:biasedAL}. When $\nu$ does not have an EIF, there is generally a tradeoff between the bias term, whose magnitude is smaller when the entries of $\beta_n$ are closer to $1$, and the linear `variance' term $P_n \phi_0^{\beta_n}$, whose magnitude scales as $O_p[\sigma_0(\beta_n)/n^{1/2}]$, where $\sigma_0(\beta_n)=O(\|\beta_n\|_{\ell^2})$ is as defined in Lemma~\ref{lem:approximateEIF}. These two terms will typically be of the same order when $\beta_n$ is selected to minimize the mean-squared error $E_{P_0^n}[\|\bar{\nu}_n^{\beta_n}-\nu(P_0)\|_{\mathcal{H}}^2]$, which makes it so that $\bar{\nu}_n^{\beta_n}-\nu(P_0)$ converges to zero in probability slower than does $n^{-1/2}$. Owing to the bias term, and also to the fact that there is generally not a scaling of 
$P_n\phi_0^{\beta_n}$ that will converge to a nondegenerate, tight random element when $\nu$ does not have an EIF (see Lemma~\ref{lem:notTight} in the appendix), our focus in this subsection will be on deriving rates of convergence for the regularized estimator $\bar{\nu}_n^{\beta_n}$. We provide a means to construct confidence sets for $\nu(P_0)$ in Section~\ref{sec:confSetsRegularized}.

To establish \eqref{eq:regularizedMainDecomp}, we introduce regularized versions of the drift and remainder terms considered in Section~\ref{sec:oneStepGuarantee}. In particular, for $j\in\{1,2\}$ and $\beta=(\beta_k)_{k=1}^\infty\in \ell^2$, define the $\mathcal{H}$-valued random elements $\mathcal{D}_n^{j,\beta}:=(P_n^j-P_0)(\phi_n^{j,\beta}-\phi_0^\beta)$ and $\mathcal{R}_n^{j,\beta}:=\mathcal{R}_{\widehat{P}_n^j}^{\beta}$, where, for $P\in\mathcal{P}$,
\begin{align}
\mathcal{R}_P^{\beta}:= \nu(P)-\nu(P_0) + P_0 \phi_P^{\beta} - \sum_{k=1}^\infty (1-\beta_k) \langle\nu(P) - \nu(P_0),h_k\rangle_{\mathcal{H}} h_k. \label{eq:remainder}
\end{align}
The main result of this subsection is as follows.
\begin{theorem}[Rate of convergence of regularized one-step estimator]\label{thm:alReg}
Suppose $\nu$ is pathwise differentiable at $P_0$, $\beta_n\in \ell_{*}^2$ for each $n\in\mathbb{N}$, and both $\mathcal{R}_n^{j,\beta_n}$ and $\mathcal{D}_n^{j,\beta_n}$ are $O_p[\|\beta_n\|_{\ell^2}/n^{1/2}]$ for $j\in\{1,2\}$. Under these conditions, \eqref{eq:regularizedMainDecomp} holds. Moreover, if $\mathcal{B}_n^{j,\beta_n}=O_p(\|\beta_n\|_{\ell^2}/n^{1/2})$ for $j\in\{1,2\}$, then
\begin{align}
\bar{\nu}_n^{\beta_n}-\nu(P_0)&= O_p\left(\|\beta_n\|_{\ell^2}/n^{1/2}\right).\label{eq:regRate}
\end{align}
\end{theorem}
Eq.~\ref{eq:regRate} suggests it is desirable to select $\beta_n$ as small as possible, while still ensuring that the drift, remainder, and bias terms are all $O_p(\|\beta_n\|_{\ell^2}/n^{1/2})$. Below we provide three general results that can aid in establishing these conditions. The first two provide ways to guarantee the drift and remainder terms are of no larger order than the variance term $P_n\phi_0^{\beta_n}$ in \eqref{eq:regularizedMainDecomp}, implying that the rate of convergence is determined by the variance and bias terms. The third makes precise our earlier statement that the bias term is smaller when the entries of $\beta_n$ are closer to $1$, thereby enforcing a lower bound on how small $\beta_n$ can be to ensure that the bias term is of the same order as the variance term and, therefore, \eqref{eq:regRate} holds.

\begin{lemma}[Sufficient condition for negligible regularized drift terms]\label{lem:driftTermRegularized}
Suppose that $\nu$ is pathwise differentiable at $P_0$ and $(r_n)_{n=1}^\infty$ is a nonnegative sequence. Fix $j\in\{1,2\}$ and, for each $n\in\mathbb{N}$, let $\beta_n\in \ell_{*}^2$. If $\|\phi_n^{j,\beta_n}-\phi_0^{\beta_n}\|_{L^2(P_0;\mathcal{H})}=o_p(r_n)$, 
then $\|\mathcal{D}_n^{j,\beta_n}\|_{\mathcal{H}}=o_p(r_n/n^{1/2})$.
\end{lemma}
By taking $r_n=\|\beta_n\|_{\ell^2}$, the above gives a condition for $\mathcal{D}_n^{j,\beta_n}$ to be $o_p(\|\beta_n\|_{\ell^2}/n^{1/2})$, and therefore $O_p(\|\beta_n\|_{\ell^2}/n^{1/2})$. 
For the regularized remainder term, the following can be useful.
\begin{lemma}[Bound on regularized remainder term]\label{lem:regRemBd}
Fix $j\in\{1,2\}$ and $\beta\in \ell_{*}^2$. If $\nu$ is pathwise differentiable at $P\in\mathcal{P}$, $P_0\ll P$, and $\phi_P^{\beta}\in L^2(P_0;\mathcal{H})$, then
\begin{align*}
\|\mathcal{R}_P^\beta\|_{\mathcal{H}}^2&= \sum_{k=1}^\infty \beta_k^2 \cdot \left(\mathscr{R}_{P,k}\right)^2\le \|\beta\|_{\ell^2}^2\sup_{k\in\mathbb{N}}\left(\mathscr{R}_{P,k}\right)^2,
\end{align*}
where $\mathscr{R}_{P,k}:=\left\langle\nu(P) - \nu(P_0),h_k\right\rangle_{\mathcal{H}} + P_0 \dot{\nu}_{P}^*(h_k)$.
\end{lemma}
For a given $k\in\mathbb{N}$, $\mathscr{R}_{P,k}$ corresponds to the remainder term in a von Mises expansion of the real-valued parameter $\psi_k : P'\mapsto \langle \nu(P'),h_k\rangle_{\mathcal{H}}$ \citep{mises1947asymptotic}, where we note that the pathwise differentiability of $\nu$ at $P$ implies the pathwise differentiability of $\psi_k$ at $P$ with canonical gradient $\dot{\nu}_P^*(h_k)$. 

We now turn to the bias term. For $u\ge 0$, let $\|\cdot\|_u : \mathcal{H}\rightarrow [0,+\infty]$ denote the norm defined by $\|h\|_u^2:=\sum_{k=1}^\infty k^{2u}\langle h,h_k\rangle_{\mathcal{H}}^2$, where the dependence of $\|\cdot\|_u$ on the basis $(h_k)_{k=1}^\infty$ used to construct the regularized one-step estimator is suppressed in the notation.
\begin{lemma}[Bound on bias term]\label{lem:biasTerm}
For any $u\ge 0$ and $\beta\in \ell_{*}^2$, $\|\mathcal{B}_P^\beta\|_{\mathcal{H}}\le \|\nu(P)-\nu(P_0)\|_u \sup_{k\in\mathbb{N}} (1-\beta_k)/k^{u}$. 
If there exists $K\in\mathbb{N}$ such that $\beta_k=1$ for all $k\le K$ and $\beta_k=0$ for all $k>K$, then this implies that
\begin{align}
    \|\mathcal{B}_P^\beta\|_{\mathcal{H}}&\le (K+1)^{-u}\|\nu(P)-\nu(P_0)\|_u\le  2(K+1)^{-u}\sup_{P'\in\mathcal{P}}\|\nu(P')\|_u. \label{eq:biasK}
\end{align}
\end{lemma}
Naturally, the upper bounds are only informative if the evaluations of $\|\cdot\|_u$ upon which they rely are finite. Conditions for the finiteness of this norm have been evaluated in several settings. In particular, $\{h\in\mathcal{H} : \|h\|_u<\infty\}$ corresponds to a periodic Sobolev space when $\mathcal{H}=L^2([0,1])$ and $(h_k)_{k=1}^\infty$ is the trigonometric basis \citep[Proposition 1.14 of][]{tsybakov2009nonparametric}, Sobolev-Laguerre space when $\mathcal{H}=L^2([0,\infty))$ and $(h_k)_{k=1}^\infty$ consists of the Laguerre functions \citep{bongioanni2008sobolev}, and Sobolev-Hermite space when $\mathcal{H}=L^2(\mathbb{R})$ and $(h_k)_{k=1}^\infty$ consists of the Hermite functions \citep{bongioanni2006sobolev}.

In Section~\ref{sec:exEst}, we study the selection of $\beta_n$ in the context of our examples. We do this by leveraging the bounds from the preceding three lemmas and then deriving the choice of $\beta_n$ that balances the variance and bias terms.

\subsection{Construction of confidence sets}\label{sec:confSetsRegularized}

In what follows we fix $\beta\in \ell_{*}^2$ and define $\Gamma_\beta : \mathcal{H}\rightarrow\mathcal{H}$ as $\Gamma_\beta(h)=\sum_{k=1}^\infty \beta_k\langle h,h_k\rangle_{\mathcal{H}}h_k$. 
The following is the key observation that we use to construct our confidence set for $\nu(P_0)$.
\begin{lemma}\label{lem:regParamEIF}
If $\nu$ is pathwise differentiable at $P$, then its transformation $\nu^{\beta}:=\Gamma_\beta\circ \nu$ is pathwise differentiable at $P$ with local parameter $\dot{\nu}_P^\beta:= \Gamma_\beta\circ\dot{\nu}_P$ and EIF $\phi_P^\beta\in L^2(P;\mathcal{H})$.
\end{lemma}
Since $\nu^\beta$ has an EIF, the methods from Section~\ref{sec:confSets} can be used to construct a confidence set for $\nu^\beta(P_0)$ based on a one-step estimator. This one-step estimator takes the form $\widetilde{\nu}_n^\beta:= \frac{1}{2}\sum_{j=1}^2 [\Gamma_\beta\circ \nu(\widehat{P}_n^j) + P_n^j\phi_n^{j,\beta}]$. By Theorem~\ref{thm:CIcoverage}, the main condition for the asymptotic validity of the resulting confidence set is that $\widetilde{\nu}_n^\beta$ is an asymptotically linear estimator of $\nu^\beta(P_0)$ with influence function $\phi_0^\beta$. 
Since $\widetilde{\nu}_n^\beta$ is a one-step estimator of $\nu^\beta(P_0)$, rather than $\nu(P_0)$, $\widetilde{\nu}_n^\beta$ generally differs from the regularized one-step estimator $\bar{\nu}_n^\beta$ of $\nu(P_0)$ --- indeed, $\widetilde{\nu}_n^\beta-\bar{\nu}_n^\beta=\frac{1}{2}\sum_{j=1}^2 [\Gamma_\beta\circ \nu(\widehat{P}_n^j)-\nu(\widehat{P}_n^j)]$. 
Nevertheless, conditions on the same regularized remainder and drift terms studied to establish rate guarantees for $\bar{\nu}_n^{\beta_n}$ can ensure the asymptotic linearity of $\widetilde{\nu}_n^\beta$ --- see Corollary~\ref{cor:nuTildeAL} in the appendix.

Let $\mathcal{C}_n^\beta(\widehat{\zeta}_n)$ denote an asymptotically valid $(1-\alpha)$-level confidence set for $\nu^\beta(P_0)$ constructed according to the methods in Section~\ref{sec:confSets}. Any standardization operator $\Omega_n$ satisfying the conditions of Theorem~\ref{thm:CIcoverage} may be used when doing this. For example, if $\Omega_n$ is taken to be the identity, then, for a cutoff $\widehat{\zeta}_n$ selected via the bootstrap, a spherical confidence set for $\nu^\beta(P_0)$ would take the form
\begin{align}
\mathcal{C}_n^\beta(\widehat{\zeta}_n) := \left\{h \in \mathcal{H} : \|\widetilde{\nu}_n^{\beta}-h\|_{\mathcal{H}}^2 \le \widehat{\zeta}_n/n\right\}. \label{eq:confSetRegularized}
\end{align}
Since the methods in Section~\ref{sec:confSets} require the parameter of interest to be fixed and not depend on sample size, when constructing $\mathcal{C}_n^\beta(\widehat{\zeta}_n)$ we require the choice of $\beta$ to remain fixed as $n\rightarrow\infty$. Handling cases where $\beta$ changes with $n$ or is selected data-adaptively is an interesting area for future work. To transform the confidence set for $\nu^\beta(P_0)$ into one for $\nu(P_0)$, we take the preimage $\Gamma_\beta^{-1}[\mathcal{C}_n^\beta(\widehat{\zeta}_n)]:=\{h\in\mathcal{H}: \Gamma_\beta(h)\in \mathcal{C}_n^\beta(\widehat{\zeta}_n)\}$. Since $\nu^{\beta}(P_0):=\Gamma_\beta\circ \nu(P_0)$, this preimage is an asymptotically valid $(1-\alpha)$-level confidence set for $\nu(P_0)$ provided $\mathcal{C}_n^\beta(\widehat{\zeta}_n)$ is an asymptotically valid $(1-\alpha)$-level confidence set for $\nu^{\beta}(P_0)$. The transformation $\Gamma_\beta$ has a left inverse if all entries of $\beta$ are nonzero, with $\Gamma_\beta^{-1}(h)=\sum_{k=1}^\infty \beta_k^{-1}\langle h,h_k\rangle_{\mathcal{H}}h_k$ for $h$ in the image of $\Gamma_\beta$. Figure~\ref{fig:CSillustration} illustrates how the map $\Gamma_\beta^{-1}$ stretches spherical and Wald-type confidence sets for the regularized parameter $\nu^\beta(P_0)$ into confidence sets for $\nu(P_0)$.

\begin{figure}[tb]
    \centering
    \includegraphics[width=\textwidth]{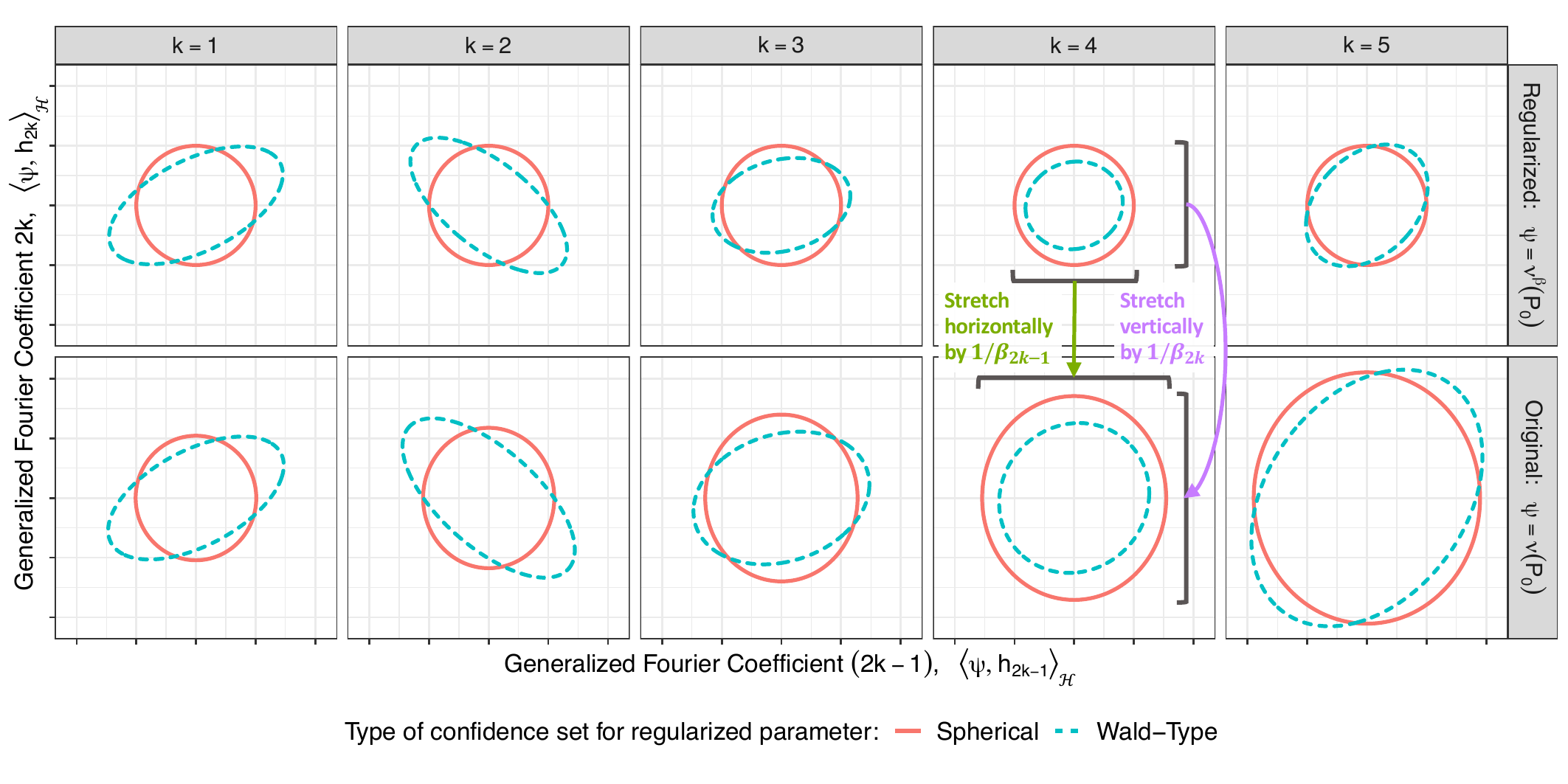}
    \caption{Projections of the boundaries of confidence sets for the regularized parameter $\nu^\beta(P_0)$ (top) and original parameter $\nu(P_0)$ (bottom), plotted via pairs of generalized Fourier coefficients with respect to the basis $(h_k)_{k=1}^\infty$. The transformation applied to confidence sets for $\nu^{\beta}(P_0)$ to obtain those for $\nu(P_0)$ stretch the axes by the reciprocals of entries of the regularization parameter $\beta$. Since $\beta_k\rightarrow 0$ as $k\rightarrow\infty$, this stretch factor becomes arbitrarily large as $k\rightarrow\infty$.}
    \label{fig:CSillustration}
\end{figure}

To simplify the discussion, hereafter we focus on the special case where $\mathcal{C}_n^\beta(\widehat{\zeta}_n)$ takes the spherical form in \eqref{eq:confSetRegularized}. In this case, $\Gamma_\beta^{-1}[\mathcal{C}_n^\beta(\widehat{\zeta}_n)]$ takes the elliptical form
\begin{align}
&\left\{h\in\mathcal{H} : {\textstyle\sum_{k=1}^\infty} \beta_k^2 \left[{\textstyle\frac{1}{2}\sum_{j=1}^2}\left\{\langle \nu(\widehat{P}_n^j),h_k\rangle_{\mathcal{H}} + P_n^j \dot{\nu}_n^{j,*}(h_k)\right\} - \langle h,h_k\rangle_{\mathcal{H}}\right]^2\le \widehat{\zeta}_n/n\right\}. \label{eq:ellipticalCS}
\end{align}
Because $\beta_k$ must tend to zero as $k\rightarrow\infty$ in order for $\beta$ to belong to $\ell^2$, the $\|\cdot\|_{\mathcal{H}}$-diameter of this confidence set, namely $\sup_{h,h'\in\Gamma_\beta^{-1}[\mathcal{C}_n^\beta(\widehat{\zeta}_n)]}\|h-h'\|_{\mathcal{H}}$, will not converge to zero with sample size. In contrast, the $\|\cdot\|_{\beta}$-diameter of this confidence set will generally shrink to zero at an $n^{-1/2}$-rate, where $\|h\|_{\beta}:=[\sum_{k=1}^\infty \beta_k^2 \langle h,h_k\rangle_{\mathcal{H}}^2]^{1/2}$. Here we note that $\|\cdot\|_{\beta}$ is a norm on $\mathcal{H}$ if all of the entries of $\beta$ are nonzero and is otherwise a seminorm.

The confidence set in \eqref{eq:ellipticalCS} also satisfies another desirable property, which can be most easily described by studying a corresponding hypothesis test. 
For fixed $h_0\in\mathcal{H}$, this test rejects the null hypothesis that $\nu(P_0)=h_0$ in favor of the complementary alternative precisely when $h_0\not\in \Gamma_\beta^{-1}[\mathcal{C}_n^\beta(\widehat{\zeta}_n)]$. 
This test asymptotically controls the type I error at level $\alpha$ when $\Gamma_\beta^{-1}[\mathcal{C}_n^\beta(\widehat{\zeta}_n)]$ has asymptotically valid coverage, and, by the triangle inequality, is consistent against fixed alternatives when $\widetilde{\nu}_n^\beta$ is a consistent estimator of $\nu^\beta(P_0)$, $\mathrm{plim}_{n\rightarrow\infty} \,\widehat{\zeta}_n<\infty$, and all entries of $\beta$ are nonzero. 
We now show that this test also achieves nontrivial power against a class of $n^{-1/2}$-rate local alternatives.  
In what follows we let $\mathbb{H}^\beta$ denote a tight $\mathcal{H}$-valued Gaussian random variable that is such that, for each $h\in\mathcal{H}$, the marginal distribution $\langle\mathbb{H}^\beta ,h \rangle_{\mathcal{H}}$ follows a $N(0,E_0[\langle \phi_0^\beta(Z),h \rangle_{\mathcal{H}}^2])$ distribution, where $\phi_0^\beta$ is as defined in Lemma~\ref{lem:regParamEIF}. Unlike in the rest of the paper, the following theorem requires all entries of $\beta$ to be nonzero, since its proof will rely on $\|\cdot\|_\beta$ being a norm.
\begin{theorem}[Local power of regularized hypothesis test]\label{thm:localAlt}
Fix $\beta\in \ell^2\cap (0,1]^{\mathbb{N}}$ and $h_0\in\mathcal{H}$. Suppose $\nu$ is pathwise differentiable at $P_0$, $\nu(P_0)=h_0$, $\widetilde{\nu}_n^\beta$ is an asymptotically linear estimator of $\nu^\beta(P_0)$ with influence function $\phi_0^\beta$, and $\widehat{\zeta}_n$ is a consistent estimator of the $(1-\alpha)$-quantile $\zeta_{1-\alpha}$ of $\|\mathbb{H}^\beta\|_{\mathcal{H}}^2$. 
Fix $\{P_\epsilon: \epsilon\}\in\mathscr{P}(P_0,\mathcal{P},s)$ such that $\|\dot{\nu}_0(s)\|_{\mathcal{H}}>0$. 
If $\Gamma_\beta^{-1}[\mathcal{C}_n^\beta(\widehat{\zeta}_n)]$ is as defined in \eqref{eq:ellipticalCS}, then
\begin{align*}
P_{\epsilon=n^{-1/2}}^n\left\{h_0\not\in \Gamma_\beta^{-1}[\mathcal{C}_n^\beta(\widehat{\zeta}_n)]\right\}&\overset{n\rightarrow\infty}{\longrightarrow} \mathrm{Pr}\left\{\|\mathbb{H}^\beta+\dot{\nu}_0^\beta(s)\|_{\mathcal{H}}^2> \zeta_{1-\alpha}\right\}> \alpha.
\end{align*}
Also, $h_n:=\nu(P_{\epsilon=n^{-1/2}})$ is an $n^{-1/2}$-rate local alternative in that $\|h_n-h_0\|_{\mathcal{H}}=O(n^{-1/2})$.
\end{theorem}
The above focuses on local alternatives that are defined via smooth parametric submodels of $\mathcal{P}$. It is worth noting, however, that by selecting such a submodel, the first-order direction of the local alternative, defined by the value of the local parameter $\dot{\nu}_0(s)$, is fixed as the sample size grows. Since the $\|\cdot\|_{\mathcal{H}}$-diameter of our confidence set does not decay with sample size, it does not appear that our test will generally have nontrivial asymptotic power against local alternatives whose direction is not fixed and whose $\|\cdot\|_{\mathcal{H}}$-magnitude decays at an $n^{-1/2}$ rate.

\subsection{Tuning parameter selection}\label{sec:regularizedOneStepTuning}

We begin by discussing tuning parameter selection for the regularized one-step estimator of $\nu(P_0)$, and then we subsequently discuss confidence set construction. 
Evaluating the regularized one-step estimator requires selecting three key components: the initial estimator $\widehat{P}_n^j$, orthonormal basis $(h_k)_{k=1}^\infty$, and regularization parameter $\beta_n$. Similarly to finite-dimensional problems, the suitability of an initial estimator $\widehat{P}_n^j$ will depend on the parameter of interest and the form of its corresponding remainder term as defined in \eqref{eq:remainder}. In the next section, we will study these remainder terms in our illustrative examples. In what follows we discuss the choice of basis $(h_k)_{k=1}^\infty$ and regularization parameter $\beta_n$.

Following the literature on series estimators \citep{chen2007large} and motivated by the bound in \eqref{eq:biasK}, we suggest choosing the basis $(h_k)_{k=1}^\infty$ so that the span of finitely many initial basis elements yields an accurate approximation of $\nu(P)$, $P\in\mathcal{P}$, provided these Hilbert random elements are smooth enough. Here, smoothness is characterized by the rate of decay of the generalized Fourier coefficients $\langle\nu(P),h_k\rangle_{\mathcal{H}}$ as $k\rightarrow\infty$. 
If $\mathcal{H}$ is an $L^2(\lambda)$ space with $\lambda$ the Lebesgue measure on the real line or a bounded subset thereof, then common choices of bases include Legendre polynomials, Laguerre functions, Hermite functions, trigonometric polynomials, and wavelets, among others. If $\mathcal{H}$ is instead an $L^2(Q)$ space with $Q$ an absolutely continuous probability measure on $\mathbb{R}^d$, then an orthonormal basis for $\mathcal{H}$ can be obtained in several ways.  
One is to multiply an orthornormal basis $(g_k)_{k=1}^\infty$ for $L^2(\mathbb{R}^d)$ by the root-density $q^{1/2}$ of $Q$; in particular, $(q^{1/2}g_k)_{k=1}^\infty$ is an orthonormal basis for $L^2(Q)$.  The suitability of these bases for characterizing the smoothness of $\nu(P_0)$ can be assessed on a case-by-case basis. 
If $d=1$, then another approach involves transforming an orthornormal basis $(g_k)_{k=1}^\infty$ for $L^2([0,1])$ via the cumulative distribution function $F_Q$ of $Q$; in particular, $(g_k\circ F_Q)_{k=1}^\infty$ is an orthonormal basis of $L^2(Q)$.
Other orthonormal bases of $L^2(Q)$-spaces are also readily available for certain choices of $Q$, such as if $Q$ is a Gaussian measure \citep[Chapter 9 of][]{da2006introduction}. 
Orthonormal bases for some non-$L^2$ spaces, such as Sobolev Hilbert spaces, are also well studied \citep{marcellan2015sobolev}.

We propose using cross-validation to choose the regularization parameter $\beta_n$. If there is uncertainty about which orthonormal basis $(h_k)_{k=1}^\infty$ should be used, this could also be selected via cross-validation, though the discussion that follows focuses on selecting $\beta_n$. Our proposal is based on the following loss for $\nu(P_0)$, which relies on an estimate $P$ of $P_0$:
\begin{align}
    \mathcal{L}_P(z;h):= \tfrac{1}{2}\|h-\nu(P)\|_{\mathcal{H}}^2 - \dot{\nu}_P^*[h-\nu(P)](z). \label{eq:lossDef}
\end{align}
The algorithm to implement the proposed cross-validation scheme can be found in Appendix~\ref{app:CV}. There, we also explain why $\mathcal{L}_P$ is a reasonable loss function to use for estimating $\nu(P_0)$. The cross-validation algorithm will be easiest to implement when the search for a regularization parameter is reduced to a search over a finite subset $B_n$ of $\ell_{*}^2$. A particularly simple choice of $B_n$ consists of the $K_n+1$ elements of $\ell^2$ that take the value $1$ in their first $k\in\{0,1,\ldots,K_n\}$ entries and zero in all remaining entries. 
Selecting over a finite set of possible values is also desirable since there are oracle inequalities for cross-validation selectors over finite sets provided the loss function satisfies appropriate regularity conditions \citep{van2003unifiedCV,vaart2006oracle}. Exploring the applicability of these conditions in our setting is an interesting area for future study.

We now turn to tuning parameter selection for confidence set construction. The considerations for selecting the basis $(h_k)_{k=1}^\infty$ are similar to those discussed above for estimation, and so we focus on selecting the regularization parameter $\beta$. As our coverage guarantees rely on the regularization parameter $\beta$ being fixed and not depending on sample size, cross-validation should not be used to select this quantity. Instead, we recommend choosing a fixed, square-summable sequence $\beta$. One natural family of choices is given by setting $\beta=(\beta_k)_{k=1}^\infty$ with $\beta_k=1/[1+(k/c)^{1/2+d}]$ for $c,d>0$. The parameters $c$ and $d$ control the stretch and polynomial rate of decay of the function $k\mapsto \beta_k$, respectively. 
Finally, we note that, to ensure computational feasibility, the infinite sum used to define the confidence set in \eqref{eq:ellipticalCS} can be truncated at a large, finite number of terms $K_n^\star$ that grows with $n$, without adversely affecting coverage. This follows from the fact that the set on the right-hand side of \eqref{eq:ellipticalCS} can only be made larger by replacing the sum from $k=1$ to $\infty$ with one from $k=1$ to $K_n^*$.

\section{Study of (regularized) one-step estimators in our examples}\label{sec:exEst}

We now revisit Examples~\ref{ex:cfdNonparametric} and \ref{ex:cfdBandlimited} from Section~\ref{sec:exPD}. For each, we evaluate the plausibility of the regularity conditions that guarantee our theoretical results hold. We revisit the other two examples from Section~\ref{sec:exPD} in Appendices~\ref{app:conacOS} and \ref{app:gkmedOS}. In what follows, $C$ denotes a generic finite constant whose value may differ from display to display. 

\begin{example}[name=Counterfactual density function,continues=ex:cfdNonparametric]
Since there is no EIF in this example, we study a regularized one-step estimator $\bar{\nu}_n^{\beta_n}$. This estimator is defined based on an orthonormal basis $(h_k)_{k=1}^\infty$ of $\mathcal{H}$ and a regularization parameter $\beta_n\in\ell_{*}^2$. Guidance on how to choose these quantities is given in Section~\ref{sec:regularizedOneStepTuning}.

Theorem~\ref{thm:alReg} relies the negligibility of regularized remainder and drift terms $\mathcal{R}_n^{j,\beta_n}$ and $\mathcal{D}_n^{j,\beta_n}$ and bias terms $\mathcal{B}_n^{j,\beta_n}$. In Appendix~\ref{app:cdRegRem}, we use Lemma~\ref{lem:regRemBd} and the strong positivity assumption to show there exists a constant $C$ that does not depend on $\beta$ such that, for all $P\in\mathcal{P}$,
\begin{align}
   \|\mathcal{R}_P^\beta\|_{\mathcal{H}}&\le C\|\beta\|_{\ell^2}\left\|g_P(1\mid \cdot)-g_0(1\mid \cdot)\right\|_{L^2(P_{0,X})}\left\|p_{Y\mid A=1,X}-p_{0,Y\mid A=1,X}\right\|_{L^2(\tau_0)}, \label{eq:cfdNonpRem}
\end{align}
where $\tau_0$ denotes the product measure $\lambda_Y\times P_{0,X}$ and, within the $L^2(\tau_0)$ norm, $p_{Y\mid A=1,X}-p_{0,Y\mid A=1,X}$ denotes the function $(y,x)\mapsto p_{Y\mid A,X}(y\mid 1,x)-p_{0,Y\mid A,X}(y\mid 1,x)$. 
The upper bound in \eqref{eq:cfdNonpRem} depends on three quantities: the $\ell^2$-magnitude of $\beta$, the $L^2(P_{0,X})$-distance between the propensities $g_P(1\mid \cdot\,)$ and $g_0(1\mid \cdot\,)$, and a root-MISE of the conditional distribution of $Y\mid A=1,X$ under $P$ relative to that under $P_0$, where the mean is taken across values of $X\sim P_0$. Applying the above inequality to study the remainder term $\mathcal{R}_n^{j,\beta_n}$ that Theorem~\ref{thm:alReg} requires to be $O_p[\|\beta_n\|_{\ell^2}/n^{1/2}]$, we see that $\mathcal{R}_n^{j,\beta_n}$ will satisfy this condition provided typical $n^{-1/4}$-rate conditions are satisfied by the estimators of two nuisance functions, namely the propensity to receive treatment and the conditional density of the outcome given treatment and covariates. Such conditions have been discussed extensively in the literature across a variety of problems \citep[e.g,.][]{van2006targeted,chernozhukov2018double}, and tend to hold when the needed nuisance functions are sufficiently smooth or parsimonious relative to the dimension of $X$ and an appropriate estimation strategy is used. For example, suppose that $X$ is continuous and $\mathbb{R}^d$-valued, $g_0(1\mid \cdot\,)$ and $(x,y)\mapsto p_{0,Y\mid A,X}(y\mid 1,x)$ are H\"{o}lder smooth with H\"{o}lder exponents $b$ and $c$, respectively \citep{robins2008higher}. If $g_0(1\mid \cdot\,)$ is estimated via a kernel regression and $(x,y)\mapsto p_{0,Y\mid A,X}(y\mid 1,x)$ is estimated via conditional kernel density estimation, each using kernels of sufficiently high orders, then the above can be used to show that $\mathcal{R}_n^{j,\beta_n}=O_p[\|\beta_n\|_{\ell^2}n^{-\frac{b}{2b+d}}n^{-\frac{c}{2c+d+1}}]$, and so $\mathcal{R}_n^{j,\beta_n}$ achieves the desired $O_p[\|\beta_n\|_{\ell^2}/n^{1/2}]$ rate provided $bc\ge d(d+1)/4$. Alternative estimation strategies that often perform well in practice even when these smoothness assumptions fail, such as those based on random forests \citep{ho1995random} or gradient boosting \citep{friedman2001greedy}, could also be used.

For the regularized drift terms, Lemma~\ref{lem:driftTermRegularized} shows that $\mathcal{D}_n^{j,\beta_n}$ is $o_p(\|\beta_n\|_{\ell^2}/n^{1/2})$ whenever $\|\phi_n^{j,\beta}-\phi_0^\beta\|_{L^2(P_0;\mathcal{H})}=o_p(\|\beta_n\|_{\ell^2})$. To provide conditions under which this is true, we use that there exists a constant $C>0$ that does not depend on $\beta$ such that, for all $P\in\mathcal{P}$, $\|\phi_P^{\beta}-\phi_0^\beta\|_{L^2(P_0;\mathcal{H})}$ is upper bounded by
\begin{align*}
C\|\beta\|_{\ell^2}\left(\left\|g_P(1\mid \cdot)-g_0(1\mid \cdot)\right\|_{L^2(P_{0,X})} + \left\|p_{Y\mid A=1,X}-p_{0,Y\mid A=1,X}\right\|_{L^2(\tau_0)}\right).
\end{align*}
Hence, $\|\phi_n^{j,\beta_n}-\phi_0^{\beta_n}\|_{L^2(P_0;\mathcal{H})}=o_p(\|\beta_n\|_{\ell^2})$ whenever the propensity and conditional density of $Y\,|\, A=1,X$ under $\widehat{P}_n^j$ are consistent according to the norms above. Consistency is a weaker requirement than the rate conditions imposed to ensure the negligibility of $\mathcal{R}_n^{j,\beta_n}$, so it is reasonable to expect that $\mathcal{D}_n^{j,\beta_n}$ will be negligible when $\mathcal{R}_n^{j,\beta_n}$ is negligible.

From Lemma~\ref{lem:biasTerm}, an upper bound on the rate at which the bias terms $\mathcal{B}_n^{j,\beta_n}$ will decay to zero can be derived by bounding either $\|\nu(\widehat{P}_n^j)-\nu(P_0)\|_u$ or $\sup_{P\in\mathcal{P}}\|\nu(P)\|_u$ for some $u\ge 0$. The latter of these quantities is no more than $c<\infty$ if the parameter space $\{\nu(P) : P\in\mathcal{P}\}$ is a subset of the Sobolev ellipsoid $\{h\in\mathcal{H} : \|h\|_u\le c\}$. In this case, when the first $K_n$ entries of $\beta_n$ are one and all others are zero, Lemma~\ref{lem:biasTerm} shows that $\|\mathcal{B}_n^{j,\beta_n}\|_{\mathcal{H}}\le c/(K_n+1)^u$; if the earlier-discussed regularity conditions hold so that the regularized remainder and drift terms are $O_p(\|\beta_n\|_{\ell^2}/n^{1/2})$, then this yields that, when $K_n$ is of the order $n^{1/(2u+1)}$,
\begin{align}
\bar{\nu}_n^{\beta_n}-\nu(P_0)&= O_p(n^{-u/(2u+1)}). \label{eq:cdRate}
\end{align}
This analysis bears similarity to the study of projection estimators \citep[Theorem~1.9 of][]{tsybakov2009nonparametric}, but with the added requirement that drift and remainder terms must be considered.

The rate of convergence in \eqref{eq:cdRate} was derived based on the looser of the two bounds in \eqref{eq:biasK}. While the former bound would give tighter bounds on the bias term when $\|\nu(\widehat{P}_n^j)-\nu(P_0)\|_u$ converges to zero in probability at some rate, it is unclear whether there are initial estimators $\widehat{P}_n^j$ of $P_0$ that would achieve this. Indeed, since $\|\cdot\|_u$ is a stronger norm than $\|\cdot\|_{\mathcal{H}}$, probabilistic convergence relative to $\|\cdot\|_{\mathcal{H}}$ is insufficient to guarantee convergence relative to $\|\cdot\|_u$. Looking to identify or develop initial estimators of $P_0$ for which $\|\nu(\widehat{P}_n^j)-\nu(P_0)\|_u\overset{p}{\rightarrow} 0$ is an interesting area for future study, since, when such an initial estimator is used, faster rates of convergence for $\bar{\nu}_n^{\beta_n}$ than that given in \eqref{eq:cdRate} may be established.

While the discussion above focused on the regularized one-step estimator, similar arguments can be used to analyze the confidence sets introduced in Section~\ref{sec:confSetsRegularized} for fixed $\beta\in \ell_{*}^2$. Indeed, Corollary~\ref{cor:nuTildeAL} in the appendix shows that the key quantities to bound to establish the validity of these confidence sets are $\mathcal{R}_n^{j,\beta}$ and $\mathcal{D}_n^{j,\beta}$ --- in particular, both of these quantities should be $o_p(n^{-1/2})$. We have already bounded these quantities above when studying the regularized one-step estimator of $\nu(P_0)$. In particular, these conditions will hold if each of $p_{0,Y\mid A,X}$ and $g_0(1\mid\cdot\,)$ is estimated at a faster-than-$n^{-1/4}$ rate according to the norms in \eqref{eq:cfdNonpRem}. 

In the special case where $\beta$ is a vector whose first $K$ entries are 1 and whose remaining entries are 0, the estimator $\widetilde{\nu}_n^\beta$ used to construct our confidence sets coincides with the $L^2(\lambda_Y)$ projection estimator studied in Corollary~2 of \cite{kennedy2021semiparametric}. In our notation, that estimator can be viewed as estimating the parameter $\Gamma_\beta\circ\nu(P_0)$, though if $K$ grows with $n$, as would typically occur under the model selection strategy described by Kennedy et al., then it can be viewed as estimating $\nu(P_0)$ as well. This estimator differs from the regularized one-step estimator $\bar{\nu}_n^\beta$ that we have recommended using for estimation of $\nu(P_0)$, with the estimators differing by the $L^2(\lambda_Y)$ projection of $\frac{1}{2}\sum_{j=1}^2 \nu(\widehat{P}_n^j)$ onto the orthogonal complement of the linear span of the first $K$ elements of the chosen basis for $L^2(\lambda_Y)$. It is not immediately clear whether one of these two estimators should be preferred over the other in general, though our upcoming simulation study supports using $\bar{\nu}_n^\beta$, especially when $n$ is small. The decision between using these estimators of $\nu(P_0)$ can be summarized as follows: $\bar{\nu}_n^\beta$ should be used to estimate $\nu(P_0)$ if $\frac{1}{2}\sum_{j=1}^2 \Gamma_\beta\circ \nu(\widehat{P}_n^j)$ attains a lower MISE for estimating $\Gamma_\beta\circ \nu(P_0)$ than does the zero function, and $\widetilde{\nu}_n^\beta$ should be preferred otherwise. 

\cite{kennedy2021semiparametric} also proposes an approach for making inference about the difference between two counterfactual densities using any of several distance metrics. For the $L^2(\lambda_Y)$ metric, this inference is based upon first-order asymptotics the parameter $\psi(P_0):=\|\nu_1(P_0)-\nu_0(P_0)\|_{L^2(\lambda_Y)}^2$, where $\nu_1$ is equal to the counterfactual density parameter $\nu$ defined in \eqref{eq:cdDef} and $\nu_0$ takes the same form but with $p_{Y\mid A,X}(y\mid 1,x)$ replaced by $p_{Y\mid A,X}(y\mid 0,x)$. There, they note an oft-confronted difficulty \citep{luedtke2019omnibus,williamson2021general} wherein their estimator of $\psi(P_0)$ converges to zero at a faster-than-$n^{-1/2}$ rate under the null hypothesis that $\nu_1(P_0)=\nu_0(P_0)$, leading them to propose a conservative threshold to test this null based on the maximum of the estimated standard error of their estimator and $n^{-1/2}$. Since the pathwise differentiability of $\nu_1$ and $\nu_0$ implies the pathwise differentiability of $\nu_1-\nu_0$ --- with efficient influence operator equal to the difference of the efficient influence operators of $\nu_1$ and $\nu_0$ --- our regularized one-step estimation framework provides an alternative, non-conservative means to test this hypothesis by constructing a confidence set for this parameter for fixed $\beta\in \ell^2\cap (0,1]^{\mathbb{N}}$ and checking whether it contains zero.

We now compare our inferential procedure in this example to that of \cite{kennedy2021semiparametric}. We start by comparing the size of the dual confidence sets. The method from Kennedy et al. can be used to construct a confidence set by inverting tests of whether $\psi_h(P_0):=\|\nu_1(P_0)-\nu_0(P_0)-h\|_{L^2(\lambda_Y)}^2$ is equal to zero across values of $h\in L^2(\lambda_Y)$; the threshold for each $h$-dependent test is determined using the same conservative threshold methodology as when $h=0$. The $L^2(\lambda_Y)$ and $\|\cdot\|_\beta$ diameters of this confidence set both decay at rates no faster than $n^{-1/4}$. In contrast, the $\|\cdot\|_\beta$ diameter of our confidence set decays at a quadratically-faster rate of $n^{-1/2}$, while the $L^2(\lambda_Y)$ diameter does not decay at all. Nevertheless, since $\|h\|_\beta=0$ if and only if $\|h\|_{L^2(\lambda_Y)}=0$, our confidence set will exclude any particular $h\not=\nu_1(P_0)-\nu_0(P_0)$ with probability tending to one. As a practical matter, our confidence set will exclude functions $h$ that differ smoothly from $\nu_1(P_0)-\nu_0(P_0)$ relative to the basis $(h_k)_{k=1}^\infty$ at smaller sample sizes than will the $n^{-1/4}$-rate confidence set, and will otherwise require larger sample sizes; here, smoothness is characterized by the decay rate of $\langle \nu_1(P_0)-\nu_0(P_0)-h,h_k\rangle_{\mathcal{H}}$ as $k\rightarrow\infty$. Our dual hypothesis test of whether $\nu_1(P_0)-\nu_0(P_0)=0$ will also satisfy the local power guarantee from Theorem~\ref{thm:localAlt}. We investigate the properties of this test and compare it to the test proposed in \cite{kennedy2021semiparametric} in our upcoming simulation study.
\end{example}

\begin{example}[name=Bandlimited counterfactual density function,continues=ex:cfdBandlimited]
Since there is an EIF in this example, we study a (non-regularized) one-step estimator. Theorem~\ref{thm:al} relies on the negligibility of the remainder and drift terms, namely that they are $o_p(n^{-1/2})$. Let $\underline{\mathcal{R}}_P:= \underline{\nu}(P) + P_0 \underline{\phi}_P -\underline{\nu}(P_0)$ denote the remainder term for a generic $P\in\mathcal{P}$. In Appendix~\ref{app:cdRem} we show that there exists a $C<\infty$ that does not depend on $P\in\mathcal{P}$ such that
\begin{align}
\|\underline{\mathcal{R}}_P\|_{\underline{\mathcal{H}}}&\le C\left\|g_P(1\mid \cdot)-g_0(1\mid \cdot)\right\|_{L^2(P_{0,X})}\left\|p_{Y\mid A=1,X}-p_{0,Y\mid A=1,X}\right\|_{L^2(\tau_0)}. \label{eq:cdRemBandlimitedBound}
\end{align}
Hence, for $\underline{\mathcal{R}}_n^j:=\underline{\mathcal{R}}_{\widehat{P}_n^j}$ to be $o_p(n^{-1/2})$, the products of the rate of convergence of $g_{\widehat{P}_n^j}(1\mid \cdot\,)$ to $g_0(1\mid \cdot\,)$ and $p_{Y\mid A=1,X}$ to $p_{0,Y\mid A=1,X}$ according to the norms above must be faster than $n^{-1/2}$. This results in the same $n^{-1/4}$-type requirement that we discussed below \eqref{eq:cfdNonpRem} for Example~\ref{ex:cfdNonparametric}, except, because we only focus on rates of convergence for regularized one-step estimators (rather than weak convergence), there we only required this product to be at least as fast as $n^{-1/2}$ rather than faster, as we require here. Also similarly to Example~\ref{sec:regularizedOneStepTuning}, for each $j\in\{1,2\}$, $\|\underline{\phi}_n^j-\underline{\phi}_0\|_{L^2(P_0;\underline{\mathcal{H}})}$ can be shown to be $o_p(1)$ provided the propensity and conditional density of $Y\mid A=1,X$ under $\widehat{P}_n^j$ converge to in probability $g_0(1\mid \cdot\,)$ and $p_{0,Y\mid A=1,X}$ according to the norms in \eqref{eq:cdRemBandlimitedBound}. Hence, Lemma~\ref{lem:driftTerm} ensures that the drift term $\underline{\mathcal{D}}_n^j$ is $o_p(n^{-1/2})$ under this condition, and so the conditions of Theorem~\ref{thm:al} hold under reasonable conditions. If the operator $\Omega_0$ used to construct a confidence set for $\underline{\nu}(P_0)$ is fixed, then the conditions of Theorem~\ref{thm:threshEst} are also satisfied, justifying the use of the bootstrap in confidence set construction. If instead the regularized covariance operator described in Appendix~\ref{app:regOmega0} is used, then the bootstrap will still yield an asymptotically valid confidence set for $\underline{\nu}(P_0)$ provided the estimator of $\Omega_0$ described in that appendix is used (see Lemma~\ref{lem:Sigman} for details).
\end{example}

\section{Simulation study}\label{sec:sims}

\subsection{Overview}

We conduct a simulation study to evaluate the finite-sample properties of our one-step estimation framework, both in settings where an EIF exists and in ones where it does not. 
All of these settings involve drawing inferences about the distributions or densities of counterfactual outcomes (Examples~\ref{ex:cfd} and \ref{ex:gkmed}). Our implemented methods are available in the HilbertOneStep R package \citep{HilbertOneStep}.

We consider multiple data-generating processes, each indexed by real-valued probability distributions $Q(0)$ and $Q(1)$. Sampling from a generic such process involves drawing $n$ iid samples from $P_0$, where $n$ takes values in $\{250,500,1000,2000,4000\}$. An observation $Z=(X,A,Y)$ from $P_0$ is sampled as follows:
\begin{align*}
    &(Y(0),Y(1))\sim Q(0)\times Q(1), \hspace{1em} V\,|\, Y(0),Y(1)\sim N(0_5,\mathrm{Id}_5), \\
    &X=\tfrac{V}{2} + (V+1)1\{Y(1)>0\}; \hspace{1em} A\mid Y(0),Y(1),V,X\sim \mathrm{Bern}\left[\tfrac{1}{20} + \tfrac{9}{10}\cdot \mathrm{expit}(X_1)\right],
\end{align*}
and then letting $Y = AY(1) + (1-A)Y(0)$. 
For $a\in \{0,1\}$, $Y(a)$ is the counterfactual outcome if treatment $A=a$ were assigned. Since $Y(a)\indep A\mid X$ and the positivity assumption is satisfied, the density of $Y(a)$ takes the form in \eqref{eq:cdDef} when $a=1$ and otherwise is the same but with $p_{Y\mid A,X}(y\mid 1,x)$ replaced by $p_{Y\mid A,X}(y\mid 0,x)$. Unless otherwise specified, estimates of performance are based on 1000 Monte Carlo repetitions.

We estimate all needed nuisance functions using the same approaches as \cite{kennedy2021semiparametric}. In particular, we estimate the marginal of $X$ with the empirical distribution, the conditional distribution of $A$ given $X$ using the ranger package \citep{ranger}, and the conditional density of $Y$ given $(A,X)$ using ranger and a Gaussian kernel weighted outcome with bandwidth selected by Silverman's rule. When implementing the quadratic forms used to define our confidence sets, we use grids of 500 points on the support of $Y$ --- see Appendix~\ref{app:practicalNonRegCS} for details.

\subsection{Performance of the regularized one-step estimator in Example~\ref{ex:cfdNonparametric}}

We evaluate the performance of the regularized one-step estimator of the counterfactual density of $Y(1)$ from Example~\ref{ex:cfdNonparametric} when $\mathcal{H}=L^2([0,1])$. 
We consider three choices of $Q(1)$, which are displayed in Figure~\ref{fig:Y1dens} in the appendix. The behavior of $Y(1)$ on the boundaries of its support, namely 0 and 1, differs across the three settings; to emphasize this, we label them `zero on both sides', `nonzero on both sides', and `spike on left side'. 
The cross-validation strategy outlined in Section~\ref{sec:regularizedOneStepTuning} is used to select the regularization parameter $\beta$ over the elements of $\ell^2$ that take the value 1 in their first $K\le 16$ entries and 0 in all others. To evaluate sensitivity to the choice of basis, we evaluate our estimator based on the cosine basis, with $h_k(y)= 2^{1/2}\mathrm{cos}[\pi(k-1)y]$, and a rescaled Legendre basis, with $h_k(y)$ proportional to the $k$-th Legendre function applied to $2y-1$. We also evaluate the use of cross-validation to select between these bases.

Performance is compared to that of a plugin estimator and also the estimation strategy implemented in the npcausal package \citep{kennedy2021semiparametric}, which is a series estimator of the $L^2([0,1])$ projection of the counterfactual density onto the first $K$ terms of the cosine basis. We use the cross-validation scheme implemented in that package to select a value of $K\le 16$. 
Though npcausal does not return the estimated density function, we tweaked its open-source code to extract this information.

\begin{figure}[tb]
    \centering
    \includegraphics[width=\textwidth]{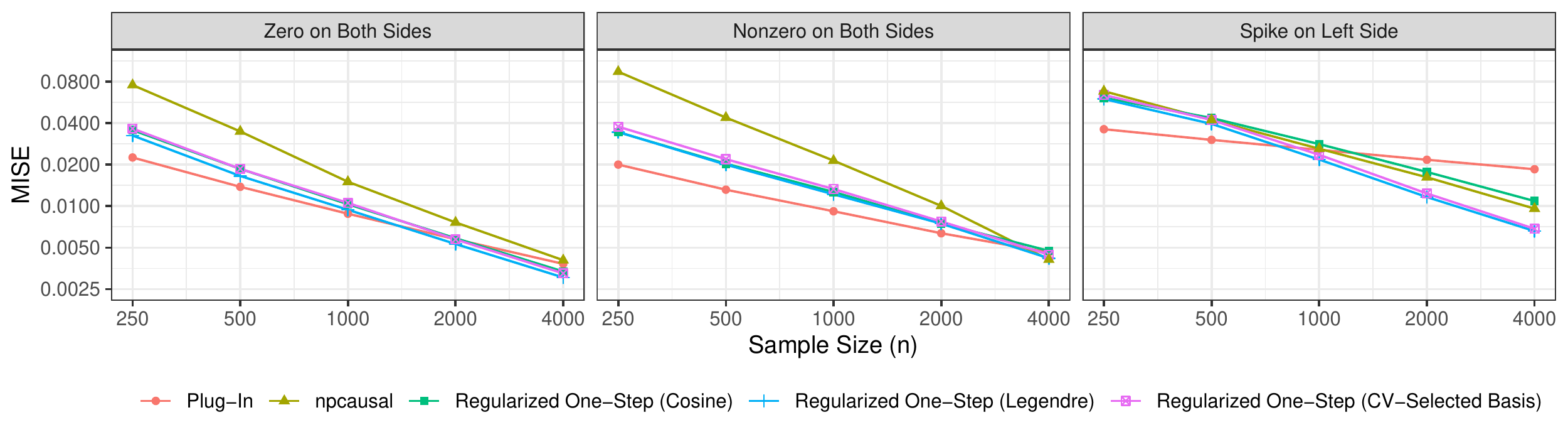}
    \caption{Mean integrated squared error (MISE) versus sample size ($n$) in Example~\ref{ex:cfdNonparametric} for five estimators across the three data generating processes considered. Both axes are log-transformed.
    }
    \label{fig:nonparametricEst}
\end{figure}

Figure~\ref{fig:nonparametricEst} displays the estimators' MISEs. In all settings, the regularized one-step estimators are outperformed by the plug-in estimator at small sample sizes, but their relative performances improve as $n$ grows and eventually exceed or trend towards exceeding that of the plugin. 
Compared to the regularized one-step estimator with the cosine basis, npcausal has MISE that is twice as large at small sample sizes in two of the three scenarios. 
In one of these scenarios, npcausal's performance improves with $n$, but is still worse than all the other estimators. In the other, its performance dramatically improves between $n=2000$ to $4000$ from the worst of all the estimators to slightly better than the others. In the remaining scenario, npcausal and the regularized one-step estimator with the cosine basis perform similarly. Among the regularized one-step estimators, using the Legendre basis outperforms using the cosine basis in one scenario, while the two perform similarly otherwise. Selecting the basis via cross-validation yields an estimator that is about as good as the one based on the Legendre basis in all scenarios.

\subsection{Properties of hypothesis tests from Examples~\ref{ex:cfdNonparametric} and \ref{ex:gkmed}}

We evaluate 5\% level tests of the null hypothesis that $Q(1) = Q(0)$ against the complementary alternative. The first class of tests uses the results from Example~\ref{ex:cfdNonparametric} to check if zero is included in an $L^2([0,1])$ confidence set for the difference $\nu(P_0) = \nu_1(P_0) - \nu_0(P_0)$ of the densities of $Q(1)$ and $Q(0)$. We obtain spherical and Wald-type $L^2([0,1])$ confidence sets for the regularized parameter $\nu^\beta(P_0)$ using cosine and Legendre bases. Both are then transformed into elliptical confidence sets for $\nu(P_0)$ using the approach from Section~\ref{sec:confSetsRegularized}. The Wald-type confidence sets are defined with the correlation-based standardization operator $\Omega_n$ from Appendix~\ref{app:practicalNonRegCS} with $\lambda=0.5$. For the regularization parameter $\beta$, we let $\beta_k = 1/[1+(k/c)^2]$ and consider values of $c\in{2.5, 5, 10}$. Results for $c=5$ are reported in the main text, while others appear in the appendix. Additionally, we examine a test based on the Gaussian kernel MMD between $Q(1)$ and $Q(0)$, which depends on a bandwidth choice of $0.5$, $1$, and $2$ times $\textnormal{median}\{|Y_i-Y_j| : 1\le i<j\le n\}$ \citep{garreau2017large}. We report results for the middle value in the main text and others in the appendix. We compare performance to the asymptotically conservative test from \cite{kennedy2021semiparametric}, implemented using npcausal.

We set $Q(1)$ to its value from the `nonzero on both sides' simulation setting and consider different values of $Q(0)$. We explore the null hypothesis with $Q(0)=Q(1)$, and, for $k\in\{1,2,\ldots,7\}$, the alternative hypothesis with $\nu_1(P_0)(y)-\nu_0(P_0)(y)=\cos(k^2\pi y)$, denoted as `Alt $k^2$'. By examining these alternatives, we assess the power decay of our tests for $\nu_1(P_0)-\nu_0(P_0)=0$ as the direction of the alternative corresponds to that of a higher-frequency function in the cosine basis.

\begin{figure}[tb]
    \centering
    \includegraphics[width=\textwidth]{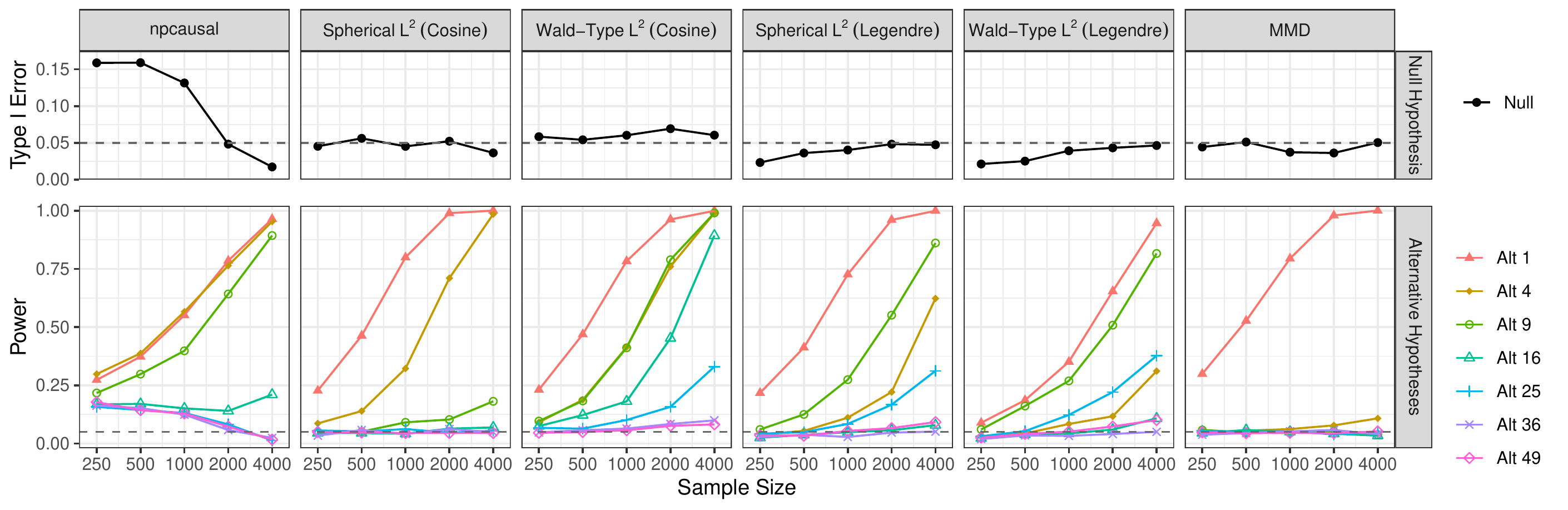}
    \caption{Type I error and power of tests from Examples~\ref{ex:cfdNonparametric} and \ref{ex:gkmed}.
    }
    \label{fig:test_rejProb_medium}
\end{figure}

Figure~\ref{fig:test_rejProb_medium} presents the type I error and power of the tests. In terms of type I error, the five tests based on our one-step estimation framework achieve or nearly achieve the nominal 5\% level. Surprisingly, the npcausal test has a type I error over three times the nominal level at smaller sample sizes, despite having an asymptotic rejection probability of zero. However, as $n$ increases, the type I error converges towards its expected conservative asymptotic behavior. As for power, our tests that regularize using the cosine basis display the anticipated power decay for rejecting Alt $k^2$ as $k$ increases. Tests using the Legendre basis do not exhibit the same monotonic dependence on $k$. The Wald-type test with the cosine basis demonstrates noticeably higher power for larger $k$ alternatives than the spherical test, while this trend is less apparent for the Legendre basis. The MMD test has high power for detecting the smoothest alternative, Alt 1, but low or no power against all others. It is important to note that the alternatives we have considered become quite nonsmooth as $k$ increases, making MMD's poor finite-sample performance for detecting them potentially acceptable. Figure~\ref{fig:test_rejProb_smallLarge} in the appendix illustrates test performance with various tuning parameter choices. Overall, the results align with expectations: for tests based on Example~\ref{ex:cfdNonparametric}, enhanced power against rougher alternatives ($k$ larger) comes at the cost of reduced power against smoother alternatives when later entries of the regularization parameter $\beta$ are increased, and vice versa. For MMD, modifying the bandwidth results in the same tradeoff.

Appendix~\ref*{app:simCfdBandlimited} presents simulation results evaluating our confidence sets for a bandlimited counterfactual density. Nominal coverage is achieved for all sample sizes considered. There, we also highlight the disadvantage of using our $L^2$ confidence sets when a point evaluation of the counterfactual density, rather than the function itself, is the true quantity of interest.

\section{Discussion}

The lack of existence of an EIF that we have confronted in parts of this work bears resemblance to the lack of existence of higher-order influence functions for many real-valued parameters \citep{robins2008higher,van2014higher,robins2017higher}. There, the nonexistence of these objects owes to the lack of a suitable Riesz-representation-type theorem for multilinear forms. In our case, it owes to the fact that the efficient influence operator $\dot{\nu}_P(\cdot)(z) : \mathcal{H}\rightarrow\mathbb{R}$ typically fails to be ($P$-a.s.) bounded and linear, even though $\dot{\nu}_P : \mathcal{H}\rightarrow \dot{\mathcal{P}}_P$ is. Though the technical details of the two problems differ, similar solutions work for both: replace the operator that does not satisfy a Riesz-type representation by an approximation that does. In our case, this involved studying the $\beta$-regularized efficient influence operator $r_P^\beta$. In future work, it would be interesting to investigate the possibility of defining higher-order influence functions for Hilbert-valued parameters. We have shown that, under mild conditions, a first-order EIF exists when the Hilbert space is an RKHS, making this case a natural starting point for exploration.

Another interesting area for future work is to develop a systematic means to select the tuning parameter $\beta$ needed to define our confidence sets when an EIF does not exist. Although our asymptotic guarantees hold for any fixed choice of $\beta\in\ell_{*}^2$ and our numerical studies shed light on how different choices of $\beta$ improve power against different alternatives, it would be desirable to have an automated means of selecting this parameter. One possible approach would involve drawing inspiration from the choice of kernel for two-sample tests based on the MMD \citep{gretton2012kernel}. Like our approach, these tests rely on selecting a fixed tuning parameter --- in their case, a kernel --- as $n$ grows. In that context, an appealing heuristic choice of the bandwidth parameter indexing the kernel has been developed, provided a radial basis function kernel is used \citep{garreau2017large}. It would be of interest to develop a similar heuristic in our setting.

Another area for future study involves extending our results to cases where the Hilbert space depends on the data-generating distribution. As a simple example, in the regression setting from Example~\ref{ex:reg} in the appendix, we may want to evaluate performance relative to $L^2(\lambda_X)$ with $\lambda_X$ equal to the marginal distribution $P_{0,X}$ of $X$ under $P_0$, rather than some fixed known measure, such as the Lebesgue measure. The definition of pathwise differentiability at $P_0$ goes through unchanged in that case. However, because the efficient influence operator at an initial estimate $\widehat{P}_n^j$ of $P_0$ depends on $P_0$ when $\mathcal{H}=L^2(P_{0,X})$, the regularized one-step estimator we have presented in this work cannot be evaluated. A natural workaround would be to modify the definition of this estimator to use the efficient influence operator of $\nu$ at $\widehat{P}_n^j$ relative to a Hilbert space that is indexed by $\widehat{P}_n^j$, rather than $P_0$, and replace $(h_k)_{k=1}^\infty$ by a basis of this space. We leave the study of this estimator to future work.

After establishing that many Hilbert-valued parameters of interest are pathwise differentiable, we focused on developing and studying a particular estimation framework that leverages this property, namely one-step estimation. This framework has the benefit that there is a closed-form expression for the resulting estimators, which simplifies the study of their convergence properties and construction of corresponding confidence sets. While this approach has advantages, it would be worth considering alternative frameworks in future work. As one example, an M-estimator based on the loss that we introduced in \eqref{eq:lossDef} could also be considered. \cite{foster2019orthogonal} offers a general method for determining the convergence rates of these estimators. Additional research is needed to investigate their weak convergence properties and the possibility of using them to construct confidence sets.

\section*{Acknowledgements}

This work was supported by the National Institutes of Health under award number DP2-LM013340 and the National Science Foundation under award number DMS-2210216. The content is solely the responsibility of the authors and does not necessarily represent the official views of the funding agencies.

{\singlespacing
\bibliography{References}

\begin{thebibliography}{}

\bibitem[\protect\citeauthoryear{Agarwal, Chen, and Sarma}{Agarwal
  et~al.}{2015}]{agarwal2015nonparametric}
Agarwal, R., Z.~Chen, and S.~V. Sarma (2015).
\newblock Nonparametric estimation of band-limited probability density
  functions.
\newblock {\em arXiv preprint arXiv:1503.06236\/}.

\bibitem[\protect\citeauthoryear{Aronszajn}{Aronszajn}{1950}]{aronszajn1950theory}
Aronszajn, N. (1950).
\newblock Theory of reproducing kernels.
\newblock {\em Transactions of the American mathematical society\/}~{\em
  68\/}(3), 337--404.

\bibitem[\protect\citeauthoryear{Baiardi and Naghi}{Baiardi and
  Naghi}{2021}]{baiardi2021value}
Baiardi, A. and A.~A. Naghi (2021).
\newblock The value added of machine learning to causal inference: Evidence
  from revisited studies.
\newblock {\em arXiv preprint arXiv:2101.00878\/}.

\bibitem[\protect\citeauthoryear{Berlinet and Thomas-Agnan}{Berlinet and
  Thomas-Agnan}{2011}]{berlinet2011reproducing}
Berlinet, A. and C.~Thomas-Agnan (2011).
\newblock {\em Reproducing kernel Hilbert spaces in probability and
  statistics}.
\newblock Springer Science \& Business Media.

\bibitem[\protect\citeauthoryear{Bickel, Klaassen, Bickel, Ritov, Klaassen,
  Wellner, and Ritov}{Bickel et~al.}{1993}]{bickel1993efficient}
Bickel, P.~J., C.~A. Klaassen, P.~J. Bickel, Y.~Ritov, J.~Klaassen, J.~A.
  Wellner, and Y.~Ritov (1993).
\newblock {\em Efficient and adaptive estimation for semiparametric models},
  Volume~4.
\newblock Johns Hopkins University Press Baltimore.

\bibitem[\protect\citeauthoryear{Bogachev}{Bogachev}{1996}]{bogachev1996gaussian}
Bogachev, V. (1996).
\newblock Gaussian measures on linear spaces.
\newblock {\em Journal of Mathematical Sciences\/}~{\em 79\/}(2), 933--1034.

\bibitem[\protect\citeauthoryear{Bongioanni and Torrea}{Bongioanni and
  Torrea}{2006}]{bongioanni2006sobolev}
Bongioanni, B. and J.~L. Torrea (2006).
\newblock Sobolev spaces associated to the harmonic oscillator.
\newblock In {\em Proceedings of the Indian Academy of Sciences-Mathematical
  Sciences}, Volume 116, pp.\  337--360. Springer.

\bibitem[\protect\citeauthoryear{Bongioanni and Torrea}{Bongioanni and
  Torrea}{2008}]{bongioanni2008sobolev}
Bongioanni, B. and J.~L. Torrea (2008).
\newblock What is a sobolev space for the laguerre function systems?
\newblock {\em Cuadernos de Matem{\'a}tica y Mec{\'a}nica\/}.

\bibitem[\protect\citeauthoryear{Cencov}{Cencov}{1962}]{cencov1962estimation}
Cencov, N.~N. (1962).
\newblock Estimation of an unknown distribution density from observations.
\newblock {\em Doklady Mathematics\/}, 1559--1562.

\bibitem[\protect\citeauthoryear{Chen}{Chen}{2007}]{chen2007large}
Chen, X. (2007).
\newblock Large sample sieve estimation of semi-nonparametric models.
\newblock {\em Handbook of econometrics\/}~{\em 6}, 5549--5632.

\bibitem[\protect\citeauthoryear{Chernozhukov, Chetverikov, Demirer, Duflo,
  Hansen, and Newey}{Chernozhukov et~al.}{2017}]{chernozhukov2017double}
Chernozhukov, V., D.~Chetverikov, M.~Demirer, E.~Duflo, C.~Hansen, and W.~Newey
  (2017).
\newblock Double/debiased/{N}eyman machine learning of treatment effects.
\newblock {\em American Economic Review\/}~{\em 107\/}(5), 261--265.

\bibitem[\protect\citeauthoryear{Chernozhukov, Chetverikov, Demirer, Duflo,
  Hansen, Newey, and Robins}{Chernozhukov
  et~al.}{2018}]{chernozhukov2018double}
Chernozhukov, V., D.~Chetverikov, M.~Demirer, E.~Duflo, C.~Hansen, W.~Newey,
  and J.~Robins (2018).
\newblock Double/debiased machine learning for treatment and structural
  parameters.

\bibitem[\protect\citeauthoryear{Chernozhukov, Newey, and Singh}{Chernozhukov
  et~al.}{2018}]{chernozhukov2018biased}
Chernozhukov, V., W.~Newey, and R.~Singh (2018).
\newblock De-biased machine learning of global and local parameters using
  regularized riesz representers.
\newblock {\em arXiv preprint arXiv:1802.08667\/}.

\bibitem[\protect\citeauthoryear{Chernozhukov, Newey, and Singh}{Chernozhukov
  et~al.}{2021}]{chernozhukov2021simple}
Chernozhukov, V., W.~K. Newey, and R.~Singh (2021).
\newblock A simple and general debiased machine learning theorem with finite
  sample guarantees.
\newblock {\em arXiv preprint arXiv:2105.15197\/}.

\bibitem[\protect\citeauthoryear{Colangelo and Lee}{Colangelo and
  Lee}{2020}]{colangelo2020double}
Colangelo, K. and Y.-Y. Lee (2020).
\newblock Double debiased machine learning nonparametric inference with
  continuous treatments.
\newblock {\em arXiv preprint arXiv:2004.03036\/}.

\bibitem[\protect\citeauthoryear{Da~Prato}{Da~Prato}{2006}]{da2006introduction}
Da~Prato, G. (2006).
\newblock {\em An introduction to infinite-dimensional analysis}.
\newblock Springer Science \& Business Media.

\bibitem[\protect\citeauthoryear{D{\'i}az and van~der Laan}{D{\'i}az and
  van~der Laan}{2013}]{diaz2013targeted}
D{\'i}az, I. and M.~J. van~der Laan (2013).
\newblock Targeted data adaptive estimation of the causal dose--response curve.
\newblock {\em Journal of Causal Inference\/}~{\em 1\/}(2), 171--192.

\bibitem[\protect\citeauthoryear{Efron}{Efron}{1979}]{efron1979bootstrap}
Efron, B. (1979).
\newblock Bootstrap methods: Another look at the jackknife.
\newblock {\em The Annals of Statistics\/}~{\em 7\/}(1), 1--26.

\bibitem[\protect\citeauthoryear{Fawkes, Hu, Evans, and Sejdinovic}{Fawkes
  et~al.}{2022}]{fawkes2022doubly}
Fawkes, J., R.~Hu, R.~J. Evans, and D.~Sejdinovic (2022).
\newblock Doubly robust kernel statistics for testing distributional treatment
  effects even under one sided overlap.
\newblock {\em arXiv preprint arXiv:2212.04922\/}.

\bibitem[\protect\citeauthoryear{Fernique}{Fernique}{1970}]{fernique1970integrabilite}
Fernique, X. (1970).
\newblock Int{\'e}grabilit{\'e} des vecteurs gaussiens.
\newblock {\em CR Acad. Sci. Paris Serie A\/}~{\em 270}, 1698--1699.

\bibitem[\protect\citeauthoryear{Foster and Syrgkanis}{Foster and
  Syrgkanis}{2019}]{foster2019orthogonal}
Foster, D.~J. and V.~Syrgkanis (2019).
\newblock Orthogonal statistical learning.
\newblock {\em arXiv preprint arXiv:1901.09036\/}.

\bibitem[\protect\citeauthoryear{Friedman}{Friedman}{2001}]{friedman2001greedy}
Friedman, J.~H. (2001).
\newblock Greedy function approximation: a gradient boosting machine.
\newblock {\em Annals of statistics\/}, 1189--1232.

\bibitem[\protect\citeauthoryear{Garreau, Jitkrittum, and Kanagawa}{Garreau
  et~al.}{2017}]{garreau2017large}
Garreau, D., W.~Jitkrittum, and M.~Kanagawa (2017).
\newblock Large sample analysis of the median heuristic.
\newblock {\em arXiv preprint arXiv:1707.07269\/}.

\bibitem[\protect\citeauthoryear{Gin{\'e} and Zinn}{Gin{\'e} and
  Zinn}{1990}]{gine1990bootstrapping}
Gin{\'e}, E. and J.~Zinn (1990).
\newblock Bootstrapping general empirical measures.
\newblock {\em The Annals of Probability\/}, 851--869.

\bibitem[\protect\citeauthoryear{Grenander}{Grenander}{1963}]{grenander1963probabilities}
Grenander, U. (1963).
\newblock {\em Probabilities on algebraic structures}.
\newblock Wiley, New York.

\bibitem[\protect\citeauthoryear{Gretton, Borgwardt, Rasch, Sch{\"o}lkopf, and
  Smola}{Gretton et~al.}{2012}]{gretton2012kernel}
Gretton, A., K.~M. Borgwardt, M.~J. Rasch, B.~Sch{\"o}lkopf, and A.~Smola
  (2012).
\newblock A kernel two-sample test.
\newblock {\em The Journal of Machine Learning Research\/}~{\em 13\/}(1),
  723--773.

\bibitem[\protect\citeauthoryear{Hill}{Hill}{2011}]{hill2011bayesian}
Hill, J.~L. (2011).
\newblock Bayesian nonparametric modeling for causal inference.
\newblock {\em Journal of Computational and Graphical Statistics\/}~{\em
  20\/}(1), 217--240.

\bibitem[\protect\citeauthoryear{Hines, Dukes, Diaz-Ordaz, and
  Vansteelandt}{Hines et~al.}{2022}]{hines2022demystifying}
Hines, O., O.~Dukes, K.~Diaz-Ordaz, and S.~Vansteelandt (2022).
\newblock Demystifying statistical learning based on efficient influence
  functions.
\newblock {\em The American Statistician\/}, 1--13.

\bibitem[\protect\citeauthoryear{Ho}{Ho}{1995}]{ho1995random}
Ho, T.~K. (1995).
\newblock Random decision forests.
\newblock In {\em Proceedings of 3rd international conference on document
  analysis and recognition}, Volume~1, pp.\  278--282. IEEE.

\bibitem[\protect\citeauthoryear{Hudson, Carone, and Shojaie}{Hudson
  et~al.}{2021}]{hudson2021inference}
Hudson, A., M.~Carone, and A.~Shojaie (2021).
\newblock Inference on function-valued parameters using a restricted score
  test.
\newblock {\em arXiv preprint arXiv:2105.06646\/}.

\bibitem[\protect\citeauthoryear{Ibragimov and Khas'minskii}{Ibragimov and
  Khas'minskii}{1983}]{ibragimov1983estimation}
Ibragimov, I. and R.~Khas'minskii (1983).
\newblock Estimation of distribution density belonging to a class of entire
  functions.
\newblock {\em Theory of Probability \& Its Applications\/}~{\em 27\/}(3),
  551--562.

\bibitem[\protect\citeauthoryear{Jung, Tian, and Bareinboim}{Jung
  et~al.}{2021}]{jung2021double}
Jung, Y., J.~Tian, and E.~Bareinboim (2021).
\newblock Double machine learning density estimation for local treatment
  effects with instruments.
\newblock {\em Advances in Neural Information Processing Systems\/}~{\em 34},
  21821--21833.

\bibitem[\protect\citeauthoryear{Kennedy}{Kennedy}{2020}]{kennedy2020optimal}
Kennedy, E.~H. (2020).
\newblock Optimal doubly robust estimation of heterogeneous causal effects.
\newblock {\em arXiv preprint arXiv:2004.14497\/}.

\bibitem[\protect\citeauthoryear{Kennedy}{Kennedy}{2022}]{kennedy2022semiparametric}
Kennedy, E.~H. (2022).
\newblock Semiparametric doubly robust targeted double machine learning: a
  review.
\newblock {\em arXiv preprint arXiv:2203.06469\/}.

\bibitem[\protect\citeauthoryear{Kennedy, Balakrishnan, and Wasserman}{Kennedy
  et~al.}{2021}]{kennedy2021semiparametric}
Kennedy, E.~H., S.~Balakrishnan, and L.~Wasserman (2021).
\newblock Semiparametric counterfactual density estimation.
\newblock {\em arXiv preprint arXiv:2102.12034\/}.

\bibitem[\protect\citeauthoryear{Kennedy, Balakrishnan, and Wasserman}{Kennedy
  et~al.}{2022}]{kennedy2022minimax}
Kennedy, E.~H., S.~Balakrishnan, and L.~Wasserman (2022).
\newblock Minimax rates for heterogeneous causal effect estimation.
\newblock {\em arXiv preprint arXiv:2203.00837\/}.

\bibitem[\protect\citeauthoryear{Kennedy, Ma, McHugh, and Small}{Kennedy
  et~al.}{2017}]{kennedy2017non}
Kennedy, E.~H., Z.~Ma, M.~D. McHugh, and D.~S. Small (2017).
\newblock Non-parametric methods for doubly robust estimation of continuous
  treatment effects.
\newblock {\em Journal of the Royal Statistical Society: Series B (Statistical
  Methodology)\/}~{\em 79\/}(4), 1229--1245.

\bibitem[\protect\citeauthoryear{Klaassen}{Klaassen}{1987}]{klaassen1987consistent}
Klaassen, C.~A. (1987).
\newblock Consistent estimation of the influence function of locally
  asymptotically linear estimators.
\newblock {\em The Annals of Statistics\/}~{\em 15\/}(4), 1548--1562.

\bibitem[\protect\citeauthoryear{K{\"u}nzel, Sekhon, Bickel, and Yu}{K{\"u}nzel
  et~al.}{2019}]{kunzel2019metalearners}
K{\"u}nzel, S.~R., J.~S. Sekhon, P.~J. Bickel, and B.~Yu (2019).
\newblock Metalearners for estimating heterogeneous treatment effects using
  machine learning.
\newblock {\em Proceedings of the national academy of sciences\/}~{\em
  116\/}(10), 4156--4165.

\bibitem[\protect\citeauthoryear{Lewandowski, Ryznar, and {\.Z}ak}{Lewandowski
  et~al.}{1995}]{lewandowski1995anderson}
Lewandowski, M., M.~Ryznar, and T.~{\.Z}ak (1995).
\newblock Anderson inequality is strict for gaussian and stable measures.
\newblock {\em Proceedings of the American Mathematical Society\/}~{\em
  123\/}(12), 3875--3880.

\bibitem[\protect\citeauthoryear{Luedtke}{Luedtke}{2023}]{HilbertOneStep}
Luedtke, A. (2023).
\newblock Hilbertonestep r package.
\newblock \url{www.github.com/alexluedtke12/HilbertOneStep}.

\bibitem[\protect\citeauthoryear{Luedtke, Carone, and van~der Laan}{Luedtke
  et~al.}{2019}]{luedtke2019omnibus}
Luedtke, A., M.~Carone, and M.~J. van~der Laan (2019).
\newblock An omnibus non-parametric test of equality in distribution for
  unknown functions.
\newblock {\em Journal of the Royal Statistical Society: Series B (Statistical
  Methodology)\/}~{\em 81\/}(1), 75--99.

\bibitem[\protect\citeauthoryear{Luedtke and Wu}{Luedtke and
  Wu}{2020}]{luedtke2020efficient}
Luedtke, A. and J.~Wu (2020).
\newblock Efficient principally stratified treatment effect estimation in
  crossover studies with absorbent binary endpoints.
\newblock {\em Journal de la soci{\'e}t{\'e} fran{\c{c}}aise de
  statistique\/}~{\em 161\/}(1), 176--200.

\bibitem[\protect\citeauthoryear{Luedtke and van~der Laan}{Luedtke and van~der
  Laan}{2016}]{luedtke2016super}
Luedtke, A.~R. and M.~J. van~der Laan (2016).
\newblock Super-learning of an optimal dynamic treatment rule.
\newblock {\em The international journal of biostatistics\/}~{\em 12\/}(1),
  305--332.

\bibitem[\protect\citeauthoryear{Marcell{\'a}n and Xu}{Marcell{\'a}n and
  Xu}{2015}]{marcellan2015sobolev}
Marcell{\'a}n, F. and Y.~Xu (2015).
\newblock On sobolev orthogonal polynomials.
\newblock {\em Expositiones Mathematicae\/}~{\em 33\/}(3), 308--352.

\bibitem[\protect\citeauthoryear{Muandet, Kanagawa, Saengkyongam, and
  Marukatat}{Muandet et~al.}{2021}]{muandet2021counterfactual}
Muandet, K., M.~Kanagawa, S.~Saengkyongam, and S.~Marukatat (2021).
\newblock Counterfactual mean embeddings.
\newblock {\em J. Mach. Learn. Res.\/}~{\em 22}, 162--1.

\bibitem[\protect\citeauthoryear{Newey and McFadden}{Newey and
  McFadden}{1994}]{newey1994large}
Newey, W.~K. and D.~McFadden (1994).
\newblock Large sample estimation and hypothesis testing.
\newblock {\em Handbook of econometrics\/}~{\em 4}, 2111--2245.

\bibitem[\protect\citeauthoryear{Nie and Wager}{Nie and
  Wager}{2021}]{nie2021quasi}
Nie, X. and S.~Wager (2021).
\newblock Quasi-oracle estimation of heterogeneous treatment effects.
\newblock {\em Biometrika\/}~{\em 108\/}(2), 299--319.

\bibitem[\protect\citeauthoryear{Pfanzagl}{Pfanzagl}{1982}]{pfanzagl1982lecture}
Pfanzagl, J. (1982).
\newblock Lecture notes in statistics.
\newblock {\em Contributions to a general asymptotic statistical theory\/}~{\em
  13}.

\bibitem[\protect\citeauthoryear{Pfanzagl}{Pfanzagl}{1990}]{pfanzagl1990estimation}
Pfanzagl, J. (1990).
\newblock Estimation in semiparametric models.
\newblock In {\em Estimation in Semiparametric Models}, pp.\  17--22. Springer.

\bibitem[\protect\citeauthoryear{Robins}{Robins}{1986}]{robins1986new}
Robins, J. (1986).
\newblock A new approach to causal inference in mortality studies with a
  sustained exposure period application to control of the healthy worker
  survivor effect.
\newblock {\em Mathematical modelling\/}~{\em 7\/}(9-12), 1393--1512.

\bibitem[\protect\citeauthoryear{Robins, Li, Mukherjee, Tchetgen, and van~der
  Vaart}{Robins et~al.}{2017}]{robins2017higher}
Robins, J., L.~Li, R.~Mukherjee, E.~T. Tchetgen, and A.~van~der Vaart (2017).
\newblock Higher order estimating equations for high-dimensional models.
\newblock {\em Annals of statistics\/}~{\em 45\/}(5), 1951.

\bibitem[\protect\citeauthoryear{Robins, Li, Tchetgen, van~der Vaart,
  et~al.}{Robins et~al.}{2008}]{robins2008higher}
Robins, J., L.~Li, E.~Tchetgen, A.~van~der Vaart, et~al. (2008).
\newblock Higher order influence functions and minimax estimation of nonlinear
  functionals.
\newblock {\em Probability and statistics: essays in honor of David A.
  Freedman\/}~{\em 2}, 335--421.

\bibitem[\protect\citeauthoryear{Rudin}{Rudin}{1987}]{rudin1987real}
Rudin, W. (1987).
\newblock {\em Real and Complex Analysis}.
\newblock Mathematics series. McGraw-Hill.

\bibitem[\protect\citeauthoryear{Schick}{Schick}{1986}]{schick1986asymptotically}
Schick, A. (1986).
\newblock On asymptotically efficient estimation in semiparametric models.
\newblock {\em The Annals of Statistics\/}, 1139--1151.

\bibitem[\protect\citeauthoryear{Schwartz}{Schwartz}{1967}]{schwartz1967estimation}
Schwartz, S.~C. (1967).
\newblock Estimation of probability density by an orthogonal series.
\newblock {\em The Annals of Mathematical Statistics\/}, 1261--1265.

\bibitem[\protect\citeauthoryear{Sz{\'e}kely and Bakirov}{Sz{\'e}kely and
  Bakirov}{2003}]{szekely2003extremal}
Sz{\'e}kely, G.~J. and N.~K. Bakirov (2003).
\newblock Extremal probabilities for gaussian quadratic forms.
\newblock {\em Probability theory and related fields\/}~{\em 126\/}(2),
  184--202.

\bibitem[\protect\citeauthoryear{Takatsu and Westling}{Takatsu and
  Westling}{2022}]{takatsu2022debiased}
Takatsu, K. and T.~Westling (2022).
\newblock Debiased inference for a covariate-adjusted regression function.
\newblock {\em arXiv preprint arXiv:2210.06448\/}.

\bibitem[\protect\citeauthoryear{Tikhonov, Goncharsky, Stepanov, and
  Yagola}{Tikhonov et~al.}{1995}]{tikhonov1995numerical}
Tikhonov, A.~N., A.~Goncharsky, V.~Stepanov, and A.~G. Yagola (1995).
\newblock {\em Numerical methods for the solution of ill-posed problems},
  Volume 328.
\newblock Springer Science \& Business Media.

\bibitem[\protect\citeauthoryear{Tsiatis}{Tsiatis}{2006}]{tsiatis2006semiparametric}
Tsiatis, A.~A. (2006).
\newblock Semiparametric theory and missing data.

\bibitem[\protect\citeauthoryear{Tsybakov}{Tsybakov}{2009}]{tsybakov2009nonparametric}
Tsybakov, A.~B. (2009).
\newblock {\em Introduction to Nonparametric Estimation}.
\newblock Springer.

\bibitem[\protect\citeauthoryear{van~der Laan}{van~der
  Laan}{2006}]{van2006statistical}
van~der Laan, M.~J. (2006).
\newblock Statistical inference for variable importance.
\newblock {\em The International Journal of Biostatistics\/}~{\em 2\/}(1).

\bibitem[\protect\citeauthoryear{van~der Laan, Bibaut, and Luedtke}{van~der
  Laan et~al.}{2018}]{van2018cv}
van~der Laan, M.~J., A.~Bibaut, and A.~R. Luedtke (2018).
\newblock Cv-tmle for nonpathwise differentiable target parameters.
\newblock In {\em Targeted Learning in Data Science}, pp.\  455--481. Springer.

\bibitem[\protect\citeauthoryear{van~der Laan and Dudoit}{van~der Laan and
  Dudoit}{2003}]{van2003unifiedCV}
van~der Laan, M.~J. and S.~Dudoit (2003).
\newblock Unified cross-validation methodology for selection among estimators
  and a general cross-validated adaptive epsilon-net estimator: Finite sample
  oracle inequalities and examples.

\bibitem[\protect\citeauthoryear{van~der Laan, Laan, and Robins}{van~der Laan
  et~al.}{2003}]{van2003unified}
van~der Laan, M.~J., M.~Laan, and J.~M. Robins (2003).
\newblock {\em Unified methods for censored longitudinal data and causality}.
\newblock Springer Science \& Business Media.

\bibitem[\protect\citeauthoryear{van~der Laan, Rose, et~al.}{van~der Laan
  et~al.}{2011}]{van2011targeted}
van~der Laan, M.~J., S.~Rose, et~al. (2011).
\newblock {\em Targeted learning: causal inference for observational and
  experimental data}, Volume~10.
\newblock Springer.

\bibitem[\protect\citeauthoryear{van~der Laan and Rubin}{van~der Laan and
  Rubin}{2006}]{van2006targeted}
van~der Laan, M.~J. and D.~Rubin (2006).
\newblock Targeted maximum likelihood learning.
\newblock {\em The international journal of biostatistics\/}~{\em 2\/}(1).

\bibitem[\protect\citeauthoryear{van~der Vaart}{van~der
  Vaart}{1991}]{van1991differentiable}
van~der Vaart, A. (1991).
\newblock On differentiable functionals.
\newblock {\em The Annals of Statistics\/}, 178--204.

\bibitem[\protect\citeauthoryear{van~der Vaart}{van~der
  Vaart}{2014}]{van2014higher}
van~der Vaart, A. (2014).
\newblock Higher order tangent spaces and influence functions.
\newblock {\em Statistical Science\/}, 679--686.

\bibitem[\protect\citeauthoryear{van~der Vaart and Wellner}{van~der Vaart and
  Wellner}{1989}]{van1989prohorov}
van~der Vaart, A. and J.~A. Wellner (1989).
\newblock Prohorov and continuous mapping theorems in the hoffmann-j{\o}rgensen
  weak convergence theory, with application to convolution and asymptotic
  minimax theorems.
\newblock {\em Tech. Rep. 157\/}.

\bibitem[\protect\citeauthoryear{van~der Vaart}{van~der
  Vaart}{2000}]{van2000asymptotic}
van~der Vaart, A.~W. (2000).
\newblock {\em Asymptotic statistics}, Volume~3.
\newblock Cambridge university press.

\bibitem[\protect\citeauthoryear{van~der Vaart, Dudoit, and van~der
  Laan}{van~der Vaart et~al.}{2006}]{vaart2006oracle}
van~der Vaart, A.~W., S.~Dudoit, and M.~J. v.~d. van~der Laan (2006).
\newblock Oracle inequalities for multi-fold cross validation.
\newblock {\em Statistics \& Decisions\/}~{\em 24\/}(3), 351--371.

\bibitem[\protect\citeauthoryear{van~der Vaart and Wellner}{van~der Vaart and
  Wellner}{1996}]{van1996weak}
van~der Vaart, A.~W. and J.~Wellner (1996).
\newblock {\em Weak convergence and empirical processes: with applications to
  statistics}.
\newblock Springer Science \& Business Media.

\bibitem[\protect\citeauthoryear{{v}on Mises}{{v}on
  Mises}{1947}]{mises1947asymptotic}
{v}on Mises, R. (1947).
\newblock On the asymptotic distribution of differentiable statistical
  functions.
\newblock {\em The annals of mathematical statistics\/}~{\em 18\/}(3),
  309--348.

\bibitem[\protect\citeauthoryear{Williamson, Gilbert, Simon, and
  Carone}{Williamson et~al.}{2021}]{williamson2021general}
Williamson, B.~D., P.~B. Gilbert, N.~R. Simon, and M.~Carone (2021).
\newblock A general framework for inference on algorithm-agnostic variable
  importance.
\newblock {\em Journal of the American Statistical Association\/}, 1--14.

\bibitem[\protect\citeauthoryear{Wright and Ziegler}{Wright and
  Ziegler}{2017}]{ranger}
Wright, M.~N. and A.~Ziegler (2017).
\newblock {ranger}: A fast implementation of random forests for high
  dimensional data in {C++} and {R}.
\newblock {\em Journal of Statistical Software\/}~{\em 77\/}(1), 1--17.

\bibitem[\protect\citeauthoryear{Yao}{Yao}{1967}]{yao1967applications}
Yao, K. (1967).
\newblock Applications of reproducing kernel hilbert spaces--bandlimited signal
  models.
\newblock {\em Information and Control\/}~{\em 11\/}(4), 429--444.

\bibitem[\protect\citeauthoryear{Zheng and Laan}{Zheng and
  Laan}{2011}]{zheng2011cross}
Zheng, W. and M.~J. Laan (2011).
\newblock Cross-validated targeted minimum-loss-based estimation.
\newblock In {\em Targeted Learning}, pp.\  459--474. Springer.

\end{thebibliography}
}


\appendix

\setcounter{equation}{0}
\renewcommand{\theequation}{S\arabic{equation}}
\setcounter{theorem}{0}
\setcounter{figure}{0}
\setcounter{table}{0}
\setcounter{lemma}{0}
\setcounter{corollary}{0}
\renewcommand{\thetheorem}{S\arabic{theorem}}
\renewcommand{\thecorollary}{S\arabic{corollary}}
\renewcommand{\thelemma}{S\arabic{lemma}}
\renewcommand{\thefigure}{S\arabic{figure}}
\renewcommand{\thetable}{S\arabic{table}}

\section*{\LARGE Appendices}

\DoToC

\section{Additional examples of pathwise differentiable parameters} \label{app:exClassical}

All derivations are given in Appendix~\ref{app:exDerivations}.

\begin{example}[Root-density function] \label{ex:den} Nonparametric density estimation is a well-studied problem. Estimating the (square root of) the density as an $L^2$ parameter has been done in \cite{cencov1962estimation}. In our setting, we suppose that $Z \sim P$ for $P \in \mathcal{P}$, where there is a $\sigma$-finite measure $\lambda$ that dominates all distributions in a locally nonparametric model $\mathcal{P}$ and we wish to estimate the square root of the density of $Z$, that is, $\nu(P)(z) := \frac{dP}{d\lambda}(z)^{1/2}$, where $\lambda$ denotes the Lebesgue measure. The parameter $\nu$ takes values in $\mathcal{H}:=L^2(\lambda)$. The local parameter takes the form $\dot{\nu}_P(s)(z) = \frac{1}{2}s(z)\nu(P)(z)$, and the efficient influence operator takes the form $$\dot{\nu}_P^\ast(h)(z) = \frac{h(z)}{2\nu(P)(z)} - E_P \left[\frac{h(Z)}{2\nu(P)(Z)} \right].$$ 
Similarly to Example~\ref{ex:cfdNonparametric}, $\dot{\nu}_P^*$ is not a bounded operator when $\lambda$ is not discrete, and so an EIF will not exist in these cases.

\end{example}

\begin{example}[Regression function] \label{ex:reg} We suppose that $Z:=(X,Y) \sim P$ for $P$ in a locally nonparametric model $\mathcal{P}$. We assume that for all $P, P' \in \mathcal{P}$, $P$ is equivalent to $P'$ in that $P \ll P'$ and $P' \ll P$, and also that
\begin{align}
\sup_{P\in\mathcal{P}}\esssup_x E_{P} [Y^2 \mid X=x]<\infty, \label{eq:regSecondMoment}
\end{align}
where $\esssup$ denotes a $P_X$-essential supremum with $P_X$ denoting the marginal distribution of $X$ under $P\in\mathcal{P}$. Let $\lambda_X$ be a measure that is dominated by the marginal distribution of $X$ under some (and, by their equivalence, all) distributions in $\mathcal{P}$. 
We wish to estimate the conditional mean $\nu(P)(x) := E_P[Y \mid X=x]$, where $\nu: \mathcal{P} \to \mathcal{H}$ with $\mathcal{H}:=L^2(\lambda_X)$. 
We establish pathwise differentiability at any $P\in\mathcal{P}$ that is such that $\frac{d\lambda_X}{dP_X}$ is bounded $P_X$-almost surely. We note that we do not require that $P_X\ll \lambda_X$; in such cases, 
 estimation error quantified via the $L^2(\lambda_X)$-norm only measures performance on a strict subset of the support of $P_X$.

The local parameter is $\dot{\nu}_P(s)(x) = \int [y-\nu(P)(x)]s(x,y) P_{Y\mid X}(dy \mid x)$ and the efficient influence operator is $\dot{\nu}_P^\ast(h)(x,y) = \frac{d\lambda_X}{dP_X}(x)[y - \nu(P)(x)]h(x)$. Similarly to Example~\ref{ex:cfdNonparametric}, $\dot{\nu}_P^*$ is not a bounded operator when $\lambda_X$ is not discrete, and so an EIF will not exist in these cases.
\end{example}

\begin{example}[Kernel mean embedding of distributions] \label{ex:kmed}
Let $\kappa: \mathcal{Z} \times \mathcal{Z} \to \mathbb{R}$ be a bounded, symmetric, positive definite function. Let $\mathcal{H}$ be the unique RKHS associated with the kernel $\kappa$ \citep{aronszajn1950theory}, and let $K_z:=\kappa(z,\cdot)$ be the associated feature map. 
Assume that $\mathcal{P}$ is a model on $\mathcal{Z}$ such that all $P \in \mathcal{P}$ are equivalent and $P \ll \lambda$ for all $P \in \mathcal{P}$.
The target of estimation is the evaluation of the kernel mean embedding $\nu: \mathcal{P} \to \mathcal{H}$ at $P$ \citep{gretton2012kernel}, where 
$$\nu(P) := \int K_z\, P(dz).$$
When the model is locally nonparametric,
the local parameter takes the form $\dot{\nu}_P(s) = \int K_z\, s(z)\, P(dz)$. The efficient influence operator takes the form $\dot{\nu}_P^\ast(h)(y) = h(y) - Ph$, and the efficient influence function is $P$-Bochner square integrable and takes the form $\phi_P(z) = K_z - \nu(P)$. Regardless of the initial estimator of $P_0$, the one-step estimator is given by $\bar{\nu}_n = \frac{1}{n} \sum_{i=1}^n K_{Z_i}$. In other words, $\bar{\nu}_n$ is the empirical kernel mean embedding as defined in \cite{gretton2012kernel}.
\end{example}

\begin{example}[Conditional average treatment effect] \label{ex:cate}
There has recently been much interest in various fields regarding the estimation of the conditional average treatment effect function \citep{hill2011bayesian,luedtke2016super,kunzel2019metalearners}.
Under conditions, this parameter corresponds to an additive causal effect between the mean outcome that would be observed among individuals with a given covariate value if, possibly contrary to fact, treatment 1 had been administered versus not administered. 
For the setting, suppose that $Z:=(X,A,Y) \sim P$ for $P \in \mathcal{P}$, where $X$ is a vector of covariates, $A$ is a binary treatment, and $Y$ is an outcome. As in Example~\ref{ex:den}, suppose that $\mathcal{P}$ is a locally nonparametric model consisting of equivalent measures. 
Define the propensity to receive treatment $a$ as $g_P(a\mid x) := P(A=a \mid X=x)$ and the outcome regression as $\mu_{P,a}(x) := E_P[Y \mid A=a,X=x]$. 
Suppose that all distributions $P\in\mathcal{P}$ are such that $\max_{a \in \{0,1\}} \esssup_{x \in \mathcal{X}} E_P(Y^2 \mid A=a, X=x) < \infty$ and $\max_{a \in \{0,1\}} \esssup_{x \in \mathcal{X}} g_P(a\mid x)^{-1}<\infty$. 
Let $\lambda_X$ be a measure that is dominated by the marginal of $X$ under the distributions in $\mathcal{P}$, and define $\mathcal{H}:=L^2(\lambda_X)$. 
The target of estimation is the conditional average treatment effect $\nu: \mathcal{P} \to \mathcal{H}$, defined as
\begin{align*}
\nu(P)(x) := \mu_{P,1}(x)- \mu_{P,0}(x).
\end{align*}
Similarly to Example~\ref{ex:den}, we establish pathwise differentiability at any $P\in\mathcal{P}$ that is such that $\frac{d\lambda_X}{dP_X}$ is bounded $P_X$-almost surely.
The local parameter takes the form
\begin{align}
\dot{\nu}_P(s)(x) = \int \frac{(2a-1)}{g_P(a\mid x)}[y-\mu_{P,a}(x)] s(y,a,x) P(dy,da \mid x). \label{eq:cateLocalP}
\end{align}
The efficient influence operator takes the form
\begin{align}
   \dot{\nu}_P^\ast(h)(y,a,x) = \frac{h(x)}{p_X(x)} \frac{(2a-1)}{g_P(a\mid x)}[y-\mu_{P,a}(x)]. \label{eq:cateEIO}
\end{align}
Similarly to Example~\ref{ex:cfdNonparametric}, $\dot{\nu}_P^*$ is not a bounded operator when $\lambda_X$ is not discrete, and so an EIF will not exist in these cases.
\end{example}

\section{Derivations for examples} \label{app:exDerivations}

\subsection{Example \ref{ex:cfdNonparametric}: counterfactual density function}

\subsubsection{Pathwise differentiability}\label{app:cdfPD}

We now show that $\nu$ is pathwise differentiable relative to a locally nonparametric model $\mathcal{P}$ at any $P\in\mathcal{P}$. To do this, we break our argument into two parts. First, we use Lemma~\ref{lem:suffCondsPD} to establish that $\nu$ is pathwise differentiable relative to a model $\mathcal{P}_g$ that is nonparametric up to the fact that the propensity to receive treatment is known to be equal to a fixed function $g$. Specifically, we consider the model $\mathcal{P}_g$ that consists of all distributions $P'$ that are such that $g_{P'}=g$ and for which there exists $P\in\mathcal{P}$ such that $P'_{Y\mid A,X}=P_{Y\mid A,X}$ and $P'_{X}=P_{X}$. 
Second, we use the fact that $\nu$ does not depend on the propensity to receive treatment to extend this pathwise differentiability result to the locally nonparametric model $\mathcal{P}$.

Throughout this subappendix we suppose that, at each $P\in\mathcal{P}$ and for a fixed $\delta>0$, $\mathcal{P}$ is large enough to contain submodels of the form $\{P_{\epsilon} : \epsilon\in [0,\delta)\}$, where $\frac{dP_{\epsilon,X}}{dP_X}(x)= 1 + \epsilon s_X(x)$, $\frac{dP_{\epsilon,A \mid X}}{dP_{A \mid X}}(a\mid x) = 1$, and $\frac{dP_{\epsilon,Y \mid A,X}}{dP_{Y \mid A,X}}(y\mid a,x) = 1 + \epsilon s_{Y \mid A,X}(y\mid a,x)$, where $s_X$ and $s_{Y\mid A,X}$ are arbitrary functions bounded in $(-\delta^{-1},\delta^{-1})$ that are such that $E_P[s_X(X)]=0$ and $E_P[s_{Y\mid A,X}(Y\mid A,X)\mid A,X] = 0$ $P$-almost surely. There is no loss in generality in assuming that $\mathcal{P}$ is this large since pathwise differentiability relative to a larger model also implies pathwise differentiability relative to a smaller model, and, when both the larger and the smaller models are locally nonparametric, the local parameters and efficient influence operators in the two models necessarily agree.

We now use Lemma~\ref{lem:suffCondsPD} to prove that $\nu$ is pathwise differentiable relative to $\mathcal{P}_g$, where for now we take $g$ to be any fixed function that is such that $g=g_{P'}$ for at least one $P'\in\mathcal{P}$. Fix two distributions $P$ and $\tilde{P}$ in $\mathcal{P}_g$. 
Let $\lambda_X$ denote a $\sigma$-finite measure that dominates the marginals in $X$ of $P$ and $\tilde{P}$, that is, $P_X \ll \lambda_X$ and $\tilde{P}\ll \lambda_X$, and let $q_X$ and $\tilde{q}_X$ denote the square root of the marginal density of $X$ relative to $\lambda_X$ under $P$ and $\tilde{P}$, respectively. For any $P'\in\mathcal{P}$, we also define $q_{P'}(\,\cdot\mid x)$ to be the square root of the conditional density of $Y$ given $(A,X)=(1,x)$ under $P'$. For brevity we let $q:=q_P$ and $\tilde{q}:=q_{\tilde{P}}$. We have that
\begin{align*}
    &\left\|\nu(\tilde{P})-\nu(P)\right\|_{\mathcal{H}}^2 \\
    &= \int \left[ \int \left\{\tilde{q}^2(y\mid x) \tilde{q}_{X}^2(x) -q^2(y\mid x) q_X^2(x)\right\} d\lambda_X(x)\right]^2 d\lambda_Y(y) \\
    &= \int \Bigg[ \int \left\{\tilde{q}(y\mid x) \tilde{q}_{X}(x)+q(y\mid x) q_X(x)\right\} \\
    &\hspace{4em}\cdot\left\{\tilde{q}(y\mid x) \tilde{q}_{X}(x) -q(y\mid x) q_X(x)\right\} d\lambda_X(x)\Bigg]^2 d\lambda_Y(y) \\
    &\le \int \left[ \int \left\{\tilde{q}(y\mid x) \tilde{q}_{X}(x)+q(y\mid x) q_X(x)\right\}^2 d\lambda_X(x)\right] \\
    &\hspace{2.5em}\cdot\left[ \int \left\{\tilde{q}(y\mid x) \tilde{q}_{X}(x) -q(y\mid x) q_X(x)\right\}^2 d\lambda_X(x)\right] d\lambda_Y(y) \\
    &\le 2\int \left[\nu(\tilde{P})(y)+\nu(P)(y)\right]\left[ \int \left\{\tilde{q}(y\mid x) \tilde{q}_{X}(x) -q(y\mid x) q_X(x)\right\}^2 d\lambda_X(x)\right] d\lambda_Y(y),
\end{align*}
where the first inequality holds by Cauchy-Schwarz and the second by the fact that $(b+c)^2\le 2(b^2+c^2)$. Using \eqref{cond:cfDensAndstrPos} and the fact that $\nu$ does not depend on the propensity to receive treatment, this shows that, for $C_1=4\sup_{P'\in\mathcal{P}}\esssup_y \nu(P')(y)$,
\begin{align*}
    &\left\|\nu(\tilde{P})-\nu(P)\right\|_{\mathcal{H}}^2 \\
    &\le C_1\iint \left\{\tilde{q}(y\mid x) \tilde{q}_{X}(x) -q(y\mid x) q_X(x)\right\}^2 d\lambda_X(x)d\lambda_Y(y) \\
    &=  C_1\iint \frac{1}{g(1\mid x)}\left\{\tilde{q}(y\mid x)g^{1/2}(1\mid x)\tilde{q}_{X}(x) -q(y\mid x)g^{1/2}(1\mid x) q_X(x)\right\}^2 d\lambda_X(x)d\lambda_Y(y).
\end{align*}
Again using \eqref{cond:cfDensAndstrPos} and letting $C_2=C_1/\inf_{P'\in\mathcal{P}}\essinf_{x \in \mathcal{X}} g_{P'}(1\mid x)$, we find that
\begin{align*}
    &\left\|\nu(\tilde{P})-\nu(P)\right\|_{\mathcal{H}}^2 \\
    &\le  C_2\iint\left\{\tilde{q}(y\mid x)g^{1/2}(1\mid x)\tilde{q}_{X}(x) -q(y\mid x)g^{1/2}(1\mid x) q_X(x)\right\}^2 d\lambda_X(x)d\lambda_Y(y).
\end{align*}
Finally, noting that the double integral on the right-hand side upper bounds by $H^2(P,\tilde{P})$, we have shown that $\|\nu(\tilde{P})-\nu(P)\|_{\mathcal{H}}\le C_2^{1/2}H(P,\tilde{P})$, which establishes \ref{it:localLip} of Lemma~\ref{lem:suffCondsPD} when the model is $\mathcal{P}_g$, where $g$ is an arbitrary value of the propensity to receive treatment for which there exists some $P\in\mathcal{P}$ such that $g=g_P$.

Hereafter we fix $P\in\mathcal{P}$ and suppose that $g=g_P$. We now establish \ref{it:closure} of Lemma~\ref{lem:suffCondsPD} at $P$ for the model $\mathcal{P}_g$ with $\eta_P(s)(y)$ equal to the right-hand side of \eqref{eq:countDenLocalP} from the main text. To do this, we use the following model: $\{P_{\epsilon} : \epsilon\in [0,\delta)\}$, where $\frac{dP_{\epsilon,X}}{dP_X}(x)= 1 + \epsilon s_X(x)$, $\frac{dP_{\epsilon,A \mid X}}{dP_{A \mid X}}(a\mid x) = 1$, and $\frac{dP_{\epsilon,Y \mid A,X}}{dP_{Y \mid A,X}}(y\mid a,x) = 1 + \epsilon s_{Y \mid A,X}(y\mid a,x)$, where $s_X$ and $s_{Y\mid A,X}$ are bounded in $[-\delta^{-1}/2,\delta^{-1}/2]$ and $E_P[s_X(X)]=0$ and $E_P[s_{Y\mid A,X}(Y\mid A,X)\mid A,X] = 0$ $P$-almost surely. The model $\{P_{\epsilon} : \epsilon\in [0,\delta)\}$ is a submodel of $\mathcal{P}$ by assumption and, due to the fact that $P_{\epsilon,A\mid X}=P_{A\mid X}$ for all $\epsilon$, is therefore also a submodel of $\mathcal{P}_g$. It can be verified that this submodel has score $s(x,a,y)=s_X(x) + s_{Y\mid A,X}(y\mid a,x)$ at $\epsilon=0$ and that the $L^2(P)$-closure of the set containing such scores corresponds to the tangent space of $\mathcal{P}_g$ at $P$. In what follows we will show that $\| \nu(P_\epsilon) - \nu(P) - \epsilon\, \eta_P(s) \|_{\mathcal{H}} = o(\epsilon)$. To this end, observe that
\begin{align*}
&\| \nu(P_\epsilon) - \nu(P) - \epsilon \eta_P(s) \|_{\mathcal{H}}^2 \\
&= \int \Bigg[ \int \left\{\left(\frac{dP_{\epsilon,Y \mid A,X}}{dP_{Y \mid A,X}}(y\mid 1,x)\frac{dP_{\epsilon,X}}{dP_{X}}(x)-1\right)q^2(y\mid x)q_X^2(x)\right\} d\lambda_X(x) \\
&\quad\quad- \epsilon\eta_P(s)(y)\Bigg]^2 d\lambda_Y(y) \\
&= \int \Bigg[ \int \big\{\big(\epsilon s_{Y\mid A,X}(y\mid 1,x) \\
&\quad\quad+ \epsilon s_X(x) + \epsilon^2 s_{Y\mid A,X}(y\mid 1,x)s_X(x)\big)q^2(y\mid x)q_X^2(x)\big\} d\lambda_X(x)  - \epsilon\eta_P(s)(y)\Bigg]^2 d\lambda_Y(y).
\end{align*}
Using that $s_{Y\mid A,X}(y\mid a,x) = s(x,a,y)-E_P[s(X,A,Y)\mid A=a,X=x]$ and $s_{X}(x) = E_P[s(X,A,Y)\mid X=x]$ and plugging in the definition of $\eta_P$, we see that
\begin{align*}
\| &\nu(P_\epsilon) - \nu(P) - \epsilon \eta_P(s) \|_{\mathcal{H}}^2 \\
&= \int \left[ \int \epsilon^2 s_{Y\mid A,X}(y\mid 1,x)s_X(x)q^2(y\mid x)q_X^2(x) d\lambda_X(x)\right]^2 d\lambda_Y(y) \\
&\le \epsilon^4\delta^{-4}\int \left[ \int q^2(y\mid x)q_X^2(x) d\lambda_X(x)\right]^2 d\lambda_Y(y) \\
&= \epsilon^4\delta^{-4}\int \nu(P)(y)^2 d\lambda_Y(y) \\
&\le\epsilon^4\delta^{-4}\left[\int \nu(P)(y) d\lambda_Y(y)\right] \sup_{P\in\mathcal{P}}\esssup_{y} \nu(P)(y) \\
&= \epsilon^4\delta^{-4}\sup_{P\in\mathcal{P}}\esssup_{y} \nu(P)(y),
\end{align*}
where the first inequality used that $s_{Y\mid A,X}$ and $s_X$ both have ranges bounded in $(-\delta^{-1},\delta^{-1})$. By \eqref{cond:cfDensAndstrPos}, the right-hand side is $O(\epsilon^4)$, and therefore is $o(\epsilon^2)$ with much to spare. Hence, $\| \nu(P_\epsilon) - \nu(P) - \epsilon\, \eta_P(s) \|_{\mathcal{H}} = o(\epsilon)$.

We now verify that $\eta_P$ is a bounded operator, which will then show that \ref{it:closure} of Lemma~\ref{lem:suffCondsPD} holds at $P$ for the model $\mathcal{P}_g$. Take any $s$ in the tangent space of $\mathcal{P}_g$ at $P$. Let $s_{Y\mid A,X}(y\mid a,x) := s(x,a,y)-E_P[s(X,A,Y)\mid A=a,X=x]$ and $s_{X}(x) := E_P[s(X,A,Y)\mid X=x]$. It can be verified that, $E_P[s(X,A,Y)\mid A,X]-E_P[s(X,A,Y)\mid X]=0$ $P$-a.s., and so $s=s_{Y\mid A,X}+s_{X}$. We write $p(\,\cdot \mid a,x) $ to denote the conditional density under $P$ of $Y$ given $(A,X)=(a,x)$. Observe that
\begin{align*}
    &\|\eta_P(s)\|_{L^2(\lambda_Y)}^2 \\
    &\quad= \int \left\{ \int [s_{Y \mid A,X}(y \mid 1,x) + s_X(x)]p(y \mid 1,x) P_X(dx)\right\}^2 \lambda_Y(dy) \\
    &\quad\le \int \left\{  [s_{Y \mid A,X}(y \mid 1,x) + s_X(x)]p(y \mid 1,x) \right\}^2 \lambda_Y(dy)P_X(dx) \\
    &\quad= \int [s_{Y \mid A,X}(y \mid 1,x) + s_X(x)] ^2p(y \mid 1,x) P_{Y\mid A,X}(dy\mid 1,x)P_X(dx) \\
    &\quad= \int [s_{Y \mid A,X}(y \mid 1,x) + s_X(x)] ^2p(y \mid 1,x) \frac{a}{P(A=1\mid X=x)} P(dz) \\
    &\quad= \int [s_{Y \mid A,X}(y \mid a,x) + s_X(x)] ^2p(y \mid a,x) \frac{a}{P(A=1\mid X=x)} P(dz) \\
    &\quad\le \left(\int [s_{Y \mid A,X}(y \mid a,x) + s_X(x)] ^2 P(dz)\right) \sup_{x,y} \left(p(y \mid 1,x) \frac{1}{P(A=1\mid X=x)}\right) \\
    &\quad\le \|s\|_{L^2(P)}^2 \sup_{x,y} \left(p(y \mid 1,x) \frac{1}{P(A=1\mid X=x)}\right).
\end{align*}
By \eqref{cond:cfDensAndstrPos}, the supremum is finite and so $\eta_P$ is a bounded linear operator. Hence, by Lemma~\ref{lem:suffCondsPD}, $\nu$ is pathwise differentiable at $P$ relative to $\mathcal{P}_g$ with $\dot{\nu}_P=\eta_P$.

We now show that $\nu$ is pathwise differentiable at $P$ relative to $\mathcal{P}$. To do this, we consider an arbitrary $s\in L_0^2(P)$ and $\{P_\epsilon: \epsilon\}\in \mathscr{P}(P,\mathcal{P},s)$. Let $s_{Y\mid A,X}(y\mid a,x) := s(x,a,y)-E_P[s(X,A,Y)\mid A=a,X=x]$ and $s_{X}(x) := E_P[s(X,A,Y)\mid X=x]$. Also define $\{P_\epsilon': \epsilon\}$ to be the submodel consisting of distributions $P_\epsilon'$ that are such that $P_{\epsilon,Y\mid A,X}'=P_{\epsilon,Y\mid A,X}$, $g_{P_\epsilon'}=g_P$, and $P_{\epsilon,X}'=P_{\epsilon,X}$. Lemma~\ref{lem:preserveDQM} can be used to verify that $\{P_\epsilon': \epsilon\}\in \mathscr{P}(P,\mathcal{P}_g,s_{Y\mid A,X}+s_X)$. Also, since $\nu$ does not depend on the propensity to receive treatment, $\nu(P_\epsilon')=\nu(P_\epsilon)$ and, by the definition of $\dot{\nu}_P$ in \eqref{eq:countDenLocalP}, $\dot{\nu}_P(s)=\dot{\nu}_P(s_{Y\mid A,X}+s_X)$. Hence,
\begin{align}
\left\|\nu(P_\epsilon)-\nu(P)-\epsilon \dot{\nu}_P(s)\right\|_{L^2(\lambda_Y)}&= \left\|\nu(P_\epsilon')-\nu(P)-\epsilon \dot{\nu}_P(s_{Y\mid A,X}+s_X)\right\|_{L^2(\lambda_Y)}. \label{eq:extendPD}
\end{align}
By the pathwise differentiability of $\nu$ at $P$ relative to $\mathcal{P}_g$, the right-hand side is $o(\epsilon)$. Recalling the left-hand side above, this shows that $\nu$ is pathwise differentiable at $P$ relative to $\mathcal{P}$ with local parameter $\dot{\nu}_P$.

\subsubsection{Efficient influence operator}
Take any $s\in L_0^2(P)$. Let $s_{Y\mid A,X}(y\mid a,x) := s(x,a,y)-E_P[s(X,A,Y)\mid A=a,X=x]$ and $s_{X}(x) := E_P[s(X,A,Y)\mid X=x]$. To compute the adjoint of $\dot{\nu}_P$, note that, for $h \in L^2(\lambda_Y)$,
$$\langle \dot{\nu}_P (s), h \rangle_{L^2(\lambda_Y)} = \int [s_{Y \mid A,X}(y \mid 1, x) + s_X(x)] h(y) P_{Y\mid A,X}(dy \mid 1, x) P_X(dx).$$
Note that
\begin{align*}
\int &s_{Y \mid A,X}(y \mid 1, x)h(y) P_{Y\mid A,X}(dy \mid 1, x) P_X(dx) \\
&= \int \frac{1\{a=1\}}{g_P(a\mid x)}\left\{h(y)-E_P\left[h(Y) \mid A=a, X=x\right]\right\} s(z)P(dz), \\
\int &s_X(x)h(y) P_{Y\mid A,X}(dy \mid 1, x) P_X(dx) \\
&= \int \left\{E_P\left[h(Y) \mid A=1, X=x\right] - E_P E_P\left[h(Y) \mid A=1, X\right]\right\} s(z) P(dz).
\end{align*}
Thus, \eqref{eq:countDenAdjoint} holds.

\subsubsection{Bounding the regularized remainder term}\label{app:cdRegRem}

For $k\in\mathbb{N}$, let $m_{P,k} : x\mapsto E_P[h_k(Y)\mid A=1,X=x]$. We now use Lemma~\ref{lem:regRemBd} to study $\mathcal{R}_n^{j,\beta_n}$. Towards this, note that, for any $P\in\mathcal{P}$ and $k\in\mathbb{N}$,
\begin{align*}
\left\langle \nu(P)-\nu(P_0),h_k\right\rangle_{\mathcal{H}} + P_0 \dot{\nu}_P^*(h_k)&= E_0\left[\left(1-\frac{g_0(1\mid X)}{g_P(1\mid X)}\right)\left(m_{P,k}(X)-m_{0,k}(X)\right)\right].
\end{align*}
Hence, by Lemma~\ref{lem:regRemBd} and the Cauchy-Schwarz inequality,
\begin{align*}
\|\mathcal{R}_P^\beta\|_{\mathcal{H}}^2&\le \left\|1-\frac{g_0(1\mid \cdot)}{g_P(1\mid \cdot)}\right\|_{L^2(P_{0,X})}^2 \sum_{k=1}^\infty \beta_k^2\left\|m_{P,k}-m_{0,k}\right\|_{L^2(P_{0,X})}^2.
\end{align*}
Combining this with \eqref{cond:cfDensAndstrPos} and letting $C:=1/\inf_{P\in\mathcal{P}}\essinf_{x \in \mathcal{X}} g_P(1\mid x)<\infty$, we then obtain
\begin{align*}
    \|\mathcal{R}_P^\beta\|_{\mathcal{H}}^2&\le C\left\|g_P(1\mid \cdot)-g_0(1\mid \cdot)\right\|_{L^2(P_{0,X})}^2 \sum_{k=1}^\infty \beta_k^2\left\|m_{P,k}-m_{0,k}\right\|_{L^2(P_{0,X})}^2.
\end{align*}
Further observe that, for any $k\in\mathbb{N}$, Cauchy-Schwarz and the fact that $h_k$ has unit length in $\mathcal{H}=L^2(\lambda_Y)$ yield that
\begin{align*}
&\left\|m_{P,k}-m_{0,k}\right\|_{L^2(P_{0,X})}^2 \\
&= \int \left(\int h_k(y) [p_{Y\mid A,X}(y\mid 1,x)-p_{0,Y\mid A,X}(y\mid 1,x)] \lambda_Y(dy)\right)^2 P_{0,X}(dx) \\
&\le \int \left(\int h_k(y)^2 \lambda_Y(dy)\right)\left(\int [p_{Y\mid A,X}(y\mid 1,x)-p_{0,Y\mid A,X}(y\mid 1,x)]^2 \lambda_Y(dy)\right) P_{0,X}(dx) \\
&= \iint [p_{Y\mid A,X}(y\mid 1,x)-p_{0,Y\mid A,X}(y\mid 1,x)]^2 \lambda_Y(dy) P_{0,X}(dx).
\end{align*}
Combining the preceding two displays gives \eqref{eq:cfdNonpRem}.

\subsection{Example~\ref{ex:cfdBandlimited}: bandlimited counterfactual density function}\label{lem:bandLimDen}

\subsubsection{Preliminaries}

Many of our arguments rely on the following result, which shows the sense in which $\mathscr{B}$ can be viewed as an $L^2(\lambda_Y)$ projection onto $\underline{\mathcal{H}}$. In this lemma and in this lemma only, we are careful to distinguish between elements of $L^2(\lambda_Y)$, which are equivalence classes of functions of the form $[f]$, and elements of $\underline{\mathcal{H}}$, which are functions. After this lemma, we return to following the usual conventions that (i) a generic $\lambda_Y$-square integrable function $h : \mathbb{R}\rightarrow\mathbb{R}$ can be treated as an element of $L^2(\lambda_Y)$ by simply replacing $h$ by $[h]$, and (ii) for a generic equivalence class $[f]$ of $L^2(\lambda_Y)$, the function $y\mapsto f(y)$ corresponds to any function belonging to the equivalence class $[f]$. In all contexts where the convention (ii) is used, the particular element of the equivalence class that is selected will not matter.

\begin{lemma} \label{lem:bandlimitedproj}
For a $\lambda_Y$-square integrable function $h : \mathbb{R}\rightarrow\mathbb{R}$, let $[h]$ denote the equivalence class of functions that are $\lambda_Y$-a.e. equal to $h$. Denoting a generic element of $L^2(\lambda_Y)$ by $[h]$, the operator $[h]\mapsto [\mathscr{B}(h)]$ is an orthogonal projection in $L^2(\lambda_Y)$ onto the closed subspace $\underline{\mathcal{H}}^\ddagger:=\{[h] : h\in \underline{\mathcal{H}}\}$.
\end{lemma}
\begin{proof}
Let $\mathscr{B}^\ddagger : [h]\mapsto [\mathscr{B}(h)]$. 
For any function $v : \mathbb{R}\rightarrow\mathbb{C}$, we define $1_{[-b,b]} \cdot v$ as the function $\xi\mapsto 1_{[-b,b]}(\xi) \cdot v(\xi)$. 

We first show that $\mathscr{B}^\ddagger$ is a linear operator. We show this using that $[f + g]=[f]+[g]$ for any $[f],[g]\in L^2(\lambda_Y)$ and also that the Fourier and inverse Fourier transforms are linear. In particular, for any $[f],[g]\in L^2(\lambda_Y)$ and $c\in\mathbb{R}$, we have that
\begin{align*}
\mathscr{B}^\ddagger([cf+g])&= \mathscr{B}^\ddagger([cf]+[g]) = [\mathscr{B}(cf+g)] = [\mathcal{F}^{-1}(1_{[-b,b]} \cdot \mathcal{F}(cf+g))] \\
&= [\mathcal{F}^{-1}(1_{[-b,b]} \cdot \mathcal{F}(cf)+1_{[-b,b]}\cdot\mathcal{F}(g))] \\
&= [\mathcal{F}^{-1}(1_{[-b,b]} \cdot \mathcal{F}(cf))+\mathcal{F}^{-1}(1_{[-b,b]}\cdot\mathcal{F}(g))] \\
&= [c\mathcal{F}^{-1}(1_{[-b,b]} \cdot \mathcal{F}(f))+\mathcal{F}^{-1}(1_{[-b,b]}\cdot\mathcal{F}(g))] \\
&= [c\mathscr{B}(f)+\mathscr{B}(g)] = c[\mathscr{B}(f)]+[\mathscr{B}(g)] = c\mathscr{B}^\ddagger([f])+\mathscr{B}^\ddagger([g]).
\end{align*}
The operator $\mathscr{B}^\ddagger$ is idempotent since, for any $[f]\in L^2(\lambda_Y)$,
\begin{align*}
\mathscr{B}^\ddagger\circ\mathscr{B}^\ddagger([f]) &= \mathscr{B}^\ddagger([\mathscr{B}(f)]) =  [\mathscr{B}\circ \mathscr{B}(f)] \\
&= [\mathcal{F}^{-1}(1_{[-b,b]} \cdot \mathcal{F} \circ \mathscr{B}(f))] = [\mathcal{F}^{-1}(1_{[-b,b]} \cdot \mathcal{F} \circ \mathcal{F}^{-1}(1_{[-b,b]} \cdot \mathcal{F}(f)))] \\
&= [\mathcal{F}^{-1}(1_{[-b,b]} \cdot 1_{[-b,b]} \cdot \mathcal{F}(f))] = [\mathcal{F}^{-1}(1_{[-b,b]} \cdot \mathcal{F}(f))] = [\mathscr{B}(f)] = \mathscr{B}^\ddagger([f]).
\end{align*}
We now show that $\mathscr{B}^\ddagger$ is self-adjoint. For any $[f], [g] \in L^2(\lambda_Y)$, the definitions of $\mathscr{B}$ and $\mathscr{B}^\ddagger$ show that
\begin{align*}
    \langle [g], \mathscr{B}^\ddagger([f]) \rangle_{L^2(\lambda_Y)} &=  \langle [g], [\mathscr{B}(f)] \rangle_{L^2(\lambda_Y)} = \int g(y)\,\mathscr{B}(f)(y)\, \lambda_Y(dy) \\
    &= \int g(y)\,\mathcal{F}^{-1} (1_{[-b,b]} \cdot \mathcal{F}(f))(y)\, \lambda_Y(dy) \\
    \intertext{Applying Plancherel's theorem, the above display continues as follows:}
    &= \int \mathcal{F} (g)(y)\,(1_{[-b,b]} \cdot \mathcal{F}(f))(y)\, \lambda_Y(dy) \\
    &= \int (1_{[-b,b]} \cdot \mathcal{F}) (g)(y)\,\mathcal{F}(f)(y)\, \lambda_Y(dy) \\
    &= \int \mathcal{F}^{-1}(1_{[-b,b]} \cdot \mathcal{F}) (g)(y)\,f(y)\, \lambda_Y(dy) = \int \mathscr{B}(g)(y)\,f(y)\, \lambda_Y(dy) \\
    &= \langle [\mathscr{B}(g)], [f] \rangle_{L^2(\lambda_Y)} = \langle \mathscr{B}^\ddagger([g]), [f] \rangle_{L^2(\lambda_Y)}.
\end{align*}
Hence, $\mathscr{B}^\ddagger$ is an orthogonal projection. Furthermore, the image $\textnormal{Im}(\mathscr{B}^\ddagger)$ of $\mathscr{B}^\ddagger$ can be seen to be equal to $\underline{\mathcal{H}}^\ddagger$. Indeed, (i) $\underline{\mathcal{H}}^\ddagger\subseteq\textnormal{Im}(\mathscr{B}^\ddagger)$ since, for any $[h]\in\underline{\mathcal{H}}^\ddagger$, $[h]=[\mathscr{B}(h)]=\mathscr{B}^\ddagger([h])\in \textnormal{Im}(\mathscr{B}^\ddagger)$ and (ii) $\textnormal{Im}(\mathscr{B}^\ddagger)\subseteq \underline{\mathcal{H}}^\ddagger$ since, for any $[h]\in\textnormal{Im}(\mathscr{B}^\ddagger)$, the idempotency of $\mathscr{B}^\ddagger$ shows that $\mathcal{F}(h)=\mathcal{F}\circ \mathscr{B}(h) = 1_{[-b,b]} \cdot \mathcal{F}(h)$, and so $h\in \underline{\mathcal{H}}$ and $[h]\in\underline{\mathcal{H}}^\ddagger$. As the image of an orthogonal projection in a Hilbert space is closed, $\underline{\mathcal{H}}^\ddagger$ is a closed subspace of $L^2(\lambda_Y)$. This completes the proof.
\end{proof}

\subsubsection{Pathwise differentiability}
Lemma \ref{lem:bandlimitedproj} and the pathwise differentiability of $\nu$, established in Example~\ref{ex:cfdNonparametric}, implies the pathwise differentiability of $\underline{\nu}$ by the following display, which holds for any $\{P_\epsilon : \epsilon \in [0,\delta)\} \in \mathscr{P}(P,\mathcal{P},s)$:
\begin{align*}
    \|\underline{\nu}(P_\epsilon)-\underline{\nu}(P_0)-\epsilon \mathscr{B}(\dot{\nu}_P(s))\|_{L^2(\lambda_Y)} &= \|\mathscr{B}(\nu(P_\epsilon)-\nu(P_0)-\epsilon \dot{\nu}_0(s))\|_{L^2(\lambda_Y)} \\
    &\leq \|\nu(P_\epsilon)-\nu(P_0)-\epsilon \dot{\nu}_0(s)\|_{L^2(\lambda_Y)} = o(\epsilon).
\end{align*}
Hence, $\underline{\dot{\nu}}_P(s) = \mathscr{B}\circ \dot{\nu}_P(s)$. 
Because $\mathscr{B}$ and $\dot{\nu}_P$ are both linear operators, $\underline{\dot{\nu}}_P : \dot{\mathcal{P}}_P\rightarrow\underline{\mathcal{H}}$ is a linear operator. This operator is also bounded since, by Lemma~\ref{lem:bandlimitedproj} and the boundedness of $\dot{\nu}_P$, the following holds for any $s\in\dot{\mathcal{P}}_P$:
\begin{align*}
    \|\underline{\dot{\nu}}_P(s)\|_{L^2(\lambda_Y)}&= \|\mathscr{B}\circ \dot{\nu}_P(s)\|_{L^2(\lambda_Y)}\le \|\dot{\nu}_P(s)\|_{L^2(\lambda_Y)}\le \|s\|_{L^2(P)} \|\dot{\nu}_P\|_{\textnormal{op}}.
\end{align*}

\subsubsection{Efficient influence operator}
Since $\mathscr{B}$ is self-adjoint, for any $h\in\underline{\mathcal{H}}$,
$$\langle \mathscr{B}\circ \dot{\nu}_P(s), h\rangle_{L^2(\lambda_Y)} = \langle \dot{\nu}_P(s), \mathscr{B}(h) \rangle_{L^2(\lambda_Y)} = \langle s, \dot{\nu}_P^\ast \circ \mathscr{B}(h) \rangle_{L^2(\lambda_Y)}.$$
Thus, $\underline{\dot{\nu}}_P^\ast (h) = \dot{\nu}_P^\ast \circ \mathscr{B} (h)$. Furthermore, as $h\in\underline{\mathcal{H}}$ and $\mathscr{B}$ is a projection onto $\underline{\mathcal{H}}$, $\underline{\dot{\nu}}_P^\ast (h) = \dot{\nu}_P^\ast(h)$.

\subsubsection{Efficient influence function}\label{app:bandLimDenEIF}

To compute the efficient influence function, we let
\begin{align*}
  &\underline{\phi}_P (y,a,x)(\tilde{y}) \\
  &\quad:= \underline{\dot{\nu}}_P^\ast (\underline{K}_{\tilde{y}})(y,a,x) \\
  &\quad= \frac{1\{a=1\}}{g_P(a\mid x)}\left\{\underline{K}_{\tilde{y}}(y)-E_P\left[\underline{K}_{\tilde{y}}(Y) \mid A=a, X=x \right]\right\} \nonumber \\
  &\quad\quad+ \left(E_P\left[\underline{K}_{\tilde{y}}(Y)\mid A=1, X=x\right] - \int E_P\left[\underline{K}_{\tilde{y}}(Y) \mid A=1, X=\tilde{x}\right]P_X(d\tilde{x})\right).
\end{align*}
By the symmetry of the kernel, $\underline{K}_{\tilde{y}}(y)=\underline{K}_y(\tilde{y})$. Plugging this into the above and noting that $\underline{\nu}(P)=\int E_P\left[\underline{K}_Y \mid A=1, X=\tilde{x}\right]P_X(d\tilde{x})$, we find that
\begin{align}
&\underline{\phi}_P(y,a,x) \nonumber\\
&= \frac{1\{a=1\}}{g_P(a\mid x)}\left\{\underline{K}_y-E_P\left[\underline{K}_Y \mid A=a, X=x \right]\right\} + E_P\left[\underline{K}_Y\mid A=1, X=x\right] - \underline{\nu}(P). \label{eq:cfDensEIF}
\end{align}
Combining \eqref{cond:cfDensAndstrPos} with the square integrability of the sinc function shows that $\underline{\phi}_P$ belongs to $L^2(P;\mathcal{H})$. Hence, Theorem~\ref{thm:eifProperties} shows that $\underline{\phi}_P$ is the EIF of $\nu$.

\subsubsection{Bounding the remainder term}\label{app:cdRem}

Fix $P\in\mathcal{P}$. Observe that, for any $y\in\mathbb{R}$, the definition of $\underline{\nu}$ and the expression in \eqref{eq:cfDensEIF} yields that
\begin{align*}
    \underline{\mathcal{R}}_P&:= \underline{\nu}(P) + P_0 \underline{\phi}_P -\underline{\nu}(P_0) \\
    &= E_0\left[\left(1-\frac{g_0(1\mid X)}{g_P(1\mid X)}\right)\left(E_P[\underline{K}_Y\mid A=1,X]-E_0[\underline{K}_Y\mid A=1,X]\right)\right] \\
    &= E_0\left[\left(1-\frac{g_0(1\mid X)}{g_P(1\mid X)}\right)\int K_y [p_{Y\mid A,X}(y\mid 1,X)-p_{0,Y\mid A,X}(y\mid 1,X)]\lambda_Y(dy) \right].
\end{align*}
Hence, letting $X_1,X_2$ denote independent draws from $P_{0,X}^2$, we find that
\begin{align*}
&\|\underline{\mathcal{R}}_P\|_{\underline{\mathcal{H}}}^2 \\
&= E_0\Bigg[\iint K_y(y')\left\{\prod_{j=1}^2 [p_{Y\mid A,X}(y'\mid 1,X_j)-p_{0,Y\mid A,X}(y'\mid 1,X_j)]\left(1-\frac{g_0(1\mid X_j)}{g_P(1\mid X_j)}\right)\right\} \\
&\hspace{5em}\lambda_Y(dy)\lambda_Y(dy')\Bigg] \\
&\le \left(\sup_{y,y'\in\mathbb{R}}|K_y(y')|\right) E_0\Bigg[\iint \Bigg|\prod_{j=1}^2 [p_{Y\mid A,X}(y'\mid 1,X_j)-p_{0,Y\mid A,X}(y'\mid 1,X_j)] \\
&\hspace{14em}\cdot\left(1-\frac{g_0(1\mid X_j)}{g_P(1\mid X_j)}\right)\Bigg|\lambda_Y(dy)\lambda_Y(dy')\Bigg].
\end{align*}
Using that $\sup_{y,y'\in\mathbb{R}}|K_y(y')|=b/\pi$ and applying Fubini's theorem to the expectation above shows that $\|\underline{\mathcal{R}}_P\|_{\underline{\mathcal{H}}}^2$ is equal to
\begin{align*}
\frac{b}{\pi}E_0\left[\int \left|[p_{Y\mid A,X}(y'\mid 1,X_1)-p_{0,Y\mid A,X}(y'\mid 1,X_1)]\left(1-\tfrac{g_0(1\mid X_1)}{g_P(1\mid X_1)}\right)\right|\lambda_Y(dy)\right]^2.
\end{align*}
Letting $C^2:=b/[\pi\inf_{P'\in\mathcal{P}}\essinf_{x \in \mathcal{X}} g_{P'}(1\mid x)^2]$ and applying Cauchy-Schwarz,
\begin{align*}
\|\underline{\mathcal{R}}_P\|_{\underline{\mathcal{H}}}^2&\le C^2\left\|g_P(1\mid \cdot)-g_0(1\mid \cdot)\right\|_{L^2(P_{0,X})}^2\left\|p_{Y\mid A=1,X}-p_{0,Y\mid A=1,X}\right\|_{L^2(\lambda_Y\times P_{0,X})}^2,
\end{align*}
where $p_{Y\mid A=1,X}-p_{0,Y\mid A=1,X}$ denotes the function $(y,x)\mapsto p_{Y\mid A,X}(y\mid 1,x)-p_{0,Y\mid A,X}(y\mid 1,x)$. 
We conclude by noting that $C^2<\infty$ by the strong positivity assumption. Hence, \eqref{eq:cdRemBandlimitedBound} holds.

\subsection{Example \ref{ex:conac}: counterfactual mean outcome under a continuous treatment}

\subsubsection{Pathwise differentiability}

We now show that $\nu$ is pathwise differentiable relative to a locally nonparametric model $\mathcal{P}$ at any $P\in\mathcal{P}$. To do this, we follow similar arguments to those used in Appendix~\ref{app:cdfPD}. In particular, we first use Lemma~\ref{lem:suffCondsPD} to establish that $\nu$ is pathwise differentiable relative to a model $\mathcal{P}_g$ that is nonparametric up to the fact that the propensity to receive treatment $g_P$ is known to be equal to a fixed function $g$. Specifically, we consider the model $\mathcal{P}_g$ that consists of all distributions $P'$ that are such that $g_{P'}=g$ and for which there exists $P\in\mathcal{P}$ such that $P'_{Y\mid A,X}=P_{Y\mid A,X}$ and $P'_X=P_X$. 
Second, we use the fact that $\nu$ does not depend on the propensity to receive treatment to extend this pathwise differentiability result to the locally nonparametric model $\mathcal{P}$.

Let $g$ be such that $g=g_{P'}$ for some fixed $P'\in\mathcal{P}$. We first show that $\nu$ is Lipschitz over $\mathcal{P}_g$. Fix $P, \tilde{P} \in \mathcal{P}_g$. For each $a\in\mathcal{A}$, let $P_a$ and $\tilde{P}_a$ denote the distributions on $\mathbb{R}$ defined so that, for any Borel set $B$, $P_a(B)=\int_{\mathcal{X}}\int_B P_{Y\mid A,X}(dy\mid a,x)P_X(dx)$ and $\tilde{P}_a(B)=\int_{\mathcal{X}}\int_B \tilde{P}_{Y\mid A,X}(dy\mid a,x)\tilde{P}_X(dx)$. We have that
\begin{align*}
  &\|\nu(\tilde{P})-\nu(P)\|^2_{L^2(\lambda_A)} \\
  &= \int \left[\int y \{P_a(dy)-\tilde{P}_a(dy)\}\right]^2 \lambda_A(da) \\
  &\le \int \left[\int y^2 \{P_a^{1/2}(dy)+\tilde{P}_a^{1/2}(dy)\}^2\right] \left[\int \{P_a^{1/2}(dy)-\tilde{P}_a^{1/2}(dy)\}^2\right]\lambda_A(da) \\
  &\le \int \left[2\int y^2 \{P_a(dy)+\tilde{P}_a(dy)\}\right] \left[\int \{P_a^{1/2}(dy)-\tilde{P}_a^{1/2}(dy)\}^2\right]\lambda_A(da) \\
  &\le \left(2\sup_{P' \in \mathcal{P}} \esssup_{a, x} E_{P'}[Y^2 \mid A=a,X=x]\right)\int \left[\int \{P_a^{1/2}(dy)-\tilde{P}_a^{1/2}(dy)\}^2\right]\lambda_A(da) \\
  &= \left(2\sup_{P' \in \mathcal{P}} \esssup_{a, x} E_{P'}[Y^2 \mid A=a,X=x]\right) \\
  &\quad \cdot\int \left[\int \frac{1}{g(a\mid x)}\{g^{1/2}(a\mid x)P_a^{1/2}(dy)-g^{1/2}(a\mid x)\tilde{P}_a^{1/2}(dy)\}^2\right]\lambda_A(da) \\
  &\le \frac{2\sup_{P' \in \mathcal{P}} \esssup_{a, x} E_{P'}[Y^2 \mid A=a,X=x]}{\inf_{P'\in\mathcal{P}}\essinf_{(a,x) \in \mathcal{X}} g_{P'}(a\mid x)}\int [P^{1/2}(dz)-\tilde{P}^{1/2}(dz)]^2 \\
  &= \frac{2\sup_{P' \in \mathcal{P}} \esssup_{a, x} E_{P'}[Y^2 \mid A=a,X=x]}{\inf_{P'\in\mathcal{P}}\essinf_{(a,x) \in \mathcal{X}} g_{P'}(a\mid x)}\,H^2(P,\tilde{P}).
\end{align*}
The first inequality above holds by Cauchy-Schwarz, the second by the fact that $(b+c)^2\le 2(b^2+c^2)$, the third by the H\"{o}lder's inequality with exponents $(p,q)=(1,\infty)$, and the fourth by the strong positivity assumption. The constant in front of $H^2(P,\tilde{P})$ above is finite by the assumptions regarding the uniform boundedness of the conditional second moment of $Y$ across distributions in $\mathcal{P}$ and the strong positivity assumption. Hence, $\nu$ is Lipschitz over $\mathcal{P}_g$. This establishes \ref{it:localLip} of Lemma~\ref{lem:suffCondsPD} when the model is $\mathcal{P}_g$, where $g$ is an arbitrary value of the propensity to receive treatment for which there exists some $P'\in\mathcal{P}$ such that $g=g_{P'}$.

Hereafter we fix $P\in\mathcal{P}$ and suppose that $g=g_P$. We now establish \ref{it:closure} of Lemma~\ref{lem:suffCondsPD} at $P$ for the model $\mathcal{P}_g$ with $\eta_P(s)(a)$ as defined on the right-hand side of \eqref{eq:conacLocalP}. To do this, we use the following model: $\{P_{\epsilon} : \epsilon\in [0,\delta)\}$, where $\frac{dP_{\epsilon,X}}{dP_X}(x)= 1 + \epsilon s_X(x)$, $\frac{dP_{\epsilon,A \mid X}}{dP_{A \mid X}}(a\mid x) = 1$, and $\frac{dP_{\epsilon,Y \mid A,X}}{dP_{Y \mid A,X}}(y\mid a,x) = 1 + \epsilon s_{Y \mid A,X}(y\mid a,x)$, where $s_X$ and $s_{Y\mid A,X}$ are bounded in $[-\delta^{-1}/2,\delta^{-1}/2]$ and $E_P[s_X(X)]=0$ and $E_P[s_{Y\mid A,X}(Y\mid A,X)\mid A,X] = 0$ $P$-almost surely. As in Appendix~\ref{app:cdfPD}, we assume that $\{P_{\epsilon} : \epsilon\in [0,\delta)\}$ is a submodel of $\mathcal{P}_g$ without loss of generality. This submodel has score $s(x,a,y)=s_X(x) + s_{Y\mid A,X}(y\mid a,x)$ at $\epsilon=0$ and the $L^2(P)$-closure of the set containing such scores corresponds to the tangent space of $\mathcal{P}_g$ at $P$. It holds that
\begin{align*}
&\left\|\nu(P_\epsilon)-\nu(P)-\epsilon\eta_P(s)\right\|_{L^2(\lambda_A)}^2 \\
&= \int\left[\nu(P_\epsilon)(a)-\nu(P)(a)-\epsilon\eta_P(s)(a)\right]^2\lambda_A(da) \\
&= \int\Bigg[\iint y[1+\epsilon s_{Y\mid A,X}(y\mid a,x)][1+\epsilon s_X(x)]P_{Y\mid A,X}(dy\mid a,x)P_X(dx) \\
&\hspace{3em}-\iint y P_{Y\mid A,X}(dy\mid a,x)P_X(dx)-\epsilon\eta_P(s)(a)\Bigg]^2\lambda_A(da) \\
&= \epsilon^4\int\Bigg[\iint y s_{Y\mid A,X}(y\mid a,x) s_X(x) P_{Y\mid A,X}(dy\mid a,x)P_X(dx)\Bigg]^2\lambda_A(da) \\
&\le \epsilon^4\delta^{-4}\int\Bigg[\iint |y| P_{Y\mid A,X}(dy\mid a,x)P_X(dx)\Bigg]^2\lambda_A(da) \\
&\le \epsilon^4\delta^{-4}\iiint y^2 P_{Y\mid A,X}(dy\mid a,x)P_X(dx)\lambda_A(da),
\end{align*}
where the first inequality holds by the bounds on the ranges of $s_{Y\mid A,X}$ and $s_X$, and the second holds by Jensen's inequality. By the bounds on the conditional second moment of $Y$ under $P$, the right-hand side is $O(\epsilon^4)$, and so is $o(\epsilon^2)$ with much to spare. Hence, $\| \nu(P_\epsilon) - \nu(P) - \epsilon\, \eta_P(s) \|_{\mathcal{H}} = o(\epsilon)$. 

We now verify that $\eta_P$ is a bounded operator. When combined with the linearity of $\eta_P$, this will then show that \ref{it:closure} of Lemma~\ref{lem:suffCondsPD} holds at $P$ for the model $\mathcal{P}_g$. Take any $s$ in the tangent space of $\mathcal{P}_g$ at $P$. Let $s_{Y\mid A,X}(y\mid a,x) := s(x,a,y)-E_P[s(X,A,Y)\mid A=a,X=x]$ and $s_{X}(x) := E_P[s(X,A,Y)\mid X=x]$. It can be verified that, $E_P[s(X,A,Y)\mid A,X]-E_P[s(X,A,Y)\mid X]=0$ $P$-a.s., and so $s=s_{Y\mid A,X}+s_{X}$. Using that $(b+c)^2\le 2(b^2+c^2)$, applying Cauchy-Schwarz and H\"{o}lder's inequalities, and leveraging Fubini's theorem, we find that
\begin{align*}
&\|\eta_P(s)\|_{L^2(\lambda_A)}^2 \\
&\le 2 \int \left[ \iint y s_{Y \mid A,X}(y\mid a,x)P_{Y \mid A,X}(dy \mid a, x) P_X(dx) \right]^2 \lambda_A(da) \\
    &\quad + 2 \int \left[\int \mu_P(a,x) s_X(x) P_X(dx)\right]^2 \lambda_A(da) \\
    &\le 2 \int \left[ \iint y^2 P_{Y \mid A,X}(dy \mid a, x) P_X(dx) \right] \\
    &\quad\quad\quad\cdot\left[ \iint  s_{Y \mid A,X}^2(y\mid a,x)P_{Y \mid A,X}(dy \mid a, x) P_X(dx) \right] \lambda_A(da) \\
    &\quad + 2 \int \left[\int \mu_P^2(a,x) P_X(dx)\right]\left[\int s_X^2(x) P_X(dx)\right] \lambda_A(da) \\
    &\le 2 \left\{\esssup_{a,x}E_P[Y^2\mid A=a,X=x]\right\} \\
    &\quad\cdot \iint \left[\int s_{Y \mid A,X}^2(y\mid a,x)P_{Y \mid A,X}(dy \mid a, x) + s_X^2(x)\right] P_X(dx) \lambda_A(da) \\
    &= 2 \left\{\esssup_{a,x}E_P[Y^2\mid A=a,X=x]\right\} \\
    &\quad\cdot \iint \frac{1}{g_P(a\mid x)}\left[\int s_{Y \mid A,X}^2(y\mid a,x)P_{Y \mid A,X}(dy \mid a, x) + s_X^2(x)\right] \\
    &\quad\quad\quad\cdot P_X(dx)g_P(a\mid x) \lambda_A(da) \\
    &\le 2 \frac{\esssup_{a,x}E_P[Y^2\mid A=a,X=x]}{\essinf_{a,x}g_P(a\mid x)} \\
    &\quad\cdot \iint \left[\int s_{Y \mid A,X}^2(y\mid a,x)P_{Y \mid A,X}(dy \mid a, x) + s_X^2(x)\right] P_X(dx)g_P(a\mid x) \lambda_A(da) \\
    &= 2 \frac{\esssup_{a,x}E_P[Y^2\mid A=a,X=x]}{\essinf_{a,x}g_P(a\mid x)}\int \left[s_{Y \mid A,X}^2(y\mid a,x) + s_X^2(x)\right]P(dz) \\
    &= 2 \frac{\esssup_{a,x}E_P[Y^2\mid A=a,X=x]}{\essinf_{a,x}g_P(a\mid x)}\|s\|_{L^2(P)}^2.
\end{align*}
Above all essential suprema and infima are under the joint distribution of $(A,X)$ implied by $P$. The fraction above is finite by the strong positivity assumption and the assumed bound on the conditional second moment of $Y$. Hence, $\eta_P$ is a bounded operator. By Lemma~\ref{lem:suffCondsPD}, $\nu$ is pathwise differentiable at $P$ relative to $\mathcal{P}_g$ with $\dot{\nu}_P=\eta_P$. In the same way as was done in \eqref{eq:extendPD} for Example~\ref{ex:cfdNonparametric}, this pathwise differentiability over $\mathcal{P}_g$ can be extended to show that $\nu$ is pathwise differentiable over the locally nonparametric model $\mathcal{P}$.

\subsubsection{Efficient influence operator}
Take any $s\in L_0^2(P)$. Let $s_{Y\mid A,X}(y\mid a,x) := s(x,a,y)-E_P[s(X,A,Y)\mid A=a,X=x]$ and $s_{X}(x) := E_P[s(X,A,Y)\mid X=x]$. To compute the adjoint of $\dot{\nu}_P$, note that, for any $h \in L^2(\lambda_A)$,
\begin{align*}
    \langle h, \dot{\nu}_P(s) \rangle_{L^2(\lambda_A)} &= \iiint \{y-\mu_P(a,x)\}s_{Y \mid A,X}(y,a,x) h(a) P_{Y \mid A,X}(dy \mid a,x) P_X(dx) \lambda_A(da) \\
    &\quad + \iint [\mu_P(a,x)-\nu(P)(a)] h(a) s_X(x) P_X(dx) \lambda_A(da).
\end{align*}
The first term may be rearranged as follows:
\begin{align*}
    &\iiint \{y-\mu_P(a,x)\}s_{Y \mid A,X}(y,a,x)h(a) P_{Y \mid A,X}(dy \mid a,x) P_X(dx) \lambda_A(da) \\
    &= \int \frac{y-\mu_P(a,x)}{g_P(a \mid x)} h(a) s_{Y \mid A,X}(y\mid a,x)P(dz) \\
    &= \int \frac{y-\mu_P(a,x)}{g_P(a \mid x)} h(a) s(z) P(dz).
\end{align*}
For the second term:
\begin{align*}
    &\iint [\mu_P(a,x)-\nu(P)(a)] h(a) s_X(x) P_X(dx) \lambda_A(da) \\
    &= \int \left[ \int [\mu_P(a,x)-\nu(P)(a)] h(a) \lambda_A(da) \right] s_X(x) P_X(dx) \\
    &= \int \left[ \int [\mu_P(a',x)-\nu(P)(a')] h(a') \lambda_A(da') \right] s(z) P(dz).
\end{align*}
Thus,
$$\dot{\nu}_P^\ast(h)(y,a,x) = \frac{y-\mu_P(a,x)}{g_P(a \mid x)}h(a) + \int [\mu_P(a',x)-\nu(P)(a')]h(a') d\lambda_A(a').$$

\subsubsection{Study of regularized one-step estimator}\label{app:conacOS}

Since there is no EIF in this example, we study a regularized one-step estimator $\bar{\nu}_n^{\beta_n}$. This estimator is defined based on an orthonormal basis $(h_k)_{k=1}^\infty$ --- guidance for selecting this basis is given in Section~\ref{sec:regularizedOneStepTuning}. We study the regularized remainder, regularized drift, and bias terms appearing in Theorem~\ref{thm:alReg} and establish a rate of convergence of $\bar{\nu}_n^{\beta_n}$ for an appropriately chosen sequence of regularization parameters $\beta_n$. In what follows, $C$ denotes a generic finite constant whose value may differ from display to display.

We begin by bounding the regularized remainder terms. We use Lemma~\ref{lem:regRemBd} to derive our bound. To this end, we note that, for any $P\in\mathcal{P}$ and $k\in\mathbb{N}$,
\begin{align*}
    \big\langle\nu(P) &- \nu(P_0),h_k\big\rangle_{\mathcal{H}} + P_0 \dot{\nu}_{P}^*(h_k) \\
    &= E_0\left[\frac{Y-\mu_P(A,X)}{g_P(A\mid X)}h_k(A) + \int [\mu_P(a,X)-\nu(P_0)(a)]h_k(a)\lambda_A(da)\right] \\
    &= E_0\left[\frac{\mu_0(A,X)-\mu_P(A,X)}{g_P(A\mid X)}h_k(A) + \int [\mu_P(a,X)-\mu_0(a,X)]h_k(a)\lambda_A(da)\right] \\
    &= E_0\left[\int \left[1-\frac{g_0(a\mid X)}{g_P(a\mid X)}\right][\mu_P(a,X)-\mu_0(a,X)]h_k(a)\lambda_A(da)\right].
\end{align*}
From here, different bounds are possible, depending on the basis $(h_k)_{k=1}^\infty$. If the functions in $h_k$ are uniformly bounded --- that is, $\sup_{a\in\mathcal{A},k\in\mathbb{N}}|h_k(a)|<\infty$ --- then the strong positivity assumption and the Cauchy-Schwarz inequality together show that there is a finite constant $C$ that does not depend on $P\in\mathcal{P}$ or $k\in\mathbb{N}$ such that
\begin{align*}
\left|\left\langle\nu(P) - \nu(P_0),h_k\right\rangle_{\mathcal{H}} + P_0 \dot{\nu}_{P}^*(h_k)\right|&\le C\|g_P-g_0\|_{L^2(\lambda_A\times P_{0,X})}\|\mu_P-\mu_0\|_{L^2(\lambda_A\times P_{0,X})},
\end{align*}
where, for $f : (a,x)\mapsto \mathbb{R}$ and $q\ge 1$, $\|f\|_{L^q(\lambda_A\times P_{0,X})}^q:=\iint f(a,x)^q\, \lambda_A(da)P_{0,X}(dx)$ and $g_P-g_0$ denotes the function $(a,x)\mapsto g_P(a\mid x)-g_0(a\mid x)$. Requiring functions in $(h_k)_{k=1}^\infty$ to be uniformly bounded is not such a strong condition, with this condition being satisfied by both the trigonometric and cosine bases for $L^2([0,1])$. If a basis is used that does not satisfy this assumption, then the following alternative bound can be derived by twice applying Cauchy-Schwarz, using Jensen's inequality, invoking the strong positivity assumption, and leveraging the fact that all elements of $(h_k)_{k=1}^\infty$ have unit length in $L^2([0,1])$: there exists a $C$ that does not depend on $P\in\mathcal{P}$ or $k\in\mathbb{N}$ such that
\begin{align*}
    \left|\left\langle\nu(P) - \nu(P_0),h_k\right\rangle_{\mathcal{H}} + P_0 \dot{\nu}_{P}^*(h_k)\right|&\le C\|g_P-g_0\|_{L^4(\lambda_A\times P_{0,X})}\|\mu_P-\mu_0\|_{L^4(\lambda_A\times P_{0,X})}.
\end{align*}
Plugging the above bounds into Lemma~\ref{lem:regRemBd} shows that, depending on whether or not the functions in $(h_k)_{k=1}^\infty$ are uniformly bounded ($q=2$) or not ($q=4$), there exists a $C<\infty$ such that
\begin{align*}
    \|\mathcal{R}_P^\beta\|_{\mathcal{H}}\le C\|\beta\|_{\ell^2}\|g_P-g_0\|_{L^q(\lambda_A\times P_{0,X})}\|\mu_P-\mu_0\|_{L^q(\lambda_A\times P_{0,X})}.
\end{align*}
Hence, for each $j\in\{1,2\}$, $\mathcal{R}_n^{j,\beta_n}:=\mathcal{R}_{\widehat{P}_n^j}^{\beta_n}$ will be $O_p(\|\beta_n\|_{\ell^2}/n^{1/2})$ provided the product of the rates of convergence in probability of $g_{\widehat{P}_n^j}$ and $\mu_{\widehat{P}_n^j}$ to $g_0$ and $\mu_0$ under the $L^q(\lambda_A\times P_{0,X})$ norm is at least $n^{-1/2}$.

We now turn to the regularized drift terms. We will bound them via Lemma~\ref{lem:driftTermRegularized}. We begin by noting that, for any $P\in\mathcal{P}$ and $\beta\in\ell_*^2$,
\begin{align*}
    \|\phi_P^\beta-\phi_0^\beta\|_{L^2(P_0;\mathcal{H})}^2&=\sum_{k=1}^\infty \beta_k^2 \left\|\dot{\nu}_P(h_k)-\dot{\nu}_0(h_k)\right\|_{L^2(P_0)}^2 \\
    &\le \|\beta\|_{\ell^2}^2\sup_{k\in\mathbb{N}}\left\|\dot{\nu}_P(h_k)-\dot{\nu}_0(h_k)\right\|_{L^2(P_0)}^2.
\end{align*}
It can further be shown that there exists a constant $C<\infty$ that does not depend on $P\in\mathcal{P}$ or $k\in\mathbb{N}$ such that
\begin{align*}
    \left\|\dot{\nu}_P(h_k)-\dot{\nu}_0(h_k)\right\|_{L^2(P_0)}&\le C\left(\|g_P-g_0\|_{L^2(\lambda_A\times P_{0,X})} + \|\mu_P-\mu_0\|_{L^2(\lambda_A\times P_{0,X})}\right).
\end{align*}
Combining the preceding two displays shows that
\begin{align*}
 \|\phi_P^\beta-\phi_0^\beta\|_{L^2(P_0;\mathcal{H})}&\le C\|\beta\|_{\ell^2}\left(\|g_P-g_0\|_{L^2(\lambda_A\times P_{0,X})} + \|\mu_P-\mu_0\|_{L^2(\lambda_A\times P_{0,X})}\right).
\end{align*}
Hence, for each $j\in\{1,2\}$, $\|\phi_n^{j,\beta_n}-\phi_0^{\beta_n}\|_{L^2(P_0;\mathcal{H})}=o_p(\|\beta_n\|_{\ell^2})$ provided $g_{\widehat{P}_n^j}$ and $\mu_{\widehat{P}_n^j}$ converge in probability to $g_0$ and $\mu_0$ under the $L^2(\lambda_A\times P_{0,X})$ norm. Since no requirement is made on the rate of convergence and the $L^4(\lambda_A\times P_{0,X})$ is stronger than the $L^2(\lambda_A\times P_{0,X})$ norm, this condition will typically be weaker than the condition required above to make the regularized remainder term negligible. In any case, under this consistency condition, Lemma~\ref{lem:driftTermRegularized} shows that $\mathcal{D}_n^{j,\beta_n}$ is $o_p(\|\beta_n\|_{\ell^2}/n^{1/2})$, as desired.

The analysis of the bias term is nearly identical to the one given for Example~\ref{ex:cfdNonparametric} in the main text. In particular, if the first $K_n$ entries of $\beta_n$ are one and all the others are zero, then, provided $\sup_{P\in\mathcal{P}}\|\nu(P)\|_u<\infty$, Lemma~\ref{lem:biasTerm} shows that $\|\mathcal{B}_n^{j,\beta_n}\|_{\mathcal{H}}\le c/(K_n+1)^u$. Hence, if $K_n$ is of the order $n^{1/(2u+1)}$ and the regularized remainder and drift terms are $O_p(\|\beta_n\|_{\ell^2}/n^{1/2})$, then
\begin{align*}
\bar{\nu}_n^{\beta_n}-\nu(P_0)&= O_p(n^{-u/(2u+1)}).
\end{align*}

\subsection{Example~\ref{ex:gkmed}: counterfactual kernel mean embedding}\label{app:gkmed}

\subsubsection{Pathwise differentiability}

We now show that $\nu$ is pathwise differentiable relative to a locally nonparametric model $\mathcal{P}$ at any $P\in\mathcal{P}$. To do this, we follow similar arguments to those used in Appendix~\ref{app:cdfPD}. In particular, first we use Lemma~\ref{lem:suffCondsPD} to establish that $\nu$ is pathwise differentiable relative to the same model $\mathcal{P}_g$ considered in Appendix~\ref{app:cdfPD}. Second, we use the fact that $\nu$ does not depend on the propensity to receive treatment to extend this pathwise differentiability result to the locally nonparametric model $\mathcal{P}$.

Let $g$ be such that $g=g_{P'}$ for some fixed $P'\in\mathcal{P}$. We first show that $\nu$ is Lipschitz over $\mathcal{P}_g$. Fix $P, \tilde{P} \in \mathcal{P}_g$. For each $a\in\mathcal{A}$, let $P_a$ and $\tilde{P}_a$ denote the distributions on $\mathbb{R}$ defined so that, for any Borel set $B$, $P_1(B)=\int_{\mathcal{X}}\int_B P_{Y\mid A,X}(dy\mid 1,x)P_X(dx)$ and $\tilde{P}_1(B)=\int_{\mathcal{X}}\int_B \tilde{P}_{Y\mid A,X}(dy\mid 1,x)\tilde{P}_X(dx)$. Observe that
\begin{align*}
    \|&\nu(P)-\nu(\tilde{P})\|_{\mathcal{H}}^2 \\
    &= \iint \kappa(y_1,y_2) \prod_{i=1}^2 (P_1 - \tilde{P}_1) (dy_i) \\
    &= \iint \kappa(y_1,y_2) \prod_{i=1}^2 \frac{a_i}{g(a_i\mid x)}(P - \tilde{P}) (dz_i) \\
    &= \iint \kappa(y_1,y_2) \prod_{i=1}^2 \frac{a_i}{g(a_i\mid x)}\left[\sqrt{dP(z_i)} + \sqrt{d\tilde{P}(z_i)}\right] \left[\sqrt{dP(z_i)} - \sqrt{d\tilde{P}(z_i)}\right]  \\
    &\le \left(\iint \kappa^2(y_1,y_2) \prod_{i=1}^2 \frac{a_i}{g^2(a_i\mid x)}\left[\sqrt{dP(z_i)} + \sqrt{d\tilde{P}(z_i)}\right]^2\right)^{1/2}  \\
    &\quad\cdot \left(\iint \prod_{i=1}^2 \left[\sqrt{dP(z_i)} - \sqrt{d\tilde{P}(z_i)}\right]^2\right)^{1/2},
\end{align*}
where the inequality holds by Cauchy-Schwarz. The latter of the two terms in the product on the right-hand side is equal to $H^2(P,\tilde{P})$. Using the inequality $(b+c)^2\le 2(b^2+c^2)$ and then applying H\"{o}lder's inequality with exponents $(p,q)=(1,\infty)$, the square of the former term in this product bounds as follows:
\begin{align*}
&\iint \kappa^2(y_1,y_2) \prod_{i=1}^2 \frac{a_i}{g^2(a_i\mid x)}\left[\sqrt{dP(z_i)} + \sqrt{d\tilde{P}(z_i)}\right]^2 \\
&\quad\le 2\iint \kappa^2(y_1,y_2) \prod_{i=1}^2 \frac{a_i}{g^2(a_i\mid x)}(P+\tilde{P})(dz_i) \\
&\quad\le \frac{2\sup_{y_1,y_2\in\mathcal{Y}}\kappa^2(y_1,y_2)}{\inf_{P'\in\mathcal{P}}\essinf_x g_{P'}^2(1\mid x)}\ .
\end{align*}
The right-hand side above is finite by the strong positivity assumption and the fact that $\kappa$ is bounded. Hence, $\nu$ is Lipschitz over $\mathcal{P}_g$. Combining the preceding two displays establishes \ref{it:localLip} of Lemma~\ref{lem:suffCondsPD} when the model is $\mathcal{P}_g$, where $g$ is an arbitrary value of the propensity to receive treatment for which there exists some $P'\in\mathcal{P}$ such that $g=g_{P'}$.

Hereafter we fix $P\in\mathcal{P}$ and suppose that $g=g_P$. We now establish \ref{it:closure} of Lemma~\ref{lem:suffCondsPD} at $P$ for the model $\mathcal{P}_g$ with $\eta_P(s)$ as defined on the right-hand side of \eqref{eq:gkmedLocalP}. To do this, we use the following model: $\{P_{\epsilon} : \epsilon\in [0,\delta)\}$, where $\frac{dP_{\epsilon,X}}{dP_X}(x)= 1 + \epsilon s_X(x)$, $\frac{dP_{\epsilon,A \mid X}}{dP_{A \mid X}}(a\mid x) = 1$, and $\frac{dP_{\epsilon,Y \mid A,X}}{dP_{Y \mid A,X}}(y\mid a,x) = 1 + \epsilon s_{Y \mid A,X}(y\mid a,x)$, where $s_X$ and $s_{Y\mid A,X}$ are bounded in $[-\delta^{-1}/2,\delta^{-1}/2]$ and $E_P[s_X(X)]=0$ and $E_P[s_{Y\mid A,X}(Y\mid A,X)\mid A,X] = 0$ $P$-almost surely. As in Appendix~\ref{app:cdfPD}, we assume that $\{P_{\epsilon} : \epsilon\in [0,\delta)\}$ is a submodel of $\mathcal{P}_g$ without loss of generality. This submodel has score $s(x,a,y)=s_X(x) + s_{Y\mid A,X}(y\mid a,x)$ at $\epsilon=0$ and the $L^2(P)$-closure of the set containing such scores corresponds to the tangent space of $\mathcal{P}_g$ at $P$. It holds that
\begin{align*}
&\left\|\nu(P_\epsilon)-\nu(P)-\epsilon\eta_P(s)\right\|_{\mathcal{H}}^2 \\
&\quad=\epsilon^4 \left\|\iint K_y s_{Y\mid A,X}(y\mid a,x)s_X(x) P_{Y\mid A,X}(dy\mid 1,x) P_X(dx)\right\|_{\mathcal{H}}^2 \\
&\quad=\epsilon^4\left\|\int \frac{a}{g_P(a\mid x)}K_y s_{Y\mid A,X}(y\mid a,x)s_X(x) P(dz)\right\|_{\mathcal{H}}^2.
\end{align*}
The right-hand side is certainly $o(\epsilon^2)$ if the squared $\mathcal{H}$-norm on that side is finite. To see that this is the case, note first that, by the strong positivity assumption and the fact that $\kappa$, $s_{Y\mid A,X}$, and $s_X$ are all bounded functions, $(x,a,y)\mapsto \frac{a}{g_P(a\mid x)}K_y s_{Y\mid A,X}(y\mid a,x)s_X(x)$ belongs to $L^2(P;\mathcal{H})$. Hence, that term satisfies the following:
\begin{align*}
&\left\|\int \frac{a}{g_P(a\mid x)}K_y s_{Y\mid A,X}(y\mid a,x)s_X(x) P(dz)\right\|_{\mathcal{H}}^2 \\
&=\iint \frac{a}{g_P(a\mid x)}\frac{a'}{g_P(a'\mid x')} \kappa(y,y')  s_{Y\mid A,X}(y\mid a,x)s_X(x) P(dz)P(dz') <\infty,
\end{align*}
where $z'=(x',a',y')$. This establishes that $\| \nu(P_\epsilon) - \nu(P) - \epsilon\, \eta_P(s) \|_{\mathcal{H}} = o(\epsilon)$.

We now verify that $\eta_P$ is a bounded operator. When combined with the linearity of $\eta_P$, this will then show that \ref{it:closure} of Lemma~\ref{lem:suffCondsPD} holds at $P$ for the model $\mathcal{P}_g$. Take any $s$ in the tangent space of $\mathcal{P}_g$ at $P$. Let $s_{Y\mid A,X}(y\mid a,x) := s(x,a,y)-E_P[s(X,A,Y)\mid A=a,X=x]$ and $s_{X}(x) := E_P[s(X,A,Y)\mid X=x]$. It can be verified that, $E_P[s(X,A,Y)\mid A,X]-E_P[s(X,A,Y)\mid X]=0$ $P$-a.s., and so $s=s_{Y\mid A,X}+s_{X}$. Since $s$ is $P$-square integrable, $s_{Y\mid A,X}$ and $s_X$ are as well. By rewriting the right-hand side of \eqref{eq:gkmedLocalP}, we see that $\eta_P$ satisfies:
\begin{align}
\eta_P(s) = \int \frac{a}{g_P(a\mid x)}K_y\,[s_{Y \mid A,X}(y \mid a,x) + s_X(x)]\,P(dz). \label{eq:gkmedAltLocalP}
\end{align}
By the strong positivity assumption, the fact that $\kappa$ is a bounded function, and the fact that $s_{Y\mid A,X}$ and $s_X$ belong to $L^2(P)$, $(x,a,y)\mapsto \frac{a}{g_P(a\mid x)}K_y\,[s_{Y \mid A,X}(y \mid a,x) + s_X(x)]$ belongs to $L^2(P;\mathcal{H})$. Hence,
\begin{align}
\|\eta_P(s)\|_{\mathcal{H}}^2&= \iint \frac{a}{g_P(a\mid x)}\frac{a'}{g_P(a'\mid x')}\kappa(y,y')\,[s_{Y \mid A,X}(y \mid a,x) + s_X(x)] \nonumber \\
&\hspace{3em}\cdot [s_{Y \mid A,X}(y' \mid a',x') + s_X(x')]\,P^2(dz,dz') \nonumber \\
&\le \iint \frac{a}{g_P(a\mid x)}\frac{a'}{g_P(a'\mid x')}\sqrt{\kappa(y,y)\kappa(y',y')}\,|s_{Y \mid A,X}(y \mid a,x) + s_X(x)| \nonumber \\
&\hspace{3em}\cdot |s_{Y \mid A,X}(y' \mid a',x') + s_X(x')|\,P^2(dz,dz') \nonumber \\
&= \left[\int \frac{a}{g_P(a\mid x)}\sqrt{\kappa(y,y)}\,|s_{Y \mid A,X}(y \mid a,x) + s_X(x)|\,P(dz)\right]^2 \nonumber \\
&\le \left[\int \frac{a}{g_P^2(a\mid x)}|\kappa(y,y)|\,P(dz)\right]\left[\int [s_{Y \mid A,X}(y \mid a,x) + s_X(x)]^2\,P(dz)\right] \nonumber \\
&\le \frac{\sup_{y\in\mathcal{Y}}|\kappa(y,y)|}{\inf_{P'\in\mathcal{P}} \essinf_x g_{P'}(1\mid x)}\left[\int [s_{Y \mid A,X}(y \mid a,x) + s_X(x)]^2\,P(dz)\right] \nonumber \\
&= \frac{\sup_{y\in\mathcal{Y}}|\kappa(y,y)|}{\inf_{P'\in\mathcal{P}} \essinf_x g_{P'}(1\mid x)}\|s\|_{L^2(P)}^2, \label{eq:bddLocalPGkmed}
\end{align}
where the first inequality holds by applying Jensen's inequality to bring the absolute value function inside the integral and then applying Cauchy-Schwarz to the (positive semidefinite) kernel $\kappa$, the second by Cauchy-Schwarz, and the third by H\"{o}lder's inequality with exponents $(p,q)=(1,\infty)$. The fraction on the right-hand side of \eqref{eq:bddLocalPGkmed} is finite by the strong positivity assumption and the fact that $\kappa$ is bounded, and so $\eta_P$ is a bounded operator. Hence, by Lemma~\ref{lem:suffCondsPD}, $\nu$ is pathwise differentiable relative to $\mathcal{P}_g$ with $\dot{\nu}_P=\eta_P$. In the same way as was done in \eqref{eq:extendPD} for Example~\ref{ex:cfdNonparametric}, this pathwise differentiability over $\mathcal{P}_g$ can be extended to show that $\nu$ is pathwise differentiable over the locally nonparametric model $\mathcal{P}$.

\subsubsection{Efficient influence operator}
Let $s \in \dot{\mathcal{P}}_P$ and $h \in \mathcal{H}$, and let $s_{Y\mid A,X}$ and $s_X$ be as defined above \eqref{eq:gkmedAltLocalP}. We have that
\begin{align*}
    \langle \dot{\nu}_P(s), h \rangle_{\mathcal{H}}&= \iint h(y) [s_{Y \mid A,X}(y\mid 1,x) + s_X(x)] P_{Y \mid A,X}(dy \mid 1,x) P_X(dx) \\
    &= \int \frac{a}{g_P(1\mid x)} \left\{h(y)-E[h(Y) \mid A=a,X=x]\right\} s(z) P(dz) \\ 
    &\quad+ \int \left(E_P[h(Y) \mid A=1,X=x]-E_P E_P[h(Y) \mid A=1,X]\right)s(z) P(dz).
\end{align*}
Hence,
\begin{align*}
    \dot{\nu}_P^\ast (h)(z)&= \frac{a}{g_P(1\mid x)}\left\{h(y)-E[h(Y) \mid A=1,X=x]\right\} \\
    &\quad+ E_P[h(Y) \mid A=1,X=x] - E_P E_P[h(Y) \mid A=1,X].
\end{align*}

\subsubsection{Efficient influence function}\label{app:gkmedEIF}
By Theorem~\ref{thm:eifProperties}, the EIF will take the form
\begin{align*}
    \phi_P(z)(y') = \dot{\nu}_P^\ast (K_{y'})(z) &= \frac{a}{g_P(1\mid x)}\left\{\kappa(y,y')-E[\kappa(Y,y') \mid A=a,X=x]\right\} \\
    &\quad+ E_P[\kappa(Y,y') \mid A=1,X=x] - E_P E_P[\kappa(Y,y') \mid A=1,X].
\end{align*}
provided we can show that this function belongs to $L^2(P;\mathcal{H})$. Defining $\mu_P^K(x) = E_P[K_Y \mid A=1,X=x]$ and noting that $E_P \mu_P^K(X) = \nu(P)$, we can rewrite the above as follows:
$$\phi_P(z) = \frac{a}{g_P(1\mid x)} [K_y - \mu_P^K (x)] + \mu_P^K (x) -  \nu(P).$$
The fact that $\phi_P\in L^2(P;\mathcal{H})$ follows from the strong positivity assumption and the fact that the kernel $\kappa$ is a bounded function.

\subsubsection{Study of one-step estimator}\label{app:gkmedOS}

The one-step estimator $\bar{\nu}_n:= \frac{1}{2}\sum_{j=1}^2 [\nu(\widehat{P}_n^j) + P_n^j \phi_n^j]$ that we study is a cross-fitted version of the estimator of the counterfactual kernel mean embedding introduced in Eq.~10 of \cite{fawkes2022doubly}. Our general results provide several new results about this estimator that did not appear in that earlier work. First, Theorem~\ref{thm:al} provides a set of conditions under which this estimator converges weakly to a tight limit. Second, when the conditions of Theorem~\ref{eq:convolution} hold, \eqref{eq:convolution} provides a precise sense in which $\bar{\nu}_n$ outperforms the inverse probability weighted estimator that was earlier introduced in \cite{muandet2021counterfactual}. The earlier work suggested that this estimator would be more efficient but did not provide any theoretical guarantees establishing this. Third, Theorems~\ref{thm:CIcoverage} and \ref{thm:threshEst} provide a means to construct bootstrap-based confidence sets and hypothesis tests regarding the counterfactual mean embedding, with accompanying theoretical guarantees. 
\cite{fawkes2022doubly} proposed using an alternative, permutation-based procedure for making inference, but no theoretical guarantees were provided ensuring type I error control, consistency, or local power of the resulting test.

The calculations needed to establish that the conditions of our Theorem~\ref{thm:al} hold are similar to those used to prove Theorem~1 in \cite{fawkes2022doubly} and those in our Appendix~\ref{app:cdRem}, and therefore we only summarize the main findings here. For the remainder term $\mathcal{R}_P:=\nu(P)+P\phi_0 - \nu(P_0)$, it holds that, for finite constants $C_1$ and $C_2$ that do not depend on $P\in\mathcal{P}$,
\begin{align*}
    &\|\mathcal{R}_P\|_{\mathcal{H}}\le C_1\|g_P(1\mid \cdot\,)-g_0(1\mid \cdot\,)\|_{L^2(P_{0,X})}\left[\int \|\mu_P^K-\mu_0^K\|_{\mathcal{H}}^2 P_{0,X}(dx)\right]^{1/2}, \\
    &\|\phi_P-\phi_0\|_{L^2(P_0;\mathcal{H})} \\
    &\quad\le C_2\left(\|g_P(1\mid \cdot\,)-g_0(1\mid \cdot\,)\|_{L^2(P_{0,X})} + \left[\int \|\mu_P^K-\mu_0^K\|_{\mathcal{H}}^2 P_{0,X}(dx)\right]^{1/2}\right).
\end{align*}
Taken together, these bounds show that the conditions of Theorem~\ref{thm:al} will be satisfied in this example when, for $j\in\{1,2\}$, $\widehat{P}_n^j$ is such that $g_{\widehat{P}_n^j}$ and $\mu_{\widehat{P}_n^j}^K$ converge to $g_0$ and $\mu_0^K$ in probability according to the norms above and, moreover, the product of their rates of convergence is faster than $n^{-1/2}$.

\subsection{Example~\ref{ex:den}: root-density function}

Fix a distribution $P\in\mathcal{P}$, score $s$ in the tangent set of $\mathcal{P}$ at $P$, and submodel $\{P_\epsilon: \epsilon\in [0,\delta)\}\in \mathscr{P}(P,\mathcal{P},s)$. Since $\nu$ is the square root of the density function in this example, the quadratic mean differentiability of $\{P_\epsilon: \epsilon\in [0,\delta)\}$ in \eqref{eq:qmd} is, by definition, equivalent to the pathwise differentiability of $\nu$ as defined in \eqref{eq:pdA}, with the local parameter $\dot{\nu}_P(s)$ equal to $\epsilon s \nu(P)/2$. This local parameter is a bounded operator since, for any $s$, $$\|\dot{\nu}_P(s)\|_{L^2(\lambda)}^2=\frac{1}{4}\int s^2\nu(P)^2 \,d\lambda = \frac{1}{4}\|s\|_{L^2(P)}^2.$$ To verify the claimed form of the efficient influence operator given in Appendix~\ref{app:exClassical}, we note that, for any $s\in \dot{\mathcal{P}}_P$ and $h\in L^2(\lambda)$,
\begin{align*}
\left\langle \dot{\nu}_P(s), h\right\rangle_{L^2(\lambda)}&= \frac{1}{2}\int h(z) s(z)\nu(P)(z) \,d\lambda(z) = \int \frac{h(z)}{2\nu(P)(z)} s(z)\,dP(z) \\
&= \int \left(\frac{h(z)}{2\nu(P)(z)} -E_P\left[\frac{h(Z)}{2\nu(P)(Z)}\right]\right)s(z)\,dP(z).
\end{align*}
As $s$ and $h$ were arbitrary, $\dot{\nu}_P^\ast(h)(z) = \frac{h(z)}{2\nu(P)(z)} - E_P \left[\frac{h(Z)}{2\nu(P)(Z)} \right]$.

\subsection{Example~\ref{ex:reg}: regression function}

\subsubsection{Pathwise differentiability}

Fix a distribution $P\in\mathcal{P}$ and suppose that $\frac{d\lambda_X}{dP_X}$ is bounded $P_X$-almost surely. We prove that $\nu$ is pathwise differentiable at $P$ relative to a locally nonparametric model and that $\dot{\nu}_P=\eta_P$, where
$$\eta_P(s)(x) = \int [y - \nu(P)(x)]s(x,y) P_{Y\mid X}(dy \mid x).$$
Fix a score function $s\in L_0^2(P)$. Let $\{P_{\epsilon}: \epsilon\}\in\mathscr{P}(P,\mathcal{P},s)$ and $q_\epsilon:= \sqrt{\frac{dP_{\epsilon}}{dP}}$. Let $q_{\epsilon,X}:= \sqrt{\frac{dP_{\epsilon,X}}{dP_X}}$ denote the square root of the marginal density, $q_{\epsilon, Y \mid X}(\cdot\mid x):=\sqrt{\frac{dP_{\epsilon,Y\mid X}}{dP_{Y\mid X}}(\,\cdot\mid x)}$ the square root of the conditional density, $s_X(x) = E_P[s(X,Y) \mid X=x]$ and  $s_{Y \mid X}(y\mid x) = s(x,y) - s_X(x)$. Observe that the following holds for $P_X$-almost all $x$:
\begin{align*}
    [\nu(P_\epsilon) &- \nu(P) - \epsilon \eta_P(s)](x) \\
    &= \int y [q_{\epsilon,Y \mid X}^2(y \mid x) - 1 - \epsilon s_{Y \mid X}(y \mid x)] P_{Y\mid X}(dy \mid x) \\
    &= \int y \left\{[q_{\epsilon,Y \mid X}(y \mid x) - 1][q_{\epsilon,Y \mid X}(y \mid x) + 1] - \epsilon s_{Y \mid X}(y \mid x)\right\} P_{Y\mid X}(dy \mid x) \\
    &= \int y \left\{\left[q_{\epsilon,Y \mid X}(y \mid x) - 1 - \frac{\epsilon}{2} s_{Y \mid X}(y \mid x)\right][q_{\epsilon,Y \mid X}(y \mid x) + 1] \right\} P_{Y\mid X}(dy \mid x) \\
    &\quad + \frac{\epsilon}{2} \int y s_{Y \mid X}(y \mid x)[q_{\epsilon,Y \mid X}(y \mid x) - 1]P_{Y\mid X}(dy \mid x).
\end{align*}
For shorthand, we refer to the first term on the right as $A_\epsilon(x)$ and the second as $\frac{\epsilon}{2}B_\epsilon(x)$. We will show that $\|A_{\epsilon}\|_{L^2(\lambda_X)} = o(\epsilon)$ and $\|B_{\epsilon}\|_{L^2(\lambda_X)} = o(1)$.
Combining this with the triangle inequality for the $L^2(\lambda_X)$ norm and the relation above will then give the result. 
For the first term, we note that the Cauchy-Schwarz inequality and the inequality $(a+b)^2 \leq 2a^2+2b^2$, we have, for $P_X$-almost all $x$,
$$|A_{\epsilon}(x)|^2 \leq \left\|q_{\epsilon, Y \mid X} - 1 - \frac{\epsilon}{2} s_{Y \mid X} \right\|_{L^2(P_{Y \mid X=x})}^2 \left( 2E_{P_{\epsilon}} [Y^2 \mid X=x] + 2 E_P [Y^2 \mid X=x]\right).$$
Integrating both sides above against $\lambda_X$, applying H\"{o}lder's inequality with exponents $(1,\infty)$, and applying Lemma \ref{lem:preserveDQM}, we find that
\begin{align*}
\|A_{\epsilon}\|_{L^2(\lambda_X)}^2 &\leq 2 \esssup_{x} \frac{d\lambda_X}{dP_X}(x)\left(E_{P_\epsilon}[Y^2 \mid X=x] + E_P [Y^2 \mid X=x]\right) \\
&\quad \cdot\left\| q_{\epsilon,Y \mid X} - 1 - \frac{\epsilon}{2} s_{Y \mid X} \right\|_{L^2(P)}^2 = o(\epsilon^2),
\end{align*}
where the essential supremum is over $P_X$. 
Above we used \eqref{eq:regSecondMoment} and the assumption that $\frac{d\lambda_X}{dP_X}$ is bounded with $P_X$-probability one. 

We now show that $\|B_\epsilon\|_{L^2(\lambda_X)}=o(1)$. Let $B_{\epsilon,1}(x) = \int 1\{|ys_{Y\mid X}(y\mid x)| \leq \epsilon^{-1/2}\} y s_{Y \mid X}(y \mid x)[q_{\epsilon,Y \mid X}(y \mid x) - 1] P_{Y\mid X}(dy \mid x)$ and $B_{\epsilon,2}(x) = B_{\epsilon}(x)-B_{\epsilon,1}(x)$. By the triangle inequality, it suffices to show that $\|B_{\epsilon,j}\|_{L^2(\lambda_X)}=o(1)$, $j\in\{1,2\}$. Using that $y^2 s_{Y\mid X}^2(y\mid x) 1\{|ys_{Y\mid X}(y\mid x)|\le \epsilon^{-1/2}\}\le \epsilon^{-1}$, Jensen's inequality, H\"{o}lder's inequality with exponents $(1,\infty)$, and Lemma \ref{lem:preserveDQM},
\begin{align*}
\|B_{\epsilon,1}\|_{L^2(\lambda_X)}^2 &\leq \epsilon^{-1} \left[\esssup_x\frac{d\lambda_X}{dP_X}(x)\right]\|q_{\epsilon,Y \mid X}-1\|_{L^2(P)}^2 = O(\epsilon),
\end{align*}
where the essential supremum is over $P_X$. 
By the Cauchy-Schwartz inequality and the inequality $(a-b)^2 \leq 2(a^2 + b^2)$, the following holds for $P_X$-almost all $x$:
\begin{align*}
    |B_{\epsilon,2}(x)|^2 &\leq 2 (E_{P_\epsilon} [Y^2 \mid X=x] + E_P [Y^2 \mid X=x]) \\
    &\quad\cdot \int s_{Y \mid X}^2(y\mid x) 1\{y > \epsilon^{-1/2}\} P_{Y\mid X}(dy \mid x).
\end{align*}
Integrating both sides over $\lambda_X$ and applying H\"{o}lder's inequality with exponents $(1,\infty)$ gives that
\begin{align*}
\|B_{\epsilon,2}\|_{L^2(\lambda_X)}^2 &\leq 2 \left[\esssup_{x} \frac{d\lambda_X}{dP_X}(x)(E_{P_\epsilon}[Y^2 \mid X=x] + E_P [Y^2 \mid X=x])\right] \\
&\quad\cdot\int s_{Y \mid X}^2(y\mid x) 1\{y > \epsilon^{-1/2}\} dP(z) \\
&= o(1),
\end{align*}
where the essential supremum is over $P_X$. Thus, $\|\nu(P_\epsilon) - \nu(P) - \epsilon \eta_P(s)\|_{L^2(\lambda_X)} = o(\epsilon)$.

We now prove that $\eta_P$ is a bounded operator. For any $s \in \dot{\mathcal{P}}_P$, applying Cauchy-Schwarz followed by H\"{o}lder's inequality with exponents $(1,\infty)$ shows that
\begin{align*}
    \|\eta_P(s)\|_{L^2(\lambda_X)}^2 &= \int \frac{d\lambda_X}{dP_X}(x)\left[ \int [y-\nu(P)(x)]s_{Y \mid X}(y \mid x) P_{Y\mid X}(dy \mid x) \right]^2 P_X(dx) \\
    &\leq \int \frac{d\lambda_X}{dP_X}(x)\text{Var}_P(Y \mid X=x) \left[\int |s_{Y \mid X}(y \mid x)|^2 P_{Y\mid X}(dy \mid x) \right] P_X(dx) \\
    &\leq \left[\esssup_{x} \frac{d\lambda_X}{dP_X}(x)\text{Var}_P(Y \mid X=x)\right]\|s\|_{L^2(P)}^2,
\end{align*}
where the essential supremum is over $P_X$. 
Hence, $$\|\eta_P\|_{\mathrm{op}}\le \esssup_{x} \sqrt{\frac{d\lambda_X}{dP_X}(x)\text{Var}_P(Y \mid X=x)},$$ which is finite by \eqref{eq:regSecondMoment} and the assumption that $\frac{d\lambda_X}{dP_X}$ is $P_X$-a.s. bounded. Since $\eta_P$ is also linear, $\nu$ is pathwise differentiable with local parameter $\dot{\nu}_P=\eta_P$.

\subsubsection{Efficient influence operator}
For any $h \in L^2(\lambda_X)$ and $s \in \dot{\mathcal{P}}_P$, 
\begin{align*}
\langle \dot{\nu}_P(s),h \rangle_{L^2(\lambda_X)}&= \iint [y - \nu(P)(x)]h(x)s(x,y) P_{Y\mid X}(dy \mid x) \lambda_X(dx) \\
&= \iint \frac{d\lambda_X}{dP_X}(x)[y - \nu(P)(x)]h(x)s(x,y) P_{Y\mid X}(dy \mid x) P_X(dx) \\
&= \int \frac{d\lambda_X}{dP_X}(x)[y - \nu(P)(x)]h(x)s(x,y) P(dz) = \langle \dot{\nu}_P^*(h),s\rangle_{L^2(P)},
\end{align*}
where $\dot{\nu}_P^\ast(h)(z) = \frac{d\lambda_X}{dP_X}(x)[y - \nu(P)(x)]h(x)$. Hence, $\dot{\nu}_P^\ast$ is the efficient influence operator.

\subsection{Example~\ref{ex:kmed}: kernel mean embedding} \label{app:kme}

The parameter considered in this example is a special case of the counterfactual kernel mean embedding parameter considered in Example~\ref{ex:gkmed} when $A=1$ almost surely. Consequently, the proof of the pathwise differentiability of the parameter in this example, and also the calculation of its efficient influence operator and EIF, follow directly from those in Appendix~\ref{app:gkmed}.

\subsection{Example~\ref{ex:cate}: conditional average treatment effect}\label{app:cate}

The proof we provide does not require a new application of Lemma~\ref{lem:suffCondsPD}, but instead leverages the application of that lemma that we already worked out in Example~\ref{ex:reg}. First, we establish that $\nu$ is pathwise differentiable with the claimed local parameter relative to a semiparametric model where the propensity is known. We do this by leveraging the result from Example~\ref{ex:reg} to establish the pathwise differentiability of a regression of a certain pseudo-outcome against the covariates. Second, we establish that working in the larger, locally nonparametric model where this quantity is not known does not change this result: $\nu$ is still pathwise differentiable with the same local parameter. An alternative argument, which we do not give here, would entail directly applying Lemma~\ref{lem:suffCondsPD} when establishing pathwise differentiability in the semiparametric model considered in the first step.

Fix a distribution $P_0$ in the locally nonparametric model $\mathcal{P}$ that is such that $\frac{d\lambda_X}{dP_{0,X}}$ is bounded $P_{0,X}$-a.s. and let $g:= g_{P_0}$. We will establish pathwise differentiability at $P_0$ with the claimed efficient influence operator, and, since $P_0$ was arbitrary, this will establish the desired result. For a distribution $P\in\mathcal{P}$, let $P_g$ denote the distribution of $Z$ that has the same conditional distribution of $Y\mid A,X$ and marginal distribution of $X$ as $P$, but has propensity equal to $g$; in other words, for all bounded, continuous functions $f : \mathcal{Z}\rightarrow\mathbb{R}$, $E_{P_g}[f(Z)]=\int \sum_{a=0}^1 \int f(x,a,y) P_{Y\mid A,X}(dy\mid a,x) g(a\mid x) P_X(dx)$. The semiparametric model that we study is given by $\mathcal{P}_g:=\{P_g : P\in\mathcal{P}\}$. Without loss of generality, we suppose that $\mathcal{P}_g\subseteq \mathcal{P}$; if this is not the case, then we can simply extend the definition of $\nu$ to $\mathcal{P}_g$ by letting $\nu(P_g)=\nu(P)$ for any $P\in\mathcal{P}$. For any $P_g\in\mathcal{P}_g$, it can be verified that, for $\lambda_X$-almost all $X$,
\begin{align*}
    \nu(P_g)(x)&= E_{P_g}\left[W\,\middle|\,X=x\right],
\end{align*}
where $W:=(2A-1)Y/g(A\mid X)$. 
The above suggests that we can use our results from the regression setting in Example~\ref{ex:reg} to derive the efficient influence operator of $\nu$ relative to $\mathcal{P}_g$. To this end, we define the model $\widetilde{\mathcal{P}}_g:=\{P_g\circ f_g^{-1} : P_g\in\mathcal{P}_g\}$, where $P_g\circ f_g^{-1}$ is the pushforward measure of $P_g$ under $f_g(x,a,y):=(x,w):=(x,[2a-1]y/g(a\mid x)])$. We also define the parameter $\widetilde{\nu} : \widetilde{\mathcal{P}}_g\rightarrow \mathcal{H}$ so that $\widetilde{\nu}(P_g\circ f_g^{-1})=\nu(P_g)$, where this definition is valid even if there are two distinct distributions $P_g,P_g'\in\mathcal{P}_g$ that make it so that $P_g\circ f_g^{-1}=P_g'\circ f_g^{-1}$ since the preceding display shows that $\nu(P_g)=\nu(P_g')$ in this case. The parameter $\widetilde{\nu}$ takes as input a distribution of features $X$ and an outcome $W$ from a locally nonparametric model and outputs a regression function. Consequently, the results of Example~\ref{ex:reg} imply that, for any $P_g$ such that $\frac{dP_{g,X}}{d\lambda_X}$ is bounded, this parameter is pathwise differentiable at $\widetilde{P}_g:=P_g\circ f_g^{-1}$ with local parameter
$$\dot{\widetilde{\nu}}_{\widetilde{P}_g}(\tilde{s})(x) = \int \left[w - \widetilde{\nu}(\widetilde{P}_g)(x) \right] \tilde{s}(x,w) \,\widetilde{P}_g(dw \mid x).$$

We now use the pathwise differentiability of $\widetilde{\nu}$ relative to $\widetilde{\mathcal{P}}_g$ to establish the pathwise differentiability of $\nu$ relative to $\mathcal{P}_g$. To this end, let $\{P_{g,\epsilon} : \epsilon\in [0,\delta)\}\in \mathscr{P}(P_g,\mathcal{P}_g,s_g)$, where $P_g\in\mathcal{P}_g$ and $s_g$ belong to the tangent set of $\mathcal{P}_g$ at $P_g$. Similar arguments to those used  to establish Lemma~\ref{lem:preserveDQM} can be used to show that $\{P_{g,\epsilon}\circ f_g^{-1} : \epsilon\in [0,\delta)\}\in \mathscr{P}(\widetilde{P}_g,\widetilde{\mathcal{P}}_g,\tilde{s}_g)$, where $\tilde{s}_g(x,w):=E_{P_g}[s_g(X,A,Y)\mid X=x,W=w]$. Combining this with the pathwise differentiability of $\widetilde{\nu}$ relative to $\widetilde{\mathcal{P}}_g$ and the definition of $\widetilde{\nu}$ shows that
\begin{align*}
    &\left\|\nu(P_{g,\epsilon})-\nu(P_g) - \epsilon \dot{\widetilde{\nu}}_{P_g\circ f_g^{-1}}(\tilde{s}_g)\right\|_{\mathcal{H}} \\
    &\quad= \left\|\widetilde{\nu}(P_{g,\epsilon}\circ f_g^{-1})-\widetilde{\nu}(P_g\circ f_g^{-1}) - \epsilon \dot{\widetilde{\nu}}_{P_g\circ f_g^{-1}}(\tilde{s}_g)\right\|_{\mathcal{H}}=o(\epsilon).
\end{align*}
Since the operator $\dot{\nu}_{P_g} : \dot{\mathcal{P}}_{g,P_g}\rightarrow \mathcal{H}$ defined by $\dot{\nu}_{P_g}(s_g):=\dot{\nu}_{\widetilde{P}_g}(\tilde{s}_g)$ is bounded and linear, where $\dot{\mathcal{P}}_{g,P_g}$ denotes the tangent space of $\mathcal{P}_g$ at $P_g$, this shows that $\nu$ is pathwise differentiable at $P_g$ relative to $\mathcal{P}_g$ with local parameter $\dot{\nu}_{P_g}$.

We now use the pathwise differentiability of $\nu$ relative to $\mathcal{P}_g$ to establish its pathwise differentiability at $P_0$ relative to $\mathcal{P}$. Let $\{P_\epsilon : \epsilon\in [0,\delta)\}\in\mathscr{P}(P_0,\mathcal{P},s)$, where $s$ is the tangent set of $\mathcal{P}$ at $P_0$. Letting $P_{g,\epsilon}$ be the distribution that has the same conditional distribution of $Y\mid A,X$ and marginal distribution of $X$ as under $P_\epsilon$ but with propensity $g$, it can be shown that $\{P_{g,\epsilon} : \epsilon\in [0,\delta)\}\in\mathscr{P}(P_0,\mathcal{P}_g,s_g)$, where $s_g(z)=s(z)-E_0[s(Z)\mid A=a,X=x] + E_0[s(Z)\mid X=x]$. Combining this with the facts that $P_{g,0}=P_0$ and $\nu$ is invariant to changes in the propensity of its input, we find that
\begin{align*}
    \left\|\nu(P_\epsilon)-\nu(P_0) - \epsilon \dot{\nu}_0(s_g)\right\|_{\mathcal{H}}&= \left\|\nu(P_{g,\epsilon})-\nu(P_g) - \epsilon \dot{\nu}_0(s_g)\right\|_{\mathcal{H}}=o(\epsilon).
\end{align*}
The above establishes that $\nu$ is pathwise differentiable at $P_0$ relative to $\mathcal{P}$, with local parameter
\begin{align*}
    s\mapsto \dot{\nu}_0(s_g) &= \dot{\widetilde{\nu}}_{P_0\circ f_g^{-1}}(\tilde{s}_g) \\
    &=\int \left[w - \nu(P_0)(x) \right] E_0[s_g(X,A,Y)\mid X=x,W=w] \,(P_0\circ f_g^{-1})(dw \mid x) \\
    &=E_0\left\{\left[W-\nu(P_0)(X)\right]E_0[s_g(X,A,Y)\mid X,W]\,\middle|\,X=x\right\} \\
    &=E_0\left\{\left[W-\nu(P_0)(X)\right]s_g(X,A,Y)\,\middle|\,X=x\right\} \\
    &=E_0\left\{\left[W-E_0(W\mid A,X) + E_0(W\mid X)-\nu(P_0)(X)\right]s_g(X,A,Y)\,\middle|\,X=x\right\} \\
    &=E_0\left\{\left[W-E_0(W\mid A,X)\right]s_g(X,A,Y)\,\middle|\,X=x\right\} \\
    &=E_0\left\{\left[W-E_0(W\mid A,X)\right]s(X,A,Y)\,\middle|\,X=x\right\} \\
    &=E_0\left\{\left[\frac{2A-1}{g_{P_0}(A\mid X)}\left\{Y-\mu_{P_0,A}(X)\right\}\right]s(X,A,Y)\,\middle|\,X=x\right\},
\end{align*}
which matches the claimed form of the local parameter from \eqref{eq:cateLocalP}. To verify that the efficient influence operator takes the form in \eqref{eq:cateEIO}, it can be directly established that $\langle \dot{\nu}_0^\ast (h),s \rangle_{L^2(P_0)} = \langle h, \dot{\nu}_0(s) \rangle_{\mathcal{H}}$ for all $h\in\mathcal{H}$ and $s\in L_0^2(P_0)$. The calculations to establish this are straightforward and so are omitted.

\section{Proofs of results from the main text, and supporting lemmas} \label{app:lem}

\subsection{Proofs for Section~\ref{sec:pd}}

\subsubsection{Proofs for Section~\ref{sec:onestep}}

\begin{lemma} \label{lem:repProp}
Suppose $\dot{\mathcal{H}}_P$ is an RKHS, $\nu : \mathcal{P}\rightarrow\mathcal{H}$ is pathwise differentiable at $P$, and $\tilde{\phi}_P$ as defined in \eqref{eq:eifDef} is $P$-Bochner square integrable. For all $h \in \mathcal{H}$, define $\langle \tilde{\phi}_P, h \rangle_{\mathcal{H}}: \mathcal{Z} \to \mathbb{R}$ so that $\langle \tilde{\phi}_P, h \rangle_{\mathcal{H}} (z) = \langle \tilde{\phi}_P(z), h \rangle_{\mathcal{H}}$. Then, $\langle \tilde{\phi}_P, h \rangle_{\mathcal{H}} \in \dot{\mathcal{P}}_P$ and $\langle \tilde{\phi}_P, h \rangle_{\mathcal{H}} = \dot{\nu}_P^\ast(h)$ $P$-almost surely.
\end{lemma}
\begin{proof}
Fix $h\in\mathcal{H}$. 
The fact that $\langle \tilde{\phi}_P, h \rangle_{\mathcal{H}} \in L^2(P)$ follows from Cauchy-Schwarz and the fact that $\tilde{\phi}_P\in L^2(P;\mathcal{H})$. In particular,
$$ \int |\langle \tilde{\phi}_P(z), h \rangle_{\mathcal{H}}|^2 P(dz) \leq \|h\|_{\mathcal{H}}^2 \|\tilde{\phi}_P\|_{L^2(P;\mathcal{H})}^2 < \infty. $$
Let $s^\perp$ be an element of the orthogonal complement of the tangent space $\dot{\mathcal{P}}_P$. Then,
\begin{align*}
    \int \langle \tilde{\phi}_P(z), h \rangle_{\mathcal{H}} s^\perp(z) P(dz) &= \int \langle t \mapsto \dot{\nu}_P^\ast (K_t)(z), h \rangle_{\mathcal{H}} s^\perp(z) P(dz) \\
    &= \left\langle t \mapsto \int \dot{\nu}_P^\ast (K_t)(z) s^\perp(z) P(dz), h \right\rangle_{\mathcal{H}} \\
    &= \langle t \mapsto \langle \dot{\nu}_P^\ast (K_t), s^\perp \rangle_{L^2(P)}, h \rangle_{\mathcal{H}} = 0,
\end{align*}
where we use the $P$-Bochner square integrability of $\tilde{\phi}_P$ to interchange the integral and the inner product and use that $\dot{\nu}_P^\ast (K_t) \in \dot{\mathcal{P}}_P$ for all $t\in\mathcal{T}$. Note also that $\langle \dot{\nu}_P^\ast (h), s^\perp \rangle_{L^2(P)} = 0$
since $\dot{\nu}_P^\ast (h) \in \dot{\mathcal{P}}_P$. Hence, for any $s^\perp$ in the orthogonal complement of $\dot{\mathcal{P}}_P$,
\begin{align}
    \left\langle \langle \tilde{\phi}_P, h \rangle_{\mathcal{H}} - \dot{\nu}_P^\ast (h), s^\perp \right\rangle_{L^2(P)} = 0. \label{eq:sperp}
\end{align}
Let $s \in \dot{\mathcal{P}}_P$. Then, following some of the same calculations as earlier,
\begin{align*}
    \int \langle \tilde{\phi}_P(z), h \rangle_{\mathcal{H}} s(z) P(dz) 
    &= \langle t \mapsto \langle \dot{\nu}_P^\ast (K_t), s \rangle_{L^2(P)}, h \rangle_{\mathcal{H}}.
\end{align*}
Furthermore, for any $t\in\mathcal{T}$, $\langle \dot{\nu}_P^\ast (K_t), s \rangle_{L^2(P)} = \langle K_t, \dot{\nu}_P(s) \rangle_{\mathcal{H}}$. Since $\dot{\nu}_P(s)$ belongs to the RKHS $\dot{\mathcal{H}}_P$, $\langle K_t, \dot{\nu}_P(s) \rangle_{\mathcal{H}}=\dot{\nu}_P(s)(t)$. Plugging these observations into the above shows that the right-hand side is equal to $\langle \dot{\nu}_P(s), h \rangle_{\mathcal{H}}  = \langle \dot{\nu}_P^\ast (h), s \rangle_{L^2(P)}$, and so
\begin{align*}
    \left\langle \langle \tilde{\phi}_P, h \rangle_{\mathcal{H}} - \dot{\nu}_P^\ast (h), s \right\rangle_{L^2(P)} = 0.
\end{align*}
As the above holds for all $s$ in the tangent space $\dot{\mathcal{P}}_P$ and \eqref{eq:sperp} holds for all $s^\perp$ in its orthogonal complement, $\langle \tilde{\phi}_P, h \rangle_{\mathcal{H}} = \dot{\nu}_P^\ast (h)$ $P$-almost surely.
\end{proof}

\begin{lemma}\label{lem:RieszRep}
In the setting of Lemma~\ref{lem:repProp}, it is $P$-a.s. true that $\sup_{h\in\mathcal{H}}|\langle \tilde{\phi}_P, h \rangle_{\mathcal{H}} - \dot{\nu}_P^\ast(h)|=0$. Hence, for $P$-almost all $z$, $\dot{\nu}_P^*(\cdot)(z) : \mathcal{H}\rightarrow\mathbb{R}$ is a bounded linear functional with Riesz representation $\tilde{\phi}_P(z)$. 
In other words, $\tilde{\phi}_P$ is the EIF of $\nu$ at $P$.
\end{lemma}
\begin{proof}
Since we have assumed throughout that a separable version of the efficient influence process is used, there exists a countable dense subset $\mathcal{H}'$ of $\mathcal{H}$ and a $P$-probability one subset $\mathcal{Z}'$ of $\mathcal{Z}$ such that, for all $h\in\mathcal{H}$ and $z\in\mathcal{Z}'$, there exists an $\mathcal{H}'$-valued sequence $(h_j')_{j=1}^\infty$ that converges to $h$ and satisfies $\dot{\nu}_P^*(h_j')(z)\rightarrow \dot{\nu}_P^*(h)(z)$ as $j\rightarrow\infty$.
For each $h\in\mathcal{H}'$, we let
\begin{align*}
    \mathcal{Z}_h''=\{z\in\mathcal{Z}' : \langle \tilde{\phi}_P(z), h \rangle_{\mathcal{H}} - \dot{\nu}_P^\ast(h)(z)=0\},
\end{align*}
and we define $\mathcal{Z}'':=\cap_{h\in\mathcal{H}'}\,\mathcal{Z}_h''$. By Lemma~\ref{lem:repProp} and the fact that $\mathcal{Z}'$ is a $P$-probability one set, $P(\mathcal{Z}_h'')=1$ for each $h\in\mathcal{H}'$ and, as $\mathcal{H}'$ is countable, $P(\mathcal{Z}'')=1$ as well. In what follows we will show that, for all $z\in\mathcal{Z}''$, $\sup_{h\in\mathcal{H}}|\langle \tilde{\phi}_P(z), h \rangle_{\mathcal{H}} - \dot{\nu}_P^\ast(h)(z)|=0$. To this end, fix $z\in\mathcal{Z}''$ and $\epsilon>0$ and let $h_\epsilon\in \mathcal{H}$ be such that
\begin{align*}
\left|\langle \tilde{\phi}_P(z), h_\epsilon \rangle_{\mathcal{H}} - \dot{\nu}_P^\ast(h_\epsilon)(z)\right|\ge \sup_{h\in\mathcal{H}}\left|\langle \tilde{\phi}_P(z), h \rangle_{\mathcal{H}} - \dot{\nu}_P^\ast(h)(z)\right| - \epsilon.
\end{align*}
By the separability of the efficient influence process, there exists an $\mathcal{H}'$-valued sequence $(h_{\epsilon,j})$ that converges to $h_\epsilon$ that is such that $\dot{\nu}_P^\ast(h_{\epsilon,j})(z)\rightarrow \dot{\nu}_P^\ast(h_\epsilon)(z)$ as $j\rightarrow\infty$. By the continuity of the inner product $\langle \tilde{\phi}_P(z),\,\cdot\,\rangle_{\mathcal{H}}$, it also holds that $\langle \tilde{\phi}_P(z),h_{\epsilon,j}\rangle_{\mathcal{H}}\rightarrow \langle \tilde{\phi}_P(z),h_\epsilon\rangle_{\mathcal{H}}$ as $j\rightarrow\infty$. Consequently, $|\langle \tilde{\phi}_P(z), h_{\epsilon,j} \rangle_{\mathcal{H}} - \dot{\nu}_P^\ast(h_{\epsilon,j})(z)|$ converges to the left-hand side above as $j\rightarrow\infty$, and so there exists a sufficiently large $j$ such that
\begin{align*}
\left|\langle \tilde{\phi}_P(z), h_{\epsilon,j} \rangle_{\mathcal{H}} - \dot{\nu}_P^\ast(h_{\epsilon,j})(z)\right|\ge \sup_{h\in\mathcal{H}}\left|\langle \tilde{\phi}_P(z), h \rangle_{\mathcal{H}} - \dot{\nu}_P^\ast(h)(z)\right| - 2\epsilon.
\end{align*}
As $z$ is in $\mathcal{Z}''$, $z$ is in $\mathcal{Z}_{h_{\epsilon,j}}''$ as well. Hence, the left-hand side above is zero, which shows that
\begin{align*}
    \sup_{h\in\mathcal{H}}\left|\langle \tilde{\phi}_P(z), h \rangle_{\mathcal{H}} - \dot{\nu}_P^\ast(h)(z)\right|\le 2\epsilon.
\end{align*}
As $\epsilon>0$ was arbitrary, the left-hand side above is equal to zero. This proves the first claim of the lemma. The second claim follows directly from the fact that $h\mapsto \langle \tilde{\phi}_P(z), h \rangle$ is a bounded linear functional and by the definition of the Riesz representation of such a functional.
\end{proof}

\begin{lemma}\label{lem:eifMustBephiTilde}
Let $\nu : \mathcal{P}\rightarrow\mathcal{H}$ be pathwise differentiable at $P$ and suppose that $\dot{\mathcal{H}}_P$ is an RKHS. If $\nu$ has EIF $\phi_P$ at $P$, then $\phi_P=\tilde{\phi}_P$ $P$-a.s., where $\tilde{\phi}_P$ is as defined in \eqref{eq:eifDef}.
\end{lemma}
\begin{proof}
Since $\phi_P$ is the EIF of $\nu$ at $P$, there exists a $P$-probability-one set $\mathcal{Z}'$ such that, for all $z\in\mathcal{Z}'$, $\dot{\nu}_P^*(\cdot)(z) : \mathcal{H}\rightarrow\mathbb{R}$ is a bounded linear functional with Riesz representation $\phi_P$. In other words, for all $z\in\mathcal{Z}'$,
\begin{align*}
    \sup_{h\in\mathcal{H}}|\langle \phi_P(z), h \rangle_{\mathcal{H}} - \dot{\nu}_P^*(h)(z)|=0.
\end{align*}
Fix $z\in\mathcal{Z}'$. Since $\dot{\mathcal{H}}_P\subseteq\mathcal{H}$, the above shows that $\langle \phi_P(z), K_t \rangle_{\mathcal{H}} - \dot{\nu}_P^*(K_t)(z)=0$ for all $t\in\mathcal{T}$. Since $\phi_P(z)\in \dot{\mathcal{H}}_P$, $\phi_P(z)(t)=\langle \phi_P(z), K_t \rangle_{\mathcal{H}}$ for all $t\in\mathcal{T}$. Combining this with the fact that $\tilde{\phi}_P(t):=\dot{\nu}_P^*(K_t)(z)$, this shows that $\phi_P(z)=\tilde{\phi}_P(z)$. As $z$ is an arbitrary element of the $P$-probability one set $\mathcal{Z}'$, this shows that $\phi_P=\tilde{\phi}_P$ $P$-almost surely.
\end{proof}

\begin{proof}[Proof of Theorem~\ref{thm:eifProperties}]
The first statement, \ref{it:eifForm}, was established in Lemma~\ref{lem:eifMustBephiTilde}. The second statement, \ref{it:eifSuffCond}, was established in Lemma~\ref{lem:RieszRep}.
\end{proof}

\begin{lemma}\label{lem:localParamToMean}
Let $\mathcal{P}$ be a statistical model of distributions that are equivalent in that, for all $P_1,P_2\in\mathcal{P}$, $P_1\ll P_2$ and $P_2\ll P_1$. Let $\{P_\epsilon : \epsilon\in [0,\delta)\} \in \mathscr{P}(P,\mathcal{P},s)$ be a quadratic mean differentiable submodel. Let $\nu : \mathcal{P}\rightarrow\mathcal{H}$ be pathwise differentiable at $P\in\mathcal{P}$ with a $P$-almost surely bounded EIF $\phi_P$, in the sense that $\|\phi_P(Z)\|_{\mathcal{H}}$ is a bounded random variable when $Z\sim P$. Under these conditions,
\begin{align}
(P_\epsilon - P)\phi_P - \epsilon \dot{\nu}_P(s) = o(\epsilon). \label{eq:EIFexpansion}
\end{align}
\end{lemma}
The above lemma requires that $\|\phi_P(Z)\|_{\mathcal{H}}$ be a bounded random variable in order to show that \eqref{eq:EIFexpansion} holds. Lemma~\ref{lem:LinearQMDRemainder} will provide an alternative condition under which \eqref{eq:EIFexpansion} holds. In particular, rather than impose a boundedness condition on the EIF, that lemma will require that the submodel be approximately linear, in the sense that $\frac{dP_\epsilon}{dP}\approx 1 + \epsilon s$ in an appropriate sense.
\begin{proof}[Proof of Lemma~\ref{lem:localParamToMean}]
Fix a quadratic mean differentiable submodel $\{P_\epsilon : \epsilon\in [0,\delta)\} \in \mathscr{P}(P,\mathcal{P},s)$. 
For $\epsilon\in [0,\delta)$, let $h_\epsilon:=(P_\epsilon - P)\phi_P - \epsilon \dot{\nu}_P(s)$. We will show that $\|h_\epsilon\|_{\mathcal{H}}=o(\epsilon)$. Let $g_\epsilon:=h_\epsilon/\|h_\epsilon\|_{\mathcal{H}}$, where we use the convention that $g_\epsilon=0$ when $h_\epsilon=0$. Observe that
\begin{align}
    \|h_\epsilon\|_{\mathcal{H}}&= \langle (P_\epsilon - P)\phi_P - \epsilon \dot{\nu}_P(s), h_\epsilon\rangle_{\mathcal{H}}= \langle (P_\epsilon - P)\phi_P, g_\epsilon\rangle_{\mathcal{H}} - \epsilon \langle s, \dot{\nu}_P^*(g_\epsilon)\rangle_{L^2(P)}. \label{eq:hEpsNorm}
\end{align}
We now study the inner product $\langle (P_\epsilon - P)\phi_P, h_\epsilon\rangle_{\mathcal{H}}$ that appears above. Since $\phi_P(Z)$ is bounded under sampling from $P$ and $P_\epsilon\ll P$, $\phi_P : \mathcal{Z}\rightarrow\mathcal{H}$ is Bochner integrable both under sampling from $P$ and $P_\epsilon$. Consequently, $ \langle (P_\epsilon - P) \phi_P, g_\epsilon \rangle_{\mathcal{H}} = \int \langle \phi_P(z), g_\epsilon \rangle_{\mathcal{H}}d(P_\epsilon-P)(z)$. Adding and subtracting terms from this identity and letting $q_\epsilon:=p_\epsilon^{1/2}$ and $q:=p^{1/2}$ yields that
\begin{align*}
    \langle &(P_\epsilon - P)\phi_P, g_\epsilon\rangle_{\mathcal{H}} \\
    &= \int \langle \phi_P(z), g_\epsilon \rangle_{\mathcal{H}}[q_\epsilon(z)+q(z)][q_\epsilon(z)-q(z)]d\lambda(z)  \\
    &= \epsilon \int \langle \phi_P(z), g_\epsilon \rangle_{\mathcal{H}}s(z) q^2(z)d\lambda(z) + \frac{1}{2}\epsilon \int \langle \phi_P(z), g_\epsilon \rangle_{\mathcal{H}}[q_\epsilon(z)-q(z)]s(z) q(z)d\lambda(z) \\
    &\quad+ \int \langle \phi_P(z), g_\epsilon \rangle_{\mathcal{H}}[q_\epsilon(z)+q(z)]\left[q_\epsilon(z)-q(z) - \frac{1}{2}\epsilon s(z) q(z)\right]d\lambda(z).
\end{align*}
By the definition of the EIF, $\langle \phi_P(z), g_\epsilon \rangle_{\mathcal{H}}=\dot{\nu}_P^*(g_\epsilon)(z)$ $P$-almost surely. Hence, the first term on the right-hand side above is equal to $\epsilon \langle s, \dot{\nu}_P^*(g_\epsilon)\rangle_{L^2(P)}$, and so \eqref{eq:hEpsNorm} shows that
\begin{align*}
\|h_\epsilon\|_{\mathcal{H}}&= \frac{1}{2}\epsilon \int \langle \phi_P(z), g_\epsilon \rangle_{\mathcal{H}}[q_\epsilon(z)-q(z)]s(z) q(z)d\lambda(z) \\
    &\quad+ \int \langle \phi_P(z), g_\epsilon \rangle_{\mathcal{H}}[q_\epsilon(z)+q(z)]\left[q_\epsilon(z)-q(z) - \frac{1}{2}\epsilon s(z) q(z)\right]d\lambda(z).
\end{align*}
By Jensen's inequality and Cauchy-Schwarz, this yields that
\begin{align}
    \|h_\epsilon\|_{\mathcal{H}}&\le \frac{1}{2}\epsilon \|g_\epsilon\|_{\mathcal{H}} \int \| \phi_P(z)\|_{\mathcal{H}}|q_\epsilon(z)-q(z)|s(z) q(z)d\lambda(z) \nonumber \\
    &\quad+ \|g_\epsilon\|_{\mathcal{H}}\int \|\phi_P(z)\|_{\mathcal{H}}[q_\epsilon(z)+q(z)]\left|q_\epsilon(z)-q(z) - \frac{1}{2}\epsilon s(z) q(z)\right|d\lambda(z). \label{eq:hEpsTwoInts}
\end{align}
Since $\|g_\epsilon\|_{\mathcal{H}}$ is either equal to 1, if $h_\epsilon\not=0$, or is equal to zero otherwise, we can establish that $\|h_\epsilon\|_{\mathcal{H}}=o(\epsilon)$ by showing that the first integral above is $o(1)$ and the second is $o(\epsilon)$. To show that the first integral is $o(1)$, we use (i) Cauchy-Schwarz, (ii) the fact that $\phi_P$ is essentially bounded, (iii) $s\in L^2(P)$, and (iv) the fact that the quadratic mean differentiability of $\{P_\epsilon: \epsilon\in [0,\delta)\}$ implies that $\|q_\epsilon-q\|_{L^2(\lambda)}=O(\epsilon)$. Combining these observations yields the display
\begin{align*}
\int &\| \phi_P(z)\|_{\mathcal{H}}|q_\epsilon(z)-q(z)|s(z) q(z)d\lambda(z) \\
&\le \big\|\|\phi_P(\cdot)\|_{\mathcal{H}}s(\cdot)q(\cdot)\big\|_{L^2(\lambda)}\|q_\epsilon-q\|_{L^2(\lambda)}= O(\epsilon)=o(1).
\end{align*}
For the second integral in \eqref{eq:hEpsTwoInts}, (i), (ii), and (iii), together with the inequality $(a+b)^2\le 2(a^2+b^2)$ and the quadratic mean differentiability of $\{P_\epsilon: \epsilon\in [0,\delta)\}$, yield that
\begin{align*}
&\int \|\phi_P(z)\|_{\mathcal{H}}[q_\epsilon(z)+q(z)]\left|q_\epsilon(z)-q(z) - \frac{1}{2}\epsilon s(z) q(z)\right|d\lambda(z) \\
&\le \big\|\|\phi_P(\cdot)\|_{\mathcal{H}}(q_\epsilon+q)\big\|_{L^2(\lambda)}\left\|q_\epsilon-q - \frac{1}{2}\epsilon s q\right\|_{L^2(\lambda)} \\
&\le 2^{1/2}\left[\int\|\phi_P(z)\|_{\mathcal{H}}^2 d(P_\epsilon+P)(z)\right]^{1/2}\left\|q_\epsilon+q\right\|_{L^2(\lambda)}\left\|q_\epsilon-q - \frac{1}{2}\epsilon s q\right\|_{L^2(\lambda)} \\
&= o(\epsilon).
\end{align*}
Plugging the preceding two displays into \eqref{eq:hEpsTwoInts} completes the proof.
\end{proof}

\begin{lemma}\label{lem:LinearQMDRemainder}
Fix a score $s$ in the tangent set of $\mathcal{P}$ at $P$ and let $\{P_\epsilon: \epsilon\in [0,\delta)\}\in \mathscr{P}(P,\mathcal{P},s)$ be such that 
\begin{align}
\|dP_\epsilon/dP - 1 - \epsilon s\|_{L^2(P)} = o(\epsilon). \label{eq:approxLinearSubmodel}
\end{align}
If $\nu$ is pathwise differentiable at $P$, then
\begin{align}
    \sup_{h\in\mathcal{H}_1}\left|P_\epsilon \dot{\nu}_P^*(h)+ \langle \nu(P) - \nu(P_\epsilon), h\rangle_{\mathcal{H}}\right| = \sup_{h\in\mathcal{H}_1}\left|P_\epsilon \dot{\nu}_P^*(h)+ \epsilon\langle \dot{\nu}_P(s), h\rangle_{\mathcal{H}}\right| + o(\epsilon) = o(\epsilon), \label{eq:biasOperatorQMD}
\end{align}
where $\mathcal{H}_1$ denotes the unit ball in $\mathcal{H}$. Moreover, if $\nu$ has $P$-Bochner square integrable EIF $\phi_P$ at $P$, then
\begin{align*}
(P_\epsilon-P) \phi_P + \epsilon \dot{\nu}_P(s)&= o(\epsilon).
\end{align*}
\end{lemma}
When $s$ is bounded and $\mathcal{P}$ is nonparametric, there is a quadratic mean differentiable submodel $\{P_\epsilon: \epsilon\in [0,\delta)\}\in \mathscr{P}(P,\mathcal{P},s)$ that is such that $\frac{dP_\epsilon}{dP} = 1+\epsilon s$. The condition in \eqref{eq:approxLinearSubmodel} holds trivially for this submodel, since $\|dP_\epsilon/dP - 1 - \epsilon s\|_{L^2(P)}=0$. This condition will generally also hold for many other quadratic mean differentiable submodels.
\begin{proof}[Proof of Lemma~\ref{lem:LinearQMDRemainder}]
Suppose that $s$ and $\{P_\epsilon : \epsilon\in [0,\delta)\}$ are as in the statement of the lemma and that $\nu$ is pathwise differentiable at $P$. For any $h\in\mathcal{H}$, the fact that $P\dot{\nu}_P^*(g)=0$ implies that
\begin{align*}
P_\epsilon \dot{\nu}_P^*(h)&= \epsilon \langle s, \dot{\nu}_P^*(h)\rangle_{L^2(P)} + P\left[\left(\frac{dP_\epsilon}{dP} - 1 - \epsilon s\right) \dot{\nu}_P^*(h)\right].
\end{align*}
Combining this with the fact that $\langle s, \dot{\nu}_P^*(h)\rangle_{L^2(P)}  = \langle \dot{\nu}_P(s), h\rangle_{\mathcal{H}}$ shows that
\begin{align*}
P_\epsilon \dot{\nu}_P^*(h)&+ \langle \nu(P) - \nu(P_\epsilon), h\rangle_{\mathcal{H}} \\
&= \langle \nu(P) - \nu(P_\epsilon) - \epsilon \dot{\nu}_P(s), h\rangle_{\mathcal{H}} + P\left[\left(\frac{dP_\epsilon}{dP} - 1 - \epsilon s\right) \dot{\nu}_P^*(h)\right].
\end{align*}
Taking an absolute value and then a supremum over $h\in\mathcal{H}_1$ and subsequently applying the triangle inequality and Cauchy-Schwarz yields that
\begin{align*}
\sup_{h\in\mathcal{H}_1}&\left|P_\epsilon \dot{\nu}_P^*(h)+ \langle \nu(P) - \nu(P_\epsilon), h\rangle_{\mathcal{H}}\right| \\
&\le \sup_{h\in\mathcal{H}_1}\left|\langle \nu(P) - \nu(P_\epsilon) - \epsilon \dot{\nu}_P(s), h\rangle_{\mathcal{H}}\right| + \sup_{h\in\mathcal{H}_1}\left|P\left[\left(\frac{dP_\epsilon}{dP} - 1 - \epsilon s\right) \dot{\nu}_P^*(h)\right]\right| \\
&\le \left\|\nu(P) - \nu(P_\epsilon) - \epsilon \dot{\nu}_P(s)\right\|_{\mathcal{H}} + \left\|\frac{dP_\epsilon}{dP} - 1 - \epsilon s\right\|_{L^2(P)}\|\dot{\nu}_P^*\|_{\mathrm{op}}
\end{align*}
The first term on the right-hand side is $o(\epsilon)$ by the pathwise differentiability of $\nu$, and the second is $o(\epsilon)$ by \eqref{eq:approxLinearSubmodel} and the fact that $\dot{\nu}_P^*$ is a bounded operator. Eq.~\ref{eq:biasOperatorQMD} follows by combining the above with the fact that, by the pathwise differentiability of $\nu$,
\begin{align*}
\sup_{h\in\mathcal{H}_1}\left|P_\epsilon \dot{\nu}_P^*(h)+ \langle \nu(P) - \nu(P_\epsilon), h\rangle_{\mathcal{H}}\right| = \sup_{h\in\mathcal{H}_1}\left|P_\epsilon \dot{\nu}_P^*(h)+ \epsilon\langle \dot{\nu}_P(s), h\rangle_{\mathcal{H}}\right| + o(\epsilon).
\end{align*}

Now suppose that $\nu$ has a $P$-Bochner square integrable EIF $\phi_P$ at $P$. We have that
\begin{align*}
&\|(P_\epsilon-P) \phi_P + \epsilon\dot{\nu}_P(s)\|_{\mathcal{H}} \\
&= \|(P_\epsilon-P) \phi_P + \nu(P_\epsilon) - \nu(P)\|_{\mathcal{H}} + o(\epsilon) \\
&= \sup_{h\in\mathcal{H}_1}\left[\langle P_\epsilon \phi_P + \nu(P_\epsilon) - \nu(P),h\rangle_{\mathcal{H}}\right] + o(\epsilon) \\
&\le \sup_{h\in\mathcal{H}_1}\left|\langle P_\epsilon \phi_P,h\rangle_{\mathcal{H}} - P_\epsilon\dot{\nu}_P^*(h)\right| + \sup_{h\in\mathcal{H}_1}\left|\langle \nu(P_\epsilon) - \nu(P),h\rangle_{\mathcal{H}} + P_\epsilon\dot{\nu}_P^*(h)\right| + o(\epsilon).
\end{align*}
The second term on the right is $o(\epsilon)$ by \eqref{eq:biasOperatorQMD}, and so it remains to show that the leading term is also $o(\epsilon)$. In fact, we will have shown that the leading term is zero if we can show that $\phi_P$ is $P_\epsilon$-Bochner integrable, since that would imply that $\langle P_\epsilon \phi_P,h\rangle_{\mathcal{H}}=P_\epsilon \langle \phi_P,h\rangle_{\mathcal{H}}=P_\epsilon\dot{\nu}_P^*(h)$. To see that $\phi_P$ is indeed $P_\epsilon$-Bochner integrable, note that
\begin{align*}
\int \|\phi_P(z)\|_{\mathcal{H}} P_\epsilon(dz)&= \int \|\phi_P(z)\|_{\mathcal{H}} \frac{dP_\epsilon}{dP}(z) P(dz) \\
&\le \left|\int \|\phi_P(z)\|_{\mathcal{H}} \left[\frac{dP_\epsilon}{dP}(z)-1\right] P(dz)\right| + \int \|\phi_P(z)\|_{\mathcal{H}} P(dz) \\
&\le  \|\phi_P\|_{L^2(P;\mathcal{H})}\left\|\frac{dP_\epsilon}{dP}-1\right\|_{L^2(P)} + \|\phi_P\|_{L^2(P;\mathcal{H})} \\
&\le \|\phi_P\|_{L^2(P;\mathcal{H})}\left(\left\|\frac{dP_\epsilon}{dP}-1-\epsilon s\right\|_{L^2(P)} + \epsilon \|s\|_{L^2(P)}\right) + \|\phi_P\|_{L^2(P;\mathcal{H})},
\end{align*}
where the first inequality holds by the triangle inequality, the second by Cauchy Schwarz and Jensen's inequality, and the third by the triangle inequality. The right-hand side is finite since $\phi_P$ is $P$-Bochner square integrable and \eqref{eq:approxLinearSubmodel} holds.
\end{proof}

\subsubsection*{Proofs for Section~\ref{sec:twostep}}

\begin{proof}[Proof of Lemma~\ref{lem:approximateEIF}]
We begin by showing that $\sum_{k=1}^\infty \beta_k^2 P\dot{\nu}_P(h_k)^2<\infty$ and that $(\beta_k \dot{\nu}_P^*(h_k)(z))_{k=1}^\infty$ is $P$-a.s. a square summable sequence. To see why this is the case, note that, by the monotone convergence theorem,
\begin{align}
E_P\left[\sum_{k=1}^\infty \beta_k^2 \dot{\nu}_P^*(h_k)(Z)^2\right]&= \lim_{K\rightarrow\infty} E_P\left[\sum_{k=1}^K \beta_k^2 \dot{\nu}_P^*(h_k)(Z)^2\right] = \lim_{K\rightarrow\infty} \sum_{k=1}^K \beta_k^2 \|\dot{\nu}_P^*(h_k)\|_{L^2(P)}^2 \nonumber \\
&\le \lim_{K\rightarrow\infty} \sum_{k=1}^K \beta_k^2 \|\dot{\nu}_P^*\|_{\mathrm{op}}^2 \|h_k\|_{\mathcal{H}}^2 \le \|\dot{\nu}_P^*\|_{\mathrm{op}}^2 \lim_{K\rightarrow\infty} \sum_{k=1}^K \beta_k^2 < \infty, \label{eq:mctSqSum}
\end{align}
where the final inequality holds because $\dot{\nu}_P^* : \mathcal{H}\rightarrow L^2(P)$ is a bounded operator and $(\beta_k)_{k=1}^\infty$ is square summable. The above implies that $\sum_{k=1}^\infty \beta_k^2 \dot{\nu}_P^*(h_k)(z)^2$ 
is finite on a $P$-probability one set $\mathcal{Z}^\beta$. Hence, $\phi_P^\beta(z)\in\mathcal{H}$ $P$-a.s. To see that $\phi_P^\beta\in L^2(P;\mathcal{H})$, note that, by the continuity and linearity of inner products and the orthonormality of $(h_k)_{k=1}^\infty$,
\begin{align*}
&\int \langle \phi_P^\beta(z), \phi_P^\beta(z) \rangle_{\mathcal{H}} P(dz) \\
&\quad= \int_{\mathcal{Z}^\beta} \left\langle \lim_{K\rightarrow\infty}\sum_{k=1}^K \beta_k \dot{\nu}_P^*(h_k)(z) h_k, \lim_{K'\rightarrow\infty}\sum_{k'=1}^{K'} \beta_{k'} \dot{\nu}_P^*(h_{k'})(z) h_{k'} \right\rangle_{\mathcal{H}} P(dz) \\
&\quad= \int_{\mathcal{Z}^\beta} \lim_{K\rightarrow\infty}\lim_{K'\rightarrow\infty}\left\langle \sum_{k=1}^K \beta_k \dot{\nu}_P^*(h_k)(z) h_k, \sum_{k'=1}^{K'} \beta_{k'} \dot{\nu}_P^*(h_{k'})(z) h_{k'} \right\rangle_{\mathcal{H}} P(dz) \\
&\quad= \int_{\mathcal{Z}^\beta} \lim_{K\rightarrow\infty}\lim_{K'\rightarrow\infty}\sum_{k=1}^K \sum_{k'=1}^{K'} \beta_k \dot{\nu}_P^*(h_k)(z)  \beta_{k'} \dot{\nu}_P^*(h_{k'})(z) \left\langle h_k,  h_{k'} \right\rangle_{\mathcal{H}} P(dz) \\
&\quad= \int_{\mathcal{Z}^\beta} \lim_{K\rightarrow\infty}\sum_{k=1}^K \beta_k^2 \dot{\nu}_P^*(h_k)(z)^2 P(dz) = E_P\left[\sum_{k=1}^\infty \beta_k^2 \dot{\nu}_P^*(h_k)(Z)^2\right],
\end{align*}
which is finite by \eqref{eq:mctSqSum}.
It remains to show that, for all $z\in\mathcal{Z}^\beta$ and 
$h\in\mathcal{H}$, $r_P^\beta(h)(z)=\langle \phi_P^\beta(z),h\rangle_{\mathcal{H}}$. This can be seen by noting that, for any $z\in\mathcal{Z}^\beta$ and 
$h\in\mathcal{H}$,
\begin{align*}
    &\langle \phi_P^\beta(z),h \rangle_{\mathcal{H}} \\
    &\quad= \left\langle \phi_P^\beta(z),\sum_{k'=1}^\infty \langle h,h_{k'}\rangle_{\mathcal{H}} h_{k'} \right\rangle_{\mathcal{H}} =  \left\langle \sum_{k=1}^\infty \beta_k \dot{\nu}_P^*(h_k)(z) h_k,\sum_{k'=1}^\infty \langle h,h_{k'}\rangle_{\mathcal{H}} h_{k'} \right\rangle_{\mathcal{H}} \\
    &\quad=  \sum_{k=1}^\infty \sum_{k'=1}^\infty \langle h,h_{k'}\rangle_{\mathcal{H}}\beta_k \dot{\nu}_P^*(h_k)(z) \left\langle h_k, h_{k'} \right\rangle_{\mathcal{H}} =  \sum_{k=1}^\infty \langle h,h_k\rangle_{\mathcal{H}}\beta_k \dot{\nu}_P^*(h_k)(z) = r_P^\beta(h)(z).
\end{align*}
As $h\in\mathcal{H}$ and $z\in\mathcal{Z}^\beta$ were arbitrary, $r_P^\beta(\cdot)(z)$ is a bounded linear functional with Riesz representation $\phi_P^\beta(z)$ for all $z\in\mathcal{Z}^\beta$.
\end{proof}

\begin{lemma}\label{lem:LinearQMDApproxRemainder}
Fix a score $s$ in the tangent set of $\mathcal{P}$ at $P$ and let $\{P_\epsilon: \epsilon\in [0,\delta)\}\in \mathscr{P}(P,\mathcal{P},s)$ satisfy \eqref{eq:approxLinearSubmodel}. If $\nu$ is pathwise differentiable at $P$, then
\begin{align*}
&\left\|\nu(P)-\nu(P_\epsilon) + P_\epsilon \phi_P^\beta - \sum_{k=1}^\infty (1-\beta_k) \langle\nu(P) - \nu(P_\epsilon),h_k\rangle_{\mathcal{H}} h_k\right\|_{\mathcal{H}} \\
&\quad= \left(1 + \|\phi_P^\beta\|_{L^2(P;\mathcal{H})}\right)\cdot o(\epsilon),
\end{align*}
where the $o(\epsilon)$ terms do not depend on the choice of $\beta$.
\end{lemma}
\begin{proof}[Proof of Lemma~\ref{lem:LinearQMDApproxRemainder}]
Suppose that $s$ and $\{P_\epsilon : \epsilon\in [0,\delta)\}$ are as in the statement of the lemma and that $\nu$ is pathwise differentiable at $P$. Let $\mathcal{H}_1$ denote the unit ball of $\mathcal{H}$. 
Since $\nu(P_\epsilon)-\nu(P_0)= \sum_{k=1}^\infty \langle\nu(P_\epsilon)-\nu(P),h_k\rangle_{\mathcal{H}} h_k$, it holds that
\begin{align*}
\nu(P)&-\nu(P_\epsilon) + P_\epsilon \phi_P^\beta - \sum_{k=1}^\infty (1-\beta_k) \langle\nu(P) - \nu(P_\epsilon),h_k\rangle_{\mathcal{H}} h_k \\
&= P_\epsilon \phi_P^\beta - \sum_{k=1}^\infty \beta_k \langle\nu(P_\epsilon) - \nu(P),h_k\rangle_{\mathcal{H}} h_k.
\end{align*}
The remainder of our analysis bounds the terms on the right of the following decomposition, which holds by the triangle inequality:
\begin{align}
    &\left\|P_\epsilon \phi_P^\beta - \sum_{k=1}^\infty \beta_k \langle\nu(P_\epsilon) - \nu(P),h_k\rangle_{\mathcal{H}} h_k\right\|_{\mathcal{H}} \label{eq:approxLinearExpansionBd} \\
    &\le \left\|P_\epsilon\phi_P^\beta - \epsilon \int s(z) \phi_P^\beta(z)P(dz)\right\|_{\mathcal{H}} \nonumber \\
    &\quad+ \epsilon \left\|\int s(z) \phi_P^\beta(z)P(dz) - \sum_{k=1}^\infty \beta_k \langle \dot{\nu}_P(s),h_k\rangle_{\mathcal{H}} h_k\right\|_{\mathcal{H}} \nonumber \\
    &\quad+ \left\|\sum_{k=1}^\infty \beta_k \langle\epsilon\dot{\nu}_P(s),h_k\rangle_{\mathcal{H}} h_k - \sum_{k=1}^\infty \beta_k \langle\nu(P_\epsilon) - \nu(P),h_k\rangle_{\mathcal{H}} h_k\right\|_{\mathcal{H}}. \nonumber
\end{align}
To bound the leading term, we use that $\phi_P^\beta\in L^2(P;\mathcal{H})$ by Lemma~\ref{lem:approximateEIF} and that $P\phi_P^\beta=0$, which give that
\begin{align*}
\left\|P_\epsilon\phi_P^\beta - \epsilon \int s(z) \phi_P^\beta(z)P(dz)\right\|_{\mathcal{H}}&= \sup_{h\in\mathcal{H}_1} \left\langle \int \left(\frac{dP_\epsilon}{dP}(z) - 1 - \epsilon s(z)\right)\phi_P^\beta(z) P(dz),h\right\rangle_{\mathcal{H}} \\
&= \sup_{h\in\mathcal{H}_1} \int \left(\frac{dP_\epsilon}{dP}(z) - 1 - \epsilon s(z)\right)\left\langle \phi_P^\beta(z),h\right\rangle_{\mathcal{H}} P(dz).
\end{align*}
By twice applying Cauchy-Schwarz, once in $\mathcal{H}$ and once in $L^2(P)$, and recalling \eqref{eq:approxLinearSubmodel} and that $\mathcal{H}_1$ is the unit ball of $\mathcal{H}$, the right-hand side can be seen to be upper bounded by $\|\phi_P^\beta\|_{L^2(P)}\cdot o(\epsilon)$, where the $o(\epsilon)$ term denotes the behavior of the term on the left-hand side of \eqref{eq:approxLinearSubmodel}, which does not depend on $\beta$.

We will show that the second term on the right of \eqref{eq:approxLinearExpansionBd} is zero. To do this, we recall that (i) $\phi_P^\beta\in L^2(P;\mathcal{H})$ by Lemma~\ref{lem:approximateEIF}, (ii) $\phi_P^\beta:= \sum_{k=1}^\infty \beta_k \dot{\nu}_P^*(h_k)(z) h_k$, (iii) inner products are continuous, (iv) $(h_k)_{k=1}^\infty$ is an orthonormal basis of $\mathcal{H}_k$, and (v) $\dot{\nu}_P^*$ is the adjoint of $\dot{\nu}_P$. Applying these facts in sequence justifies the following for any basis element $h_k$:
\begin{align*}
&\left\langle \int s(z) \phi_P^\beta(z)P(dz),h_k\right\rangle_{\mathcal{H}} \\
&\quad= \int s(z) \langle \phi_P^\beta(z),h_k\rangle_{\mathcal{H}}P(dz) 
= \int s(z) \left\langle \sum_{k'=1}^\infty \beta_{k'} \dot{\nu}_P^*(h_{k'})(z) h_{k'},h_k\right\rangle_{\mathcal{H}}P(dz) \\
&\quad= \int s(z) \left[\sum_{k'=1}^\infty \beta_{k'} \dot{\nu}_P^*(h_{k'})(z) \left\langle h_{k'},h_k\right\rangle_{\mathcal{H}}\right] P(dz) \\
&\quad= \beta_{k} \langle h_k,h_k\rangle_{\mathcal{H}}\int s(z) \dot{\nu}_P^*(h_{k})(z)P(dz) = \beta_{k} \langle h_k,h_k\rangle_{\mathcal{H}}\langle \dot{\nu}_P(s),h_k\rangle_{\mathcal{H}} \\
&\quad= \left\langle\sum_{k'=1}^\infty \beta_{k'} \langle \dot{\nu}_P(s),h_{k'}\rangle_{\mathcal{H}} h_{k'},h_k\right\rangle_{\mathcal{H}}
\end{align*}
As $h_k$ was an arbitrary element of an orthonormal basis of $\mathcal{H}$, this shows that the second term on the right of \eqref{eq:approxLinearExpansionBd} is zero.

We will show that the third term on the right of \eqref{eq:approxLinearExpansionBd} is $o(\epsilon)$. We begin by noting that
\begin{align*}
&\left\|\sum_{k=1}^\infty \beta_k \langle\epsilon\dot{\nu}_P(s),h_k\rangle_{\mathcal{H}} h_k - \sum_{k=1}^\infty \beta_k \langle\nu(P_\epsilon) - \nu(P),h_k\rangle_{\mathcal{H}} h_k\right\|_{\mathcal{H}} \\
&= \left\|\sum_{k=1}^\infty \beta_k \langle\nu(P_\epsilon) - \nu(P) - \epsilon \dot{\nu}_P(s),h_k\rangle_{\mathcal{H}} h_k\right\|_{\mathcal{H}} \\
&= \sup_{h\in\mathcal{H}_1}\left\langle\sum_{k=1}^\infty \beta_k \langle\nu(P_\epsilon) - \nu(P) - \epsilon \dot{\nu}_P(s),h_k\rangle_{\mathcal{H}} h_k,h\right\rangle_{\mathcal{H}} \\
&= \sup_{h\in\mathcal{H}_1}\sum_{k=1}^\infty \beta_k \langle\nu(P_\epsilon) - \nu(P) - \epsilon \dot{\nu}_P(s),h_k\rangle_{\mathcal{H}} \left\langle h_k,h\right\rangle_{\mathcal{H}}.
\end{align*}
Combining the above with the fact that all $\beta_k$ belong to $[0,1]$, the Cauchy-Schwarz inequality in $\ell^2$, and Parseval's identity, the above shows that
\begin{align*}
&\left\|\sum_{k=1}^\infty \beta_k \langle\epsilon\dot{\nu}_P(s),h_k\rangle_{\mathcal{H}} h_k - \sum_{k=1}^\infty \beta_k \langle\nu(P_\epsilon) - \nu(P),h_k\rangle_{\mathcal{H}} h_k\right\|_{\mathcal{H}} \\
&\le \sup_{h\in\mathcal{H}_1}\sum_{k=1}^\infty \left|\langle\nu(P_\epsilon) - \nu(P) - \epsilon \dot{\nu}_P(s),h_k\rangle_{\mathcal{H}} \left\langle h_k,h\right\rangle_{\mathcal{H}}\right| \\
&\le \left(\sum_{k=1}^\infty \langle\nu(P_\epsilon) - \nu(P) - \epsilon \dot{\nu}_P(s),h_k\rangle_{\mathcal{H}}^2\right)^{1/2}\left( \sup_{h\in\mathcal{H}_1}\sum_{k=1}^\infty \left\langle h_k,h\right\rangle_{\mathcal{H}}^2\right)^{1/2} \\
&= \left\|\nu(P_\epsilon) - \nu(P) - \epsilon \dot{\nu}_P(s)\right\|_{\mathcal{H}}.
\end{align*}
The right-hand side does not depend on $\beta$ and, by the pathwise differentiability of $\nu$, is $o(\epsilon)$. This completes the proof.
\end{proof}

\subsection{Proofs for Section~\ref{sec:exPD}}

We now prove the sufficient condition for pathwise differentiability that we presented in the main text. We refer the interested reader to Remark~2 in Appendix A.5 of \cite{bickel1993efficient} for an alternative characterization of pathwise differentiability that may also be useful in some contexts.

\begin{proof}[Proof of Lemma~\ref{lem:suffCondsPD}]
Suppose that \ref{it:closure} and \ref{it:localLip} hold. Let $\{P_\epsilon: \epsilon\in [0,\delta)\}\in \mathscr{P}(P,\mathcal{P},s)$ for some $s$ in the tangent set of $\mathcal{P}$ at $P$. Since $\mathcal{S}(P)$ is dense in $\dot{\mathcal{P}}_P$, there exists a $\mathcal{S}(P)$-valued sequence $(s_n)_{n=1}^\infty$ such that $s_n\rightarrow s$ in $L^2(P)$. For each $n$, let $\{P_\epsilon^{[n]} : \epsilon\}$ be the element of $\mathscr{P}(P,\mathcal{P},s_n)$ satisfying $\| \nu(P_\epsilon^{[n]}) - \nu(P) - \epsilon\, \eta_P(s) \|_{\mathcal{H}} = o(\epsilon)$ that is guaranteed to exist by \ref{it:closure}. Observe that, for any $\epsilon>0$ and any map $N : (0,\delta]\rightarrow\mathbb{N}$, the triangle inequality and the linearity of $\eta_P$ show that
\begin{align}
    &\left\|\epsilon^{-1}\left[\nu(P_{\epsilon})-\nu(P)\right]-\eta_P(s)\right\|_{\mathcal{H}} \nonumber \\
    &\le \left\|\epsilon^{-1}\left[\nu(P_{\epsilon}^{[N(\epsilon)]})-\nu(P)\right]-\eta_P(s_{N(\epsilon)})\right\|_{\mathcal{H}} \nonumber \\
    &\quad + \epsilon^{-1}\left\|\nu(P_{\epsilon}^{[N(\epsilon)]})-\nu(P_\epsilon)\right\|_{\mathcal{H}}+ \left\|\eta_P(s_{N(\epsilon)}-s)\right\|_{\mathcal{H}}. \label{eq:pdViaSn}
\end{align}
Our argument will be based on the above with a map $N  : (0,\delta]\rightarrow\mathbb{N}$ that we construct to satisfy all of the following properties as $\epsilon\rightarrow 0$:
\begin{enumerate}[label=\alph*)]
    \item\label{it:Ndiverges} $N(\epsilon)\rightarrow\infty$;
    \item\label{it:Nqmd} $\|\epsilon^{-1}[\sqrt{dP_{\epsilon}^{[N(\epsilon)]}}-\sqrt{dP}] - \frac{1}{2} s_{N(\epsilon)}\sqrt{dP}\|_{L^2(\lambda)}\rightarrow 0$;
    \item\label{it:Npd} $\|\epsilon^{-1}[\nu(P_{\epsilon}^{[N(\epsilon)]})-\nu(P)]-\eta_P(s_{N(\epsilon)})\|_{\mathcal{H}}\rightarrow 0$.
\end{enumerate}
We will first establish that such a map exists, and then show that, when combined with \eqref{eq:pdViaSn}, this establishes that $\nu$ is pathwise differentiable at $P$ with local parameter $\dot{\nu}_P=\eta_P$.

Now, for any fixed $n\in\mathbb{N}$, the quadratic mean differentiability of $\{P_\epsilon^{[n]} : \epsilon\}$ paired with the fact that $\| \nu(P_\epsilon^{[n]}) - \nu(P) - \epsilon\, \eta_P(s) \|_{\mathcal{H}} = o(\epsilon)$ implies that $\lim_{\epsilon\rightarrow 0}f_n(\epsilon)=0$, where
\begin{align*}
    f_n(\epsilon)&:= \left\|\epsilon^{-1}\left[\sqrt{dP_{\epsilon}^{[n]}}-\sqrt{dP}\right] - \frac{1}{2} s_{n}\sqrt{dP}\right\|_{L^2(\lambda)} \\
    &\quad+ \left\|\epsilon^{-1}\left[\nu(P_{\epsilon}^{[n]})-\nu(P)\right]-\eta_P(s_n)\right\|_{\mathcal{H}}.
\end{align*}
We use this fact to define $N : (0,\delta]\rightarrow\mathbb{N}$ via a recursive formulation. For a strictly decreasing positive sequence $(\epsilon_\ell)_{\ell=1}^\infty$ that we will define momentarily, we let $N(\epsilon):=\max\{\ell\in\mathbb{N} : \epsilon\le \epsilon_\ell\}$. This sequence will be constructed so that $\epsilon_\ell\downarrow 0$ as $\ell\uparrow\infty$, which ensures that the maximum used to define $N(\epsilon)$ is well-defined and that $N(\epsilon)\rightarrow \infty$ as $\epsilon\rightarrow 0$ or, in other words, condition \ref{it:Ndiverges} holds. The construction of the sequence will also ensure that $f_{N(\epsilon)}(\epsilon)\rightarrow 0$ as $\epsilon\rightarrow 0$, which will guarantee that conditions \ref{it:Nqmd} and \ref{it:Npd} hold as well. We now construct this sequence. We first let $\epsilon_1:=\delta$. Then, recursively from $\ell=1,2,\ldots$, we let $\epsilon_{\ell+1}:=\frac{1}{2}\sup\{\epsilon\in (0,\epsilon_\ell] : \sup_{\epsilon'\in (0,\epsilon]}f_{\ell+1}(\epsilon')\le 2^{-(\ell+1)}\}$; since $f_{\ell+1}(\epsilon)\rightarrow 0$ as $\epsilon\rightarrow 0$, $\epsilon_{\ell+1}$ is well-defined and positive. Also, by its definition, $\epsilon_{\ell+1}\le \epsilon_\ell/2$ and $f_{\ell+1}(\epsilon)\le 2^{-(\ell+1)}$ for all $\epsilon\le \epsilon_{\ell+1}$. As a consequence, $f_{N(\epsilon)}(\epsilon)\le 2^{-N(\epsilon)}$ for all $\epsilon\le\epsilon_2$. Since $N(\epsilon)\rightarrow\infty$ as $\epsilon\rightarrow 0$, this implies that $f_{N(\epsilon)}(\epsilon)\rightarrow 0$ as $\epsilon\rightarrow 0$, which implies that \ref{it:Nqmd} and \ref{it:Npd} hold.

Having now defined $N : (0,\delta]\rightarrow\mathbb{N}$, we return to \eqref{eq:pdViaSn}. Because $s$ was an arbitrary element of the tangent set, we will have established that $\nu$ is pathwise differentiable at $P$ with local parameter $\dot{\nu}_P=\eta_P$ if we can show that the right-hand side of that display converges to zero as $\epsilon\rightarrow 0$. Since the choice of $N$ ensured that \ref{it:Npd} holds, the first term on the right-hand side of that display goes to zero as $\epsilon\rightarrow 0$. 
Since $\eta_P : \dot{\mathcal{P}}_P\rightarrow \mathcal{H}$ is a bounded linear operator by \ref{it:closure} and $\lim_{\epsilon\rightarrow 0}s_{N(\epsilon)}= s$ by virtue of the fact that $\lim_{n\rightarrow\infty} s_n=s$ and and $N(\epsilon)\rightarrow \infty$ as $\epsilon\rightarrow 0$, the third term on the right-hand side of that display goes to zero as $\epsilon\rightarrow 0$. 
It remains to study the second term. For this term, we will first show that the Hellinger distance between $P_\epsilon^{[N(\epsilon)]}$ and $P_\epsilon$ is $o(\epsilon)$, and then we will leverage the local Lipschitz property of $\nu$ that holds by \ref{it:localLip}. Beginning by studying the Hellinger distance, we use the triangle inequality to show that
\begin{align*}
    H\left(P_\epsilon^{[N(\epsilon)]},P_\epsilon\right)&:= \left\|\sqrt{dP_{\epsilon}^{[N(\epsilon)]}}-\sqrt{dP_\epsilon}\right\|_{L^2(\lambda)} \\
    &\le \left\|\sqrt{dP_{\epsilon}}-\sqrt{dP} - \frac{1}{2}\epsilon s\sqrt{dP}\right\|_{L^2(\lambda)} \\
    &\quad + \left\|\sqrt{dP_{\epsilon}^{[N(\epsilon)]}}-\sqrt{dP} - \frac{1}{2}\epsilon s_{N(\epsilon)}\sqrt{dP}\right\|_{L^2(\lambda)} + \frac{1}{2}\epsilon \left\|s_{N(\epsilon)}-s\right\|_{L^2(P)}.
\end{align*}
The first term on the right is $o(\epsilon)$ by the quadratic mean differentiability of $\{P_\epsilon : \epsilon\in [0,\delta)\}$, the second is $o(\epsilon)$ by \ref{it:Nqmd}, and the third is $o(\epsilon)$ by the fact that $s_{N(\epsilon)}\rightarrow s$ in $L^2(P)$. Hence, $H(P_\epsilon^{[N(\epsilon)]},P_\epsilon)$ is $o(\epsilon)$. Letting $c$ be the constant from \ref{it:localLip}, this implies that, for all $\epsilon$ small enough, the second term in \eqref{eq:pdViaSn} bounds as follows:
\begin{equation*}
    \epsilon^{-1}\left\|\nu(P_{\epsilon}^{[N(\epsilon)]})-\nu(P_\epsilon)\right\|_{\mathcal{H}}\le \epsilon^{-1} c H(P_{\epsilon}^{[N(\epsilon)]},P_\epsilon) = o(1).
\end{equation*}
\end{proof}

We now establish a lemma concerning the preservation of quadratic mean differentiability, which we used several times in the derivations for the examples provided in Appendix~\ref{app:exDerivations}.

\begin{lemma} \label{lem:preserveDQM}
Let $\mathcal{P}$ be a statistical model of distributions of $Z=(X,Y)$ that are equivalent in that, for all $P_1, P_2 \in \mathcal{P}$, $P_1 \ll P_2$ and $P_2 \ll P_1$. Let $\{P_{\epsilon}: \epsilon\in [0,\delta)\}\in\mathscr{P}(P,\mathcal{P},s)$, where $s$ is in the tangent set of $\mathcal{P}$ at $P$. Let $q_\epsilon = \sqrt{\frac{dP_{\epsilon}}{dP}}$ and $q_{\epsilon,X}$, $q_{\epsilon,Y \mid X}$ be the square root of the marginal density and conditional density of $P_\epsilon$ relative to $P$, respectively. Let $s_X(x) = E_P[s(Z) \mid X=x]$ and $s_{Y \mid X}(y\mid x) = s(z) - s_X(x)$. Then, both of the following hold:
\[
    \left\|q_{\epsilon,X} - 1 - \frac{\epsilon}{2} s_X\right\|_{L^2(P)} = o(\epsilon), \quad
    \left\|q_{\epsilon,Y \mid X} - 1 - \frac{\epsilon}{2} s_{Y \mid X} \right\|_{L^2(P)} = o(\epsilon).
\]
\end{lemma}
\begin{proof}[Proof of Lemma~\ref{lem:preserveDQM}]
Applying Proposition A.5.5 in \cite{bickel1993efficient} and taking $(x,y) \mapsto x$ as the statistic, we have the first inequality, namely that $\|q_{\epsilon,X} - 1 - \frac{\epsilon}{2} s_X\|_{L^2(P)} = o(\epsilon)$.

We now establish the second equality. Let $s_X^{(\epsilon)}:= s_X 1\{|s_X| \leq \epsilon^{-1/2}\}$ and $f(\epsilon):= \|q_{\epsilon,Y \mid X} - 1 - \frac{\epsilon}{2} s_{Y \mid X}\|_{L^2(P)}^2$. We have that
\begin{align}
    &f(\epsilon)= \int \left(q_{\epsilon,Y\mid X}[1-q_{\epsilon,X}] + q_\epsilon-1 - \frac{\epsilon}{2}s_{Y\mid X}\right)^2 dP \nonumber \\
    &= \int \Bigg(q_{\epsilon,Y\mid X}\left[1-q_{\epsilon,X}+\frac{\epsilon}{2}s_X^{(\epsilon)}\right] + q_\epsilon-1 - \frac{\epsilon}{2}[s_{Y\mid X}+s_X^{(\epsilon)}] \nonumber \\
    &\hspace{5em}- \frac{\epsilon}{2} \left[q_{\epsilon,Y\mid X}-1-\frac{\epsilon}{2}s_{Y\mid X}\right]s_X^{(\epsilon)} - \frac{\epsilon^2}{4} s_{Y\mid X}s_X^{(\epsilon)}\Bigg)^2 dP \nonumber \\
    &\le 4\int \left(q_{\epsilon,Y\mid X}\left[1-q_{\epsilon,X}+\frac{\epsilon}{2}s_X^{(\epsilon)}\right]\right)^2 dP + 4 \int \left(q_\epsilon-1 - \frac{\epsilon}{2}[s_{Y\mid X}+s_X^{(\epsilon)}]\right)^2 dP \nonumber \\
    &\quad+ \epsilon^2\int \left[q_{\epsilon,Y\mid X}-1-\frac{\epsilon}{2}s_{Y\mid X}\right]^2 (s_X^{(\epsilon)})^2 dP + \frac{\epsilon^4}{4}\int s_{Y\mid X}^2 (s_X^{(\epsilon)})^2 dP,
    \label{eq:qmdCond}
\end{align}
where the final inequality uses that $(a+b+c+d)^2 \leq 4(a^2+b^2+c^2+d^2)$. We consider each of the four terms above separately, showing that the first two are $o(\epsilon^2)$, the third is no more than $\epsilon f(\epsilon)$ and the last is $O(\epsilon^3)$. Subtracting the third term from both sides and dividing both sides by $1-\epsilon$ will then show that $f(\epsilon) = o(\epsilon^2)$. For the first term, note that
\begin{align*}
    4&\int \left(q_{\epsilon,Y\mid X}\left[1-q_{\epsilon,X}+\frac{\epsilon}{2}s_X^{(\epsilon)}\right]\right)^2 dP \\
    &= 4 \iint \left(1-q_{\epsilon,X}(x)+\frac{\epsilon}{2}s_X^{(\epsilon)}(x)\right)^2 P_{\epsilon,Y\mid X}(dy\mid x)P_X(dx) \\
    &= 4\int \left(1-q_{\epsilon,X}+\frac{\epsilon}{2}s_X^{(\epsilon)}\right)^2 dP_X \\
    &\le 8\int \left(1-q_{\epsilon,X}+\frac{\epsilon}{2}s_X\right)^2 dP_X + 8\epsilon^2\int s_X^21\{|s_X|>\epsilon^{-1/2}\} dP,
\end{align*}
where we used that $(a+b)^2\le 2(a^2+b^2)$. The first term on the right is $o(\epsilon^2)$ since quadratic mean differentiability is preserved under marginalization as proved earlier, and the second is $o(\epsilon^2)$ by the dominated convergence theorem. For the second term in \eqref{eq:qmdCond}, similar arguments show that
\begin{align*}
    4 \int \left(q_\epsilon-1 - \frac{\epsilon}{2}[s_{Y\mid X}+s_X^{(\epsilon)}]\right)^2 dP&= o(\epsilon^2).
\end{align*}
For the third integral in \eqref{eq:qmdCond},
\begin{align*}
    \epsilon^2\int \left[q_{\epsilon,Y\mid X}-1-\frac{\epsilon}{2}s_{Y\mid X}\right]^2 (s_X^{(\epsilon)})^2 dP&\le \epsilon \int \left[q_{\epsilon,Y\mid X}-1-\frac{\epsilon}{2}s_{Y\mid X}\right]^2 dP = \epsilon f(\epsilon)
\end{align*}
and, for the final term in \eqref{eq:qmdCond},
\begin{align*}
    \frac{\epsilon^4}{4}\int s_{Y\mid X}^2 (s_X^{(\epsilon)})^2 dP&\le \frac{\epsilon^3}{4}\int s_{Y\mid X}^2 dP.
\end{align*}
This completes the proof.
\end{proof}

\subsection{Proofs for Section~\ref{sec:perfGuarantees}}

\subsubsection{Proofs for Section~\ref{sec:oneStepGuarantee}}

\begin{proof}[Proof of Lemma~\ref{lem:driftTerm}]
Fix $\delta>0$. Suppose that $\|\phi_n^1-\phi_0\|_{L^2(P_0;\mathcal{H})}=o_p(1)$. We will show that
\begin{align}
\lim_{n\rightarrow\infty}P_0^n\{\|\mathcal{D}_n^1\|_{\mathcal{H}}>n^{-1/2}\delta\}= 0. \label{eq:Dn1Delta}
\end{align}
As $\delta$ was arbitrary, this will show that $\mathcal{D}_n^1=o_p(n^{-1/2})$. An analogous argument can be used to show that $\|\phi_n^2-\phi_0\|_{L^2(P_0;\mathcal{H})}=o_p(1)$ implies that $\mathcal{D}_n^2=o_p(n^{-1/2})$.

Let $1_{\mathcal{E}_n}$ denote the indicator that the event $\mathcal{E}_n$ that $\|\phi_n^1-\phi_0\|_{L^2(P_0;\mathcal{H})}^2\le \delta^2/2$ and let $\mathcal{E}_n^c$ denote the complement of $\mathcal{E}_n$.  
We will leverage the following decomposition when showing \eqref{eq:Dn1Delta}:
\begin{align}
P_0^n\{\|\mathcal{D}_n^1\|_{\mathcal{H}}>n^{-1/2}\delta\}&\le P_0^n\left(\{\|\mathcal{D}_n^1\|_{\mathcal{H}}>n^{-1/2}\delta\}\cap\mathcal{E}_n\right) + P_0^n(\mathcal{E}_n^c) \nonumber \\
&= E_0^n\left[1_{\mathcal{E}_n}P_0^n\left\{\|\mathcal{D}_n^1\|_{\mathcal{H}}>n^{-1/2}\delta\,\middle|\, Z_1,\ldots,Z_{n/2}\right\}\right] + o(1), \label{eq:Dn1Prob}
\end{align}
where $E_{0}^n$ denotes an expectation under sampling from the $n$-fold product measure $P_0^n$ and $o(1)$ denotes a deterministic term that goes to zero as $n\rightarrow\infty$. The equality above holds by the law of total probability, the fact that $\mathcal{E}_n$ is measurable with respect to the $\sigma$-field generated by $Z_1,\ldots,Z_{n/2}$, and the assumption that $\|\phi_n^1-\phi_0\|_{L^2(P_0;\mathcal{H})}=o_p(1)$ implies that $P_0^n(\mathcal{E}_n^c)=o(1)$. To show \eqref{eq:Dn1Delta}, the above shows that it suffices to show that the first term on the right-hand side is $o(1)$. To this end, note that Chebyshev's inequality for Hilbert-valued random variables \citep{grenander1963probabilities} and the bilinearity of inner products shows that
\begin{align}
&1_{\mathcal{E}_n}P_0^n\left\{\|\mathcal{D}_n^1\|_{\mathcal{H}}>n^{-1/2}\delta\,\middle|\, Z_1,\ldots,Z_{n/2}\right\} \label{eq:Dn1Event} \\
&\le 1_{\mathcal{E}_n}\frac{E_{P_0^n}[\|(P_n^1-P_0)(\phi_n^1-\phi_0)\|_{\mathcal{H}}^2\mid Z_1,\ldots,Z_{n/2}]}{n^{-1}\delta^2} \nonumber \\
&= 1_{\mathcal{E}_n}\frac{(n/2)^{-1}E_{P_0^n}[P_n^1\|(I-P_0)(\phi_n^1-\phi_0)\|_{\mathcal{H}}^2\mid Z_1,\ldots,Z_{n/2}]}{n^{-1}\delta^2} \nonumber \\
&\quad+ 1_{\mathcal{E}_n}\frac{4}{n^2}\sum_{(i,j)\in \{n/2+1,\ldots,n\}^2 : i\not=j}\tfrac{E_{P_0^n}[\langle (I-P_0)(\phi_n^1-\phi_0)(Z_i),(I-P_0)(\phi_n^1-\phi_0)(Z_k)\rangle_{\mathcal{H}}\mid Z_1,\ldots,Z_{n/2}]}{n^{-1}\delta^2}, \nonumber
\end{align}
where $(I-P_0)(\phi_n^1-\phi_0)$ denotes the map $z\mapsto (\phi_n^1-\phi_0)(z) - P_0(\phi_n^1-\phi_0)$ and each expectation in the summand on the right-hand side above is well-defined since $(\phi_n^1-\phi_0)\in L^2(P_0;\mathcal{H})$ implies that $(I-P_0)(\phi_n^1-\phi_0)\in L^2(P_0;\mathcal{H})$ as well. In fact, each expectation in the summand on the right-hand side is zero since, by the fact that $(I-P_0)(\phi_n^1-\phi_0)\in L^2(P_0;\mathcal{H})$ and Fubini's theorem,
\begin{align*}
&E_{P_0^n}[\langle (I-P_0)(\phi_n^1-\phi_0)(Z_i),(I-P_0)(\phi_n^1-\phi_0)(Z_k)\rangle_{\mathcal{H}}\mid Z_1,\ldots,Z_{n/2}] \\
&= \int \langle (I-P_0)(\phi_n^1-\phi_0)(z_1),(I-P_0)(\phi_n^1-\phi_0)(z_2)\rangle_{\mathcal{H}} P_0^2(dz_1,dz_2) \\
&= \iint \langle (I-P_0)(\phi_n^1-\phi_0)(z_1),(I-P_0)(\phi_n^1-\phi_0)(z_2)\rangle_{\mathcal{H}} P_0(dz_1)P_0(dz_2) \\
&= \int \left\langle \int (I-P_0)(\phi_n^1-\phi_0)(z_1)P_0(dz_1),(I-P_0)(\phi_n^1-\phi_0)(z_2)\right\rangle_{\mathcal{H}} P_0(dz_2) \\
&= \int \left\langle 0,(I-P_0)(\phi_n^1-\phi_0)(z_2)\right\rangle_{\mathcal{H}} P_0(dz_2) = 0.
\end{align*}
Returning to \eqref{eq:Dn1Event} and simplifying the first term on the right-hand side of that expression, this shows that
\begin{align*}
&1_{\mathcal{E}_n}P_0^n\left\{\|\mathcal{D}_n^1\|_{\mathcal{H}}>n^{-1/2}\delta\,\middle|\, Z_1,\ldots,Z_{n/2}\right\}\le 1_{\mathcal{E}_n}\frac{2\|(I-P_0)(\phi_n^1-\phi_0)\|_{L^2(P_0;\mathcal{H})}^2}{\delta^2}.
\end{align*}
Using that $P_0(\phi_n^1-\phi_0)$ is a minimizer over $h\in\mathcal{H}$ of $\|\phi_n^1-\phi_0 - h\|_{L^2(P;\mathcal{H})}^2$ and subsequently leveraging the definition of the event $\mathcal{E}_n$, this shows that
\begin{align*}
&1_{\mathcal{E}_n}P_0^n\left\{\|\mathcal{D}_n^1\|_{\mathcal{H}}>n^{-1/2}\delta\,\middle|\, Z_1,\ldots,Z_{n/2}\right\} \\
&\quad\le 1_{\mathcal{E}_n}\frac{2\|\phi_n^1-\phi_0\|_{L^2(P_0;\mathcal{H})}^2}{\delta^2}\le \min\left\{1,\frac{2\|\phi_n^1-\phi_0\|_{L^2(P_0;\mathcal{H})}^2}{\delta^2}\right\}.
\end{align*}
Note that the right-hand side above is no larger than $1$. 
Taking an expectation of both sides over $Z_1,\ldots,Z_{n/2}\iidsim P_0$ and recalling that $\|\phi_n^1-\phi_0\|_{L^2(P_0;\mathcal{H})}=o_p(1)$, the dominated convergence theorem shows that the first term on the right-hand side of \eqref{eq:Dn1Prob} is $o(1)$. This completes the proof.
\end{proof}

\begin{proof}[Proof of Theorem~\ref{thm:al}]
In this argument, we will let $\widetilde{\mathcal{H}}$ denote the Hilbert space of elements $(h,r)\in\mathcal{H}\times \mathbb{R}$ that is equipped with inner product $\langle (h_1,r_1),(h_2,r_2)\rangle_{\widetilde{\mathcal{H}}} = \langle h_1, h_2\rangle_{\mathcal{H}} + r_1r_2$. Fix $s\in\dot{\mathcal{P}}_P$. Since $\mathcal{R}_n^j = o_P(n^{-1/2})$ and $\mathcal{D}_n^j = o_P(n^{-1/2})$ for $j\in\{1,2\}$, \eqref{eq:cfExpansion} shows that
\begin{align}
\begin{pmatrix}
\bar{\nu}_n - \nu(P_0) \\
P_n s
\end{pmatrix}
= 
\begin{pmatrix}
P_n\phi_0 \\
P_n s
\end{pmatrix} +
o_p(n^{-1/2}) 
= P_n\begin{pmatrix}
\phi_0 \\
 s
\end{pmatrix} + o_p(n^{-1/2}). \label{eq:regularityJoint}
\end{align}
Moreover, $\|(\phi_0,s)\|_{L^2(P;\widetilde{\mathcal{H}})}^2 = \|\phi_0\|_{L^2(P;\mathcal{H})}^2 + \|s\|_{L^2(P)}^2 <\infty$. Hence, by Slutsky's lemma and a central limit theorem for Hilbert-valued random variables \citep[see Examples~1.4.7 and 1.8.5 in][]{van1996weak}, it holds that
\begin{align}
n^{1/2}\begin{pmatrix}
\bar{\nu}_n - \nu(P_0) \\
P_n s
\end{pmatrix}&\rightsquigarrow \begin{pmatrix}
\mathbb{H} \\
S
\end{pmatrix}, \label{eq:nunHatRegJoint}
\end{align}
where $(\mathbb{H},S)$ is a tight $\widetilde{\mathcal{H}}$-valued Gaussian random variable that is such that
\begin{align*}
\langle (h,r) ,(\mathbb{H},S)\rangle_{\widetilde{\mathcal{H}}}\sim N\left(0,E_0\left[\left\{\langle \phi_0(Z),h \rangle_{\mathcal{H}} + rs(Z)\right\}^2\right]\right).
\end{align*}
Marginalizing the $P_n s$ term on the left-hand side of \eqref{eq:regularityJoint} shows that \eqref{eq:oneStepWeakConv} holds.

We will use \eqref{eq:nunHatRegJoint} along with Theorem~3 in Chapter 5.2 of \cite{bickel1993efficient} to establish the regularity of $\bar{\nu}_n$. To use this result, it suffices to show that $E[S\mathbb{H}]=\dot{\nu}_P(s)$. We will establish this by showing that $\langle h,E[S\mathbb{H}]\rangle_{\mathcal{H}}=\langle h,\dot{\nu}_P(s)\rangle_{\mathcal{H}}$ for all $h\in\mathcal{H}$. To see that this holds, first note that $\langle h,E[S\mathbb{H}]\rangle_{\mathcal{H}} = E[S\langle h,\mathbb{H}\rangle_{\mathcal{H}}]$ since 
\begin{align*}
E[\|S\mathbb{H}\|_{\mathcal{H}}^2] = E[S^2\|\mathbb{H}\|_{\mathcal{H}}^2]\le E[S^4]^{1/2}E[\|\mathbb{H}\|_{\mathcal{H}}^4]^{1/2}<\infty,
\end{align*}
where the first inequality holds by Cauchy-Schwarz and the second by Fernique's theorem \citep{fernique1970integrabilite}. 
Hence, it suffices to show that $E[S\langle h,\mathbb{H}\rangle_{\mathcal{H}}]=\langle h,\dot{\nu}_P(s)\rangle_{\mathcal{H}}$. To show that this is the case, we let $f_h : \widetilde{\mathcal{H}}\rightarrow\mathbb{R}^2$ be defined so that $f_h(h_1,r_1)= (\langle h_1,h \rangle_{\mathcal{H}},r_1)$. For any $(a,b)\in\mathbb{R}^2$, the dot product $(a,b)\cdot f_h(\mathbb{H},S)$ is equal to $a \langle h, \mathbb{H}\rangle_{\mathcal{H}} + bS = \langle (ah,b),(\mathbb{H},S)\rangle_{\widetilde{\mathcal{H}}}$, which follows a mean-zero normal distribution with variance
\begin{align*}
E_0\left[\left\{a\langle \phi_0(Z),h \rangle_{\mathcal{H}} + bs(Z)\right\}^2\right] = (a,b)^\top \Sigma\,(a,b),
\end{align*}
where
\begin{align*}
\Sigma:=\begin{pmatrix}E_0\left[\langle \phi_0(Z),h \rangle_{\mathcal{H}}^2\right] & E_0\left[s(Z)\langle \phi_0(Z),h \rangle_{\mathcal{H}}\right] \\
E_0\left[s(Z)\langle \phi_0(Z),h \rangle_{\mathcal{H}}\right] & E_0[s(Z)^2]\end{pmatrix}.
\end{align*}
As $(a,b)\in\mathbb{R}^2$ was arbitrary, it follows that $(\langle h, \mathbb{H} \rangle_{\mathcal{H}},S)\sim N((0,0),\Sigma)$. Hence, $E[S\langle h, \mathbb{H} \rangle_{\mathcal{H}}]$ is equal to $E_0\left[s(Z)\langle \phi_0(Z),h \rangle_{\mathcal{H}}\right]$. Finally, note that $\langle \phi_0(Z),h \rangle_{\mathcal{H}} = \dot{\nu}_P^*(h)(Z)$ $P$-a.s. since $\phi_0$ is the EIF of $\nu$, and so
\begin{align*}
E_0\left[s(Z)\langle \phi_0(Z),h \rangle_{\mathcal{H}}\right] = E_0\left[s(Z)\dot{\nu}_P^*(h)(Z)\right] = \langle s, \dot{\nu}_P^*(h)\rangle_{L^2(P)} = \langle h, \dot{\nu}_P(s)\rangle_{\mathcal{H}}.
\end{align*}
\end{proof}

\subsubsection*{Proofs for Section~\ref{sec:confSets}}

\begin{proof}[Proof of Theorem~\ref{thm:CIcoverage}]
Suppose the conditions of Theorem~\ref{thm:al} hold and that $\|\phi_0\|_{L^2(P_0;\mathcal{H})}$ is strictly positive, $\Omega_n\in\mathcal{O}$, $\Omega_0\in\mathcal{O}$, and $\|\Omega_n - \Omega_0\|_{\mathrm{op}}=o_p(1)$. For brevity, we will let $w_n(\cdot):=w(\,\cdot\,;\Omega_n)$ and $w_0(\cdot):=w(\,\cdot\,;\Omega_0)$ in this proof. 
By Theorem~\ref{thm:al}, $n^{1/2}[\bar{\nu}_n - \nu(P_0)]\rightsquigarrow \mathbb{H}$, where $\mathbb{H}$ is as defined in that theorem. Slutsky's lemma and the continuous mapping theorem can further be used to show that $n\cdot w_n[\bar{\nu}_n-\nu(P_0)]\rightsquigarrow w_0(\mathbb{H})$. To see this, first note that
\begin{align}
n\cdot w_n[\bar{\nu}_n-\nu(P_0)]&= n\cdot w_0[\bar{\nu}_n-\nu(P_0)] + n\cdot w[\bar{\nu}_n-\nu(P_0);\Omega_n - \Omega_0]. \label{eq:Slutsky}
\end{align}
The first of the two terms on the right converges weakly to $w_0(\mathbb{H})$ by the continuous mapping theorem, where we have used that, by virtue of belonging to $\mathcal{O}$, $\Omega_0$ is a continuous operator, and therefore $w_0 : \mathcal{H}\rightarrow\mathbb{R}$ is a continuous functional. The second term on the right is $o_p(1)$, since, by Cauchy-Schwarz, the definition of the operator norm, and the continuous mapping theorem,
\begin{align*}
|n\cdot w[\bar{\nu}_n-\nu(P_0);\Omega_n - \Omega_0]|&\le n\|(\Omega_n - \Omega_0)[\bar{\nu}_n-\nu(P_0)]\|_{\mathcal{H}}\|\bar{\nu}_n-\nu(P_0)\|_{\mathcal{H}} \\
&\le \|\Omega_n - \Omega_0\|_{\mathrm{op}}\|n^{1/2}[\bar{\nu}_n-\nu(P_0)]\|_{\mathcal{H}}^2 = o_p(1) O_p(1) = o_p(1).
\end{align*}
Plugging this into \eqref{eq:Slutsky} and applying Slutsky's lemma shows that $n\cdot w_n[\bar{\nu}_n-\nu(P_0)]\rightsquigarrow w_0(\mathbb{H})$. 

We apply Corollary~3.3 of \cite{bogachev1996gaussian} to show that $w_0(\mathbb{H})$ is absolutely continuous. To apply this corollary, it suffices to show that $w_0 : \mathcal{H}\rightarrow\mathbb{R}$ is locally Lipschitz and that
the image of the Gateaux derivative $dw_0(h;\,\cdot\,)$ is $\mathbb{P}$-a.s. equal to $\mathbb{R}$, where $\mathbb{P}$ is the distribution of $\mathbb{H}$. 
For the Gateaux derivative condition, we note that, for any $g\in\mathcal{H}$, $dw_0(h;g)=\frac{d}{d\epsilon}w(h+\epsilon g;\Omega_0)=\langle\Omega_0(g),h\rangle_{\mathcal{H}} + \langle\Omega_0(h),g\rangle_{\mathcal{H}} = 2\langle\Omega_0(g),h\rangle_{\mathcal{H}}$, where the latter equality used that $\Omega_0$ is self-adoint. Since $\Omega_0$ is positive definite, $\langle\Omega_0(h),h\rangle_{\mathcal{H}}>0$ for all $h\in\mathcal{H}$. Hence, for any $h\in\mathcal{H}\backslash\{0\}$, the image of $dw_0(h;\cdot)$ is equal to $\mathbb{R}$; this can be seen by considering $dw_0(h;ch)$ with $c$ varying over $\mathbb{R}$. As $\|\phi_0\|_{L^2(P_0;\mathcal{H})}>0$, $\mathcal{H}\backslash\{0\}$ is a $\mathbb{P}$-probability one set. 
The locally Lipschitz property follows from the fact that, for all $h\in\mathcal{H}$ and all $g_1$ and $g_2$ in the unit ball $\mathcal{H}_1$ of $\mathcal{H}$,
\begin{align*}
&\left|w_0(h+g_1)-w_0(h+g_2)\right| \\
&\quad= \left|\int_0^1 [dw_0(h;\epsilon g_1)-dw_0(h;\epsilon g_2)] d\epsilon\right|\le \int_0^1 \left|dw_0(h;\epsilon g_1)-dw_0(h;\epsilon g_2)\right| d\epsilon \\
&\quad= 2 \left|\langle\Omega_0(g_1-g_2),h\rangle_{\mathcal{H}}\right|\int_0^1 \epsilon\, d\epsilon = \left|\langle\Omega_0(g_1-g_2),h\rangle_{\mathcal{H}}\right|\le \|\Omega_0(g_1-g_2)\|_{\mathcal{H}}\|h\|_{\mathcal{H}} \\
&\quad\le \|\Omega_0\|_{\mathrm{op}}\|g_1-g_2\|_{\mathcal{H}}\|h\|_{\mathcal{H}} = \|\Omega_0\|_{\mathrm{op}}\|(h+g_1)-(h+g_2)\|_{\mathcal{H}}\|h\|_{\mathcal{H}}.
\end{align*}
Hence, $w_0$ is $\|\Omega_0\|_{\mathrm{op}}\|h\|_{\mathcal{H}}$-Lipschitz continuous in the radius-one ball centered at $h$. As $h$ was arbitrary, $w_0$ is locally Lipschitz. 
Corollary~3.3 of \cite{bogachev1996gaussian} thus shows that $w_0(\mathbb{H})$ is an absolutely continuous random variable. 

Because convergence in distribution implies convergence of cumulative distribution functions at continuity points, $n\cdot w_n[\bar{\nu}_n-\nu(P_0)]\rightsquigarrow w_0(\mathbb{H})$ implies that
\begin{align}
P_0^n\left\{n\cdot w_n[\bar{\nu}_n - \nu(P_0)]\le \zeta_{1-\alpha}+\delta\right\} \overset{n\rightarrow\infty}{\longrightarrow} \mathbb{F}(\zeta_{1-\alpha}+\delta) \label{eq:cdfContPoints}
\end{align}
for all $\delta\in\mathbb{R}$, where $\mathbb{F}$ is the cumulative distribution function of $w_0(\mathbb{H})$. In what follows, we will use this fact twice when establishing the asymptotic validity of the $(1-\alpha)$-confidence sets $\mathcal{C}_n(\widehat{\zeta}_n)$. When doing so, we will also use that the event $\{\nu(P_0)\in\mathcal{C}_n(\widehat{\zeta}_n)\}$ is the same as the event $\{n\cdot w_n[\bar{\nu}_n - \nu(P_0)]\le \widehat{\zeta}_n\}$.

We now establish \ref{it:consCoverage}. To do this, we use that, for any $\delta>0$,
\begin{align*}
P_0^n&\left\{n\cdot w_n[\bar{\nu}_n - \nu(P_0)]> \widehat{\zeta}_n\right\} \\
&\le P_0^n\left\{n\cdot w_n[\bar{\nu}_n - \nu(P_0)]> \widehat{\zeta}_n,\widehat{\zeta}_n-\zeta_{1-\alpha}\ge -\delta\right\} + P_0^n\left\{\widehat{\zeta}_n-\zeta_{1-\alpha}< -\delta\right\} \\
&\le P_0^n\left\{n\cdot w_n[\bar{\nu}_n - \nu(P_0)]> \zeta_{1-\alpha}-\delta\right\} + P_0^n\left\{\widehat{\zeta}_n-\zeta_{1-\alpha}< -\delta\right\}.
\end{align*}
Subtracting both sides from $1$ yields that
\begin{align*}
P_0^n&\left\{n\cdot w_n[\bar{\nu}_n - \nu(P_0)]\le \widehat{\zeta}_n\right\} \\
&\ge P_0^n\left\{n\cdot w_n[\bar{\nu}_n - \nu(P_0)] \le \zeta_{1-\alpha}-\delta\right\} - P_0^n\left\{\widehat{\zeta}_n-\zeta_{1-\alpha}< -\delta\right\}.
\end{align*}
Taking $n\rightarrow\infty$, applying \eqref{eq:cdfContPoints}, and then taking $\delta\downarrow 0$ shows that, if $\widehat{\zeta}_n$ is an asymptotically conservative estimator of $\zeta_{1-\alpha}$ in the sense stated in \ref{it:consCoverage}, then $$\liminf_n P_0^n\left\{n\cdot w_n[\bar{\nu}_n - \nu(P_0)]\le \widehat{\zeta}_n\right\}\ge \mathbb{F}(\zeta_{1-\alpha})=1-\alpha.$$ This establishes \ref{it:consCoverage}.

We now establish \ref{it:exactCoverage}. To do this, we use that, for any $\delta>0$,
\begin{align*}
P_0^n&\left\{n\cdot w_n[\bar{\nu}_n - \nu(P_0)]\le \widehat{\zeta}_n\right\} \\
&\le P_0^n\left\{n\cdot w_n[\bar{\nu}_n - \nu(P_0)]\le \widehat{\zeta}_n,\widehat{\zeta}_n-\zeta_{1-\alpha}\le \delta\right\} + P_0^n\left\{\widehat{\zeta}_n-\zeta_{1-\alpha}> \delta\right\} \\
&\le P_0^n\left\{n\cdot w_n[\bar{\nu}_n - \nu(P_0)]\le \zeta_{1-\alpha}+\delta\right\} + P_0^n\left\{\widehat{\zeta}_n-\zeta_{1-\alpha}> \delta\right\}.
\end{align*}
Taking $n\rightarrow\infty$, applying \eqref{eq:cdfContPoints}, and then taking $\delta\downarrow 0$ shows that, if $\widehat{\zeta}_n$ is a consistent estimator of $\zeta_{1-\alpha}$, then $\limsup_n P_0^n\{n\cdot w_n[\bar{\nu}_n - \nu(P_0)]\le \widehat{\zeta}_n\}\le \mathbb{F}(\zeta_{1-\alpha})=1-\alpha$. Combining this with \ref{it:consCoverage} gives \ref{it:exactCoverage}.
\end{proof}

In the following result, we write $o_p(1)$ to denote a term that converges to zero in probability marginally over the randomness both in the original sample $(Z_1,\ldots,Z_n)$ and the bootstrap sample $(Z_1^\#,\ldots,Z_n^\#)$.
\begin{lemma}\label{lem:bootstraDriftTerm}
If $\|\phi_n^j-\phi_0\|_{L^2(P_0;\mathcal{H})}=o_p(1)$ for each $j\in\{1,2\}$, then $\|\mathbb{H}_n^\#-\mathbb{H}_{n,0}^\#\|_{\mathcal{H}}=o_p(1)$, where $\mathbb{H}_{n,0}^\#:= n^{1/2}\frac{1}{2}\sum_{j=1}^2 (P_n^{j,\#}-P_n^j) \phi_0$.
\end{lemma}
\begin{proof}[Proof of Lemma~\ref{lem:bootstraDriftTerm}]
This proof bears resemblance to that of Lemma~\ref{lem:driftTerm}. In what follows we use $Z_1^n$ as shorthand for the sample $(Z_1,Z_2,\ldots,Z_n)$. Note that, for any $\delta>0$,
\begin{align*}
\mathrm{Pr}\left\{\left\|\mathbb{H}_n^\#-\mathbb{H}_{n,0}^\#\right\|_{\mathcal{H}} > \delta\,\middle|\,Z_1^n\right\}&= \mathrm{Pr}\left\{\left\|\frac{1}{2}\sum_{j=1}^2 (P_n^{j,\#}-P_n^j) [\phi_n^j - \phi_0]\right\|_{\mathcal{H}}>n^{-1/2}\delta\,\middle|\,Z_1^n\right\} \\
&\le \sum_{j=1}^2 \mathrm{Pr}\left\{\left\|(P_n^{j,\#}-P_n^j) [\phi_n^j - \phi_0]\right\|_{\mathcal{H}}>n^{-1/2}\delta\,\middle|\,Z_1^n\right\}.
\end{align*}
Taking an expectation of both sides over $Z_1,Z_2,\ldots,Z_n\iidsim P_0$,
\begin{align*}
\mathrm{Pr}\left\{\left\|\mathbb{H}_n^\#-\mathbb{H}_{n,0}^\#\right\|_{\mathcal{H}} > \delta\right\}&= \sum_{j=1}^2 \mathrm{Pr}\left\{\left\|(P_n^{j,\#}-P_n^j) [\phi_n^j - \phi_0]\right\|_{\mathcal{H}}>n^{-1/2}\delta\right\}.
\end{align*}
In what follows we show that $\mathrm{Pr}\left\{\left\|(P_n^{j,\#}-P_n^j) [\phi_n^j - \phi_0]\right\|_{\mathcal{H}}>n^{-1/2}\delta\right\}=o(1)$ when $j=1$. An analogous argument holds for $j=2$. As $\delta>0$ was arbitrary, this will complete the proof.

We begin by noting that
\begin{align*}
&\mathrm{Pr}\left\{\left\|(P_n^{1,\#}-P_n^1) [\phi_n^1 - \phi_0]\right\|_{\mathcal{H}}>n^{-1/2}\delta\,\middle|\,Z_1^n\right\} \\
&\le \frac{n}{\delta^2}E\left[\left\|(P_n^{1,\#}-P_n^1) (\phi_n^1 - \phi_0)\right\|_{\mathcal{H}}^2\,\middle|\,Z_1^n\right] \\
&= \frac{2}{\delta^2}E\left[P_n^{1,\#}\left\|(I-P_n^1) (\phi_n^1 - \phi_0)\right\|_{\mathcal{H}}^2\,\middle|\,Z_1^n\right] \\
&\quad+ 4 \sum_{(i,j)\in \{n/2+1,\ldots,n\}^2 : i\not=j}\tfrac{E\left[\langle (I-P_n^1)(\phi_n^1-\phi_0)(Z_i^\#),(I-P_n^1)(\phi_n^1-\phi_0)(Z_k^\#)\rangle_{\mathcal{H}}\mid Z_1^n\right]}{n\delta^2},
\end{align*}
where $(I-P_n^1)(\phi_n^1 - \phi_0)$ denotes the map $z\mapsto (\phi_n^1 - \phi_0)(z) - P_n^1(\phi_n^1 - \phi_0)$. 
Because $Z_i^\#$ and $Z_k^\#$ are independent draws from $P_n^1$ conditional on $Z_1^n$ when $i\not=k$, each term in the summand on the right-hand side is exactly equal to zero. Since $P_n^{1,\#}$ is the empirical distribution of an iid sample from $P_n^1$, the first term on the right rewrites as $(2/\delta^2)\|(I-P_n^1)(\phi_n^1-\phi_0)\|_{L^2(P_n^1;\mathcal{H})}^2$. As $P_n^1(\phi_n^1-\phi_0)$ is a minimizer over $h\in\mathcal{H}$ of $\|\phi_n^1-\phi_0 - h\|_{L^2(P_n^1;\mathcal{H})}^2$, this term upper bounds by $(2/\delta^2)\|\phi_n^1-\phi_0\|_{L^2(P_n^1;\mathcal{H})}^2$. Plugging this bound into the above yields that
\begin{align*}
\mathrm{Pr}\left\{\left\|(P_n^{1,\#}-P_n^1) [\phi_n^1 - \phi_0]\right\|_{\mathcal{H}}>n^{-1/2}\delta\,\middle|\,Z_1^n\right\}&\le \frac{2}{\delta^2} \|\phi_n^1-\phi_0\|_{L^2(P_n^1;\mathcal{H})}^2.
\end{align*}
Taking the mean of both sides over the sample $Z_{n/2+1},Z_{n/2+2},\ldots,Z_n\iidsim P_0$ used to define $P_n^1$ shows that 
\begin{align*}
\mathrm{Pr}\left\{\left\|(P_n^{1,\#}-P_n^1) [\phi_n^1 - \phi_0]\right\|_{\mathcal{H}}>n^{-1/2}\delta\,\middle|\,Z_1,\ldots,Z_{n/2}\right\}&\le \frac{2}{\delta^2} \|\phi_n^1-\phi_0\|_{L^2(P_0;\mathcal{H})}^2.
\end{align*}
Using the trivial bound that probabilities are no more than $1$ and subsequently taking an expectation on both sides over $Z_1,Z_2,\ldots,Z_{n/2}\iidsim P_0$ yields that
\begin{align*}
\mathrm{Pr}\left\{\left\|(P_n^{1,\#}-P_n^1) [\phi_n^1 - \phi_0]\right\|_{\mathcal{H}}>n^{-1/2}\delta\right\}&\le E\left[\min\left\{1,\frac{2}{\delta^2} \|\phi_n^1-\phi_0\|_{L^2(P_0;\mathcal{H})}^2\right\}\right].
\end{align*}
Using that $\|\phi_n^1-\phi_0\|_{L^2(P_0;\mathcal{H})}=o_p(1)$ by assumption and applying the dominated convergence theorem shows that the right-hand side is $o(1)$, which gives the result.
\end{proof}

\begin{proof}[Proof of Theorem~\ref{thm:threshEst}]
By Lemma~\ref{lem:bootstraDriftTerm}, $\|\mathbb{H}_n^\#-\mathbb{H}_{n,0}^\#\|_{\mathcal{H}}=o_p(1)$. By Remark~2.5 of \cite{gine1990bootstrapping} and the fact that $\phi_0\in L^2(P_0;\mathcal{H})$, $\mathbb{H}_{n,0}^\#\rightsquigarrow \mathbb{H}$ weakly a.s., that is, weakly conditionally on the iid sequence $(Z_i)_{i=1}^\infty$ with $P_0$-probability one. By the continuous mapping theorem, this implies that $w(\mathbb{H}_{n,0}^\#;\Omega_0)\rightsquigarrow w(\mathbb{H};\Omega_0)$ weakly a.s. as well. In what follows we will use these facts, along with the fact that $\|\Omega_n-\Omega_0\|_{\mathrm{op}}=o_p(1)$, to show that $w(\mathbb{H}_n^\#;\Omega_n)\rightsquigarrow w(\mathbb{H};\Omega_0)$ weakly, conditionally on $(Z_i)_{i=1}^\infty$, in probability, in the sense defined in Chapter~23.2.1 of \cite{van2000asymptotic}. To show this, we begin by noting that, as $\Omega_n$ and $\Omega_0$ belong to $\mathcal{O}$,
\begin{align}
w(\mathbb{H}_n^\#;\Omega_n) - w(\mathbb{H}_{n,0}^\#;\Omega_0) 
&=  w(\mathbb{H}_n^\#-\mathbb{H}_{n,0}^\#;\Omega_0) + w(\mathbb{H}_n^\#;\Omega_n-\Omega_0). \label{eq:wBootDecomp}
\end{align}
We now show that each of the terms on the right are marginally $o_p(1)$. For the first, this follows from the fact that $|w(\mathbb{H}_n^\#-\mathbb{H}_{n,0}^\#;\Omega_0)| = |\langle\Omega_0(\mathbb{H}_n^\#-\mathbb{H}_{n,0}^\#),\mathbb{H}_n^\#-\mathbb{H}_{n,0}^\#\rangle_{\mathcal{H}}|\le \|\mathbb{H}_n^\#-\mathbb{H}_{n,0}^\#\|_{\mathcal{H}}^2\|\Omega_0\|_{\mathrm{op}}$, and this upper bound is $o_p(1)$ by Lemma~\ref{lem:bootstraDriftTerm}. For the second, this follows from the fact that $|w(\mathbb{H}_n^\#;\Omega_n-\Omega_0)|\le \|\mathbb{H}_n^\#\|_{\mathcal{H}}^2 \|\Omega_n-\Omega_0\|_{\mathrm{op}}$, combined with the fact that $\|\Omega_n-\Omega_0\|_{\mathrm{op}}=o_p(1)$, by assumption, and $\|\mathbb{H}_n^\#\|_{\mathcal{H}}^2=O_p(1)$, by virtue of the fact that $\|\mathbb{H}_n^\#-\mathbb{H}_{n,0}^\#\|_{\mathcal{H}}=o_p(1)$ and $\mathbb{H}_{n,0}^\#\rightsquigarrow \mathbb{H}$ weakly almost surely. 

We now derive a form of Slutsky's lemma to show that $w(\mathbb{H}_n^\#;\Omega_0)\rightsquigarrow w(\mathbb{H};\Omega_0)$ weakly, conditionally on $(Z_i)_{i=1}^\infty$, in probability. In particular, taking $f : \mathbb{R}\rightarrow [-1,1]$ to be a bounded, 1-Lipschitz function, letting $Z_1^n:=(Z_i)_{i=1}^n$, and recalling \eqref{eq:wBootDecomp}, we see that
\begin{align*}
&\left|E\left[ f\left(w[\mathbb{H}_n^\#;\Omega_n]\right)\,\middle|\,Z_1^n\right] - E[f\left(w[\mathbb{H};\Omega_0]\right)] \right| \\
&\le \left|E\left[ f\left(w[\mathbb{H}_n^\#;\Omega_n]\right)\,\middle|\,Z_1^n\right] - E\left[ f\left(w[\mathbb{H}_{n,0}^\#;\Omega_0]\right)\,\middle|\,Z_1^n\right]\right| \\
&\quad+ \left|E\left[ f\left(w[\mathbb{H}_{n,0}^\#;\Omega_0]\right)\,\middle|\,Z_1^n\right]- E[f\left(w[\mathbb{H};\Omega_0]\right)] \right| \\
&\le E\left[\min\left\{2,\left|w(\mathbb{H}_n^\#-\mathbb{H}_{n,0}^\#;\Omega_0) + w(\mathbb{H}_n^\#;\Omega_n-\Omega_0)\right|\right\}\,\middle|\,Z_1^n\right] \\
&\quad+ \left|E\left[ f\left(w[\mathbb{H}_{n,0}^\#;\Omega_0]\right)\,\middle|\,Z_1^n\right]- E\left[f\left(w[\mathbb{H};\Omega_0]\right)\right] \right|.
\end{align*}
Taking a supremum over all 1-Lipschitz $f : \mathbb{R}\rightarrow[-1,1]$ on both sides and using that $w[\mathbb{H}_{n,0}^\#;\Omega_0]\rightsquigarrow w[\mathbb{H};\Omega_0]$ weakly a.s. is equivalent to the supremum over such $f$ of the latter term on the right being $P_0$-a.s. $o(1)$ \citep[Chapter 23.2.1 of][]{van2000asymptotic}, we  see that it is $P_0$-a.s. true that
\begin{align*}
\sup_f &\left|E\left[ f\left(w[\mathbb{H}_n^\#;\Omega_n]\right)\,\middle|\,Z_1^n\right] - E[f\left(w[\mathbb{H};\Omega_0]\right)] \right| \\
&\le E\left[\min\left\{2,\left|w(\mathbb{H}_n^\#-\mathbb{H}_{n,0}^\#;\Omega_0) + w(\mathbb{H}_n^\#;\Omega_n-\Omega_0)\right|\right\}\,\middle|\,Z_1^n\right]  + o(1).
\end{align*}
Taking an expectation of the first term on the right over $Z_1,Z_2,\ldots,Z_n\iidsim P_0$, recalling that $w(\mathbb{H}_n^\#-\mathbb{H}_{n,0}^\#;\Omega_0) + w(\mathbb{H}_n^\#;\Omega_n-\Omega_0)$ is marginally $o_p(1)$, and applying the dominated convergence theorem shows that this nonnegative conditional expectation converges to zero in mean, and therefore also in probability. Hence, the above shows that $w(\mathbb{H}_n^\#;\Omega_n)$ converges weakly to $w(\mathbb{H};\Omega_0)$, given $(Z_i)_{i=1}^\infty$, in probability. This implies that the $(1-\alpha)$ quantile of $w(\mathbb{H}_n^\#;\Omega_n)$ conditional on $Z_1^n$, namely $\widehat{\zeta}_n$, converges in probability to the $(1-\alpha)$-quantile of $w(\mathbb{H};\Omega_0)$, namely $\zeta_{1-\alpha}$.
\end{proof}

\subsection{Proofs for Section~\ref{sec:perfGuaranteesRegularized}}

\subsubsection*{Proofs for Section~\ref{sec:regularizedOneStepGuarantee}}

Recall that the (squared) Hilbert-Schmidt norm is defined as $\|\dot{\nu}_P^*\|_{\mathrm{HS}}^2:= \sum_{k=1}^\infty \|\dot{\nu}_P^*(h_k)\|_{L^2(P)}^2$.
\begin{lemma}\label{lem:L2PHEquiv}
Suppose $\nu$ is pathwise differentiable at $P$ with EIF $\phi_P$. Then, $\|\phi_P\|_{L^2(P;\mathcal{H})}^2=\|\dot{\nu}_P^*\|_{\mathrm{HS}}^2$.
\end{lemma}
\begin{proof}[Proof of Lemma~\ref{lem:L2PHEquiv}]
Suppose that $\nu$ has EIF $\phi_P$. Using that (i) since $\phi_P$ is the EIF, it is $P$-a.s. true that $\dot{\nu}_P^*(h_k)(z) = \langle\phi_P(z),h_k\rangle_{\mathcal{H}}$ for all $k\in\mathbb{N}$, and (ii) for any $h\in\mathcal{H}$, $\|h\|_{\mathcal{H}}^2=\sum_{k=1}^\infty \langle h,h_k\rangle_{\mathcal{H}}^2$, we see that
\begin{align}
E_P\left[\sum_{k=1}^\infty \dot{\nu}_P^*(h_k)(Z)^2\right]&= E_P\left[\sum_{k=1}^\infty \langle \phi_P(Z),h_k\rangle_{\mathcal{H}}^2\right] = E_P\left[\|\phi_P(Z)\|_{\mathcal{H}}^2\right] = \left\|\phi_P\right\|_{L^2(P;\mathcal{H})}^2. \label{eq:FourierMainDisplay}
\end{align}
\end{proof}

In the following lemma, the little-oh and big-Omega notation both denote behavior as $n\rightarrow\infty$.
\begin{lemma}[No tight, non-zero weak limit for a scaling of the the regularized one-step estimator when $\sum_{k=1}^\infty P_0 \dot{\nu}_0^*(h_k)^2=+\infty$]\label{lem:notTight}
Suppose that $\nu$ is pathwise differentiable at $P_0$. 
Let $(\beta_n)_{n=1}^\infty$ be an $\ell^2$-valued sequence that grows to $(1,1,\ldots)$ pointwise as $n\rightarrow\infty$ and let $(c_n)_{n=1}^\infty$ be a nonnegative real-valued sequence. All of the following hold:
\begin{enumerate}[label=(\roman*)]
    \item\label{it:slowScaling} if $c_n=o[n^{1/2}/\sigma_0(\beta_n)]$, then $c_n P_n \phi_0^{\beta_n}\overset{p}{\rightarrow} 0$;
    \item\label{it:rightScaling} if $c_n=o(n^{1/2})$, then either $c_n P_n \phi_0^{\beta_n}$ does not converge weakly in $\mathcal{H}$ to a tight random element or $c_n P_n \phi_0^{\beta_n}\overset{p}{\rightarrow} 0$;
    \item\label{it:fastScaling} if $c_n=\Omega(n^{1/2})$ and $\sum_{k=1}^\infty P_0 \dot{\nu}_0^*(h_k)^2=+\infty$, then $c_n P_n \phi_0^{\beta_n}$ does not converge weakly in $\mathcal{H}$ to a tight random element.
\end{enumerate}
\end{lemma}
Before giving the proof, we note that the condition that $\sum_{k=1}^\infty P_0 \dot{\nu}_0^*(h_k)^2=+\infty$ holds in all of the examples we exhibit in this work for which there does not exist an EIF. Moreover, if there does exist an EIF $\phi_0$, then Lemma~\ref{lem:L2PHEquiv} shows that $\sum_{k=1}^\infty P_0 \dot{\nu}_0^*(h_k)^2<+\infty$ if and only if $\phi_0$ is $P_0$-Bochner square integrable.
\begin{proof}[Proof of Lemma~\ref{lem:notTight}]
Let $\nu$ and $(\beta_n)_{n=1}^\infty$ be as in the statement of the lemma.

We first prove \ref{it:slowScaling}. Suppose that $c_n=o[n^{1/2}/\sigma_0(\beta_n)]$. By the definition of $\sigma_0^2(\beta_n)$, we have that
\begin{align*}
&E_{P_0^n}\|n^{1/2} P_n \phi_0^{\beta_n}/\sigma_0(\beta_n)\|_{\mathcal{H}}^2 \\
&\quad=E_{P_0^n}\left[\frac{1}{n}\sum_{i=1}^n \sum_{k=1}^\infty \beta_{n,k}^2 \dot{\nu}_0^*(h_k)(Z_i)^2\right]/\sigma_0^2(\beta_n) = \sigma_0^2(\beta_n)/\sigma_0^2(\beta_n)=1.
\end{align*}
As $c_n=o[n^{1/2}/\sigma_0(\beta_n)]$, this implies that $E_{P_0^n}\|c_n P_n \phi_0^{\beta_n}\|_{\mathcal{H}}^2 = o(1)$, which in turn implies that $\|c_n P_n \phi_0^{\beta_n}\|_{\mathcal{H}}^2$ is $o_p(1)$, as desired.

We now prove \ref{it:rightScaling}. Suppose that $c_n=o(n^{1/2})$ and $c_n P_n \phi_0^{\beta_n}$ converges weakly in $\mathcal{H}$ to a tight random element $\mathbb{H}_0$. We will show that this can only be true if $\mathbb{H}_0$ is equal to the zero element of $\mathcal{H}$ almost surely. By Theorem~1.8.4 of \cite{van1996weak}, $c_n P_n \phi_0^{\beta_n}\rightsquigarrow\mathbb{H}_0$ implies that $\langle c_n P_n \phi_0^{\beta_n},h_k\rangle_{\mathcal{H}}\rightsquigarrow \langle \mathbb{H}_0,h_k\rangle_{\mathcal{H}}$ for all $k\in\mathbb{N}$. Moreover, since $\langle c_n P_n \phi_0^{\beta_n},h_k\rangle_{\mathcal{H}}=c_n P_n \langle \phi_0^{\beta_n},h_k\rangle_{\mathcal{H}}=\beta_{n,k}[c_n/n^{1/2}][n^{1/2} P_n \dot{\nu}_0^*(h_k)] $, this shows that $\beta_{n,k}[c_n/n^{1/2}][n^{1/2} P_n \dot{\nu}_0^*(h_k)] \rightsquigarrow \langle \mathbb{H}_0,h_k\rangle_{\mathcal{H}}$ for all $k\in\mathbb{N}$. Since $c_n/n^{1/2}=o(1)$, $\beta_{n,k}\overset{n\rightarrow\infty}{\longrightarrow} 1$, and, by the central limit theorem, $n^{1/2} P_n \dot{\nu}_0^*(h_k)=O_p(1)$, it holds that $\beta_{n,k}[c_n/n^{1/2}][n^{1/2} P_n \dot{\nu}_0^*(h_k)]\overset{p}{\rightarrow} 0$. 
As weak limits must share the same distribution, this shows that $\langle \mathbb{H}_0,h_k\rangle_{\mathcal{H}}$ is degenerate at zero for all $k$. Hence, $\mathbb{H}_0$ is almost surely equal to the zero element of $\mathcal{H}$.

We now prove \ref{it:fastScaling}. It suffices to show that $c_n P_n \phi_0^{\beta_n}$ does not converge weakly to a tight random element when $c_n=n^{1/2}$. We argue this by contradiction. To this end, suppose that there exists a tight random element $\mathbb{H}_0$ such that $n^{1/2} P_n \phi_0^{\beta_n}\rightsquigarrow \mathbb{H}_0$. By Lemma~1.8.4 of \cite{van1996weak}, $\mathbb{H}_0$ is then such that $\langle n^{1/2}P_n \phi_0^{\beta_n},h_k\rangle_{\mathcal{H}}\rightsquigarrow \langle \mathbb{H}_0,h_k\rangle_{\mathcal{H}}$ for all $k\in\mathbb{N}$. Combining this with the fact that $\langle n^{1/2}P_n \phi_0^{\beta_n},h_k\rangle_{\mathcal{H}}= \beta_{n,k}n^{1/2}P_n \dot{\nu}_0^*(h_k)$, $\beta_{n,k}\overset{k\rightarrow\infty}{\longrightarrow}1$, and a univariate central limit theorem, this shows that $\mathbb{H}_0$ is such that $\langle \mathbb{H}_0,h_k\rangle_{\mathcal{H}}\sim N[0,\dot{\nu}_0^*(h_k)^2]$ for all $k\in\mathbb{N}$. Hence, $\mathbb{H}_0$ is a Gaussian random element. Also, by Fernique's theorem \citep{fernique1970integrabilite}, $\mathbb{H}_0\in L^2(P_0;\mathcal{H})$. But $\|\mathbb{H}_0\|_{L^2(P_0;\mathcal{H})}^2=E[\sum_{k=1}^\infty \langle \mathbb{H}_0,h_k\rangle_{\mathcal{H}}^2] = \sum_{k=1}^\infty P_0 \dot{\nu}_0^*(h_k)^2$, which is equal to $+\infty$ by assumption. Contradiction.
\end{proof}

\begin{proof}[Proof of Theorem~\ref{thm:alReg}]
By the definitions of $\bar{\nu}_n^{\beta_n}$, $\phi_0^{\beta_n}$, $\mathcal{R}_n^{j,\beta_n}$, and $\mathcal{D}_n^{j,\beta_n}$,
\begin{align*}
\bar{\nu}_n^{\beta_n}&-\nu(P_0) - P_n\phi_0^{\beta_n} - \frac{1}{2}\sum_{j=1}^2 \sum_{k=1}^\infty (1-\beta_{n,k}) \langle\nu(\widehat{P}_n^j)-\nu(P_0),h_k\rangle_{\mathcal{H}} h_k \\
&= \frac{1}{2}\sum_{j=1}^2[\mathcal{R}_n^{j,\beta_n} + \mathcal{D}_n^{j,\beta_n}].
\end{align*}
Taking an $\mathcal{H}$-norm of both sides, applying the triangle inequality on the right, and upper bounding averages by maxima yields
\begin{align*}
\left\|\bar{\nu}_n^{\beta_n}-\nu(P_0) - P_n\phi_0^{\beta_n} - \frac{1}{2}\sum_{j=1}^2 \mathcal{B}_n^{j,\beta_n}\right\|_{\mathcal{H}}&\le \max_j \|\mathcal{R}_n^{j,\beta_n}\|_{\mathcal{H}} + \max_j \|\mathcal{D}_n^{j,\beta_n}\|_{\mathcal{H}}, 
\end{align*}
which bears resemblance to \eqref{eq:cfExpansion} but contains an extra bias term $\frac{1}{2}\sum_{j=1}^2 \mathcal{B}_n^{j,\beta_n}$. Plugging in the assumption that $\|\mathcal{R}_n^{j,\beta_n}\|_{\mathcal{H}}$ and $\|\mathcal{D}_n^{j,\beta_n}\|_{\mathcal{H}}$ are $O_p(\|\beta_n\|_{\ell^2}/n^{1/2})$ for each $j\in\{1,2\}$ gives \eqref{eq:regularizedMainDecomp}. 
Combining this with the assumption that $\mathcal{B}_n^{j,\beta_n}=O_p[\|\beta_n\|_{\ell^2}/n^{1/2}]$ for each $j\in\{1,2\}$ and the fact that $P_n \phi_0^{\beta_n}$ is $O_p[\sigma_0(\beta_n)/n^{1/2}]=O_p[\|\beta_n\|_{\ell^2}/n^{1/2}]$ by Chebyshev's inequality \citep{grenander1963probabilities} and Lemma~\ref{lem:approximateEIF} then gives \eqref{eq:regRate}.
\end{proof}

\begin{proof}[Proof of Lemma~\ref{lem:driftTermRegularized}]
This proof is similar to that of Lemma~\ref{lem:driftTerm}. Fix $\delta>0$ and an $\ell^2$-valued sequence $(\beta_n)_{n=1}^\infty$ that is such that $\|\phi_n^{1,\beta_n}-\phi_0^{\beta_n}\|_{L^2(P_0;\mathcal{H})}=o_p(r_n)$ holds. We will show that
\begin{align}
\lim_{n\rightarrow\infty}P_0^n\{\|\mathcal{D}_n^{1,\beta_n}\|_{\mathcal{H}}>r_n n^{-1/2}\delta\}= 0. \label{eq:Dn1DeltaRegularized}
\end{align}
As $\delta$ was arbitrary, this will show that $\mathcal{D}_n^{1,\beta_n}=o_p(r_n/n^{1/2})$. This will establish the stated result in the case where $j=1$, and an analogous argument can be used to handle the case where $j=2$.

Let $1_{\mathcal{E}_n}$ denote the indicator that the event $\mathcal{E}_n$ that $\|\phi_n^{1,\beta_n}-\phi_0^{\beta_n}\|_{L^2(P_0;\mathcal{H})}^2\le r_n^2\delta^2/2$ and let $\mathcal{E}_n^c$ denote the complement of $\mathcal{E}_n$.  We will leverage the following decomposition when showing \eqref{eq:Dn1DeltaRegularized}:
\begin{align}
P_0^n\{\|\mathcal{D}_n^{1,\beta_n}\|_{\mathcal{H}}>n^{-1/2}\delta\}&\le P_0^n\left(\{\|\mathcal{D}_n^{1,\beta_n}\|_{\mathcal{H}}>r_n  n^{-1/2}\delta\}\cap\mathcal{E}_n\right) + P_0^n(\mathcal{E}_n^c) \nonumber \\
&= E_0^n\left[1_{\mathcal{E}_n}P_0^n\left\{\|\mathcal{D}_n^{1,\beta_n}\|_{\mathcal{H}}>r_n n^{-1/2}\delta\,\middle|\, Z_1,\ldots,Z_{n/2}\right\}\right] + o(1), \label{eq:Dn1ProbRegularized}
\end{align}
where $E_{0}^n$ denotes an expectation under sampling from the $n$-fold product measure $P_0^n$ and $o(1)$ denotes a deterministic term that goes to zero as $n\rightarrow\infty$. The equality above holds by the law of total probability, the fact that $\mathcal{E}_n$ is measurable with respect to the $\sigma$-field generated by $Z_1,\ldots,Z_{n/2}$, and the assumption that $\|\phi_n^{1,\beta_n}-\phi_0^{\beta_n}\|_{L^2(P_0;\mathcal{H})}=o_p(r_n)$ implies that $P_0^n(\mathcal{E}_n^c)=o(1)$. To show \eqref{eq:Dn1DeltaRegularized}, the above shows that it suffices to show that the first term on the right-hand side is $o(1)$. To this end, note that Chebyshev's inequality for Hilbert-valued random variables \citep{grenander1963probabilities} and the bilinearity of inner products shows that
\begin{align}
&1_{\mathcal{E}_n}P_0^n\left\{\|\mathcal{D}_n^{1,\beta_n}\|_{\mathcal{H}}>r_n n^{-1/2}\delta\,\middle|\, Z_1,\ldots,Z_{n/2}\right\} \label{eq:Dn1EventRegularized} \\
&\le 1_{\mathcal{E}_n}\frac{E_{P_0^n}[\|(P_n^1-P_0)(\phi_n^{1,\beta_n}-\phi_0^{\beta_n})\|_{\mathcal{H}}^2\mid Z_1,\ldots,Z_{n/2}]}{r_n^2 n^{-1}\delta^2} \nonumber \\
&= 1_{\mathcal{E}_n}\frac{(n/2)^{-1}E_{P_0^n}[P_n^1\|(I-P_0)(\phi_n^{1,\beta_n}-\phi_0^{\beta_n})\|_{\mathcal{H}}^2\mid Z_1,\ldots,Z_{n/2}]}{r_n^2 n^{-1}\delta^2} \nonumber \\
&\quad+ 1_{\mathcal{E}_n}\frac{4}{n^2}\sum_{i\not=j}\tfrac{E_{P_0^n}[\langle (I-P_0)(\phi_n^{1,\beta_n}-\phi_0^{\beta_n})(Z_i),(I-P_0)(\phi_n^{1,\beta_n}-\phi_0^{\beta_n})(Z_k)\rangle_{\mathcal{H}}\mid Z_1,\ldots,Z_{n/2}]}{r_n^2 n^{-1}\delta^2}, \nonumber
\end{align}
where $(I-P_0)(\phi_n^{1,\beta_n}-\phi_0^{\beta_n})$ denotes the map $z\mapsto (\phi_n^{1,\beta_n}-\phi_0^{\beta_n})(z) - P_0(\phi_n^{1,\beta_n}-\phi_0^{\beta_n})$ and the sum is over $(i,j)\in \{n/2+1,\ldots,n\}^2$ such that $i\not=j$.  Each expectation in the summand on the right-hand side above is well-defined since $(\phi_n^{1,\beta_n}-\phi_0^{\beta_n})\in L^2(P_0;\mathcal{H})$ implies that $(I-P_0)(\phi_n^{1,\beta_n}-\phi_0^{\beta_n})\in L^2(P_0;\mathcal{H})$ as well. In fact, each expectation in the summand on the right-hand side is zero since, by the fact that $(I-P_0)(\phi_n^{1,\beta_n}-\phi_0^{\beta_n})\in L^2(P_0;\mathcal{H})$ and Fubini's theorem,
\begin{align*}
&E_{P_0^n}[\langle (I-P_0)(\phi_n^{1,\beta_n}-\phi_0^{\beta_n})(Z_i),(I-P_0)(\phi_n^{1,\beta_n}-\phi_0^{\beta_n})(Z_k)\rangle_{\mathcal{H}}\mid Z_1,\ldots,Z_{n/2}] \\
&= \int \langle (I-P_0)(\phi_n^{1,\beta_n}-\phi_0^{\beta_n})(z_1),(I-P_0)(\phi_n^{1,\beta_n}-\phi_0^{\beta_n})(z_2)\rangle_{\mathcal{H}} P_0^2(dz_1,dz_2) \\
&= \iint \langle (I-P_0)(\phi_n^{1,\beta_n}-\phi_0^{\beta_n})(z_1),(I-P_0)(\phi_n^{1,\beta_n}-\phi_0^{\beta_n})(z_2)\rangle_{\mathcal{H}} P_0(dz_1)P_0(dz_2) \\
&= \int \left\langle \int (I-P_0)(\phi_n^{1,\beta_n}-\phi_0^{\beta_n})(z_1)P_0(dz_1),(I-P_0)(\phi_n^{1,\beta_n}-\phi_0^{\beta_n})(z_2)\right\rangle_{\mathcal{H}} P_0(dz_2) \\
&= \int \left\langle 0,(I-P_0)(\phi_n^{1,\beta_n}-\phi_0^{\beta_n})(z_2)\right\rangle_{\mathcal{H}} P_0(dz_2) = 0.
\end{align*}
Returning to \eqref{eq:Dn1EventRegularized} and simplifying the first term on the right-hand side of that expression, this shows that
\begin{align*}
&1_{\mathcal{E}_n}P_0^n\left\{\|\mathcal{D}_n^{1,\beta_n}\|_{\mathcal{H}}>r_n n^{-1/2}\delta\,\middle|\, Z_1,\ldots,Z_{n/2}\right\}\le 1_{\mathcal{E}_n}\frac{2\|(I-P_0)(\phi_n^{1,\beta_n}-\phi_0^{\beta_n})\|_{L^2(P_0;\mathcal{H})}^2}{r_n^2 \delta^2}.
\end{align*}
Using that $P_0(\phi_n^{1,\beta_n}-\phi_0^{\beta_n})$ is a minimizer over $h\in\mathcal{H}$ of $\|\phi_n^{1,\beta_n}-\phi_0^{\beta_n} - h\|_{L^2(P_0;\mathcal{H})}^2$ and subsequently leveraging the definition of the event $\mathcal{E}_n$, this shows that
\begin{align*}
1_{\mathcal{E}_n}P_0^n\left\{\|\mathcal{D}_n^{1,\beta_n}\|_{\mathcal{H}}>r_n n^{-1/2}\delta\,\middle|\, Z_1,\ldots,Z_{n/2}\right\}&\le 1_{\mathcal{E}_n}\frac{2\|\phi_n^{1,\beta_n}-\phi_0^{\beta_n}\|_{L^2(P_0;\mathcal{H})}^2}{r_n^2 \delta^2} \\
&\le \min\left\{1,\frac{2\|\phi_n^{1,\beta_n}-\phi_0^{\beta_n}\|_{L^2(P_0;\mathcal{H})}^2}{r_n^2 \delta^2}\right\}.
\end{align*}
Note that the right-hand side above is no larger than $1$. 
Taking an expectation of both sides over $Z_1,\ldots,Z_{n/2}\iidsim P_0$ and recalling that $\|\phi_n^{1,\beta_n}-\phi_0^{\beta_n}\|_{L^2(P_0;\mathcal{H})}=o_p(r_n)$, the dominated convergence theorem shows that the first term on the right-hand side of \eqref{eq:Dn1ProbRegularized} is $o(1)$. This completes the proof.
\end{proof}

\begin{proof}[Proof of Lemma~\ref{lem:regRemBd}]
Since $\nu(P),\nu(P_0)\in\mathcal{H}$ and $\phi_P^\beta\in L^2(P_0;\mathcal{H})$, it holds that $\mathcal{R}_P^\beta\in\mathcal{H}$. We begin by showing that, for any $\ell\in\mathbb{N}$,
\begin{align}
\left\langle \mathcal{R}_P^{\beta},h_{\ell}\right\rangle_{\mathcal{H}}&= \beta_{\ell} \left[\left\langle \nu(P)-\nu(P_0),h_{\ell}\right\rangle_{\mathcal{H}} + P_0 \dot{\nu}_P^*(h_\ell)\right]. \label{eq:regRCoordProj}
\end{align}
As $\ell$ was arbitrary and $(h_k)_{k=1}^\infty$ is an orthonormal basis of $\mathcal{H}$, this will then show that
\begin{align}
\mathcal{R}_P^{\beta}&= \sum_{k=1}^\infty \beta_k \left[\left\langle \nu(P)-\nu(P_0),h_k\right\rangle_{\mathcal{H}} + P_0 \dot{\nu}_P^*(h_k)\right]h_k. \label{eq:regRForm}
\end{align}
We now establish \eqref{eq:regRCoordProj} for a fixed $\ell\in\mathbb{N}$. Note that
\begin{align*}
&\left\langle \mathcal{R}_P^{\beta},h_{\ell}\right\rangle_{\mathcal{H}} \\
&:= \left\langle \nu(P)-\nu(P_0) + P_0 \phi_P^{\beta} - \sum_{k=1}^\infty (1-\beta_k) \langle\nu(P) - \nu(P_0),h_k\rangle_{\mathcal{H}} h_k,h_{\ell}\right\rangle_{\mathcal{H}} \\
&= \left\langle \sum_{k=1}^\infty \left\langle \nu(P)-\nu(P_0),h_k\right\rangle_{\mathcal{H}} h_k + P_0 \phi_P^{\beta} - \sum_{k=1}^\infty (1-\beta_k) \langle\nu(P) - \nu(P_0),h_k\rangle_{\mathcal{H}} h_k,h_{\ell}\right\rangle_{\mathcal{H}} \\
&= \left\langle \sum_{k=1}^\infty \beta_k \left\langle \nu(P)-\nu(P_0),h_k\right\rangle_{\mathcal{H}} h_k + P_0 \phi_P^{\beta},h_{\ell}\right\rangle_{\mathcal{H}} \\
&= \beta_{\ell} \left\langle \nu(P)-\nu(P_0),h_{\ell}\right\rangle_{\mathcal{H}} + \left\langle P_0 \phi_P^{\beta},h_{\ell}\right\rangle_{\mathcal{H}}.
\end{align*}
It remains to show that $\langle P_0 \phi_P^{\beta},h_{\ell}\rangle_{\mathcal{H}}=\beta_\ell P_0 \dot{\nu}_P^*(h_\ell)$. To see that this holds, note that, since $\nu$ is pathwise differentiable at $P$, Lemma~\ref{lem:approximateEIF} ensures that $\phi_P^{\beta}(z)$ is the Riesz representation of $r_P^\beta(\cdot)(z)$ on a set $\mathcal{Z}^\beta$ of $P$-probability one. Since $P_0\ll P$, $\mathcal{Z}^\beta$ has $P_0$-probability one as well. Hence, $\int \langle\phi_P^{\beta}(z),h_\ell\rangle_{\mathcal{H}} P_0(dz)=\int r_P^\beta(h_\ell)(z) P_0(dz) = P_0 r_P^\beta(h_\ell)$. Furthermore, since $\phi_P^{\beta}\in L^2(P_0;\mathcal{H})$, $\int \langle\phi_P^{\beta}(z),h_\ell\rangle_{\mathcal{H}} P_0(dz)=\langle P_0\phi_P^{\beta},h_\ell\rangle_{\mathcal{H}}$, and so $\langle P_0\phi_P^{\beta},h_\ell\rangle_{\mathcal{H}} = P_0 r_P^\beta(h_\ell)$. Plugging in the definition of $r_P^\beta$ shows that $\langle P_0\phi_P^{\beta},h_\ell\rangle_{\mathcal{H}}=\beta_\ell P_0 \dot{\nu}_P^*(h_\ell)$, as desired. This establishes \eqref{eq:regRCoordProj}, which in turn establishes \eqref{eq:regRForm}. Using the form of $\mathcal{R}_P^\beta$ given in \eqref{eq:regRForm} establishes the equality in the statement of the lemma. The inequality follows by Cauchy-Schwarz.
\end{proof}

\begin{proof}[Proof of Lemma~\ref{lem:biasTerm}]
We have that
\begin{align*}
&\|\mathcal{B}_P^\beta\|_{\mathcal{H}}^2 \\
&= \sum_{k=1}^\infty (1-\beta_k)^2 \langle\nu(P)-\nu(P_0),h_k\rangle_{\mathcal{H}}^2 = \sum_{k=1}^\infty k^{-2u}(1-\beta_k)^2 k^{2u} \langle\nu(P)-\nu(P_0),h_k\rangle_{\mathcal{H}}^2 \\
&\le \left[\sup_{k\in\mathbb{N}} \frac{(1-\beta_k)^2}{k^{2u}}\right]\sum_{k=1}^\infty k^{2u} \langle\nu(P)-\nu(P_0),h_k\rangle_{\mathcal{H}}^2=\left[\sup_{k\in\mathbb{N}} \frac{(1-\beta_k)^2}{k^{2u}}\right]\|\nu(P)-\nu(P_0)\|_u^2.
\end{align*}
Taking a square root of both sides above gives the inequality from the lemma statement for general $\beta$. 
In the special case where $\beta_k=1$ for all $k\le K$ and $\beta_k=0$ for all $k>K$, \eqref{eq:biasK} follows by plugging this value of $\beta$ and then applying the triangle inequality.
\end{proof}

\subsubsection{Proofs for Section~\ref{sec:confSetsRegularized}}

\begin{proof}[Proof of Lemma~\ref{lem:regParamEIF}]
In this proof, we will use that $\Gamma_\beta : \mathcal{H}\rightarrow\mathcal{H}$ is a linear mapping, and also that $\|\Gamma_\beta(h)\|_{\mathcal{H}}\le \|h\|_{\mathcal{H}}$ for all $h$.

We now show that $\Gamma_\beta\circ \nu$ is pathwise differentiable with local parameter $\Gamma_\beta\circ \dot{\nu}_P$. To see that this holds, fix a quadratic mean differentiable submodel $\{P_\epsilon : \epsilon \in [0,\delta)\} \in \mathscr{P}(P,\mathcal{P},s)$. Note that
\begin{align*}
    \| \Gamma_\beta\circ \nu(P_\epsilon) - \Gamma_\beta\circ \nu(P) - \epsilon \Gamma_\beta\circ \dot{\nu}_P(s) \|_{\mathcal{H}}&= \| \Gamma_\beta\circ [\nu(P_\epsilon) - \nu(P) - \epsilon \dot{\nu}_P(s)] \|_{\mathcal{H}} \\
    &\le \| \nu(P_\epsilon) - \nu(P) - \epsilon \dot{\nu}_P(s) \|_{\mathcal{H}} = o(\epsilon), 
\end{align*}
where the final equality holds by the pathwise differentiability of $\nu$. Hence, $\Gamma_\beta\circ \nu$ is pathwise differentiable with local parameter $\Gamma_\beta\circ \dot{\nu}_P$. The efficient influence operator is equal to $r_P^\beta$, where this quantity is as defined above Lemma~\ref{lem:approximateEIF}. Indeed, for any $s\in\dot{\mathcal{P}}_P$ and $h\in\mathcal{H}$,
\begin{align*}
\left\langle s,r_P^\beta(h) \right\rangle_{L^2(P)}&= \left\langle s,\sum_{k=1}^\infty \beta_k\langle h,h_k\rangle_{\mathcal{H}}\dot{\nu}_P^*(h_k)\right\rangle_{L^2(P)} = \sum_{k=1}^\infty \beta_k\langle h,h_k\rangle_{\mathcal{H}}\left\langle s,\dot{\nu}_P^*(h_k)\right\rangle_{L^2(P)} \\
&= \sum_{k=1}^\infty \beta_k\langle h,h_k\rangle_{\mathcal{H}}\left\langle \dot{\nu}_P(s),h_k\right\rangle_{\mathcal{H}} = \left\langle h,\sum_{k=1}^\infty \beta_k \left\langle \dot{\nu}_P(s),h_k\right\rangle_{\mathcal{H}}h_k\right\rangle_{\mathcal{H}} \\
&= \left\langle h,\Gamma_\beta\circ \dot{\nu}_P(s)\right\rangle_{\mathcal{H}},
\end{align*}
where above we have used the linearity and continuity of inner products and the definition of the efficient influence operator $\dot{\nu}_P^*$ of $\nu$. By Lemma~\ref{lem:approximateEIF}, $r_P^\beta(\cdot)(z)$ is $P$-almost surely a bounded linear operator with Riesz representation $\phi_P^\beta\in L^2(P;\mathcal{H})$, and therefore $\phi_P^\beta$ is the EIF of $\Gamma_\beta\circ \nu$.
\end{proof}

The following is a consequence of Theorem~\ref{thm:al}, specialized to the case where the pathwise differentiable parameter of interest takes the form $\nu^\beta:=\Gamma_\beta\circ \nu$. Below $\mathcal{R}_n^{j,\beta}$ and $\mathcal{D}_n^{j,\beta}$ are the regularized remainder and drift terms for the $\beta$-regularized one-step estimator $\bar{\nu}_n^\beta$ of $\nu(P_0)$, as defined in Section \ref{sec:regularizedOneStepGuarantee}.
\begin{corollary}[Asymptotic linearity of $\widetilde{\nu}_n^\beta$]\label{cor:nuTildeAL}
Fix $\beta\in \ell_{*}^2$. Suppose that $\nu$ is pathwise differentiable at $P_0$ and that $\mathcal{R}_n^{j,\beta}$ and $\mathcal{D}_n^{j,\beta}$ are both $o_p(n^{-1/2})$ for $j\in\{1,2\}$. Under these conditions,
\begin{align}
    \widetilde{\nu}_n^\beta - \nu^\beta(P_0)&= \frac{1}{n}\sum_{i=1}^n \phi_0^\beta(Z_i) + o_p(n^{-1/2}), \label{eq:nuTildeAL}
\end{align}
$\widetilde{\nu}_n^\beta$ is regular, and $n^{1/2}[\widetilde{\nu}_n-\nu^\beta(P_0)]\rightsquigarrow \mathbb{H}$, where $\mathbb{H}$ is a tight $\mathcal{H}$-valued Gaussian random variable that is such that, for each $h\in\mathcal{H}$, the marginal distribution $\langle\mathbb{H} ,h \rangle_{\mathcal{H}}$ follows a $N(0,E_0[\langle \phi_0^\beta(Z),h \rangle_{\mathcal{H}}^2])$ distribution.
\end{corollary}
Since the above imposes conditions on the regularized remainder and drift terms for the regularized one-step estimator $\bar{\nu}_n^\beta$, any analysis that is performed to bound these terms when studying $\bar{\nu}_n^\beta$ can also be used to bound these terms when studying $\widetilde{\nu}_n^\beta$. In particular, Lemmas~\ref{lem:driftTermRegularized} and \ref{lem:regRemBd} can be used to study these terms.
\begin{proof}[Proof of Corollary~\ref{cor:nuTildeAL}]
   We establish that the conditions of Theorem~\ref{thm:al} are satisfied. By Lemma~\ref{lem:regParamEIF}, the EIF of $\nu^\beta$ is equal to $\phi_0^\beta\in L^2(P_0;\mathcal{H})$ at $P_0$ and $\phi_n^{j,\beta}$ at $\widehat{P}_n^j$, $j\in \{1,2\}$. Hence, for each $j\in\{1,2\}$, the regularized drift term $\mathcal{D}_n^{j,\beta}:=(P_n^j-P_0)(\phi_n^{j,\beta}-\phi_0^\beta)$ for the $\beta$-regularized one-step estimator $\bar{\nu}_n^\beta$ of $\nu(P_0)$ is identical to the drift term for the one-step estimator $\widetilde{\nu}_n^\beta$ of $\nu^\beta(P_0)$. Moreover, since
   \begin{align*}
    \mathcal{R}_n^{j,\beta}&:= \nu(\widehat{P}_n^j)-\nu(P_0) + P_0\phi_n^{j,\beta} - \sum_{k=1}^\infty (1-\beta_k)\langle \nu(\widehat{P}_n^j)-\nu(P_0),h_k\rangle_{\mathcal{H}} h_k \\
    &= \nu^\beta(\widehat{P}_n^j)-\nu^\beta(P_0) + P_0\phi_n^{j,\beta},
   \end{align*}
   the regularized remainder term $\mathcal{R}_n^{j,\beta}$ for the $\beta$-regularized one-step estimator $\bar{\nu}_n^\beta$ of $\nu(P_0)$ is also identical to the remainder term for the one-step estimator $\widetilde{\nu}_n^\beta$ of $\nu^\beta(P_0)$. As we have assumed that both $\mathcal{D}_n^{j,\beta}$ and $\mathcal{R}_n^{j,\beta}$ are $o_p(n^{-1/2})$, $j\in\{1,2\}$, Theorem~\ref{thm:al} implies all the claims in the statement of this corollary.
\end{proof}

\begin{proof}[Proof of Theorem~\ref{thm:localAlt}]

By Lemma~\ref{lem:regParamEIF}, $\phi_0^\beta$ is the EIF of $\nu^\beta$ at $P_0$. By the same arguments used to establish Theorem~\ref{thm:al}, $\widetilde{\nu}_n^\beta$ is a regular estimator of $\nu^\beta(P_0)$, and, in particular, $n^{1/2}[\widetilde{\nu}_n^\beta-\nu^\beta(P_\epsilon)]\rightsquigarrow \mathbb{H}^\beta$ under the sampling of $n$ iid draws from $P_{\epsilon=n^{-1/2}}$. Combining this with the pathwise differentiability of $\nu^\beta$, this shows that $n^{1/2}[\widetilde{\nu}_n^\beta-\nu^\beta(P_0)]\rightsquigarrow \mathbb{H}^\beta +  \dot{\nu}_0^\beta(s)$, where $\dot{\nu}_0^\beta:=\Gamma_\beta\circ\dot{\nu}_0$ is the local parameter of $\nu^\beta$ at $P_0$. By similar arguments to those used in the proof of Theorem~\ref{thm:CIcoverage}, $\|\mathbb{H}^\beta +  \dot{\nu}_0^\beta(s)\|_{\mathcal{H}}^2$ is a continuous random variable. Combining this with the fact that $\widehat{\zeta}_n\rightarrow\zeta_{1-\alpha}$ shows that
\begin{align*}
P_{\epsilon=n^{-1/2}}^n\left\{h_0\not\in \Gamma_\beta^{-1}[\mathcal{C}_n^\beta(\widehat{\zeta}_n)]\right\}&\overset{n\rightarrow\infty}{\longrightarrow} \mathrm{Pr}\left\{\|\mathbb{H}^\beta+\dot{\nu}_0^\beta(s)\|_{\mathcal{H}}^2> \zeta_{1-\alpha}\right\}.
\end{align*}
By Corollary~2 of \cite{lewandowski1995anderson}, the definition of $\zeta_{1-\alpha}$, and the fact that $\|\dot{\nu}_0(s)\|_{\mathcal{H}}>0$ and $\beta>0$ entrywise together imply that $\|\dot{\nu}_0^\beta(s)\|_{\mathcal{H}}>0$, the right-hand side above is strictly larger than $\alpha$.

To see that $\nu(P_{\epsilon=n^{-1/2}})$ is an $n^{-1/2}$-rate local alternative, note that
\begin{align*}
\|\nu(P_{\epsilon=n^{-1/2}})-h_0\|_{\mathcal{H}}&= \|\nu(P_{\epsilon=n^{-1/2}})-\nu(P_0)\|_{\mathcal{H}} \\
&= n^{-1/2}\|\dot{\nu}_0(s)\|_{\mathcal{H}} + o(n^{-1/2}) = O(n^{-1/2}).
\end{align*}
Above we used that $\nu(P_0)=h_0$, $\{P_\epsilon: \epsilon\}$ is quadratic mean differentiable, and $\dot{\nu}_0$ is the local parameter of $\nu$ at $P_0$.
\end{proof}

\section{A regularized inverse covariance operator and a consistent estimator thereof}\label{app:regOmega0}

Let $\Sigma_0 : h\mapsto E[\langle h,\mathbb{H}\rangle_{\mathcal{H}}\mathbb{H}]$ denote the covariance operator of the Gaussian random element $\mathbb{H}$ from Theorem~\ref{thm:al}. In this appendix, we study the regularized inverse $\Omega_0 = [(1-\lambda)\Sigma_0 + \lambda I]^{-1}$, where $\lambda>0$ and $I$ denotes the identity operator on $\mathcal{H}$. Though it would be interesting to study the behavior of our confidence set in cases where $\lambda$ shrinks to zero with sample size, doing so may be challenging since the inverse covariance operator $\Sigma_0^{-1}$ may not exist and, even if it does, it will generally be unbounded, which will complicate the use of the continuous mapping theorem that we use to justify the proof of the asymptotic validity of our confidence set (Theorem~\ref{thm:CIcoverage}). Hence, while studying the case where $\lambda$ shrinks to zero slowly with sample size is an interesting area for future work, here we focus on the case where $\lambda$ is a fixed constant that does not depend on sample size.

The regularized inverse of interest writes as $\Omega_0=f_\lambda(\Sigma_0)$, where, for a positive semidefinite linear operator $\Sigma : \mathcal{H}\rightarrow\mathbb{H}$, $f_\lambda(\Sigma):= [(1-\lambda)\Sigma+\lambda I]^{-1}$. The operator $f_\lambda$ can be seen to be Lipschitz continuous relative to the operator norm with Lipschitz constant $(1-\lambda)/\lambda^2$, which holds since, for positive semidefinite $\Sigma_1$ and $\Sigma_2$,
\begin{align*}
&\left\|f_\lambda(\Sigma_1)-f_\lambda(\Sigma_2)\right\|_{\mathrm{op}} \\
&= \left\|[(1-\lambda)\Sigma_1 + \lambda I]^{-1}\circ \{[(1-\lambda)\Sigma_2 + \lambda I] - [(1-\lambda)\Sigma_1 + \lambda I]\}\circ [(1-\lambda)\Sigma_2 + \lambda I]^{-1}\right\|_{\mathrm{op}} \\
&= (1-\lambda)\left\|[(1-\lambda)\Sigma_1 + \lambda I]^{-1}\circ (\Sigma_1-\Sigma_2)\circ [(1-\lambda)\Sigma_2 + \lambda I]^{-1}\right\|_{\mathrm{op}} \\
&\le (1-\lambda)\left\|(1-\lambda)\Sigma_1 + \lambda I\right\|_{\mathrm{op}}\left\|\Sigma_1-\Sigma_2\right\|_{\mathrm{op}}\left\|(1-\lambda)\Sigma_2 + \lambda I\right\|_{\mathrm{op}} \\
&\le (1-\lambda)\lambda^{-2}\left\|\Sigma_1-\Sigma_2\right\|_{\mathrm{op}}.
\end{align*}
By the continuous mapping theorem, an operator-norm-consistent estimator $\Omega_n$ of $\Omega_0$ --- that is, one for which $\|\Omega_n-\Omega_0\|_{\mathrm{op}}=o_p(1)$ --- is thus given by $\Omega_n=f_\lambda(\Sigma_n)$, where $\Sigma_n$ is any operator-norm consistent estimator of $\Sigma_0$. The following lemma shows that one such estimator is given by $h\mapsto \frac{1}{2}\sum_{j=1}^2 E_{P_n^j}[\langle h,\phi_n^j(Z)\rangle_{\mathcal{H}} \phi_n^j(Z)]$.
\begin{lemma}\label{lem:Sigman}
Fix $\lambda>0$. Suppose that $\|\phi_0\|_{L^2(P_0;\mathcal{H})}<\infty$ and $\|\phi_n^j-\phi_0\|_{L^2(P_0;\mathcal{H})}=o_p(1)$ for each $j\in\{1,2\}$. If $\Sigma_n : h\mapsto \frac{1}{2}\sum_{j=1}^2 E_{P_n^j}[\langle h,\phi_n^j(Z)\rangle_{\mathcal{H}} \phi_n^j(Z)]$, then $\|\Sigma_n-\Sigma_0\|_{\mathrm{op}}=o_p(1)$.
\end{lemma}
In what follows, for a function $\phi : \mathcal{Z}\rightarrow\mathcal{H}$, we let $\langle h,\phi\rangle_{\mathcal{H}}$ denote the map $z\mapsto \langle h,\phi(z)\rangle_{\mathcal{H}}$. We also recall that $\mathcal{H}_1$ denotes the unit ball of $\mathcal{H}$. We give the proof of the above result after we prove the following supporting lemma.
\begin{lemma}\label{lem:gcInnerProd}
In the setting of Lemma~\ref{lem:Sigman}, $\left\{z\mapsto \langle h,\phi_0(z)\rangle_{\mathcal{H}}^2 : h\in\mathcal{H}_1\right\}$ is $P_0$-Glivenko-Cantelli. Hence, $\sup_{h\in\mathcal{H}_1}(P_n^1-P_0) \langle h,\phi_0\rangle_{\mathcal{H}}^2=o_p(1)$.
\end{lemma}
\begin{proof}[Proof of Lemma~\ref{lem:gcInnerProd}]
In what follows we let $\mathcal{F}:= \left\{\langle h,\phi_0\rangle_{\mathcal{H}}^2 : h\in\mathcal{H}_1\right\}$. We will show that this collection of functions is $P_0$-Glivenko-Cantelli, which is the first result in the statement of the lemma. Combining this with the fact that $P_n^1$ is the empirical distribution of an iid sample from $P_0$ will then give the second result. To establish that $\mathcal{F}$ is Glivenko-Cantelli, we will show that the conditions of Theorem~2.4.3 in \cite{van1996weak} are satisfied. These conditions follow from $\mathcal{F}$ having a $P_0$-integrable envelope function $F$ and, moreover, satisfying an appropriate covering number condition --- we will define this condition in the next paragraph. Before doing so, we note that $\mathcal{F}$ has $P_0$-integrable envelope function $F(z):=\|\phi_0(z)\|_{\mathcal{H}}^2$. To see that this function is indeed an envelope of $\mathcal{F}$, note that, for any $h\in\mathcal{H}_1$, the Cauchy-Schwarz inequality shows that $\langle h,\phi_0(z) \rangle_{\mathcal{H}}^2\le \|\phi_0(z)\|_{\mathcal{H}}^2 = F(z)$. To see that $F$ is $P_0$-integrable, note that $P_0 F = \|\phi_0\|_{L^2(P_0;\mathcal{H})}^2$, which is finite by an assumption of Lemma~\ref{lem:Sigman}.

In the remainder of this proof, we will establish a covering number condition on $\mathcal{F}$ that implies the covering number condition from Theorem~2.4.3 in \cite{van1996weak}. In particular, in what follows we will show that, for any $\epsilon>0$, there exists an $N\in\mathbb{N}$ such that, with probability tending to one, the $L^1(P_n^1)$-covering number of $\mathcal{F}$ is no more than $N$; here we recall that, for fixed $\epsilon>0$, the corresponding $L^1(P_n^1)$ covering number of $\mathcal{F}$ denotes the size of the minimal $\epsilon$-cover of $\mathcal{F}$ relative to the $L^1(P_n^1)$ metric. Problem 2.4.2 in \cite{van1996weak} justifies why this condition suffices to establish the covering number condition in Theorem~2.4.3 of that reference. 

Fix $\epsilon>0$ and an orthonormal basis $(h_k)_{k=1}^\infty$ of $\mathcal{H}$. By the monotone convergence theorem and the $P_0$-Bochner square integrability of $\phi_0$,
\begin{align*}
\lim_{K'\rightarrow\infty} P_0 \left[\sum_{k=1}^{K'} \langle h_k,\phi_0 \rangle_{\mathcal{H}}^2\right]=P_0 \left[\sum_{k=1}^\infty \langle h_k ,\phi_0\rangle_{\mathcal{H}}^2\right]=\|\phi_0\|_{L^2(P_0;\mathcal{H})}^2<\infty.
\end{align*}
Hence, there exists a $K<\infty$ such that $P_0 \left[\sum_{k=1}^K \langle h_k,\phi_0 \rangle_{\mathcal{H}}^2\right]> \|\phi_0\|_{L^2(P_0;\mathcal{H})}^2-\epsilon/8$, and so, for this $K$, $P_0 \left[\sum_{k=K+1}^\infty \langle h_k,\phi_0 \rangle_{\mathcal{H}}^2\right]\le \epsilon/8$. By the weak law of large numbers, we further have that
\begin{align*}
P_n^1 \left[\sum_{k=K+1}^\infty \langle h_k ,\phi_0\rangle_{\mathcal{H}}^2\right]&= P_0 \left[\sum_{k=K+1}^\infty \langle h_k,\phi_0\rangle_{\mathcal{H}}^2\right] + o_p(1).
\end{align*}
Hereafter we work on the event $\mathcal{E}_n$ where (i) the $o_p(1)$ term above is less than $\epsilon/8$, so that the left-hand side above is no more than $\epsilon/4$, and (ii) $\|\phi_0\|_{L^2(P_n^1;\mathcal{H})}^2\le \|\phi_0\|_{L^2(P_0;\mathcal{H})}^2+\epsilon$; note that $\mathcal{E}_n$ holds with probability tending to one as $n\rightarrow\infty$. We now show that there exists a fixed subset $\widetilde{\mathcal{H}}_1$ of $\mathcal{H}_1$ such that, on this event, $\mathcal{F}_\epsilon:= \{\langle h,\phi_0\rangle_{\mathcal{H}}^2 : h\in \widetilde{\mathcal{H}}_1\}$ is an $\epsilon$-cover of $\mathcal{F}$. In particular, we take $\widetilde{\mathcal{H}}_1$ to be a finite $\delta$-cover of the finite-dimensional subset $\widetilde{\mathcal{H}}_1:=\mathcal{H}_1\cap\,\mathrm{span}\{h_1,\ldots,h_K\}$ of $\mathcal{H}_1$ relative to the $\mathcal{H}$-norm, where $\delta:= \epsilon/[4(\|\phi_0\|_{L^2(P_0;\mathcal{H})}^2+\epsilon)]$. Such a finite $\delta$-cover is guaranteed to exist because the unit ball in a finite-dimensional Hilbert space is necessarily totally bounded in the norm topology. To see that $\mathcal{F}_\epsilon$ is indeed an $\epsilon$-cover of $\mathcal{F}$, fix $h\in\mathcal{H}$ and let $\tilde{h}\in\widetilde{\mathcal{H}}_1$ be such that $\|\pi_K h-\tilde{h}\|_{\mathcal{H}}\le \delta$, where $\pi_K h:=\Pi_{\mathcal{H}}(h\mid \mathrm{span}\{h_1,\ldots,h_K\})$. Observe that
\begin{align}
&\left\|\langle h,\phi_0\rangle_{\mathcal{H}}^2 - \langle \tilde{h},\phi_0\rangle_{\mathcal{H}}^2 \right\|_{L^1(P_n^1)} \nonumber\\
&\quad\le \left\|\langle h,\phi_0\rangle_{\mathcal{H}}^2 - \langle \pi_K h,\phi_0\rangle_{\mathcal{H}}^2 \right\|_{L^1(P_n^1)} + \left\|\langle \pi_K h,\phi_0\rangle_{\mathcal{H}}^2 - \langle \tilde{h},\phi_0\rangle_{\mathcal{H}}^2 \right\|_{L^1(P_n^1)}. \label{eq:triangleIneqProj}
\end{align}
We now show each of the two terms on the right-hand side is no more than $\epsilon/2$. For the first term, we let $\pi_K \phi_0(z):=\Pi_{\mathcal{H}}[\phi_0(z)\mid \mathrm{span}\{h_1,\ldots,h_K\}]$ and note that
\begin{align*}
&\left\|\langle h,\phi_0\rangle_{\mathcal{H}}^2 - \langle \pi_K h,\phi_0\rangle_{\mathcal{H}}^2 \right\|_{L^1(P_n^1)} \\
&= \int |\langle h + \pi_K h,\phi_0(z)\rangle_{\mathcal{H}}\langle h-\pi_K h,\phi_0(z)\rangle_{\mathcal{H}}|  P_n^1(dz) \\
&\le \left[\int \langle h+\pi_K h,\phi_0(z)\rangle_{\mathcal{H}}^2 P_n^1(dz)\right]^{1/2}\left[\int \langle h-\pi_K h,\phi_0(z)\rangle_{\mathcal{H}}^2 P_n^1(dz)\right]^{1/2} \\
&= \left[\int \langle h+\pi_K h,\phi_0(z)\rangle_{\mathcal{H}}^2 P_n^1(dz)\right]^{1/2}\left[\int \langle h-\pi_K h,\phi_0(z)-\pi_K\phi(z)\rangle_{\mathcal{H}}^2 P_n^1(dz)\right]^{1/2} \\
&\le \|h+\pi_K h\|_{\mathcal{H}}\|h-\pi_K h\|_{\mathcal{H}}\|\phi_0-\pi_K\phi_0\|_{L^2(P_n^1;\mathcal{H})}^2. 
\end{align*}
The first equality holds by the definition of the $L^1(P_n^1)$ norm, both inequalities hold by Cauchy-Schwarz, and the second equality holds because $h-\pi_K h$ is orthogonal to $\mathrm{span}\{h_1,\ldots,h_K\}$. Now, since $h\in\mathcal{H}_1$, the triangle inequality and the fact that orthogonal projections cannot increase length show that $\|h+\pi_K h\|_{\mathcal{H}}\|h-\pi_K h\|_{\mathcal{H}}\le 2$. Furthermore, by the choice of $K$ and the fact that we are working on the event $\mathcal{E}_n$, $\|\phi_0-\pi_K\phi_0\|_{L^2(P_n^1;\mathcal{H})}^2\le \epsilon/4$. Hence, the first term on the right-hand side of \eqref{eq:triangleIneqProj} is no more than $\epsilon/2$. For the second term in \eqref{eq:triangleIneqProj}, two consecutive applications of the Cauchy-Schwarz inequality yield that
\begin{align*}
&\left\|\langle \pi_K h,\phi_0\rangle_{\mathcal{H}}^2 - \langle \tilde{h},\phi_0\rangle_{\mathcal{H}}^2 \right\|_{L^1(P_n^1)} \\
&\quad= \int |\langle \pi_K h+\tilde{h},\phi_0(z)\rangle_{\mathcal{H}}\langle \pi_K h-\tilde{h},\phi_0(z)\rangle_{\mathcal{H}}|  P_n^1(dz) \\
&\quad\le \left[\int \langle \pi_K h+\tilde{h},\phi_0(z)\rangle_{\mathcal{H}}^2 P_n^1(dz)\right]^{1/2}\left[\int \langle \pi_K h-\tilde{h},\phi_0(z)\rangle_{\mathcal{H}}^2 P_n^1(dz)\right]^{1/2} \\
&\quad\le \|\pi_K h+\tilde{h}\|_{\mathcal{H}}\|\pi_K h-\tilde{h}\|_{\mathcal{H}}\|\phi_0\|_{L^2(P_n^1;\mathcal{H})}^2.
\end{align*}
Now, by the triangle inequality and the fact that $\pi_K h$ and $\tilde{h}$ belong to $\mathcal{H}_1$, $\|\pi_K h+\tilde{h}\|_{\mathcal{H}}\le 2$. Moreover, because we are working on the event $\mathcal{E}_n$, $\|\phi_0\|_{L^2(P_n^1;\mathcal{H})}^2\le \|\phi_0\|_{L^2(P_0;\mathcal{H})}^2+\epsilon$. Combining these bounds with the fact that $\|\pi_K h-\tilde{h}\|_{\mathcal{H}}\le \delta:= \epsilon/[4(\|\phi_0\|_{L^2(P_0;\mathcal{H})}^2+\epsilon)]$ gives that $\|\langle \pi_K h,\phi_0\rangle_{\mathcal{H}}^2 - \langle \tilde{h},\phi_0\rangle_{\mathcal{H}}^2 \|_{L^1(P_n^1)}\le \epsilon/2$. Returning to \eqref{eq:triangleIneqProj}, this shows that $\|\langle h,\phi_0\rangle_{\mathcal{H}}^2 - \langle \tilde{h},\phi_0\rangle_{\mathcal{H}}^2 \|_{L^1(P_n^1)}\le \epsilon$. As $h\in\mathcal{H}_1$ was arbitrary and $\langle \tilde{h},\phi_0\rangle_{\mathcal{H}}^2\in\mathcal{F}_\epsilon$, this shows that $\mathcal{F}_\epsilon$ is an $\epsilon$-cover of $\mathcal{H}_1$ on the event $\mathcal{E}_n$. Since $\mathcal{F}_\epsilon$ contains finitely many functions, we can invoke Theorem~2.4.3 of \cite{van1996weak} to show that $\mathcal{F}$ is a $P_0$-Glivenko Cantelli class.
\end{proof}

We conclude this appendix with the proof of Lemma~\ref{lem:Sigman}.
\begin{proof}[Proof of Lemma~\ref{lem:Sigman}]
For $j\in\{1,2\}$, let $\Sigma_n^j(h):= E_{P_n^j}[\langle h,\phi_n^j(Z)\rangle_{\mathcal{H}} \phi_n^j(Z)]$. By the triangle inequality,
\begin{align*}
\|\Sigma_n-\Sigma_0\|_{\mathrm{op}}&=\left\|\frac{1}{2}\sum_{j=1}^2 \Sigma_n^j - \Sigma_0\right\|_{\mathrm{op}}\le \frac{1}{2}\sum_{j=1}^2 \left\|\Sigma_n^j - \Sigma_0\right\|_{\mathrm{op}}.
\end{align*}
Hence, it suffices to show that $\left\|\Sigma_n^j - \Sigma_0\right\|_{\mathrm{op}}=o_p(1)$. We show this for the case where $j=1$, and the case where $j=2$ follows by analogous arguments.

Because $\Sigma_n^1-\Sigma_0$ is a positive, self-adjoint operator, it holds that
\begin{align}
\|\Sigma_n^1 - \Sigma_0\|_{\mathrm{op}}&= \sup_{h\in\mathcal{H}_1}\langle\Sigma_n^1(h) - \Sigma_0(h),h\rangle_{\mathcal{H}}. \label{eq:svOpNorm}
\end{align}
We will bound the right-hand side above in what follows. To do this, we will use that, for any $h\in\mathcal{H}_1$,
\begin{align*}
\langle\Sigma_n^1(h) - \Sigma_0(h),h\rangle_{\mathcal{H}} &= \left\langle P_n^1 \langle h,\phi_n^1\rangle_{\mathcal{H}}\phi_n^1 - P_0 \langle h,\phi_0\rangle_{\mathcal{H}}\phi_0,h\right\rangle_{\mathcal{H}} \\
&= \left\langle P_n^1 \left[\langle h,\phi_n^1\rangle_{\mathcal{H}}\phi_n^1 - \langle h,\phi_0\rangle_{\mathcal{H}}\phi_0\right] + (P_n^1-P_0) \langle h,\phi_0\rangle_{\mathcal{H}}\phi_0,h\right\rangle_{\mathcal{H}} \\
&= P_n^1 \left[\langle h,\phi_n^1\rangle_{\mathcal{H}}^2 - \langle h,\phi_0\rangle_{\mathcal{H}}^2\right] + (P_n^1-P_0) \langle h,\phi_0\rangle_{\mathcal{H}}^2 \\
&= P_n^1 \langle h,\phi_n^1-\phi_0\rangle_{\mathcal{H}}\langle h,\phi_n^1+\phi_0\rangle_{\mathcal{H}} + (P_n^1-P_0) \langle h,\phi_0\rangle_{\mathcal{H}}^2.
\end{align*}
Applying the triangle and Cauchy-Schwarz inequalities to the above and combining the result with \eqref{eq:svOpNorm} shows that
\begin{align*}
\|\Sigma_n^1-\Sigma_0\|_{\mathrm{op}}&\le \left[\sup_{h\in\mathcal{H}_1} P_n^1 \langle h,\phi_n^1-\phi_0\rangle_{\mathcal{H}}^2\right]^{1/2}\left[\sup_{h\in\mathcal{H}_1} P_n^1 \langle h,\phi_n^1+\phi_0\rangle_{\mathcal{H}}^2\right]^{1/2} \\
&\quad+ \sup_{h\in\mathcal{H}_1}(P_n^1-P_0) \langle h,\phi_0\rangle_{\mathcal{H}}^2.
\end{align*}
Using that $\sup_{h\in\mathcal{H}_1} P_n^1 \langle h,\phi\rangle_{\mathcal{H}}^2\le \|\phi\|_{L^2(P_n^1;\mathcal{H})}^2$ for $\phi : \mathcal{Z}\rightarrow\mathbb{H}$ and then subsequently applying the triangle inequality in $L^2(P_n^1;\mathcal{H})$, we find that
\begin{align*}
\|\Sigma_n^1-\Sigma_0\|_{\mathrm{op}}&\le \|\phi_n^1-\phi_0\|_{L^2(P_n^1;\mathcal{H})}\left(2\|\phi_0\|_{L^2(P_n^1;\mathcal{H})} + \|\phi_n^1-\phi_0\|_{L^2(P_n^1;\mathcal{H})}\right) \\
&\quad+ \sup_{h\in\mathcal{H}_1}(P_n^1-P_0) \langle h,\phi_0\rangle_{\mathcal{H}}^2.
\end{align*}
The second term is $o_p(1)$ by Lemma~\ref{lem:gcInnerProd}. We now show that the first term is also $o_p(1)$. To see this, first note that $\|\phi_0\|_{L^2(P_n^1;\mathcal{H})}=O_p(1)$ by the fact that $\|\phi_0\|_{L^2(P_0;\mathcal{H})}<\infty$ and since, by the weak law of large numbers, $\|\phi_0\|_{L^2(P_n^1;\mathcal{H})}^2= \|\phi_0\|_{L^2(P_0;\mathcal{H})}^2 + o_p(1)$. Hence, it suffices to show that $\|\phi_n^1-\phi_0\|_{L^2(P_n^1;\mathcal{H})}=o_p(1)$. To see that this holds, note that, for any $\delta>0$, the probability that $\|\phi_n^1-\phi_0\|_{L^2(P_n^1;\mathcal{H})}^2$ exceeds $\delta$ conditional on the data used to create the estimate $\phi_n^1$ of $\phi_0$ satisfies the following:
\begin{align*}
P_0^n&\left\{\|\phi_n^1-\phi_0\|_{L^2(P_n^1;\mathcal{H})}^2 > \delta\,\middle|\,Z_1,\ldots,Z_{n/2}\right\} \\
&\le \min\left\{1,\delta^{-1}E_{P_0^n}\left[\|\phi_n^1-\phi_0\|_{L^2(P_n^1;\mathcal{H})}^2\,\middle|\,Z_1,\ldots,Z_{n/2}\right]\right\} \\
&= \min\left\{1,\delta^{-1}\|\phi_n^1-\phi_0\|_{L^2(P_0;\mathcal{H})}^2\right\}.
\end{align*}
Taking an expectation of both sides over $Z_1,\ldots,Z_{n/2}$, using that $\|\phi_n^1-\phi_0\|_{L^2(P_0;\mathcal{H})}^2=o_p(1)$, and applying the dominated convergence theorem shows that the right-hand side is $o(1)$. As $\delta>0$ was arbitrary, this shows that $\|\phi_n^1-\phi_0\|_{L^2(P_n^1;\mathcal{H})}=o_p(1)$, which gives the result.
\end{proof}

\section{A conservative estimator of the threshold used to define our confidence sets that does not require the bootstrap}\label{app:szekely}

We now present a conservative estimator of the threshold $\zeta_{1-\alpha}$ that is used to construct the confidence sets described in Section~\ref{sec:confSets}. This estimator is applicable in settings where $\Omega_0$ is the identity function. 
Its form is motivated by Theorem~1 in \cite{szekely2003extremal}, which concerns tail probabilities for Gaussian quadratic forms of the type $\sum_{k=1}^\infty c_k N_k^2$, where $(N_k)_{k=1}^\infty$ is an iid sequence of standard normal random variables and $(c_k)_{k=1}^\infty$ is a sequence of nonnegative constants. This result is applicable when $\Omega_0$ is the identity operator since, in that case, $\|\mathbb{H}\|_{\mathcal{H}}^2$ has the same distribution as $\sum_{k=1}^\infty E_0[\langle \phi_0(Z),h_{0,k}\rangle_{\mathcal{H}}^2] N_k^2$, where $(h_{0,k})_{k=1}^\infty$ are the unit eigenvectors of the covariance operator $E[\langle\mathbb{H},\,\cdot\,\rangle_{\mathcal{H}}\mathbb{H}]$ of $\mathbb{H}$. When $\alpha\le 0.2$, as it will be in most practical settings, Theorem~1 in  \cite{szekely2003extremal} can be used to show that $\mathrm{Pr}\{\|\mathbb{H}\|_{\mathcal{H}}^2> \chi_{1-\alpha}^2\|\phi_0\|_{L^2(P_0;\mathcal{H})}^2\}\le \alpha$, where $\chi_{1-\alpha}^2$ denotes the $(1-\alpha)$-quantile of a chi-squared distribution with 1 degree of freedom. Hence, if $s_n^2$ is a consistent estimator of $\|\phi_0\|_{L^2(P_0;\mathcal{H})}^2$, then Theorem~\ref{thm:CIcoverage} shows that $\mathcal{C}_n(\chi_{1-\alpha}^2\cdot s_n^2)$ is an asymptotically valid, albeit conservative, $(1-\alpha)$-confidence set for $\nu(P_0)$. If $\phi_n^j\rightarrow \phi_0$ in probability in $L^2(P_0;\mathcal{H})$ for $j\in\{1,2\}$, then Lemma~\ref{lem:phinConv}, given below, shows that the cross-fitted estimator $\frac{1}{2}\sum_{j=1}^2\|\phi_n^j\|_{L^2(P_n^j;\mathcal{H})}^2$ will converge in probability to $\|\phi_0\|_{L^2(P_0;\mathcal{H})}^2$, so that $s_n^2$ can be taken to be equal to this estimator.
\begin{lemma}\label{lem:phinConv}
If $\|\phi_0\|_{L^2(P_0;\mathcal{H})}<\infty$ and $\|\phi_n^j-\phi_0\|_{L^2(P_0;\mathcal{H})}=o_p(1)$ for each $j\in\{1,2\}$, then $\|\phi_n^j\|_{L^2(P_n^j;\mathcal{H})}^2$ converges to $\|\phi_0\|_{L^2(P_0;\mathcal{H})}^2$ in probability for each $j\in\{1,2\}$ and, consequently, $\frac{1}{2}\sum_{j=1}^2\|\phi_n^j\|_{L^2(P_n^j;\mathcal{H})}^2\rightarrow\|\phi_0\|_{L^2(P_0;\mathcal{H})}^2$ in probability as well.
\end{lemma}
\begin{proof}[Proof of Lemma~\ref{lem:phinConv}]
We begin by showing that $\|\phi_n^j\|_{L^2(P_n^j;\mathcal{H})}^2\overset{p}{\rightarrow}\|\phi_0\|_{L^2(P_0;\mathcal{H})}^2$ for fixed $j\in\{1,2\}$. Note that
\begin{align*}
&\left|\|\phi_n^j\|_{L^2(P_n^j;\mathcal{H})}^2 - \|\phi_0\|_{L^2(P_0;\mathcal{H})}^2\right| \\
&\quad\le \left|\int \|\phi_n^j(z)\|_{\mathcal{H}}^2 (P_n^j-P_0)(dz)\right| + \left|\|\phi_n^j\|_{L^2(P_0;\mathcal{H})}^2 - \|\phi_0\|_{L^2(P_0;\mathcal{H})}^2\right|.
\end{align*}
We study these two terms separately. The second term is $o_p(1)$ since (i) by the reverse triangle inequality and the assumption of this theorem, $|\|\phi_n^j\|_{L^2(P_0;\mathcal{H})} - \|\phi_0\|_{L^2(P_0;\mathcal{H})}|\le \|\phi_n^j - \phi_0\|_{L^2(P_0;\mathcal{H})} = o_p(1)$ and (ii) by the continuous mapping theorem, $\|\phi_n^j\|_{L^2(P_0;\mathcal{H})}\overset{p}{\rightarrow} \|\phi_0\|_{L^2(P_0;\mathcal{H})}$ implies that $\|\phi_n^j\|_{L^2(P_0;\mathcal{H})}^2\overset{p}{\rightarrow} \|\phi_0\|_{L^2(P_0;\mathcal{H})}^2$. In what follows, we show that the first term above is $o_p(1)$ as well.

Combining the fact that $\|\phi_n^j(z)\|_{\mathcal{H}}\le \|\phi_n^j(z)-\phi_0\|_{\mathcal{H}} + \|\phi_0(z)\|_{\mathcal{H}}$ with the basic inequality that $(a+b)^2\le 2(a^2+b^2)$, and subsequently applying the triangle inequality, yields that
\begin{align}
\frac{1}{2}&\left|\int \|\phi_n^j(z)\|_{\mathcal{H}}^2 (P_n^j-P_0)(dz)\right| \nonumber \\
&\le \left|\int \left(\|\phi_n^j(z)-\phi_0(z)\|_{\mathcal{H}}^2 + \|\phi_0(z)\|_{\mathcal{H}}^2\right) (P_n^j-P_0)(dz)\right| \nonumber \\
&\le \left|\int \|\phi_n^j(z)-\phi_0(z)\|_{\mathcal{H}}^2 P_n^j(dz)\right| + \left|\int \|\phi_n^j(z)-\phi_0(z)\|_{\mathcal{H}}^2 P_0(dz)\right| \nonumber \\
&\quad+ \left|\int \|\phi_0(z)\|_{\mathcal{H}}^2 (P_n^j-P_0)(dz)\right|. \label{eq:phinjEmpProc}
\end{align}
The second term is equal to $\|\phi_n^j(z)-\phi_0(z)\|_{L^2(P_0;\mathcal{H})}^2$ and so is $o_p(1)$ by assumption. 
The third term is $o_p(1)$ by the weak law of large numbers, which is applicable since $\|\phi_0\|_{L^2(P_0;\mathcal{H})}^2<\infty$ by assumption. By Markov's inequality, the fact that $P_n^j$ and $\widehat{P}_n^j$ are fitted on different subsamples, and the fact that probabilities are no more than $1$, the conditional probability that the first term exceeds any fixed $\delta>0$ satisfies the following:
\begin{align*}
P_0^n&\left\{\left|\int \|\phi_n^j(z)-\phi_0(z)\|_{\mathcal{H}}^2 P_n^j(dz)\right| > \delta\,\middle|\,Z_{(j-1)n/2 + 1},\ldots,Z_{(j-1)n/2 + n/2}\right\} \\
&\le \min\left\{1,\frac{1}{\delta}\|\phi_n^j-\phi_0\|_{L^2(P_0;\mathcal{H})}^2\right\}.
\end{align*}
Taking an expected value of both sides over $Z_{(j-1)n/2 + 1},\ldots,Z_{(j-1)n/2 + n/2}\iidsim P_0$ and using that $\|\phi_n^j(z)-\phi_0(z)\|_{L^2(P_0;\mathcal{H})}^2=o_p(1)$, the dominated convergence theorem shows that $\left|\int \|\phi_n^j(z)-\phi_0(z)\|_{\mathcal{H}}^2 P_n^j(dz)\right| > \delta$ occurs with probability tending to zero. As $\delta>0$ was arbitrary, this shows that the first term on the right-hand side of \eqref{eq:phinjEmpProc} is $o_p(1)$, which completes the proof of the fact that $\|\phi_n^j\|_{L^2(P_n^j;\mathcal{H})}^2\overset{p}{\rightarrow}\|\phi_0\|_{L^2(P_0;\mathcal{H})}^2$ for $j\in\{1,2\}$.

Since $\|\phi_n^j\|_{L^2(P_n^j;\mathcal{H})}^2\overset{p}{\rightarrow}\|\phi_0\|_{L^2(P_0;\mathcal{H})}^2$ for $j\in\{1,2\}$, the continuous mapping theorem shows that $\frac{1}{2}\sum_{j=1}^2\|\phi_n^j\|_{L^2(P_n^j;\mathcal{H})}^2\rightarrow\|\phi_0\|_{L^2(P_0;\mathcal{H})}^2$.
\end{proof}

\section{Numerical considerations for computing the proposed confidence sets}\label{app:practicalNonRegCS}

Evaluating whether some $h_0\in\mathcal{H}$ belongs to the confidence set in \eqref{eq:confSet} requires computing the quadratic form $\langle \Omega_n(\bar{\nu}_n-h_0),\bar{\nu}_n-h_0\rangle_{\mathcal{H}}$ in a possibly infinite dimensional Hilbert space. 
In many cases, computing this quadratic form will require leveraging some form of numerical approximation. 
One way of doing this is to replace the computation of the quadratic form in \eqref{eq:confSet} by a finite-dimensional approximation thereof. To this end, for each $m\in\mathbb{N}$ we let $D_m : \mathcal{H}\rightarrow\mathbb{R}^m$ denote a linear operator. This linear operator should have the property that, for any $h_1$ and $h_2$ in $\mathcal{H}$, $\langle h_1,h_2\rangle_{\mathcal{H}}\overset{m\rightarrow\infty}{\longrightarrow} D_m(h_1)^\top D_m(h_2)$, where here and in all subsequent calculations all vectors are taken to be equal to column vectors when involved in matrix operations. In practice, for a given sample size $n$, $m$ can be chosen to be some large constant. One natural choice of $D_m$ corresponds to the map from $h$ to the vector of the first $m$ generalized Fourier coefficients of $h$ with respect to some orthonormal basis $(h_k)_{k=1}^\infty$, so that $D_m(h)=(\langle h_k,h\rangle_{\mathcal{H}})_{k=1}^m$. 
If $\mathcal{H}=L^2([0,1])$ and it is known that $D_m$ will only be evaluated on elements of $\mathcal{H}$ that have a continuous version, as occurs if $h_0$ is continuous and $\bar{\nu}_n$ is continuous with probability one, then another natural choice is to take $D_m(h)$ to be equal to $(h(t_k)/m^{1/2})_{k=1}^m$, where $t_k=k/(m+1)$ and $h(t_k)$ is taken to be the evaluation of the continuous version of $h$ at $t_k$. If instead $\mathcal{H}=L^2(\mathbb{R})$ and $h$ still has a continuous version, then $D_m$ can be taken equal to $(h(t_k)/[m\varphi_{\mu,\sigma}(t_k)]^{1/2})_{k=1}^m$, where, for $\mu\in\mathbb{R}$ and $\sigma>0$, $\{t_k\}_{k=1}^m$ are such that $\Phi_{\mu,\sigma}(t_k)=k/(m+1)$ with $\Phi_{\mu,\sigma}$ and $\varphi_{\mu,\sigma}$ denoting the cumulative distribution function and probability density function of a $N(\mu,\sigma^2)$ distribution, respectively. In practice $\mu$ and $\sigma$ may be selected based on the data, which can be justified theoretically so long as their random values converge to some limits in probability asymptotically --- for example, in our simulation implementation of the bandlimited density estimator from Example~\ref{ex:cfdBandlimited}, we take $\mu$ and $\sigma/4$ to be the empirical mean and standard deviation of $Y$ given $A=1$, respectively.

The linear operator $D_m$ can be used to approximate the infinite-dimensional quadratic form in \eqref{eq:confSet} with a finite-dimensional quadratic form. In particular, $\langle \Omega_n(\bar{\nu}_n-h_0),\bar{\nu}_n-h_0\rangle_{\mathcal{H}}$ can be replaced by $D_m(\bar{\nu}_n-h_0)^\top\, \widetilde{\Omega}_{n,m}\, D_m(\bar{\nu}_n-h_0)$, where $\widetilde{\Omega}_{n,m}$ is an $m$-dimensional positive definite Hermitian matrix whose value will depend on the standardization operator $\Omega_n$ that it is meant to approximate. If $\Omega_n$ is the identity operator, then $\widetilde{\Omega}_{n,m}$ can be taken to be equal to the $m$-dimensional identity matrix $I_m$. If $\Omega_n$ is instead the estimator of the regularized covariance operator described in Appendix~\ref{app:regOmega0}, then it can instead be approximated by a regularized empirical covariance matrix. In particular, we can let $\widetilde{\Omega}_{n,m}:=[(1-\lambda)\Sigma_{n,m}+\lambda I_m]^{-1}$, where $\Sigma_{n,m}:=\frac{1}{2}\sum_{j=1}^2 P_n^j [D_m(\phi_n^j(\cdot))D_m(\phi_n^j(\cdot))^\top]$, where $P_n^j [D_m(\phi_n^j(\cdot))D_m(\phi_n^j(\cdot))^\top]$ corresponds to the empirical covariance matrix of the random variable $D_m(\phi_n^j(Z))$ computed using the empirical distribution $P_n^j$. Alternatively, $P_n^j [D_m(\phi_n^j(\cdot))D_m(\phi_n^j(\cdot))^\top]$ may be replaced by the empirical correlation matrix of $D_m(\phi_n^j(Z))$ under $P_n^j$ in the definition of $\widetilde{\Omega}_{n,m}$. Though using an empirical correlation matrix rather than an empirical covariance matrix changes the quadratic form used to define the confidence set,  doing so can make selecting the parameter $\lambda$ simpler because, in that case, the matrices $\Sigma_{n,m}$ and $I_m$ are on the same scale in the sense that both have trace $m$.

We conclude by noting that, when $\mathcal{H}$ is an RKHS on $\mathcal{T}$ with feature map $t\mapsto K_t$, it will be possible to compute the quadratic form $\langle \Omega_n(\bar{\nu}_n-h_0),\bar{\nu}_n-h_0\rangle_{\mathcal{H}}$ explicitly in some cases. One particularly interesting case occurs when $\bar{\nu}_n$ and $h_0$ are both contained in the linear span of $\{K_{t_k}\}_{k=1}^m$ and $\Omega_n$ is the identity operator, where the set $\{K_{t_k}\}_{k=1}^m$ may depend on the observed data. In such cases, $\bar{\nu}_n-h_0=\sum_{k=1}^m c_k K_{t_k}$ for some $c:=(c_k)_{k=1}^m\in\mathbb{R}^m$, and so, letting $G$ denote the Gram matrix with $G_{jk}=K_{t_j}(t_k)$, it holds that $\langle \bar{\nu}_n-h_0,\bar{\nu}_n-h_0\rangle_{\mathcal{H}}= c^\top G c$. If $h_0=0$, which would be the key value of $h_0$ to consider when the confidence set is being used to test the null hypothesis that $\nu(P_0)=0$ against the complementary alternative, it is necessarily the case that $h_0$ is in the linear span of $\{K_{t_k}\}_{k=1}^m$ for any collection $\{t_k\}_{k=1}^m$. Hence, in these cases, it suffices that $\bar{\nu}_n$ be in the linear span of $\{K_{t_k}\}_{k=1}^m$ for some $m$. Such cases arise, for example, when using the MMD to test for the equality of two distributions \citep{gretton2012kernel}; when conducting these tests, $\bar{\nu}_n$ is equal to the difference of the one-step estimators of the kernel mean embeddings of two distributions.

\section{Cross-validated selection of the regularization parameter}\label{app:CV}

The key observation that motivates the risk we use is that, along any quadratic mean differentiable submodel $\{P_\epsilon: \epsilon\in [0,\delta)\}\in \mathscr{P}(P,\mathcal{P},s)$, $\nu(P_\epsilon)-\nu(P)$ should approximately be equal to $\epsilon \dot{\nu}_P(s)$ or, put another way, should approximately be equal to
\begin{align*}
\argmin_{h\in\mathcal{H}}\left\|h - \epsilon\dot{\nu}_P(s)\right\|_{\mathcal{H}}^2 &= \argmin_{h\in\mathcal{H}}\left[\frac{1}{2}\|h\|_{\mathcal{H}}^2 - \epsilon\langle h, \dot{\nu}_P(s)\rangle_{\mathcal{H}}\right] \\
&= \argmin_{h\in\mathcal{H}}\left[\frac{1}{2}\|h\|_{\mathcal{H}}^2 - \epsilon\langle \dot{\nu}_P^*(h), s\rangle_{L^2(P)}\right].
\end{align*}
By Lemma~\ref{lem:LinearQMDRemainder} in Appendix~\ref{app:lem}, $\epsilon\langle \dot{\nu}_P^*(h), s\rangle_{L^2(P)}= P_\epsilon \dot{\nu}_P^*(h)+o(\epsilon)$ under appropriate conditions. This suggests that $\nu(P_\epsilon)-\nu(P)$ should approximately equal $\argmin_{h\in\mathcal{H}}[\frac{1}{2}\|h\|_{\mathcal{H}}^2 - P_\epsilon \dot{\nu}_P^*(h)]$. Letting $P_0$ play the role of $P_\epsilon$ and an estimate $\widehat{P}_n^{\mathrm{loss}}$ of $P_0$ play the role of $P$ suggests that $\nu(P_0)-\nu(\widehat{P}_n^{\mathrm{loss}})$ should approximately minimize $\frac{1}{2}\|h\|_{\mathcal{H}}^2 - P_0 \dot{\nu}_{\widehat{P}_n^{\mathrm{loss}}}(h)$ over $h\in \mathcal{H}$. Put another way, $\nu(P_0)$ should approximately minimize $E_0[\mathscr{L}_{\widehat{P}_n^{\mathrm{loss}}}(Z;h)]=\frac{1}{2}\|h-\nu(\widehat{P}_n^{\mathrm{loss}})\|_{\mathcal{H}}^2 - P_0 \dot{\nu}_{\widehat{P}_n^{\mathrm{loss}}}[h-\nu(\widehat{P}_n^{\mathrm{loss}})]$ over $h\in\mathcal{H}$. This suggests using the loss $\mathscr{L}_{\widehat{P}_n^{\mathrm{loss}}}$ when performing cross-validation to select the regularization parameter $\beta_n$. Such an approach is presented in Algorithm~\ref{alg:cv}. We refer the reader to \cite{van2003unifiedCV} for arguments that can be used to establish oracle guarantees for this cross-validation selector.

\begin{algorithm}[tb]
   \caption{Cross-validated selection of the regularization parameter $\beta_n$}
   \label{alg:cv}
   \linespread{1.5}\selectfont
\begin{algorithmic}[1]
   \STATE \textbf{Inputs:} Data $Z_1,Z_2,\ldots,Z_n$, estimator to be used to estimate the nuisance $P_0$ and a finite subset $B_n$ of $\ell_{*}^2$ of candidate values for the regularization parameter
   \STATE \textbf{Generate folds:} partition the multiset $\{Z_i\}_{i=1}^n$ into multisets $\mathcal{Z}_1,\mathcal{Z}_2,\mathcal{Z}_3,\mathcal{Z}_4$ of roughly equal size
   \FOR{all folds $j=1,2,3,4$}
    \STATE \textbf{Nuisance estimation:} using only data in $\mathcal{Z}_j$, estimate $P_0$ as $\widehat{P}_n^j$.
   \ENDFOR
   \FOR{all permutations $j=(j(1),j(2),j(3),j(4))$ of $\{1,2,3,4\}$}
        \STATE \textbf{Nuisance for regularized one-step:} let $\hat{P}_n^{\mathrm{os}}=\hat{P}_n^{j(1)}$
        \STATE \textbf{Nuisance for loss function:} let $\hat{P}_n^{\mathrm{loss}}=\hat{P}_n^{j(2)}$
        \FOR{all candidate regularization parameters $\beta:=(\beta_k)_{k=1}^\infty\in B_n$}
            \STATE \textbf{Define regularized one-step estimator:} $\widetilde{\nu}_{n,j}^\beta:=\nu(\hat{P}_n^{\mathrm{os}}) + \frac{1}{|\mathcal{Z}_{j(3)}|}\sum_{z\in\mathcal{Z}_{j(3)}} \phi_{\hat{P}_n^{\mathrm{os}}}^\beta(z)$
            \STATE \textbf{Compute risk $(j,\beta)$-specific risk:} $R_j^\beta:= \frac{1}{|\mathcal{Z}_{j(4)}|}\sum_{z\in\mathcal{Z}_{j(4)}} \mathcal{L}_{\hat{P}_n^{\mathrm{loss}}}(z;\widetilde{\nu}_{n,j}^\beta)$
        \ENDFOR
    \ENDFOR
    \FOR{all candidate regularization parameters $\beta:=(\beta_k)_{k=1}^\infty\in B_n$}
        \STATE \textbf{Aggregate the risks:} $R^\beta:=\frac{1}{24}\sum_{\textnormal{permutations $j$ of $\{1,2,3,4\}$}} R_j^\beta$
    \ENDFOR
    \STATE \textbf{return} $\beta^\star\in \argmin_{\beta\in B_n}R^\beta$
\end{algorithmic}
\end{algorithm}

\section{Inefficient influence operators and influence functions}\label{app:inefficient}

When $\mathcal{H}$ is finite-dimensional and the model is semiparametric at $P_0$ --- in the sense that its tangent space $\dot{\mathcal{P}}_{P_0}$ is a strict subspace of $L_0^2(P_0)$ --- there are generally many influence functions that can be used to construct a one-step estimator of $\nu(P_0)$. In our more general Hilbert-valued setting, the same can be done by replacing the efficient influence operator that we use to construct our (regularized) one-step estimators by an inefficient influence operator. Each inefficient influence operator is the Hermitian adjoint of a bounded linear extension $\dot{\nu}_{\mathrm{ext},P}$ of the local parameter $\dot{\nu}_P$ from $\dot{\mathcal{P}}_P$ to $L_0^2(P)$. All such extensions take the form $\dot{\nu}_{\mathrm{ext},P}(s)=\dot{\nu}_P(\Pi_{L_0^2(P)}[s\mid \dot{\mathcal{P}}_P]) + \dot{\xi}_P(\Pi_{L_0^2(P)}[s\mid \dot{\mathcal{P}}_P^\perp])$, where $\dot{\mathcal{P}}_P^\perp$ is the orthogonal complement of $\dot{\mathcal{P}}_P\subset L_0^2(P)$ and $\dot{\xi}_P : \dot{\mathcal{P}}_P^\perp\rightarrow\mathcal{H}$ is bounded and linear. The corresponding influence operator is the Hermitian adjoint of $\dot{\nu}_{\mathrm{ext},P}$, which takes the form $\dot{\nu}_{\mathrm{ext},P}^*=\dot{\nu}_P^* + \dot{\xi}_P^*$, where $\dot{\nu}_P^*$ is the efficient influence operator and $\dot{\xi}_P^*$ is the Hermitian adjoint of $\dot{\xi}_P$. The efficient influence operator is recovered by taking $\dot{\xi}_P$ to be the zero operator. If $\dot{\xi}_P\not=0$, which we assume hereafter, then $\dot{\nu}_{\mathrm{ext},P}^*\not=\dot{\nu}_P^*$ and we call $\dot{\nu}_{\mathrm{ext},P}^*$ an inefficient influence operator.

In some cases, $\dot{\nu}_{\mathrm{ext},P}^*$ will have an associated inefficient influence function $\phi_{\mathrm{ext},P}$. Concretely, this holds if and only if $\dot{\nu}_{\mathrm{ext},P}^*(\cdot)(z)$ is bounded and linear $P$-almost surely; in these cases, $\phi_{\mathrm{ext},P}(z)$ is the Riesz representation of this operator. The following result shows that an efficient influence function must exist for an inefficient one to exist, and also provides a means to derive the form of the EIF based on the form of an inefficient influence function. We let $\Phi$ denote the $L^2(P;\mathcal{H})$-closure of the linear span of $\{z\mapsto s(z)h : h\in\mathcal{H},s\in\dot{\mathcal{P}}_P\}$.
    \begin{lemma}[Expressing the EIF in terms of an inefficient influence function]\label{lem:ineffIF}
        If $\nu$ is pathwise differentiable at $P\in\mathcal{P}$ with inefficient influence function $\phi_{\mathrm{ext},P}\in L^2(P;\mathcal{H})$, then $\nu$ has EIF $\phi_P=\Pi_{L^2(P;\mathcal{H})}\left[\phi_{\mathrm{ext},P}\,|\, \Phi\right]\in L^2(P;\mathcal{H})$.
    \end{lemma}
    \noindent The proof of this lemma is given at the end of this appendix.

    One-step estimators can be constructed using inefficient influence operators. Beginning with cases where an inefficient influence function exists, we define a cross-fitted one-step estimator as $\bar{\nu}_{\mathrm{ext},n}:= \frac{1}{2}\sum_{j=1}^2 [\nu(\widehat{P}_n^j) + P_n^j \phi_{\mathrm{ext},n}^j]$, where $\phi_{\mathrm{ext},n}^j:=\phi_{\mathrm{ext},\widehat{P}_n^j}$, where here and in the following we use notation from Section \ref{sec:cf}. Under similar conditions to those of Theorem~\ref{thm:al}, it can be shown that $\bar{\nu}_{\mathrm{ext},n}$ is regular and asymptotically linear with influence function $\phi_{\mathrm{ext},0}:=\phi_{\mathrm{ext},P_0}$ and
    \begin{align*}
    n^{1/2}\left[\bar{\nu}_{\mathrm{ext},n} - \nu(P_0)\right]&\rightsquigarrow \mathbb{H}_{\mathrm{ext}}, 
    \end{align*}
    where $\mathbb{H}_{\mathrm{ext}}$ is a tight $\mathcal{H}$-valued Gaussian random variable that is such that, for each $h\in\mathcal{H}$, $\langle\mathbb{H}_{\mathrm{ext}} ,h \rangle_{\mathcal{H}}\sim N(0,E_0[\langle \phi_{\mathrm{ext},0}(Z),h \rangle_{\mathcal{H}}^2])$. The above weak convergence facilitates the construction of confidence sets for $\nu(P_0)$ using analogous methods to those used in Section~\ref{sec:confSets}. The main distinction between $\bar{\nu}_{\mathrm{ext},n}$ and $\bar{\nu}_n$ is that, since $\phi_{\mathrm{ext},0}$ is not the EIF, the conditions of the convolution theorem fail to hold \citep[Theorem~3.11.2 and Lemma~3.11.4 of][]{van1996weak}, and so $\bar{\nu}_{\mathrm{ext},n}$ will not be efficient --- e.g., \eqref{eq:convolution} will not generally hold.

    Moving now to cases where an inefficient influence function does not exist, we define a $\beta_n$-regularized one-step estimator as $\bar{\nu}_{\mathrm{ext},n}^{\beta_n}:= \frac{1}{2}\sum_{j=1}^2 [\nu(\widehat{P}_n^j) + P_n^j \phi_{\mathrm{ext},n}^{j,\beta_n}]$ with $\phi_{\mathrm{ext},n}^{j,\beta_n}(z):=\sum_{k=1}^\infty \beta_{n,k} \dot{\nu}_{\mathrm{ext},\widehat{P}_n^j}^*(h_k)(z) h_k$. This regularized one-step estimator satisfies similar guarantees to those satisfied by the one based on the efficient influence operator: it achieves a $\|\beta_n\|_{\ell^2}/n^{1/2}$-rate of convergence when a drift, regularized remainder, and bias terms are small; the drift term will be small if $\phi_{\mathrm{ext},n}^{j,\beta_n}$ is close to $\phi_{\mathrm{ext},0}^{\beta_n}(z):=\sum_{k=1}^\infty \beta_{n,k} \dot{\nu}_{\mathrm{ext},P_0}^*(h_k)(z) h_k$ in $L^2(P_0;\mathcal{H})$; the remainder will be small if $\max_j \sup_{k\in\mathbb{N}}|\langle\nu(\widehat{P}_n^j) - \nu(P_0),h_k\rangle_{\mathcal{H}} + P_0 \dot{\nu}_{\mathrm{ext},\widehat{P}_n^j}^*(h_k)|$ is $o_p(n^{-1/2})$; and the bias term will be small if $\nu(P)$ is sufficiently smooth for all $P\in\mathcal{P}$. As for confidence sets, the same methods as described in Section~\ref{sec:confSetsRegularized} can be used once one notes that, for any fixed $\beta\in \ell^2\cap (0,1]^{\mathbb{N}}$, $\nu^\beta:=\Gamma_\beta\circ \nu$ is pathwise differentiable at $P_0$ with inefficient influence function $\phi_{\mathrm{ext},0}^\beta$. This can be used to justify, for example, the asymptotic validity of the $(1-\alpha)$ confidence set
    \begin{align*}
    &\left\{h\in\mathcal{H} : {\textstyle\sum_{k=1}^\infty} \beta_k^2 \left[{\textstyle\frac{1}{2}\sum_{j=1}^2}\left\{\langle \nu(\widehat{P}_n^j),h_k\rangle_{\mathcal{H}} + P_n^j \dot{\nu}_{\mathrm{ext},\widehat{P}_n^j}^{*}(h_k)\right\} - \langle h,h_k\rangle_{\mathcal{H}}\right]^2\le \widehat{\zeta}_{\mathrm{ext},n}/n\right\},
    \end{align*}
    where $\widehat{\zeta}_{\mathrm{ext},n}$ is selected via the bootstrap. Since $\dot{\nu}_{\mathrm{ext},\widehat{P}_n^j}^{*}$ is an inefficient influence operator, the threshold $\widehat{\zeta}_{\mathrm{ext},n}$ will generally be asymptotically larger than the one used for the confidence set built based on the efficient influence operator given in \eqref{eq:ellipticalCS} \citep[see Lemma~3.11.4 of][]{van1996weak}.

    \begin{proof}[Proof of Lemma~\ref{lem:ineffIF}]
        Denote by $\dot{\nu}_{\mathrm{ext},P}^*$ the inefficient influence operator to which $\phi_{\mathrm{ext},P}$ corresponds, and let its Hermitian adjoint $\dot{\nu}_{\mathrm{ext},P}$ denote the extension of the local parameter $\dot{\nu}_P$ used to define this inefficient influence operator. Throughout this proof we let $\phi_P^\diamond:=\Pi_{L^2(P;\mathcal{H})}[\phi_{\mathrm{ext},P}\,|\, \Phi]\in L^2(P;\mathcal{H})$. Our goal is to show that $\nu$ has EIF $\phi_P=\phi_P^\diamond$.
        
        Since we have assumed throughout that a separable version of the efficient influence process is used, there exists a countable dense subset $\mathcal{H}'$ of $\mathcal{H}$ and a $P$-probability one subset $\mathcal{Z}'$ of $\mathcal{Z}$ such that, for all $h\in\mathcal{H}$ and $z\in\mathcal{Z}'$, there exists an $\mathcal{H}'$-valued sequence $(h_j')_{j=1}^\infty$ that converges to $h$ and satisfies $\dot{\nu}_P^*(h_j')(z)\rightarrow \dot{\nu}_P^*(h)(z)$ as $j\rightarrow\infty$. Fix $\epsilon>0$ and $z\in\mathcal{Z}'$. Let $h_\epsilon$ be such that
        \begin{align*}
            &\left|\dot{\nu}_P^*(h_\epsilon)(z)-\left\langle h_\epsilon,\phi_P^\diamond(z)\right\rangle_{\mathcal{H}}\right|\ge \sup_{h\in\mathcal{H}}\left|\dot{\nu}_P^*(h)(z)-\left\langle h,\phi_P^\diamond(z)\right\rangle_{\mathcal{H}}\right| - \epsilon.
        \end{align*}
        Fix an $\mathcal{H}'$-valued sequence $(h_{\epsilon,j})_{j=1}^\infty$ that converges to $h_\epsilon$ and is such that $\dot{\nu}_P^*(h_{\epsilon,j})(z)\rightarrow \dot{\nu}_P^*(h_\epsilon)(z)$ as $j\rightarrow\infty$. By the choice of this sequence and the fact that $\langle \,\cdot\,,\phi_P^\diamond(z)\rangle_{\mathcal{H}}$ is continuous, there exists a large enough $j$ such that
        \begin{align*}
            &\left|\dot{\nu}_P^*(h_{\epsilon,j})(z)-\left\langle h_{\epsilon,j},\phi_P^\diamond(z)\right\rangle_{\mathcal{H}}\right| \ge \sup_{h\in\mathcal{H}}\left|\dot{\nu}_P^*(h)(z)-\left\langle h,\phi_P^\diamond(z)\right\rangle_{\mathcal{H}}\right| - 2\epsilon.
        \end{align*}
        Taking a supremum over $h_{\epsilon,j}\in\mathcal{H}'$ on the left and then recalling that $\epsilon>0$ was arbitrary shows that
        \begin{align*}
            &\sup_{h\in\mathcal{H}'}\left|\dot{\nu}_P^*(h)(z)-\left\langle h,\phi_P^\diamond(z)\right\rangle_{\mathcal{H}}\right| = \sup_{h\in\mathcal{H}}\left|\dot{\nu}_P^*(h)(z)-\left\langle h,\phi_P^\diamond(z)\right\rangle_{\mathcal{H}}\right|.
        \end{align*}
        Now, for each $h\in\mathcal{H}'$, let $\mathcal{Z}_{h}'':=\{z\in\mathcal{Z} : \dot{\nu}_P^*(h)(z)=\langle h,\phi_P^\diamond(z)\rangle_{\mathcal{H}}\}$. In the remainder we shall show that $\mathcal{Z}_h''$ has $P$-probability one. Combining this with the above display then shows that $\nu$ has EIF $\phi_P=\phi_P^\diamond$, since that will show that the right-hand side above is $0$ on the $P$-probability one set $\mathcal{Z}'\cap[\cap_{h\in\mathcal{H}'}\mathcal{Z}_h'']$.

        Fix $h\in\mathcal{H}'$ and let $\langle h,\phi_P^\diamond\rangle_{\mathcal{H}}$ denote the function $z\mapsto \langle h,\phi_P^\diamond(z)\rangle_{\mathcal{H}}$. We will show that $ \langle\dot{\nu}_P^*(h)-\langle h,\phi_P^\diamond\rangle_{\mathcal{H}},u\rangle_{L^2(P)}=0$ for a generic $u\in L^2(P)$, which will then establish that $\mathcal{Z}_h''$ is a $P$-probability one set and complete our proof. Writing $u=s+s^\perp$ with $s\in\dot{\mathcal{P}}_P$ and $s^\perp$ belonging to the orthogonal complement of $\dot{\mathcal{P}}_P$ in $L^2(P)$, it suffices to show that $ \langle\dot{\nu}_P^*(h)-\langle h,\phi_P^\diamond\rangle_{\mathcal{H}},s\rangle_{L^2(P)}=0$ and $ \langle\dot{\nu}_P^*(h)-\langle h,\phi_P^\diamond\rangle_{\mathcal{H}},s^\perp\rangle_{L^2(P)}=0$. Beginning with the former equality and using that the restriction of $\dot{\nu}_{\mathrm{ext},P}$ to $\dot{\mathcal{P}}_P$ is equal to $\dot{\nu}_P$,
        \begin{align*}
           \left\langle\dot{\nu}_P^*(h),s\right\rangle_{L^2(P)}&= \left\langle h,\dot{\nu}_P(s)\right\rangle_{\mathcal{H}}= \left\langle h,\dot{\nu}_{\mathrm{ext},P}(s)\right\rangle_{\mathcal{H}} = \left\langle \dot{\nu}_{\mathrm{ext},P}^*(h),s\right\rangle_{L^2(P)} \\
           &= \int \dot{\nu}_{\mathrm{ext},P}^*(h)(z)s(z) P(dz)= \int \left\langle h,\phi_{\mathrm{ext},P}(z)\right\rangle_{\mathcal{H}} s(z) P(dz) \\
           &= \int \left\langle s(z)h,\phi_{\mathrm{ext},P}(z)\right\rangle_{\mathcal{H}} P(dz) = \left\langle z\mapsto s(z)h,\phi_{\mathrm{ext},P}\right\rangle_{L^2(P;\mathcal{H})}.
        \end{align*}
        Since $z\mapsto s(z)h\in\Phi$ and $\phi_P^\diamond$ is an orthogonal projection of $\phi_{\mathrm{ext},P}$ onto $\Phi$, the right-hand side equals $\left\langle z\mapsto s(z)h,\phi_P^\diamond\right\rangle_{L^2(P;\mathcal{H})}$. By similar calculations to those used above, this in turn equals $\langle \langle h,\phi_P^\diamond\rangle_{\mathcal{H}},s\rangle_{L^2(P)}$. Hence, $ \langle\dot{\nu}_P^*(h)-\langle h,\phi_P^\diamond\rangle_{\mathcal{H}},s\rangle_{L^2(P)}=0$. To see that $ \langle\dot{\nu}_P^*(h)-\langle h,\phi_P^\diamond\rangle_{\mathcal{H}},s^\perp\rangle_{L^2(P)}=0$, observe that $\langle\dot{\nu}_P^*(h),s^\perp\rangle_{L^2(P)}=0$ since $\dot{\nu}_P^*$ has codomain $\dot{\mathcal{P}}_P$, and
        \begin{align*}
            \langle\langle h,\phi_P^\diamond\rangle_{\mathcal{H}},s^\perp\rangle_{L^2(P)} = \langle z\mapsto s^\perp(z)h,\phi_P^\diamond\rangle_{L^2(P;\mathcal{H})}=0,
        \end{align*}
        where we have used that all elements of $\Phi\subset L^2(P;\mathcal{H})$ are orthogonal to $z\mapsto s^\perp(z)h$ by virtue of the fact that $s^\perp$ is orthogonal to $\dot{\mathcal{P}}_P\subset L^2(P)$.
    \end{proof}

\section{Additional simulation results}

\subsection{Simulation results for Example~\ref{ex:cfdBandlimited}}\label{app:simCfdBandlimited}

We evaluate the coverage of our spherical $L^2(\mathbb{R})$ confidence sets for a bandlimited counterfactual density of $Y(1)$. When doing this, we take $Q(1)$ to be the distribution of $\sigma_M S + \mu_M$, where $(\mu_1,\mu_2,\mu_3)=(-4,0,4)$, $(\sigma_1,\sigma_2,\sigma_3)=(2,2,1)$, and $M\sim \textnormal{Unif}\{1,2,3\}$ is drawn independently of the random variable $S$, which has density function $3\mathrm{sinc}^4(\cdot)/(2\pi)$. The density of $Q(1)$ is depicted in the top left corner of Figure~\ref{fig:onestepIllustration}. It is bandlimited, with the support of its Fourier transform equal to $[-2,2]$. We focus on the case where the bandlimiting parameter $b$ used to define $\underline{\nu}$ in Eq.~\ref{eq:underlineNuDef} is equal to $2$, so that the target of inference $\underline{\nu}(P_0)$ coincides with the density of $Q(1)$.

Figure~\ref{fig:bandlimitedCoverage} displays the coverage of our confidence sets at different nominal levels. At all nominal levels larger than 75\%, the confidence sets are slightly conservative at the sample sizes considered, with the actual coverage probability approaching the nominal level as $n$ grows. A similar improvement with $n$ holds across the full 0-100\% range of nominal levels, which both supports our theoretical weak convergence guarantees for the one-step estimator and their finite-sample utility. In Table~\ref{tab:bandlimitedMISE}, we also verified that, as anticipated by our theory for cases where an EIF exists, the mean integrated squared error of the one-step estimator decays at an $n^{-1}$ rate.

We conclude by comparing the size our $L^2(\mathbb{R})$ confidence set $\mathcal{C}_n$ to those of a pointwise confidence interval for the counterfactual density function at zero, namely $\underline{\nu}(P_0)(0)$. To make this comparison, we first note that, for any $y\in\mathbb{R}$, including $y=0$, and any $h\in\mathcal{C}_n:=\{h\in\mathcal{H} : \|h-\underline{\bar{\nu}}_n\|_{L^2(\lambda_Y)}^2\le \widehat{\zeta}_n/n\}$,
\begin{align}
    |h(y) - \underline{\bar{\nu}}_n(y)|&= \left|\int_{-\infty}^\infty \underline{K}_y(\tilde{y})\,[h(\tilde{y})-\underline{\bar{\nu}}_n(y)]\, \lambda_Y(d\tilde{y})\right| \nonumber \\
    &\le \left\|\underline{K}_y\right\|_{L^2(\lambda_Y)}\left\|h-\underline{\bar{\nu}}_n\right\|_{L^2(\lambda_Y)}\le (b\widehat{\zeta}_n/[n\pi])^{1/2}, \label{eq:bandlimitedUniform}
\end{align}
where we used that $\left\|\underline{K}_y\right\|_{L^2(\lambda_Y)}=(b/\pi)^{1/2}$. Consequently, our $L^2(\mathbb{R})$ confidence set yields an interval for $\underline{\nu}(P_0)(0)$ of the form $\underline{\bar{\nu}}_n(0)\pm (b\widehat{\zeta}_n/[n\pi])^{1/2}$. In our simulation setting, this confidence interval was about $2.3$ times wider than an efficient Wald-type confidence interval for the real-valued quantity $\underline{\nu}(P_0)(0)$ when $\alpha=0.05$. Hence, if a point evaluation of $\underline{\nu}(P_0)$ is truly the target of inference, then there would be a benefit to directly pursuing inference for this quantity, rather than the function as a whole. However, if the function $\underline{\nu}(P_0)$ is the target of inference, then our $L^2(\mathbb{R})$ confidence set is likely the preferred method for making inference. This may be especially true in this bandlimited density example since the fact that \eqref{eq:bandlimitedUniform} holds for all $y\in\mathbb{R}$ shows that a uniform confidence band for $\underline{\nu}(P_0)$ is given by $\mathcal{C}_{n,\infty}:= \{h\in\mathcal{H} : \sup_{y\in\mathbb{R}}|h(y)-\bar{\nu}_n(y)|\le (b\widehat{\zeta}_n/[n\pi])^{1/2}\}$.

\begin{figure}[tb]
    \centering
    \includegraphics[width=0.5\textwidth]{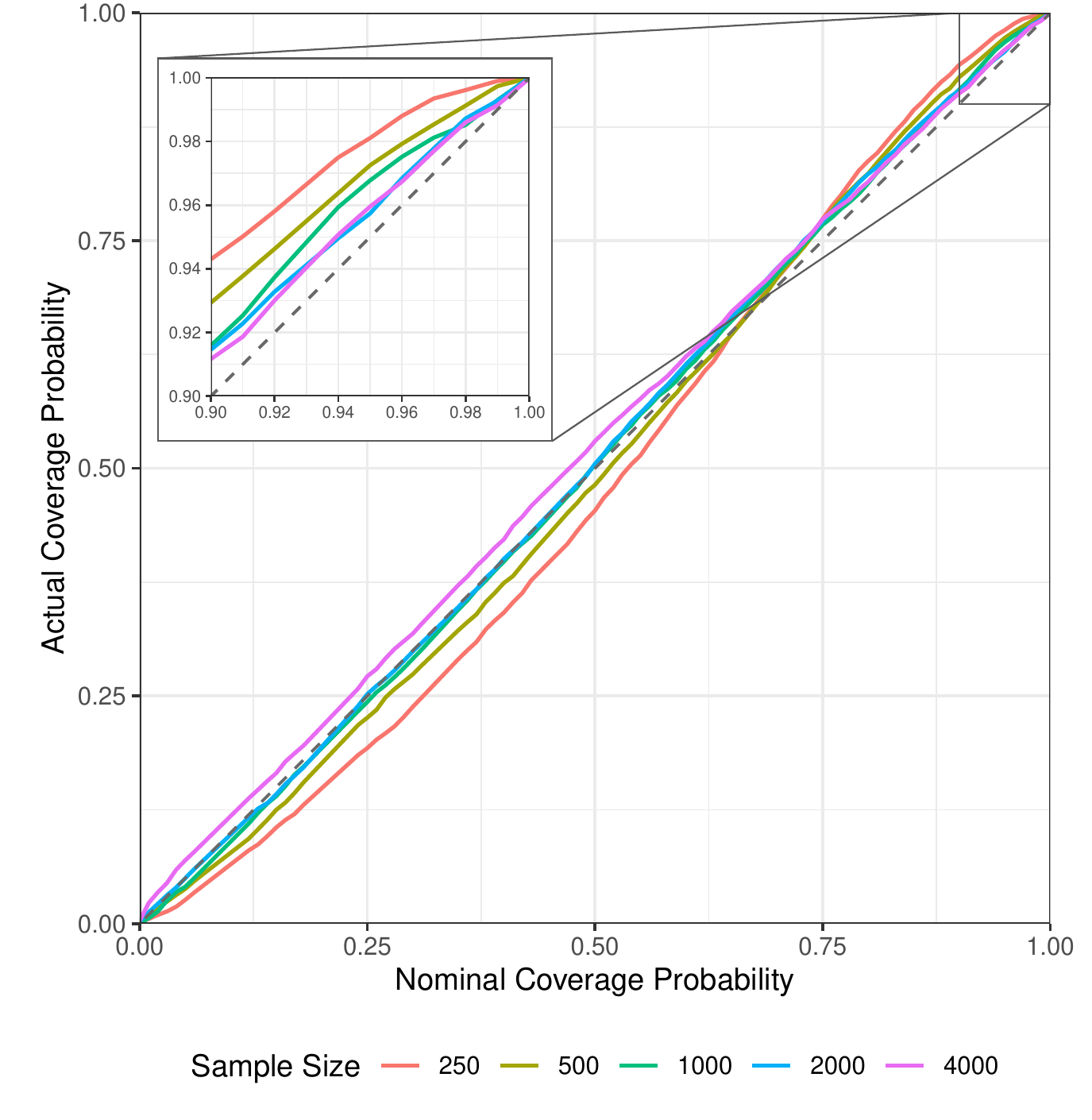}
    \caption{Actual versus nominal coverage of confidence sets for the bandlimited density function based on 5000 Monte Carlo repetitions. The inset displays nominal coverage values that are of particular interest in practice.}
    \label{fig:bandlimitedCoverage}
\end{figure}

\begin{table}[tb]
    \centering
\begin{tabular}{rrrrrr}
\hline
& \multicolumn{5}{c}{Sample Size ($n$)} \\
 & 250 & 500 & 1000 & 2000 & 4000 \\\hline
  Plug-In & 1.11 & 1.59 & 2.22 & 3.04 & 4.15 \\ 
  One-Step & 3.58 & 3.30 & 3.11 & 2.97 & 2.84 \\ 
   \hline
\end{tabular}
    \caption{Performance of the plug-in and one-step estimators at different sample sizes $n$, where performance is measured in terms of $n$ times the mean integrated squared error. As would be predicted by theory, this criterion appears to stabilize with $n$ for the one-step estimator. In contrast, it grows with $n$ for the plug-in estimator.}
    \label{tab:bandlimitedMISE}
\end{table}

\subsection{Supplemental tables and figures for simulation studies}

\begin{figure}[ht]
    \centering
    \includegraphics[width=.5\textwidth]{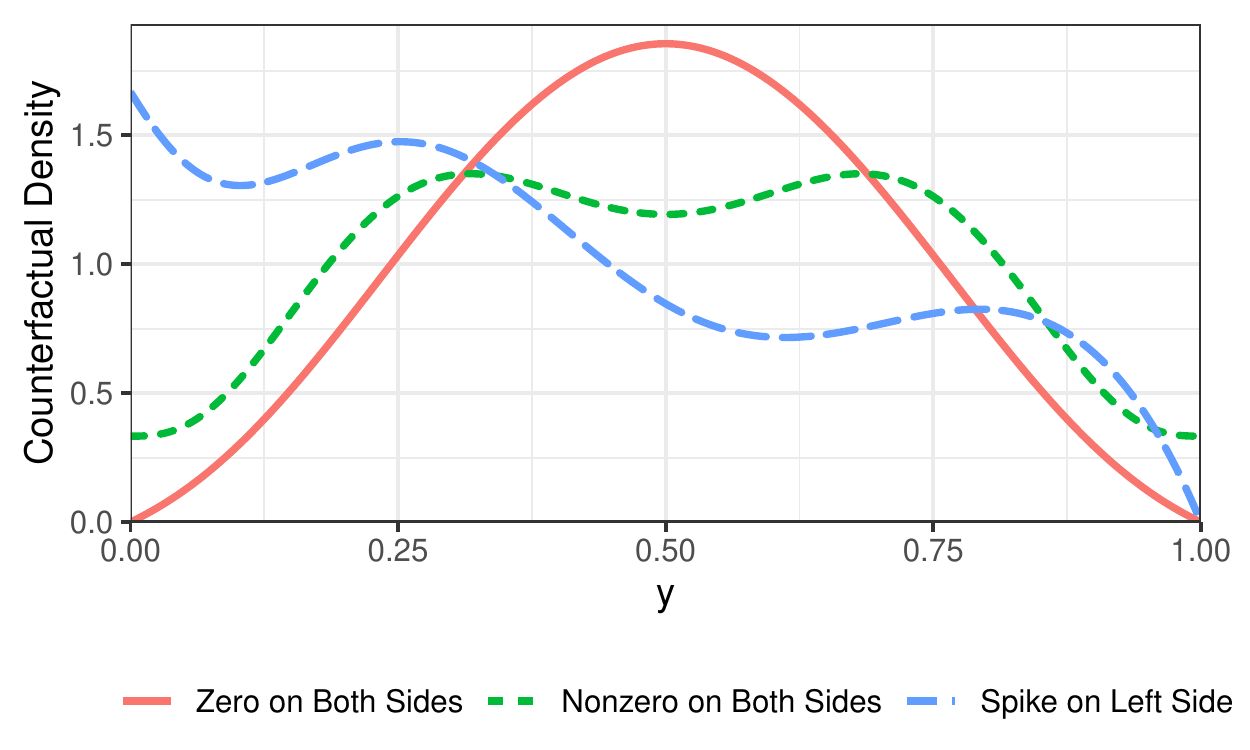}
    \caption{Densities $Q(1)$ of $Y(1)$ used in the three settings considered in the evaluation of estimators in Example~\ref{ex:cfdNonparametric}. Each of these densities is a uniform mixture $\frac{1}{3}\sum_{k=1}^3 \mathrm{Beta}(c_k,d_k)$, where the parameters indexing the beta distributions in this mixture differ across the three settings plotted in the figure. In particular, $\{(c_1,d_1),(c_2,d_2),(c_3,d_3)\}$ is equal to $\{(2,2),(3,3),(4,4)\}$ for `Zero on Both Sides', $\{(1,1),(8,4),(4,8)\}$ for `Nonzero on Both Sides', and $\{(1,5),(5,2),(4,8)\}$ for `Spike on Left Side'. When evaluating the mean integrated squared error performance of estimators of the density of $Y(1)$, $Q(0)$ is set equal to $Q(1)$.}
    \label{fig:Y1dens}
\end{figure}

\begin{figure}[ht]
    \centering
    \begin{subfigure}{.8\textwidth}
    \centering
    \includegraphics[width=\textwidth]{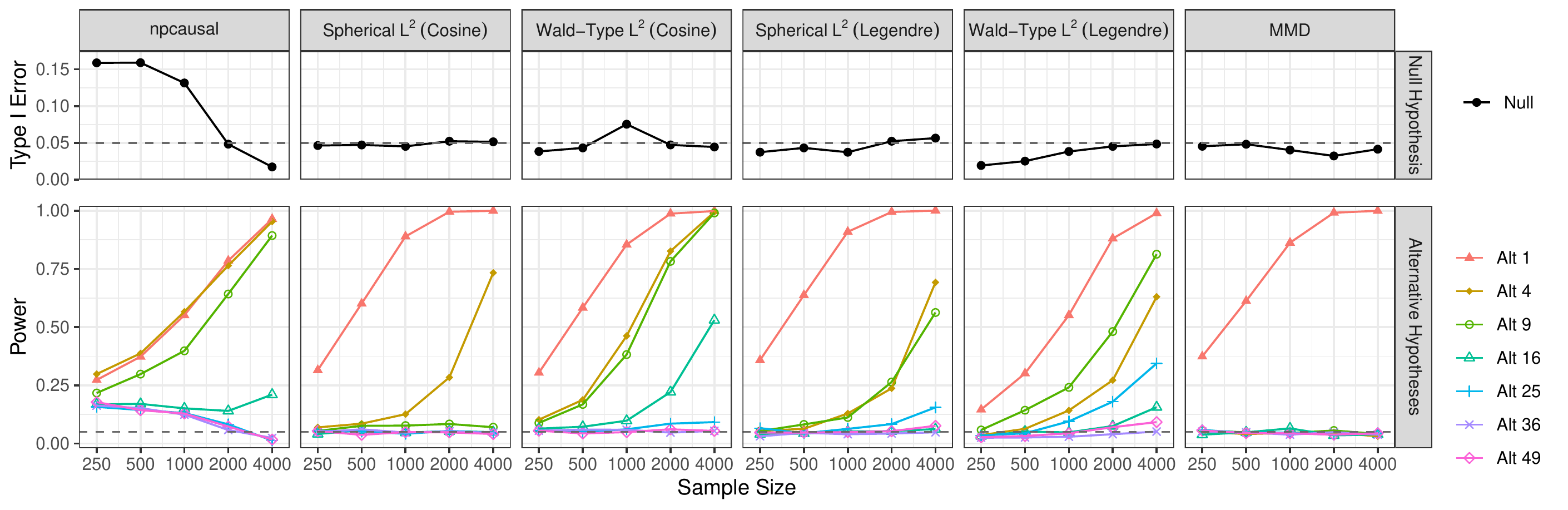}
    \caption{Parameter choices leading to improved power against smoother alternatives: $c=2.5$ and $s=2$.}
    \end{subfigure}
    \par\bigskip
    \begin{subfigure}{.8\textwidth}
    \centering
    \includegraphics[width=\textwidth]{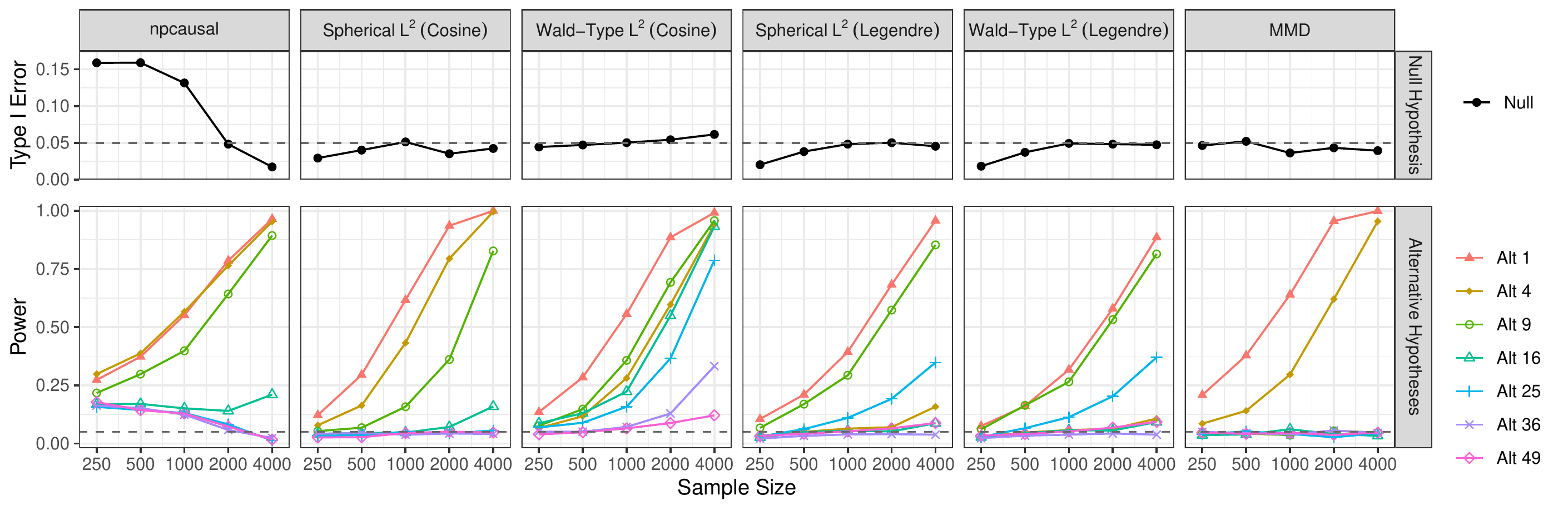}
    \caption{Parameter choices leading to improved power against rougher alternatives: $c=10$ and $s=0.5$.}
    \end{subfigure}
    \caption{Same as Figure~\ref{fig:test_rejProb_medium}, but at different choices of the tuning parameters indexing the tests. The tests based on Example~\ref{ex:cfdNonparametric} set the regularization parameter so that $\beta_k=1/[1+(k/c)^2]$ for a constant $c$ and the MMD test selects a bandwidth equal to some constant $s$ times $\textnormal{median}\{Y_1,\ldots,Y_n\}$. The same npcausal panel from Figure~\ref{fig:test_rejProb_medium} is shown as a benchmark in both subfigures.}
    \label{fig:test_rejProb_smallLarge}
\end{figure}

\end{document}